\documentclass{article}

\usepackage{arxiv}

\usepackage[utf8]{inputenc} % allow utf-8 input
\usepackage[T1]{fontenc}    % use 8-bit T1 fonts
\usepackage{hyperref}       % hyperlinks
\usepackage{url}            % simple URL typesetting
\usepackage{booktabs}       % professional-quality tables
\usepackage{amsfonts}       % blackboard math symbols
\usepackage{nicefrac}       % compact symbols for 1/2, etc.
\usepackage{microtype}      % microtypography
\usepackage{lipsum}		% Can be removed after putting your text content
\usepackage{graphicx}
\usepackage{natbib}
\usepackage{doi}
\usepackage{amsmath}
\usepackage{amsthm}
\usepackage{bbm}
\usepackage{caption}
\usepackage{subcaption}

\DeclareMathOperator*{\argmin}{arg\,min}
\usepackage{multirow}
\usepackage{longtable}
\usepackage{float}

\title{Controlling the low-temperature Ising model using spatiotemporal Markov decision theory}

\date{January 7, 2025}	% Here you can change the date presented in the paper title
\author{ \href{https://orcid.org/0000-0000-0000-0000}{\includegraphics[scale=0.06]{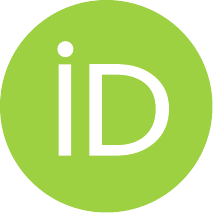}\hspace{1mm}M.C. de Jongh} \\
	Department of Applied Mathematics\\
	University of Twente\\
	P.O. Box 217, NL-7500 AE Enschede \\
	\texttt{m.c.dejongh@utwente.nl} \\
	\And
	\href{https://orcid.org/0000-0000-0000-0000}{\includegraphics[scale=0.06]{orcid.pdf}\hspace{1mm}Richard J. Boucherie} \\
	Department of Applied Mathematics\\
	University of Twente\\
	   P.O. Box 217, NL-7500 AE Enschede \\
	\texttt{r.j.boucherie@utwente.nl} \\
	\AND
    \href{https://orcid.org/0000-0000-0000-0000}{\includegraphics[scale=0.06]{orcid.pdf}\hspace{1mm}M.N.M. van Lieshout}\thanks{M.N.M. van Lieshout is also affiliated with the University of Twente} \\
	Stochastics Department\\
	Centrum Wiskunde \& Informatica\\
	P.O. Box 94079, NL-1090 GB Amsterdam    \\
	\texttt{Marie-Colette.van.Lieshout@cwi.nl} \\
}

\renewcommand{\shorttitle}{Controlling the low-temperature Ising model using STMDPs}

\hypersetup{
pdftitle={A template for the arxiv style},
pdfsubject={q-bio.NC, q-bio.QM},
pdfauthor={David S.~Hippocampus, Elias D.~Striatum},
pdfkeywords={First keyword, Second keyword, More},
}

\allowdisplaybreaks

\begin{document}

\theoremstyle{plain}
\newtheorem{axiom}{Axiom}
\newtheorem{claim}[axiom]{Claim}
\newtheorem{theorem}{Theorem}[section]
\newtheorem{lemma}[theorem]{Lemma}
\newtheorem{corollary}[theorem]{Corollary}

\theoremstyle{remark}
\newtheorem{definition}[theorem]{Definition}
\newtheorem*{example}{Example}
\newtheorem*{fact}{Fact}
\newtheorem{assumption}{Assumption}
\newtheorem{remark}{Remark}

\maketitle

\begin{abstract}
	We introduce the spatiotemporal Markov decision process (STMDP), a special type of Markov decision process that models sequential decision-making problems which are not only characterized by temporal, but also by spatial interaction structures.
To illustrate the framework, we construct an STMDP inspired by the low-temperature two-dimensional Ising model on a finite, square lattice, evolving according to the Metropolis dynamics. We consider the situation in which an external decision maker aims to drive the system towards the all-plus configuration by flipping spins at specified moments in time. In order to analyze this problem, we construct an auxiliary MDP by means of a reduction of the configuration space to the local minima of the Hamiltonian. Leveraging the convenient form of this auxiliary MDP, we uncover the structure of the optimal policy by solving the Bellman equations in a recursive manner. Finally, we conduct a numerical study on the performance of the optimal policy obtained from the auxiliary MDP in the original Ising STMDP. 
\end{abstract}

% keywords can be removed

\section{Introduction}
\label{Introduction}
The \textit{Markov decision process} (MDP) is a well-established framework for modeling and solving problems that involve sequential decision-making under uncertainty \citep{Puterman}. Various domains, such as epidemic management \citep{Diaz, Palopoli}, wild fire prevention \citep{Altamimi, Roozbeh} and agricultural economics \citep{Chi, Swinton}, face decision-making problems that feature not only temporal but also spatial structures. If the state space, and possibly the action space, can be factored, such problems can be covered by the \textit{Factored MDP} (FMDP) \citep{Boutilier} or \textit{graph-based MDP} (GMDP) \citep{Sabbadin1} frameworks. However, a framework for problems that do not obey such a factorisation assumption is lacking in the literature. In this article, we introduce the \textit{spatiotemporal Markov decision process} (STMDP) as a special type of MDP in which the state variables adhere to local interaction structures that cannot necessarily be written in factored form. Unlike the FMDP and GMDP models, the STMDP framework includes processes with asynchronous dynamics, providing a higher level of flexibility. 

To illustrate the framework, we formulate an STMDP based on the two-dimensional Ising model on a finite, square lattice that evolves according to asynchronous Metropolis dynamics \citep[Section 3]{Newman}. We assume that an external decision maker has the power to flip a spin at specified points in time, with the aim to drive the process towards a predetermined target configuration.

In our analysis, we focus on the low-temperature regime, where the behaviour of the Ising model is well studied \citep{Cerf, Neves, Cirillo, Cerf2, Kotecky, Nardi, Nardi2, Cirillo2}. Specifically, if the system is initialized from the all-minus configuration and subjected to a small positive magnetic field, it will take an exponentially long time to reach the stable all-plus configuration. This phenomenon is known as \textit{metastability}. The transition to the stable state typically occurs through the formation of a critical droplet of $+$-spins, which eventually nucleates the lattice. In this article, we use the Ising STMDP to optimize the trajectory for reaching the stable state in the low-temperature regime. To analyze the STMDP, we construct an auxiliary MDP by reducing the configuration space to the local minima of the Hamiltonian and show that this MDP is an accurate approximation of the original process if the time between two actions of the decision maker is sufficiently long. Exploiting the convenient form of this auxiliary MDP, we recursively solve the Bellman equations to unravel the structure of the optimal policy to optimally speed up the nucleation process. 

The structure of this paper is as follows. In Section \ref{The STMDP}, we introduce the STMDP framework and discuss its relation to the FMDP and GMDP models. In addition, we provide some insights in the connection between the value function and first hitting times for MDPs with a reachability objective. In Section \ref{Ising STMDP}, we formulate the two-dimensional low-temperature Ising STMDP and outline the construction of the auxiliary MDP. In addition, we state our main result, which concerns the structure of the optimal policy in the auxiliary MDP. Section~\ref{Numerical results} provides several numerical experiments that give insight in the performance of the optimal policy derived for the auxiliary MDP in the original Ising STMDP and draws a comparison with two alternative policies. In Section \ref{Analysis of the low-temperature Ising STMDP sec}, we formalize our analysis of the low-temperature Ising STMDP and give an outline of the proof of the result stated in Section~\ref{Definition of the Ising STMDP}. Details of the proof are deferred to the Supplementary Material.

\section{The spatiotemporal Markov decision process}
\label{The STMDP}
\subsection{Introduction to Markov decision processes}
\label{Introduction Markov decision processes}
A Markov decision process is defined as a tuple $(S, A, P, r)$, in which the elements are specified as follows \citep{Puterman}.
\begin{itemize}
    \item The \textit{state space} $S$ contains the possible states that the system can occupy. We assume $S$ to be finite. 
    \item Let $A_s$ denote the set of possible actions that can be taken from a state $s \in S$. The \textit{action space} is defined as $A = \cup_{s \in S} A_s$. We assume the action space to be finite.
    \item The \textit{transition probability kernel} $P: S \times A \times S$ specifies the dynamics of the MDP. Here, $P(s'|s, a)$ denotes the probability that the system will make a transition to state $s' \in S$ given that it is currently in state $s \in S$ and the decision maker selected action $a \in A_s$. 
    \item The \textit{reward function} $r: S \times A \rightarrow \mathbb{R}$ specifies the immediate reward $r(s, a)$ which the decision maker receives if he selects action $a \in A_s$ from state $s \in S$. We assume that the reward function is bounded, i.e., that there is a constant $M > 0$ such that $|r(s,a)| \leq M$ for each $s \in S$, $a \in A$. 
\end{itemize}
Let $T$ denote the set of \textit{decision epochs}, i.e., the moments in time at which a decision can be made. We assume that $T = \{0, 1, 2, \ldots\}$.
The behaviour of the decision maker is governed by a \textit{policy}. We restrict our attention to policies that are \textit{stationary} and \textit{deterministic}. Such a policy applies at each decision epoch a \textit{deterministic decision rule} $d: S \rightarrow A$ that specifies the action $d(s)$ that ought to be selected in state $s \in S$. The resulting policy is written as $\pi = d^{\infty}$. Let $\Pi$ denote the set of stationary deterministic policies.  The expected total discounted reward, or the value function $v_{\lambda}^{\pi}(s)$ of a policy $\pi = d^{\infty}$ in a state $s \in S$ is defined as    
\begin{equation*}
    v^{\pi}_{\lambda}(s) = \mathbb{E}^{\pi}_s\left[\sum\limits_{t=0}^{\infty}\lambda^{t}r_t(X_t, Y_t)\right],
\end{equation*}
where $X_t$ denotes the state of the system at decision epoch $t$, $Y_t$ the selected action at decision epoch $t$ and $\lambda \in (0,1)$ is a discount factor. By conditioning on the state reached at the first decision epoch, we obtain the following set of equations for the value function $v^{\pi}_{\lambda}$:
\begin{equation*}
    v^{\pi}_{\lambda}(s) = r(s, d(s)) + \lambda \sum\limits_{s' \in S} P(s'|s, d(s))v^{\pi}_{\lambda}(s'), \quad s \in S.
\end{equation*}
In vector notation, this equation reads
\begin{equation*}
    \mathbf{v}^{\pi}_{\lambda} = \mathbf{r}^{\pi} + \lambda \mathbf{P}^{\pi}\mathbf{v}^{\pi}_{\lambda},
\end{equation*}
where $\mathbf{v}^{\pi}_{\lambda} \in \mathbb{R}^{|S|}$, $\mathbf{r}^{\pi} \in \mathbb{R}^{|S|}$ and $\mathbf{P}^{\pi} \in \mathbb{R}^{|S| \times |S|}$ denote the vector of values, the vector of rewards and the transition probability matrix specified by policy $\pi$. Now, given a stationary deterministic policy $\pi = d^{\infty}$, let the operator $\mathcal{F}^{\pi}_{\lambda}: \mathbb{R}^{|S|} \rightarrow \mathbb{R}^{|S|}$ be defined as
\begin{equation}\label{operator}
    \mathcal{F}^{\pi}_{\lambda}(\mathbf{x}) := \mathbf{r}^{\pi} + \lambda \mathbf{P}^{\pi}\mathbf{x}, \quad \mathbf{x} \in \mathbb{R}^{|S|}.
\end{equation}
Standard results from Markov decision theory assert that the vector of expected total discounted rewards $\mathbf{v}^{\pi}_{\lambda}$ is the unique fixed point of the operator $\mathcal{F}^{\pi}_{\lambda}$ \citep[p.\ 151, Thm.\ 6.2.5]{Puterman}.
A policy $\pi^*$ is \textit{optimal} if it satisfies
\begin{equation*}
    v^{\pi^*}_{\lambda}(s) \geq v^{\pi}_{\lambda}(s), \quad s \in S, 
\end{equation*}
for each $\pi \in \Pi$. The optimal value function of the MDP is now defined as
\begin{equation*}
    v^*_\lambda(s) = \sup_{\pi \in \Pi}v^{\pi}_{\lambda}(s), \quad s \in S.
\end{equation*}
Under the listed assumptions on the state space, action space and reward function, there exists a stationary, deterministic policy which is optimal under the expected total discounted reward optimality criterion \citep[p.\ 154, Thm.\ 6.2.10]{Puterman}. The optimal values and policies for infinite horizon models can be characterized by the so-called \textit{optimality equations} or \textit{Bellman equations}:
\begin{equation}\label{Bellman_equations}
    v_{\lambda}(s) = \sup_{a \in A_s}\{r(s, a) + \lambda \sum\limits_{j \in S} p(j|s, a)v_{\lambda}(j)\}.
\end{equation}
Note that the assumptions made on the state and action spaces and the reward function guarantee the attainment of the supremum. Hence, in the remainder of the paper, we will replace it by a maximum.   
The following theorem establishes the usefulness of these equations in identifying optimal policies.
\begin{theorem}{\rm\citep[p. 152]{Puterman}}\label{Bellman_equations_thm}
A policy $\pi^* \in \Pi$ is optimal if and only if $v^{\pi^*}_{\lambda}$ is a solution to the optimality equations.
\end{theorem}

\subsection{A reachability objective}
\label{A reachability objective}
We focus on decision processes with a reachability objective, i.e., with the aim to reach some target state. This section provides some relations between the value function of such an MDP and the first hitting time to the target state. 

Consider an MDP $(S, A, P, r)$. Let $\tau^{s, \pi}_{B}$ denote the \textit{first hitting time} to a set $B \subseteq S$ of the state process $X^{\pi}_t$ induced by policy $\pi \in \Pi$ starting from state $s \in S$. That is,

\begin{equation*}
\tau^{s, \pi}_{B} = \inf_{t \in \mathbb{N}}\{X^{\pi}_t \in B|X^{\pi}_0 = s\}.
\end{equation*}
Simplifying notation, we write $\tau^{s, \pi}_{\{s'\}}$ as $\tau^{s, \pi}_{s'}$ for $s, s' \in S$. The following theorem expresses the expected total discounted reward of a stationary, deterministic policy $\pi$ in terms of expected first hitting times.

\begin{theorem}
For each stationary, deterministic policy $\pi = d^{\infty}$, we have

\begin{equation}\label{value_function_hitting_time}
v^{\pi}_{\lambda}(s) = r(s, d(s)) + \sum\limits_{s' \in S} \dfrac{r(s', d(s'))\mathbb{E}[\lambda^{\tau^{s, \pi}_{s'}}]}{1- \mathbb{E}[\lambda^{\tau^{s', \pi}_{s'}}]}.
\end{equation}
\end{theorem}  
\begin{proof}
Let $\{(X^{\pi}_t, r(X^{\pi}_t, Y^{\pi}_t)); t = 0, 1, \ldots\}$ denote the stochastic process induced by policy $\pi$, where $X^{\pi}_t$ represents the state at time $t$, $Y^{\pi}_t$ the action taken at time $t$ and $r(X^{\pi}_t, Y^{\pi}_t)$ the reward obtained after taking action $Y^{\pi}_t$ from state $X^{\pi}_t$. First of all, note that for any pair of states $s, s' \in S$, conditioning on the first hitting time from $s$ to $s'$ yields
\begin{align}
    \nonumber \mathbb{E}_s \left[\sum\limits_{t=1}^{\infty} \lambda^tr(s', d(s'))\mathbbm{1}\{X^{\pi}_t = s'\}\right] &= \sum\limits_{t'=1}^{\infty}    \mathbb{E}_{s} \left[\sum\limits_{t=1}^{\infty} \lambda^tr(s', d(s'))\mathbbm{1}\{X^{\pi}_t = s'\}\Big|\tau^{s, \pi}_{s'} = t'\right] \mathbb{P}(\tau^{s, \pi}_{s'} = t') \\
    \label{hit_help} &= \mathbb{E}\left[\lambda^{\tau^{s, \pi}_{s'}}\right]\mathbb{E}_{s'}\left[\sum\limits_{t=0}^{\infty} \lambda^t r(s', d(s'))\mathbbm{1}\{X^{\pi}_t = s'\} \right].
\end{align}
For $s = s'$, it follows from expression (\ref{hit_help}) that
\begin{align*}
\mathbb{E}_{s}\left[\sum\limits_{t=0}^{\infty} \lambda^t r(s, d(s))\mathbbm{1}\{X^{\pi}_t = s\} \right] &= r(s, d(s)) + \mathbb{E}\left[\lambda^{\tau^{s, \pi}_{s}}\right]\mathbb{E}_{s}\left[\sum\limits_{t=0}^{\infty} \lambda^t r(s, d(s))\mathbbm{1}\{X^{\pi}_t = s\} \right].
\end{align*}
This yields, for each $s \in S$,
\begin{equation}\label{hit_help_2}
\mathbb{E}_{s}\left[\sum\limits_{t=0}^{\infty} \lambda^t r(s, d(s))\mathbbm{1}\{X^{\pi}_t = s\} \right] = \dfrac{r(s, d(s))}{1- \mathbb{E}[\lambda^{\tau^{s, \pi}_{s}}]}.
\end{equation}
We now obtain, for $s \in S$, 
\begin{align*}
v^{\pi}_{\lambda}(s) =  \mathbb{E}_s\left[\sum\limits_{t=0}^{\infty} \lambda^{t}r(X^{\pi}_t, Y^{\pi}_t)\right] &= r(s, d(s)) + \sum\limits_{s' \in S}   \mathbb{E}_s\left[\sum\limits_{t=1}^{\infty} \lambda^t r(s', d(s'))\mathbbm{1}\{X^{\pi}_t = s'\}\right]. 
\end{align*}
Invoking expression (\ref{hit_help}) yields
\begin{equation*}
    v^{\pi}_{\lambda}(s) = r(s, d(s)) + \sum\limits_{s' \in S} \mathbb{E}\left[\lambda^{\tau^{s, \pi}_{s'}}\right] \mathbb{E}_{s'}\left[\sum\limits_{t=0}^{\infty} \lambda^t r(s', d(s'))\mathbbm{1}\{X^{\pi}_t = s'\} \right].
\end{equation*}
Inserting (\ref{hit_help_2}) now completes the proof.
\end{proof}

\begin{corollary}\label{Value_hitting_time_indicator_reward}
Consider the reward function $r: S \rightarrow \mathbb{R}$ given by
\begin{equation}\label{indicator_reward}
r(s) = \begin{cases}
1, &\text{ if } s = s^*, \\
0, &\text{ otherwise}, 
\end{cases}
\quad s \in S,
\end{equation}
for some target state $s^* \in S$. The value function of a stationary deterministic policy $\pi = d^{\infty}$ is given by

\begin{equation*}
v^{\pi}_{\lambda}(s) = \begin{cases}
1 + \dfrac{\mathbb{E}[\lambda^{\tau^{s^*, \pi}_{s^*}}]}{1- \mathbb{E}[\lambda^{\tau^{s^*, \pi}_{s^*}}]}, &\text{ if } s = s^*, \\
\dfrac{\mathbb{E}[\lambda^{\tau^{s, \pi}_{s^*}}]}{1- \mathbb{E}[\lambda^{\tau^{s^*, \pi}_{s^*}}]}, &\text{otherwise}.
\end{cases}
\end{equation*}
\end{corollary}
\begin{proof}
    The result follows immediately from substituting the indicator reward function into expression (\ref{value_function_hitting_time}).
\end{proof}

\subsection{The spatiotemporal Markov decision process}
We introduce the \textit{spatiotemporal Markov decision process} (STMDP) as an MDP in which the state space is multidimensional and local dependencies between state variables are represented by an undirected finite graph $G(V, E)$ with vertex set $V$ and edge set $E$. Given a set $W \subseteq V$, let $N(W)$ denote the neighborhood of $W$, i.e.,
\begin{equation*}
    N(W) = \{v \in V\setminus W \big| (v,w) \in E \text{ for some } w \in W\}.
\end{equation*}
Each vertex $v \in V$ has a local state space $S_v$. A \textit{configuration} $\sigma$ is defined as a function that maps each vertex $v \in V$ to a state $\sigma(\{v\})$ in its local state space $S_v$. To simplify notation, we write $\sigma(\{v\}) = \sigma(v)$ for singletons $\{v\} \subset V$. Given a certain ordering of the vertices, we often regard a configuration as a vector of which the $i$th element specifies the state of the $i$th vertex. The global state space, or \textit{configuration space} is defined as the Cartesian product $S = \times_{v \in V} S_v$. We denote by $\sigma(W)$, where $W = \{w_1, w_2, \ldots, w_k\} \subseteq V$, the configuration on the vertices in $W$, i.e., $\sigma(W) = (\sigma(w_1), \sigma(w_2), \ldots, \sigma(w_k))$. 
The action space, the reward function and the set of decision epochs are defined in the same way as in the classic MDP setting. 

The main characteristic of the STMDP framework is the following assumption on the transition probability kernel.

\begin{assumption}
For each $W \subseteq V$ and configuration $\eta' \in \times_{v \in W} S_v$, we have
\begin{equation}\label{STprop}
    \sum\limits_{\substack{\eta \in S,\\\eta(W) = \eta'}} P(\eta | \sigma, a) = \sum\limits_{\substack{\eta \in S,\\\eta(W) = \eta'}} P(\eta | \sigma', a),
\end{equation}
for all $\sigma, \sigma' \in S$ with $\sigma(W \cup N(W)) = \sigma'(W \cup N(W))$ and $a \in A_{\sigma} \cap A_{\sigma'}$. 
\end{assumption}
%\begin{equation}
%    P(\eta(W)|\sigma, a) = P(\eta(W)|(\sigma(W \cup \mathcal{N}(W)), a).
%\end{equation}
\noindent This property, which can be considered a Markov property in space, ensures that for any action $a \in A$, the restriction of the configuration at time $t+1$ to any set $W \subseteq V$ depends on the configuration at time $t$ only through its restriction to the set $W$ and its neighborhood $N(W)$. 

\subsection{Related frameworks}
The STMDP is related to the \textit{factored} MDP (FMDP) \citep{Boutilier} and \textit{graph-based} MDP (GMDP) \citep{Sabbadin1} models. As opposed to the latter two models, however, the STMDP framework does not require the assumption that local interaction structures can be written in factored form and thus offers greater flexibility. In this section, we further discuss the similarities and differences between the three frameworks. 

In an FMDP, the dependencies between the state variables are expressed by means of a directed graph $G'(V', E')$ with vertex set $V'$ and edge set $E'$. The transition probability kernel can be written in the following factored form:
\begin{equation*}
    P(\eta|\sigma, a) = \prod\limits_{v \in V}P_v(\eta(v)|\sigma(N_{\text{in}}(v)), a), \quad \sigma, \eta \in S, a \in A,
\end{equation*}
where $P_v: \times_{w \in N_{\text{in}}(v)}S_w \times A \times S_v \rightarrow [0,1]$, $v \in V'$, are local transition probabilities and 
\begin{equation*}
    N_{\text{in}}(v) = \{w \in V'|(w, v) \in E'\}, \quad v \in V'.
\end{equation*}
Observe that the class of FMDPs defined on symmetric directed graphs (i.e., directed graphs with the property that for each directed edge, the edge in the opposite direction is present as well) in which each vertex has a self-loop is a subclass of the STMDPs. After all, it coincides with the class of STDMPs for which the transition probability kernel satisfies, in addition to the spatial Markov property given by expression (\ref{STprop}), the assumption of synchronous and independent transitions, i.e., 
\begin{equation*}
    P(\eta|\sigma, a) = \prod\limits_{v \in V} P_v(\eta(v)|\sigma(v \cup N(v)), a), \text{ for all } \sigma, \eta \in S, a \in A,
\end{equation*}
for some local transition probabilities $P_v: \times_{w \in \{v\} \cup N(v)}S_w \times A \times S_v \rightarrow [0,1]$.
A GMDP is a special type of FMDP that imposes additional structure on the action space and the reward function. Let $G'(V', E')$ again denote a directed graph. In the GMDP framework, a local action space $A_v$ is defined for each vertex $v \in V'$ and the global action space is the Cartesian product $A = \times_{v \in V'}A_v$. Similar to the FMDP model, the transition probability kernel can be written as follows:
\begin{equation*}
    P(\eta|\sigma, a) = \prod\limits_{v \in V}P_v(\eta(v)|\sigma(N_{\text{in}}(v)), a(v)), \quad \sigma, \eta \in S, a \in A,
\end{equation*}
where $P_v: \times_{w \in N_{\text{in}}(v)}S_w \times A_v \times S_v \rightarrow [0,1]$, $v \in V'$ are again local transition probabilities. In addition, the reward function can be decomposed into a sum of local reward functions $r_v: \times_{w \in N_{\text{in}}(v)} S_w \times A_v \rightarrow \mathbb{R}$, i.e.,
\begin{equation*}
    r(\sigma, a) = \sum\limits_{v \in V'} r_v(\sigma(N_{\text{in}}(v)), a(v)), \quad \sigma \in S, a \in A.
\end{equation*}
Note that the class of GMDPs defined on symmetric directed graphs in which each vertex has a self-loop is again a subclass of the class of STMDPs.  

What sets the STMDP framework apart from the FMDP and GMDP models is the fact that it includes dynamics defined by asynchronous update rules. This flexibility comes at the cost of rendering solution methods designed for FMDPs and GMDPs in general inapplicable to STMDPs, since such methods typically rely on the factored structure of the state space  \citep{Boutilier, Sabbadin1, Guestrin, Sabbadin2, Sabbadin3, Sabbadin4}.

\section{The two-dimensional low-temperature Ising STMDP}
\label{Ising STMDP}
In this section, we formulate and analyze an STMDP based on the two-dimensional Ising model under Metropolis dynamics in the low-temperature regime. Section \ref{Definition of the Ising STMDP} provides a formal definition of the model and a rough description of an auxiliary MDP, which serves as an approximation of the original STMDP. In addition, we state our main result, which concerns the structure of the optimal policy in this auxiliary MDP. A complete analysis can be found in Section \ref{Analysis of the low-temperature Ising STMDP sec}.

\subsection{Definition of the Ising STMDP}
\label{Definition of the Ising STMDP}

Let $V = \{0, \ldots, N-1\}^2$ denote the vertex set of a finite, square, two-dimensional lattice. We formulate an STMDP based on the Ising model defined on $V$ with periodic boundary conditions, evolving according to the discrete-time Metropolis dynamics.

In line with the literature on the Ising model, we will refer to the vertices as spins \citep{Liggett}. The local state space of each spin $i \in V$ is $S_i = \{-1, +1\}$. We denote a configuration on $V$ by $\sigma = (\sigma(i))_{i\in V}$.

The Hamiltonian of a configuration $\sigma \in S$ is given by
\begin{equation}\label{Hamiltonian}
    H(\sigma) = -\dfrac{1}{2}\sum\limits_{\substack{i,j \in V \\ j \in N(i)}} \sigma(i)\sigma(j) - h \sum\limits_{i \in V} \sigma(i),
\end{equation}
where the first summation ranges over all pairs of neighboring spins $i, j \in V$ and $h \in (0,1)$ represents an external magnetic field. Given a configuration $\sigma \in S$ and a spin $i \in V$, let $\sigma^i$ denote the configuration that is obtained after flipping spin $i$ from $\sigma$, and given a set $W \subseteq V$, let $\sigma^W$ denote the configuration obtained after flipping all spins in $W$.

The Metropolis dynamics is defined as a discrete-time Markov chain $\{X_t\}_{t\geq0}$ on $S$ that, given inverse temperature $\beta > 0$, evolves according to the following transition probabilities:
\begin{equation}\label{Metropolis}
    p_{\beta}(\sigma, \sigma^i) = \begin{cases}
        1/N^2, & \text{if } H(\sigma^i) \leq H(\sigma),\\
        (1/N^2)\exp(-\beta(H(\sigma^i) - H(\sigma))), & \text{otherwise,}
    \end{cases}
\end{equation}
and $p_{\beta}(\sigma, \sigma) = 1-\sum\limits_{i \in V} p_{\beta}(\sigma, \sigma^i)$. The Metropolis dynamics is reversible with respect to the Gibbs measure associated with $H$, which is given by
\begin{equation}\label{Gibbs}
    \mu_{\beta}(\sigma) = Z_{\beta}^{-1} \exp(-\beta H(\sigma)),
\end{equation}
where $Z_{\beta}$ is the partition function, i.e.,
\begin{equation}\label{Partition_function}
    Z_{\beta} = \sum\limits_{\sigma \in S} \exp(-\beta H(\sigma)).
\end{equation}

Suppose now that an external decision maker can control the dynamics of the Ising model by flipping a spin every $\kappa$ time steps. Each decision epoch, the decision maker can either decide to flip one of the spins in $V$ or leave the configuration intact. After the execution of the action selected by the decision maker, the process evolves according to the Metropolis dynamics for a period of $\kappa$ time steps. We refer to $\kappa$ as the \textit{adjustment time}. To keep track of the time that elapsed since the previous decision moment, we introduce an additional vertex $v_c$, referred to as the clock vertex, with state space $S_{v_c} = \{0, \ldots, \kappa -1\}$. The state of the clock vertex increases by 1 each time step until it reaches state $\kappa -1$. Then, it is set back to 0, at which point the decision maker can decide to flip a spin. We denote by $\bar{\sigma} = (\sigma(i))_{i \in V \cup \{v_c\}}$ a configuration defined on the lattice augmented with the clock vertex and by $\bar{S}$ the configuration space on this set of vertices. The action space is specified by $A = \cup_{\bar{\sigma} \in \bar{S}}A_{\bar{\sigma}}$, where
\begin{equation*}
    A_{\bar{\sigma}} = \begin{cases}
        V \cup \{0\},  &\text{ if } \bar{\sigma}_{v_c} = 0, \\
        \{0\}, &\text{otherwise,}
    \end{cases}
\end{equation*}
where the action $a = 0$ corresponds to not flipping any spins. Let $\sigma^a$ denote the configuration obtained immediately after taking action $a \in A$ from $\sigma \in S$. We refer to $\sigma^a$ as the \textit{post-decision configuration}. The transition probability kernel $\bar{P}_{\beta}: \bar{S}\times A \times \bar{S} \rightarrow [0,1]$ for inverse temperature $\beta > 0$ is defined as
\begin{equation*}
    \bar{P}_{\beta}(\bar{\sigma}'|\bar{\sigma}, a) = \begin{cases}
        p_{\beta}(\sigma^a, \sigma), &\text{if } \bar{\sigma}'(v_c) = (\bar{\sigma}(v_c) + 1)\text{mod }\kappa,\\
        0, &\text{otherwise,}
    \end{cases}
\end{equation*}
where $\sigma, \sigma' \in S$ are the restrictions of configurations $\bar{\sigma}$ and $\bar{\sigma}'$ respectively to $V$ and $\sigma^0 = \sigma$.

The objective of the Ising STMDP is to enforce a certain desirable behaviour on the process, for example to steer the process towards a certain target configuration $\sigma^* \in S$. This goal is reflected by means of a reward function of the form:
\begin{equation}\label{barreward}
    \bar{r}(\bar{\sigma}, a) = \begin{cases}
        r(\sigma, a), &\text{ if } \bar{\sigma}(v_c) = 0,\\
        0, &\text{ otherwise,}
    \end{cases} \quad \bar{\sigma} \in \bar{S}, \quad a \in A,
\end{equation}
where $r(\sigma, a) \in \mathbb{R}$, $\sigma \in S$, $a \in A$, is a suitably chosen function satisfying $|r(\sigma, a)| \leq M$ for all $\sigma \in S$, $a \in A$, and some $M > 0$. 
We consider the situation in which the decision maker aims to drive the Ising model as quickly as possible towards the all-plus configuration, denoted by $\sigma^+$. Accordingly, we choose the function $r:S \times A \rightarrow \mathbb{R}$ to be
\begin{equation}\label{reward}
    r(\sigma, a) = \begin{cases}
        1, &\text{if } \sigma = \sigma^+,\\
        0, &\text{otherwise.}
    \end{cases}
\end{equation}
In addition, we assume that the process starts in a configuration in which the spins in state $+1$ form a single cluster, which means that each pair of spins in state $+1$ is connected by a path that consists only of spins in state $+1$.

Since the decision maker can only flip a spin if the clock vertex occupies state $0$, we define a decision rule $d: S \rightarrow A$ as a deterministic function on the configuration space without the clock vertex. Also, we define $P_{\beta}: S \times A \times S \rightarrow [0, 1]$ as the counterpart of $\bar{P}_{\beta}$ that disregards the clock vertex, i.e.,
\begin{equation*}
    P_{\beta}(\sigma'|\sigma, a) = p_{\beta}(\sigma^a, \sigma), \quad \sigma, \sigma' \in S, \quad a \in A.
\end{equation*}
Using expression (\ref{barreward}), the expected total discounted reward of a configuration $\sigma \in S$ under policy $\pi = d^{\infty}$, with discount factor $\lambda \in (0,1)$, inverse temperature $\beta > 0$ and adjustment time $\kappa \in \mathbb{N}$, can now be written as 
\begin{equation*}
    v^{\pi}_{\lambda, \beta, \kappa}(\sigma) = \mathbb{E}_{\sigma, \beta}^{\pi}\left[\sum\limits_{t = 0}^{\infty}\lambda^{\kappa t}r(X^{\pi}_{\kappa t}, Y^{\pi}_{\kappa t})\right].
\end{equation*}
By conditioning on the configuration at time $\kappa$, it follows that the value function satisfies the relation 
\begin{equation}\label{value_function_equation}
    v^{\pi}_{\lambda, \beta, \kappa}(\sigma) = r(\sigma, d(\sigma)) + \lambda^{\kappa} \sum\limits_{\sigma' \in S}P^{(\kappa)}_{\beta}(\sigma'|\sigma, d(\sigma))v^{\pi}_{\lambda, \beta, \kappa}(\sigma'), \quad \sigma \in S,
\end{equation}
where $P^{(\kappa)}_{\beta}(\sigma'|\sigma, d(\sigma))$ is the $\kappa$-step transition probability that the process will occupy configuration $\sigma'$ after $\kappa$ time steps, given that action $d(\sigma)$ was selected in configuration $\sigma$. 
Note that $v^{\pi}_{\lambda, \beta, \kappa}(\sigma) < \infty$ for each $\sigma \in S$ due to the fact that the reward function is bounded and $\lambda < 1$. 

\subsection{The low-temperature Ising STMDP}
In analyzing the Ising STMDP, we will restrict ourselves to the low-temperature regime. Hence, we study the low-temperature limit of the value function, i.e.,
\begin{equation*}
    \tilde{v}^{\pi}_{\lambda, \kappa}(\sigma) := \lim_{\beta \rightarrow \infty}v^{\pi}_{\lambda, \beta, \kappa}(\sigma), \quad \sigma \in S.
\end{equation*}
Similarly, we define the low-temperature $\kappa$-step transition probability kernel as 
\begin{equation*}
    \tilde{P}_{\kappa}(\sigma'|\sigma, a) := \lim_{\beta \rightarrow \infty} P^{(\kappa)}_{\beta}(\sigma'|\sigma, a), \quad \sigma, \sigma' \in S, \quad a \in A.
\end{equation*}
Taking the low-temperature limit in both the left- and right-hand side of expression (\ref{value_function_equation}) now yields
\begin{equation}\label{low_temp_value_function_equation}
    \tilde{v}^{\pi}_{\lambda, \kappa}(\sigma) = r(\sigma, d(\sigma)) + \lambda^{\kappa} \sum\limits_{\sigma' \in S}\tilde{P}_{\kappa}(\sigma'|\sigma, d(\sigma))\tilde{v}^{\pi}_{\lambda, \kappa}(\sigma'), \quad \sigma \in S. 
\end{equation}
In vector notation, this expression reads 
\begin{equation*}
    \tilde{\mathbf{v}}^{\pi}_{\lambda, \kappa} = \mathbf{r^{\pi}} + \lambda^{\kappa} \mathbf{\tilde{P}^{\pi}}_{\kappa}\tilde{\mathbf{v}}^{\pi}_{\lambda, \kappa},
\end{equation*}
where $\tilde{\mathbf{v}}_{\lambda, \kappa}^{\pi} \in \mathbb{R}^{|S|}$ and $\mathbf{r^{\pi}} \in \mathbb{R}^{|S|}$ and $\mathbf{\tilde{P}^{\pi}}_{\kappa}\in \mathbb{R}^{|S|\times|S|}$ denote the vector of values, the vector of rewards and the low-temperature $\kappa$-step transition probability matrix specified by policy $\pi$. As established in Section \ref{Introduction Markov decision processes}, the vector of expected total discounted rewards $\tilde{\mathbf{v}}^{\pi}_{\lambda, \kappa}$ of a policy $\pi \in \Pi$ is the unique fixed point of the operator $\mathcal{F}^{\pi}_{\lambda, \kappa}: \mathbb{R}^{|S|} \rightarrow \mathbb{R}^{|S|}$ defined as 
\begin{equation}\label{operator_F}
    \mathcal{F}^{\pi}_{\lambda, \kappa}\mathbf{x} := \mathbf{r}^{\pi} + \lambda^{\kappa} \mathbf{\tilde{P}}^{\pi}_{\kappa}\mathbf{x}, \quad \mathbf{x} \in \mathbf{R}^{|S|}.
\end{equation}
One of the notorious obstacles in solving these equations for the value function is the fact that the state space can be very large. In the Ising STMDP, the state space is of size $2^{N^2}$, which qualifies traditional methods to find the value function as intractable for all but small values of $N$. Observe, however, that this issue is less severe in the low-temperature regime, as the definition of the Metropolis dynamics implies that $p_{\beta}(\sigma, \sigma') \rightarrow 0$ as $\beta \rightarrow \infty$ for all $\sigma'\in S$ that satisfy $H(\sigma') > H(\sigma)$, $\sigma \in S$. Hence, within a finite number of time steps, the probability of the process making a transition that leads to a higher energy configuration tends to 0 in the low-temperature limit, which greatly simplifies the Bellman equations. Also, observe that for large values of $\kappa$, the process is likely to be found in a local minimum of the energy function at the end of the adjustment period. Hence, the dominating terms in the Bellman equations are those that correspond to local minima of the Hamiltonian. Based on these observations, we construct an auxiliary MDP, which ranges only over the local minima of the Hamiltonian in which the spins in state $+1$ form a single cluster. In Section \ref{The auxiliary MDP}, we show that these configurations correspond to the set of configurations in which the spins in state $+1$ form a rectangle.

\subsection{The auxiliary MDP and main result}
This section gives a brief description of the auxiliary MDP and reports the structure of the optimal policy in this MDP. Details of the construction and the analysis of the auxiliary MDP are provided in Sections~\ref{The auxiliary MDP}~and~\ref{Proof of main thm}.

Let the auxiliary MDP be denoted by $(\hat{S}, \hat{A}, \hat{P}, \hat{r})$. Motivated by the geometrical characterization of the local minima of the Hamiltonian as configurations in which the spins in state $+1$ form a rectangle, we define the state space $\hat{S}$ as 
\begin{equation*}
    \hat{S} = \{(i,j)|i,j = 2, 3, \ldots, N-3, N-2, N\} \cup \{(0,0)\},
\end{equation*}
where each vector represents the size of a rectangle in the Ising STMDP and the vector $(0,0)$ corresponds to the all-minus configuration. The action space $\hat{A} = \cup_{(i,j) \in \hat{S}} \hat{A}(i,j)$ is defined as 
\begin{equation*}
    \hat{A} = \{a_{11}, a_{12}, a_{21}, a_{22}, a_{11}', a_{12}', a_{21}', a_{22}', a_0, \tilde{a}, 0\},
\end{equation*}
where $a_{1\ell}$ and $a_{2\ell}$, $\ell = 1,2$, represent the actions of flipping a spin at distance $\ell$ from the horizontal and the vertical side of the rectangle, for which the closest spin belonging to the rectangle is not a corner, actions $a_{1\ell}'$, $a_{2\ell}'$, $\ell = 1,2$, represent the actions of flipping a spin at distance $\ell$ from the horizontal and the vertical side of the rectangle, where the closest spin belonging to the rectangle is a corner, action $a_0$ corresponds to the action of flipping a spin that is diagonally adjacent to the rectangle, action $\tilde{a}$ represents the action of flipping a corner spin of the rectangle and action $0$ corresponds to any other spin or doing nothing. Figure \ref{auxiliary_action_space} provides a visualization of the action space. 

\begin{figure}
\centering
\includegraphics[width=0.3\linewidth]{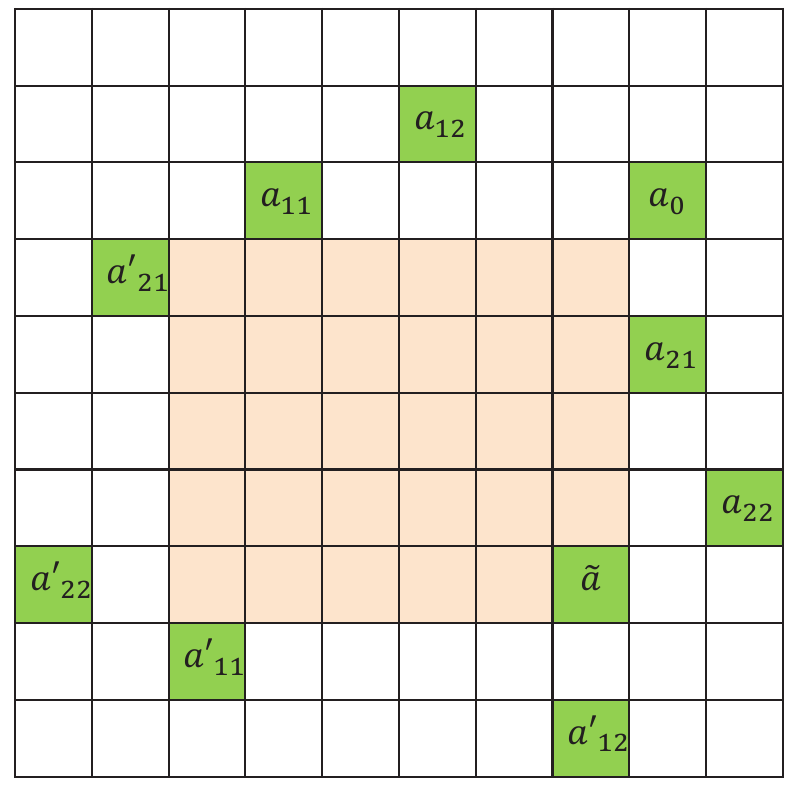}
\caption{Visualization of the action space of the auxiliary MDP.}
\label{auxiliary_action_space}
\end{figure} 
The transition probability kernel $\hat{P}: \hat{S} \times \hat{A} \times \hat{S} \rightarrow [0,1]$, derived from that of the original Ising STMDP, is provided in explicit form in Lemma \ref{explicit_Phats}. The reward function of the auxiliary process is given by 
\begin{equation*}
    \hat{r}(s, a) = \begin{cases}
        1, &\text{ if } s = (N, N), \\
        0, &\text{ otherwise},
    \end{cases}
\end{equation*}
for all $s \in \hat{S}$, $a \in \hat{A}$, reflecting the reward function (\ref{reward}) of the original STMDP, as a rectangle of size $N \times N$ is equivalent to the all-plus configuration. Our main result, which specifies the structure of the optimal policy in the auxiliary MDP, is stated in Theorem \ref{Optimal_policy}. A visualization of this result is provided in Figure \ref{fig_optimal_policy}.

\newpage
\begin{theorem}\label{Optimal_policy}
    Let $A^*_k: \hat{S} \rightarrow P(\hat{A})$, $k = 1, 2$, where $P(\hat{A})$ denotes the power set of the action space $\hat{A}$, be defined as
    \begin{itemize}
        \item $A^*_k(N, N) = \{0\}$, for $k = 1, 2$,
        \item $A^*_k(N, j) = \{a_{11}\}$, if $j = N-2, N-4$, for $k = 1, 2$,
        \item $A^*_k(i, N) = \{a_{21}\}$, if $i = N-2, N-4$, for $k = 1, 2$, 
        \item $A^*_k(N, N-3) = \{a_{12}\}$, for $k = 1, 2$, 
        \item $A^*_k(N-3, N) = \{a_{22}\}$, for $k = 1, 2$,
        \item $A^*_k(N, j) = \begin{cases}
                \{a_{11}\}, &\text{if } k = 1, \\
                \{a_{12}\}, &\text{if } k = 2,
            \end{cases} \quad \text{if } j = 2, \ldots, N-5,$
        \item $A^*_k(i, N) = \begin{cases}
               \{a_{21}\}, &\text{if } k = 1, \\
               \{a_{22}\}, &\text{if } k = 2,
                \end{cases}, \quad \text{if } i = 2, \ldots, N-5,$
        \item $A^*_k(N-2, N-3) = \{a_{12}\}$, for $k = 1, 2$,
        \item $A^*_k(N-3, N-2) = \{a_{22}\}$, for $k = 1, 2$,
        \item $A^*_k(N-2, j) = A^*_k(j, N-2) = \{a_0\}$, if $j = 2, \ldots, N-4$ or $j = N-2$, for $k = 1, 2$,
        \item $A^*_k(N-3, N-3) = \{a_{12}, a_{22}\}$, for $k = 1, 2$,
        \item $A^*_k(N-3, j) = A^*_k(j, N-3) = \{a_0\}$, if $j = 2, \ldots, N-4$ for $k = 1, 2$,
        \item $A^*_k(i, j) = \{a_0\}$, if $i, j = 2, \ldots, N-4$ for $k = 1, 2$.
\end{itemize}
A stationary, deterministic policy $\pi^* = (d^*)^{\infty}$ is optimal in the auxiliary MDP if and only if 
\begin{equation*}
d^*(i,j) \in \begin{cases} 
A^*_{1}(i,j), &\text{if } \lambda \in (\lambda_c, 1), \\
A^*_1(i,j) \cup A^*_2(i,j), &\text{if } \lambda = \lambda_c, \\
A^*_2(i,j), &\text{if } \lambda \in (0, \lambda_c),
\end{cases}
\end{equation*}
for all $(i,j) \in \hat{S}$, where $\lambda_c = 15/17$.
\end{theorem}

\vspace{5mm}

\begin{remark}
    An interesting aspect of the result of Theorem \ref{Optimal_policy} is the phase transition with respect to the discount factor for rectangles of size $(N, j)$ or $(i, N)$, $i, j = 2, \ldots, N-5$. This phenomenon has a very intuitive explanation. Note that actions $a_{12}$ and $a_{22}$ involve more risk but potentially higher gains compared to actions $a_{11}$ and $a_{21}$. If the discount factor is high, a reward obtained in the future still retains a fair portion of its value in the present. In this scenario, it is most favorable to choose the safe options $a_{11}$ or $a_{21}$. On the other hand, if the discount factor is low, a reward obtained in the future holds less value in the present. In this case, it is more advantageous to opt for a riskier strategy that has the potential to reach the target sooner.\qed 
\end{remark}

\begin{figure}
    \centering
    \captionsetup[subfloat]{justification=centering, labelformat=empty, singlelinecheck=false}

    % First row
    \parbox{\linewidth}{
        \centering
        \subfloat[{\small Optimal action in state $(N,N)$.}]{
            \includegraphics[width=0.2\linewidth]{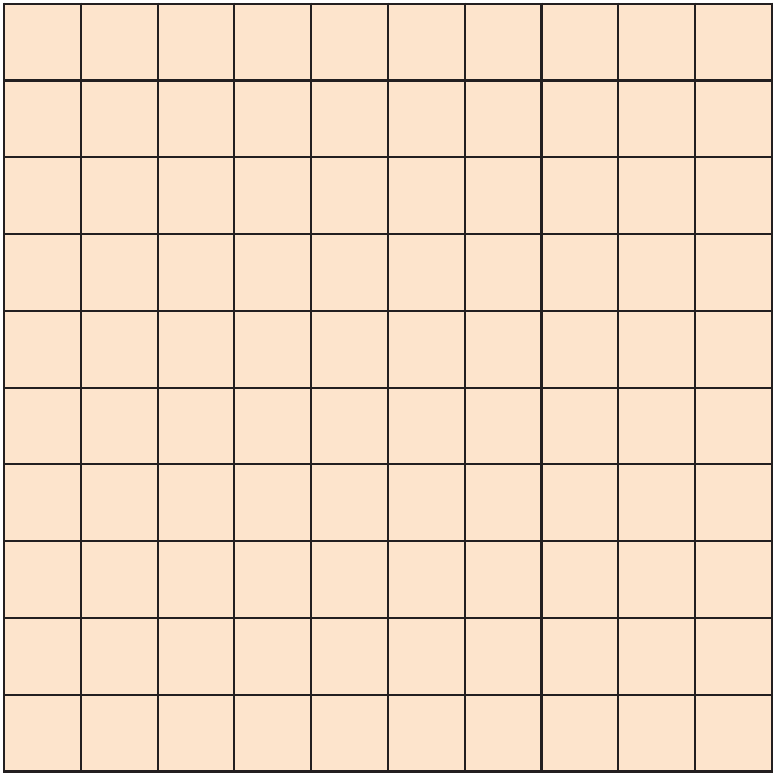}
        } \hspace{0.02\linewidth}
        \subfloat[{\small Optimal action in state $(N,N-2)$.}]{
            \includegraphics[width=0.2\linewidth]{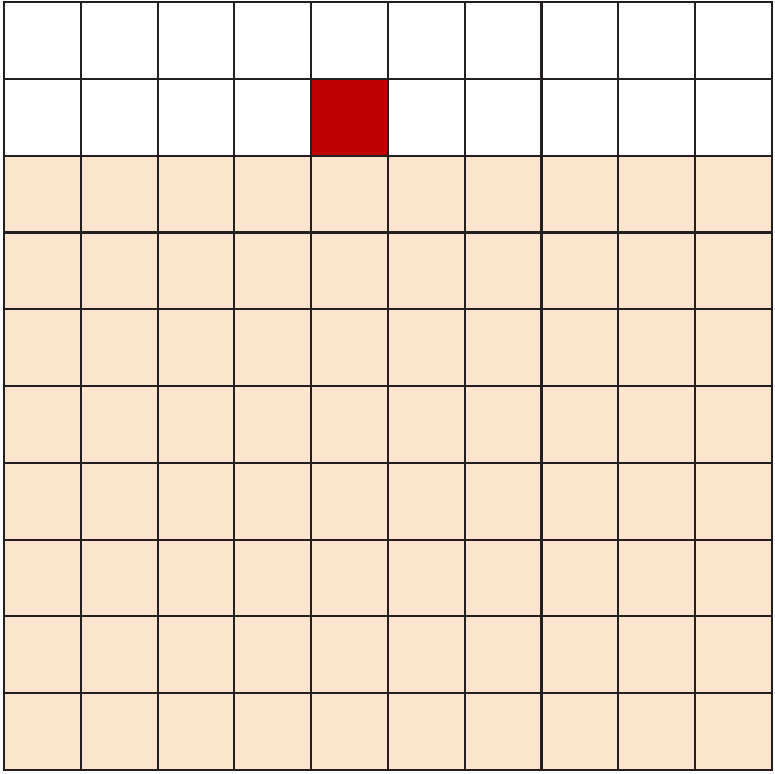}
        } \hspace{0.02\linewidth}
        \subfloat[{\small Optimal action in state $(N,N-3)$.}]{
            \includegraphics[width=0.2\linewidth]{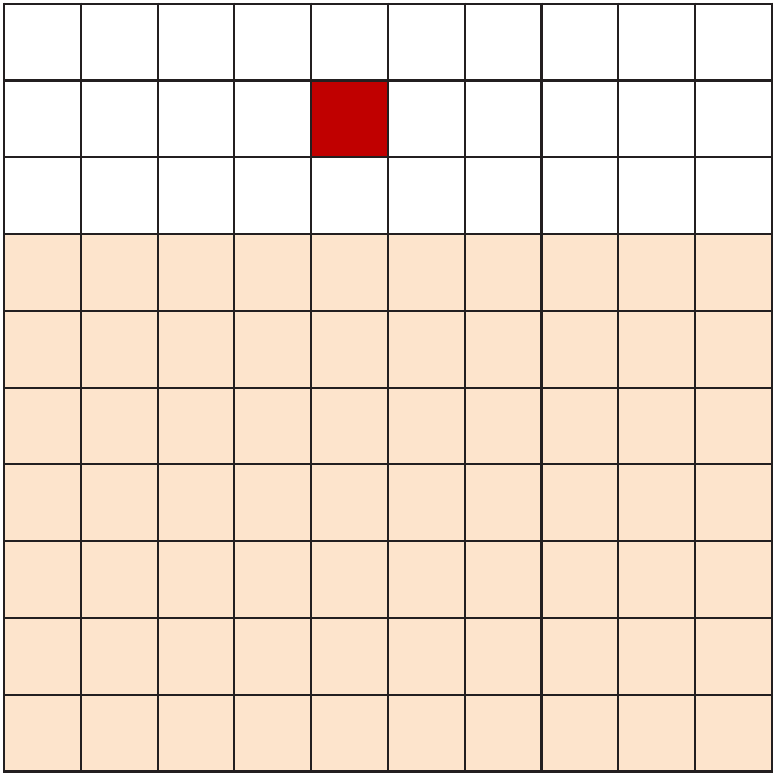}
        } \hspace{0.02\linewidth}
        \subfloat[{\small Optimal action in state $(N,N-4)$.}]{
            \includegraphics[width=0.2\linewidth]{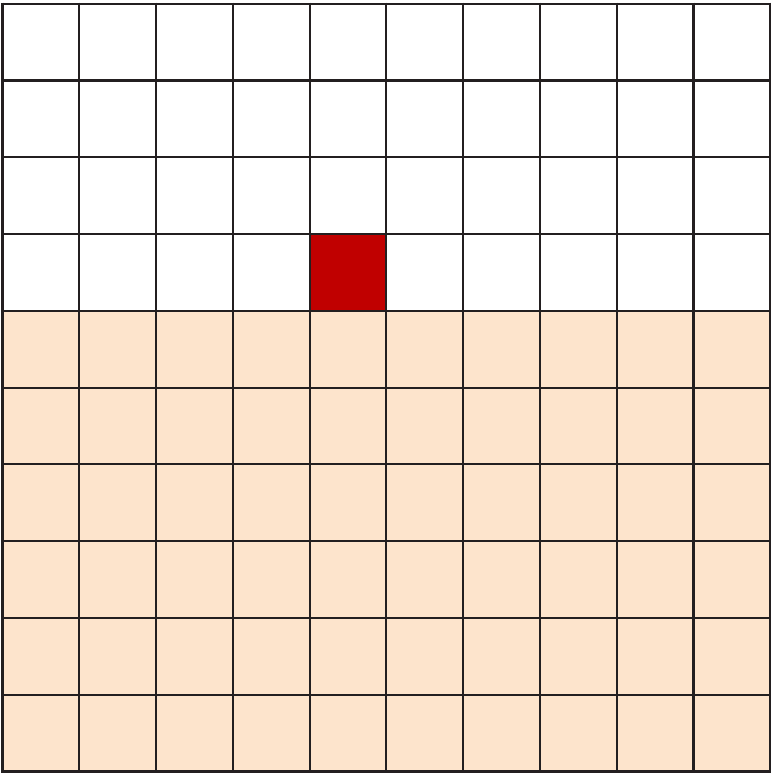}
        }
    }
    
    \vspace{0.5cm} % Adjust for spacing between rows

    % Second row
    \parbox{\linewidth}{
        \centering
        \subfloat[{\small Optimal actions in states $(N,j)$, $j = 2, 3, \ldots, N-5$, depending on value of $\lambda$.}]{
            \includegraphics[width=0.2\linewidth]{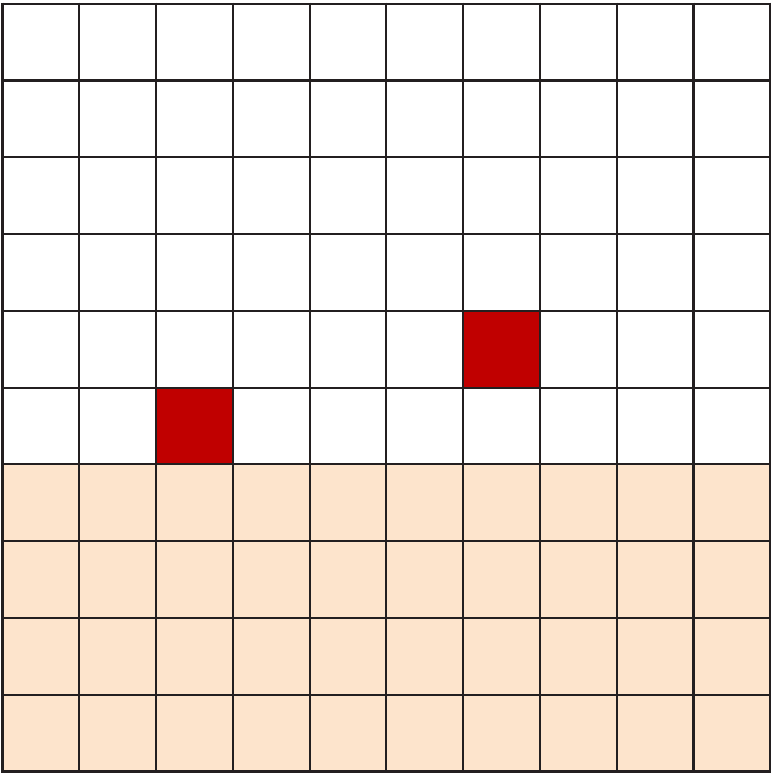}
        } \hspace{0.02\linewidth}
        \subfloat[{\small Optimal action in state $(N-2,N-2)$.}]{
            \includegraphics[width=0.2\linewidth]{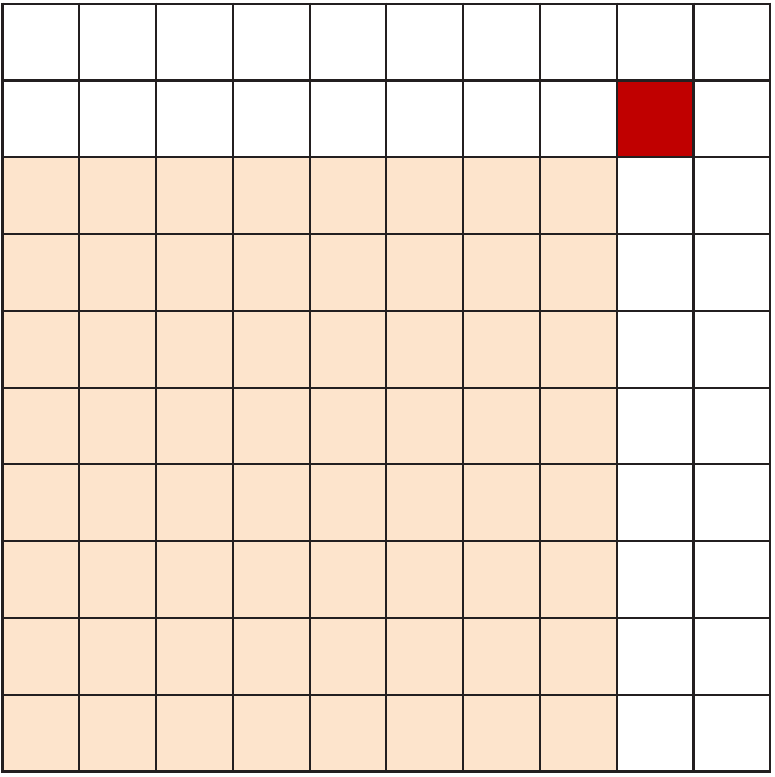}
        } \hspace{0.02\linewidth}
        \subfloat[{\small Optimal action in state $(N-2, N-3)$.}]{
            \includegraphics[width=0.2\linewidth]{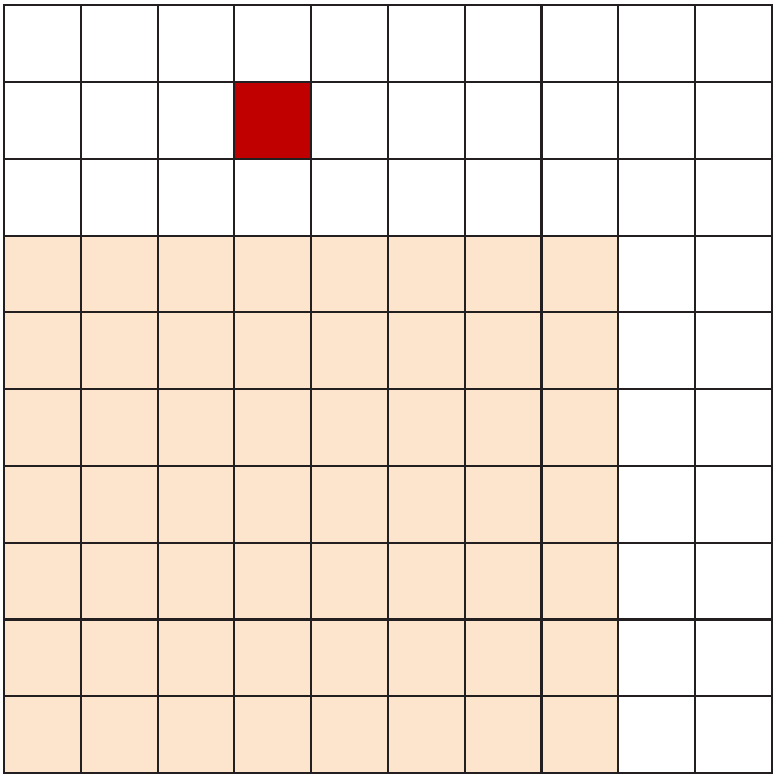}
        } \hspace{0.02\linewidth}
        \subfloat[{\small Optimal action in states $(N-2,j)$, $j = 2, 3, \ldots, N-4$.}]{
            \includegraphics[width=0.2\linewidth]{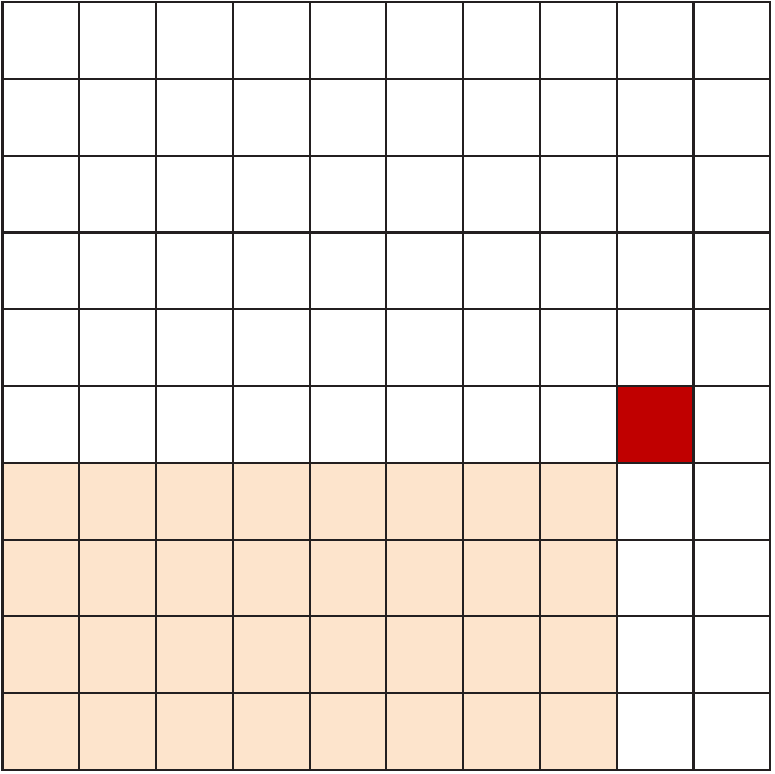}
        }
    }

    \vspace{0.5cm}

    % Third row with phantom figure
    \parbox{\linewidth}{
        \centering
        \subfloat[{\small Optimal action in state $(N-3,N-3)$.}]{
            \includegraphics[width=0.2\linewidth]{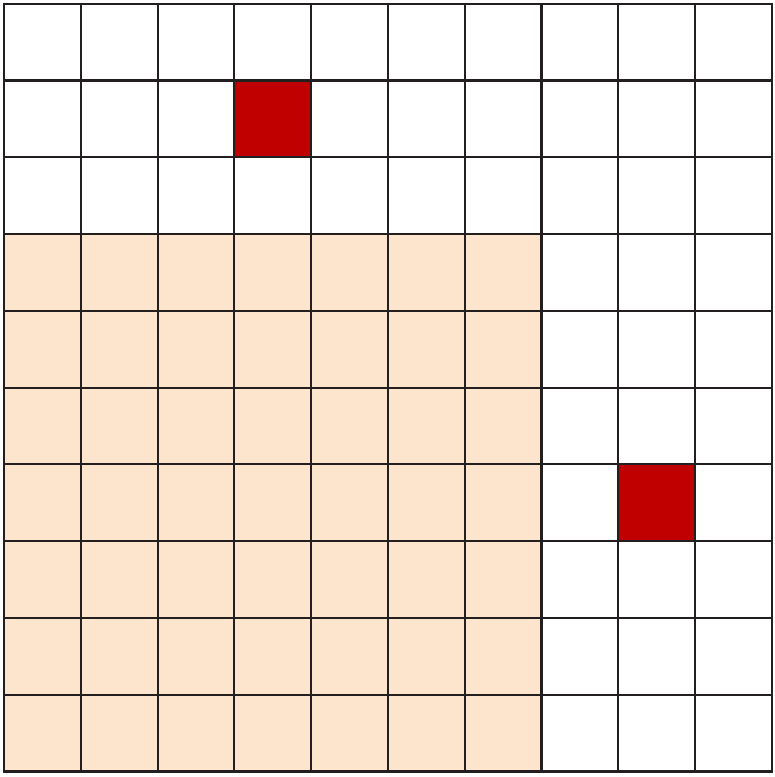}
        } \hspace{0.02\linewidth}
        \subfloat[{\small Optimal action in state $(N-3,j)$, $j = 2, 3, \ldots, N-4$.}]{
            \includegraphics[width=0.2\linewidth]{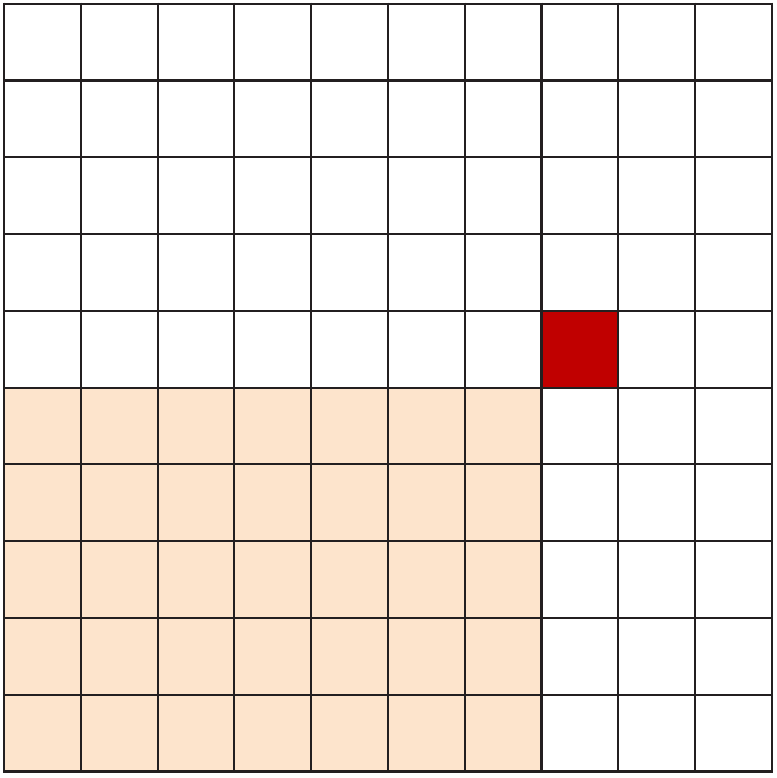}
        } \hspace{0.02\linewidth}
        \subfloat[{\small Optimal action in states $(i,j)$, $i,j = 2, 3, \ldots, N-4$.}]{
            \includegraphics[width=0.2\linewidth]{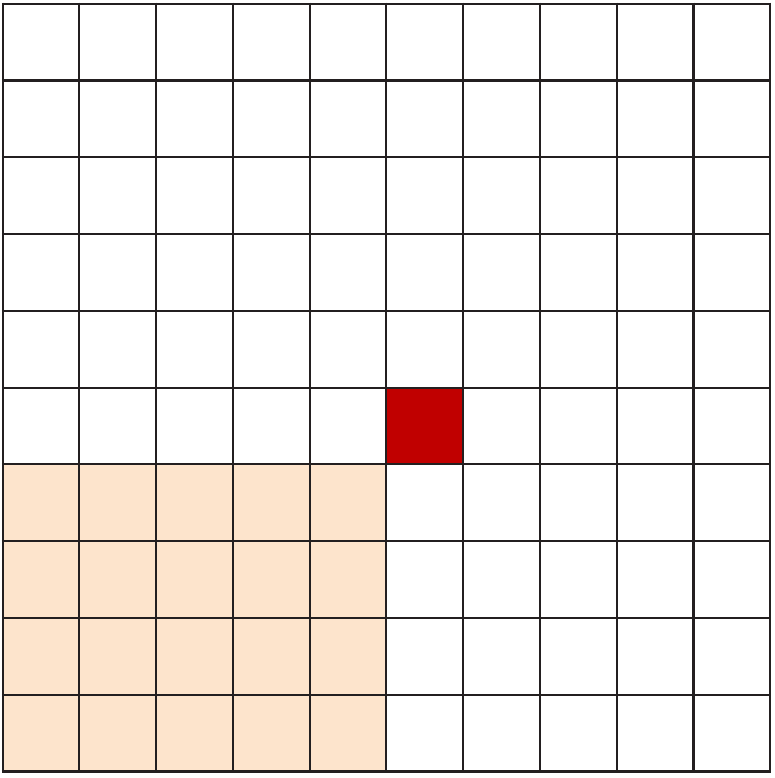}
        } \hspace{0.02\linewidth}
        \subfloat[]{\phantom{\includegraphics[width=0.2\linewidth]{Opt_pol_i,j.pdf}}} % Phantom figure
    }

    \caption[The average and standard deviation of critical parameters]
    {\small Visualization of the optimal policy of the auxiliary MDP (Theorem \ref{Optimal_policy}).} 
    \label{fig_optimal_policy}
\end{figure}

\section{Numerical results}
\label{Numerical results}
In this section, we numerically investigate the performance of an analogue of the optimal policy for the auxiliary MDP in the original low-temperature Ising STMDP. In addition, we compare the results to those of two other policies. \\
Let $\hat{\pi}^* = (\hat{d}^*)^{\infty}$ denote an optimal policy in the auxiliary MDP. Furthermore, define policies $\hat{\pi}_1 = (\hat{d}^1_t)_{t \in \mathbb{N}}$ and $\hat{\pi}_2 = (\hat{d}^2_t)_{t \in \mathbb{N}}$ in the auxiliary MDP as
\begin{align*}
    \hat{d}^1_t(i, j) &= \begin{cases}
        a_{11}, &\text{if } t \text{ even}, \\
        a_{21}, &\text{if } t \text{ odd},
    \end{cases} \quad \text{for } i, j = 2, \ldots, N-2, \\
    \hat{d}^1_t(i,j) &= \begin{cases}
        a_{11}, &\text{if } i = N, \quad j = 2, \ldots, N-2, \\
        a_{21}, &\text{if } i = 2, \ldots, N-2, \quad j = N, \\
        0, &\text{if } i = j = N
    \end{cases}
\end{align*}
and
\begin{align*}
    \hat{d}^2_t(i,j) &= \begin{cases}
        a_{12}, &\text{if } t \text{ even and } j = 2, \ldots, N-3, \\
        a_{11}, &\text{if } t \text{ even and } j = N-2, \\
        a_{22}, &\text{if } t \text{ odd and } i = 2, \ldots, N-3, \\
        a_{21}, &\text{if } t \text{ odd and } i = N-2,
    \end{cases} \quad \text{for } i, j = 2, \ldots, N-2, \\
    \hat{d}^2_t(i,j) &= \begin{cases}
        a_{12}, &\text{if } i = N, \quad j = 2, \ldots, N-3, \\
        a_{11}, &\text{if } i = N, \quad j = N-2, \\
        a_{22}, &\text{if } i = 2, \ldots, N-3, \quad j = N, \\
        a_{21}, &\text{if } i = N, \quad j = N-2, \\
        0, &\text{if } i = j = N.
    \end{cases}
\end{align*}
Policies $\hat{\pi}_1$ and $\hat{\pi}_2$ flip a spin at distance 1 and 2, respectively, from the rectangle each time step, alternating sides. 
Figure \ref{policy_comparison_auxiliary_MDP} depicts the values of these policies in state $(3, 3)$ for $N = 20$ as a function of the discount factor $\lambda$. It can be seen that $\hat{\pi}^*$ indeed achieves higher expected total discounted reward than the policies $\hat{\pi}_1$ and $\hat{\pi}_2$ for any value of the discount factor. 

\begin{figure}
\centering
        \includegraphics[width=1.1\linewidth]{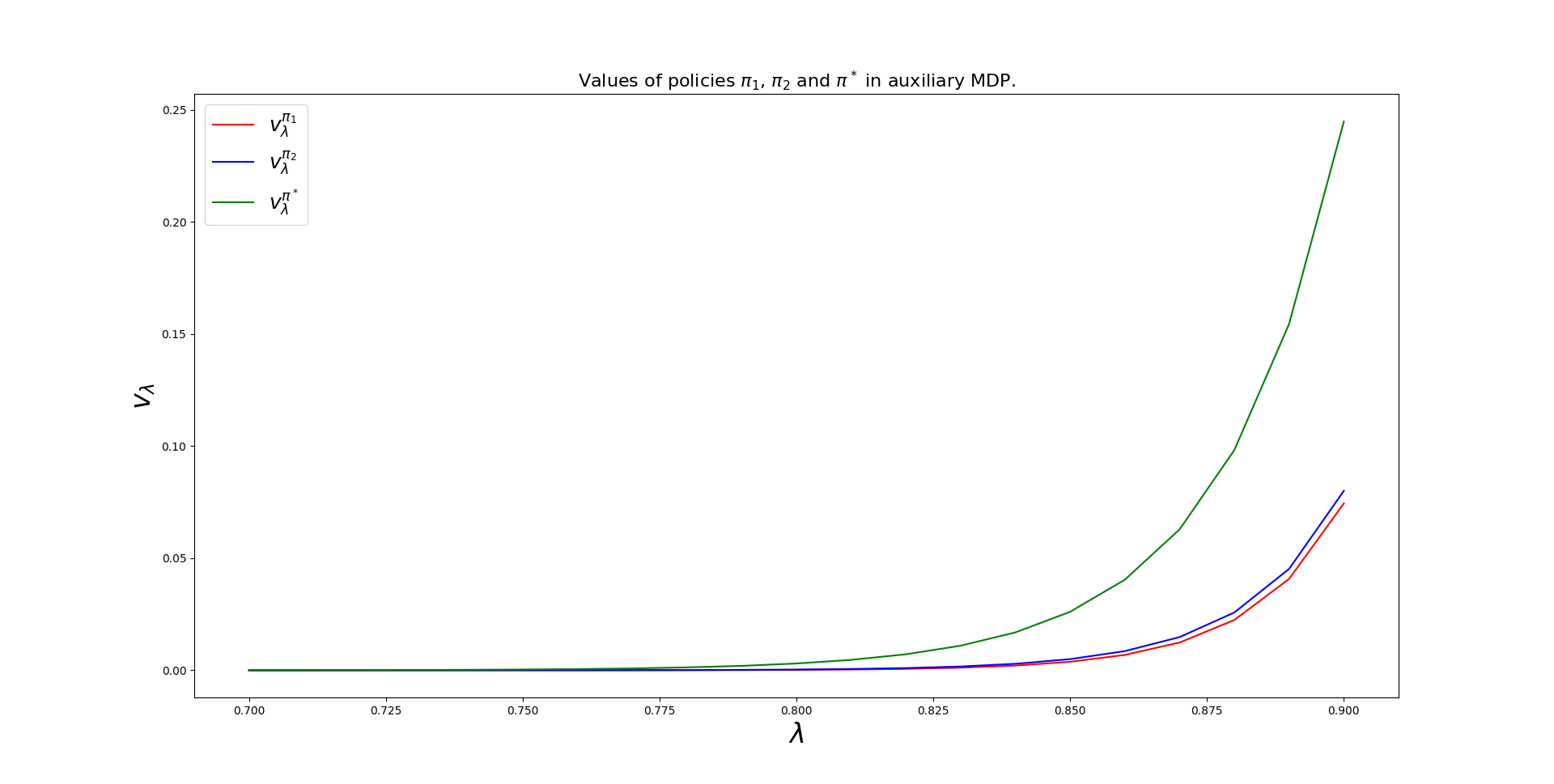}
        \caption{Values of policies $\hat{\pi}_1$, $\hat{\pi}_2$ and $\hat{\pi}^*$ in state $(3, 3)$ in the auxiliary MDP as a function of the discount factor.}
        \label{policy_comparison_auxiliary_MDP}
    \end{figure} 

We now construct the analogues of policies $\hat{\pi}^*$, $\hat{\pi}_1$ and $\hat{\pi}_2$ in the original low-temperature Ising STMDP by identifying a cluster of spins in state $+1$ with its circumscribed rectangle and selecting a spin that corresponds to the action specified for this rectangle by the policy for the auxiliary MDP. Specifically, we define policies $\tilde{\pi}^* = (\tilde{d}^*)^{\infty}$, $\pi_1 = (d_t^1)_{t \in \mathbb{N}}$ and $\pi_2 = (d_t^2)_{t \in \mathbb{N}}$, where $\tilde{d}^*(\sigma)$, $d^1_t(\sigma)$, $d^2_t(\sigma)$ for a configuration $\sigma \in S$ in which the circumscribed rectangle has size $i \times j$, are spins corresponding to the actions $\hat{d}^*(i,j)$, $\hat{d}^1_t(i,j)$ and $\hat{d}^2_t(i,j)$. We choose such a spin as the one that is closest to the spin in the middle of the side of the circumscribed rectangle (or to one of the two spins in the middle if the side length is even). Furthermore, whenever we select a spin on a horizontal side of the circumscribed rectangle, we always choose the same horizontal side. The same principle applies to spins on the vertical side.

Figure \ref{policy_comparison_numerical_k5000} shows simulation results for the three different policies for $\beta = 10$, $N = 100$ and an adjustment time of $\kappa = 5000$. The policy $\tilde{\pi}^*$ indeed seems to reach the target configuration faster than the other two policies.

% Figure environment
\begin{figure}
    \centering
    \captionsetup[subfloat]{justification=centering, labelformat=empty, singlelinecheck=false, font=small}
    
    % Three images in a row
    \subfloat[{\small Simulation of policy $\tilde{\pi}^*$ for $\beta = 10$, $N = 100$, and $\kappa = 5000$.} \label{subfig:piopt_k5000}]{
        \includegraphics[width=0.3\linewidth]{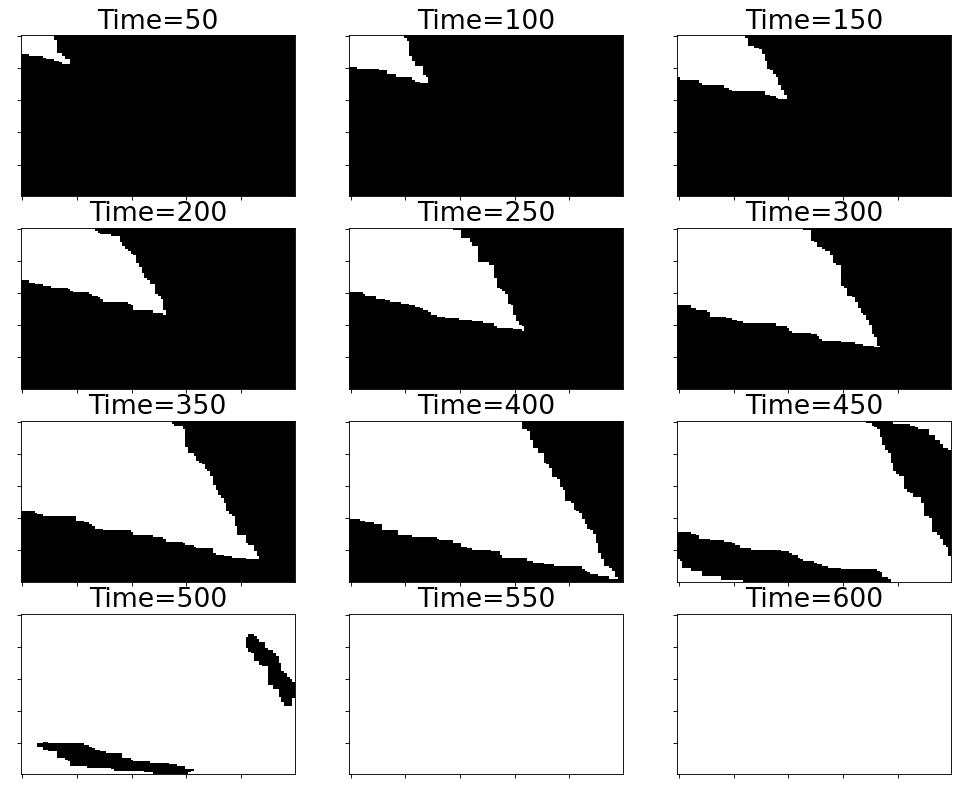}
    } \hspace{0.02\linewidth}
    \subfloat[{\small Simulation of policy $\pi_1$ for $\beta = 10$, $N = 100$, and $\kappa = 5000$.} \label{subfig:pi1_k5000}]{
        \includegraphics[width=0.3\linewidth]{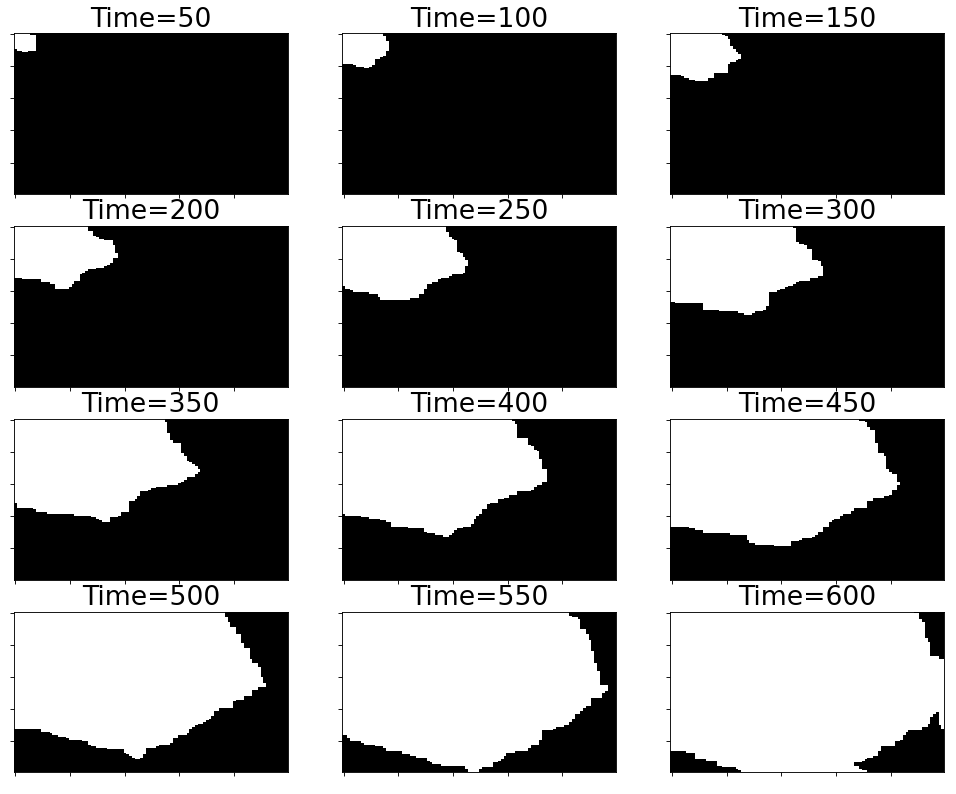}
    } \hspace{0.02\linewidth}
    \subfloat[{\small Simulation of policy $\pi_2$ for $\beta = 10$, $N = 100$, and $\kappa = 5000$.} \label{subfig:pi2_k5000}]{
        \includegraphics[width=0.3\linewidth]{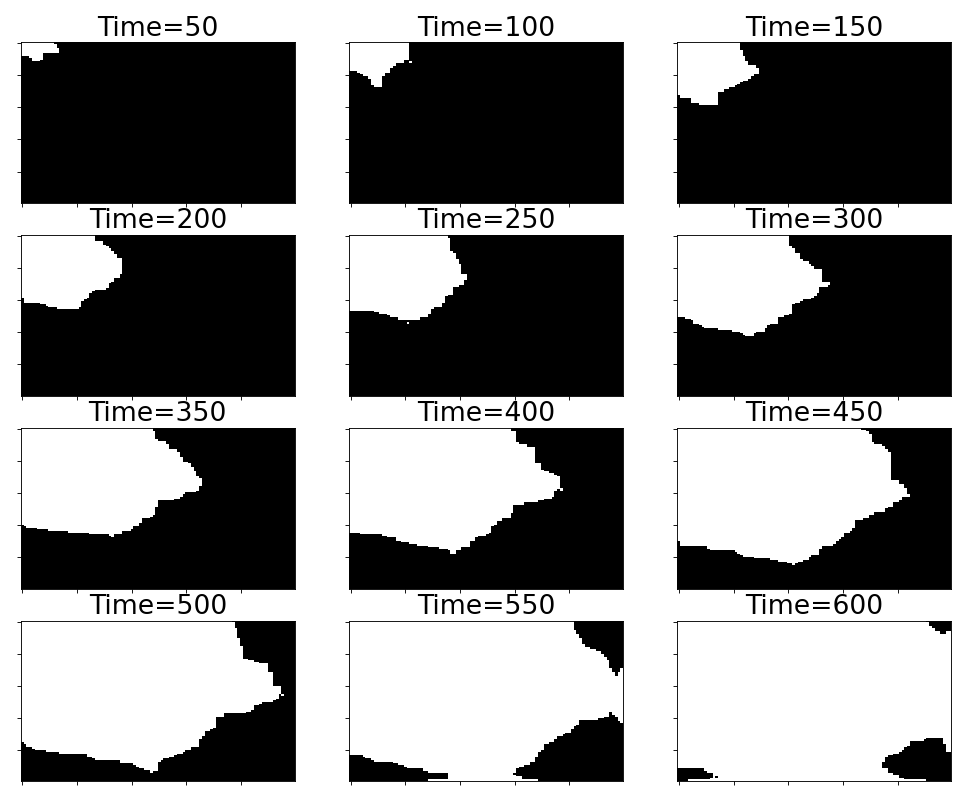}
    }

    \caption[Simulations of policies $\tilde{\pi}^*$, $\pi_1$, and $\pi_2$ for $\kappa = 5000$]
    {\small Simulations of policies $\tilde{\pi}^*$, $\pi_1$, and $\pi_2$ for $\beta = 10$, $N = 100$, and $\kappa = 5000$.} 
    \label{policy_comparison_numerical_k5000}
\end{figure}

Note that the adjustment time in Figure \ref{policy_comparison_numerical_k5000} is clearly not long enough for a cluster of $+$-spins to fully grow into a rectangle before the decision maker again decides to flip a spin. Figure \ref{policy_comparison_numerical_adjustment_times} illustrates the behaviour of the policy $\tilde{\pi}^*$ for various different adjustment times. For larger adjustment times, the shapes of the clusters more closely resemble rectangles.

\begin{figure}
    \centering
    \captionsetup[subfloat]{justification=centering, labelformat=empty, singlelinecheck=false, font=small}
    
    % Three images in a single row
    \subfloat[{\small Simulation of policy $\tilde{\pi}^*$ for $\beta = 10$, $N = 100$, and $\kappa = 5000$.} \label{subfig:piopt_k50000}]{
        \includegraphics[width=0.3\linewidth]{Policy_comparison_numerical_piopt_k5000_new.png}
    } \hspace{0.02\linewidth}
    \subfloat[{\small Simulation of policy $\tilde{\pi}^*$ for $\beta = 10$, $N = 100$, and $\kappa = 10000$.} \label{subfig:pi1_k50000}]{
        \includegraphics[width=0.3\linewidth]{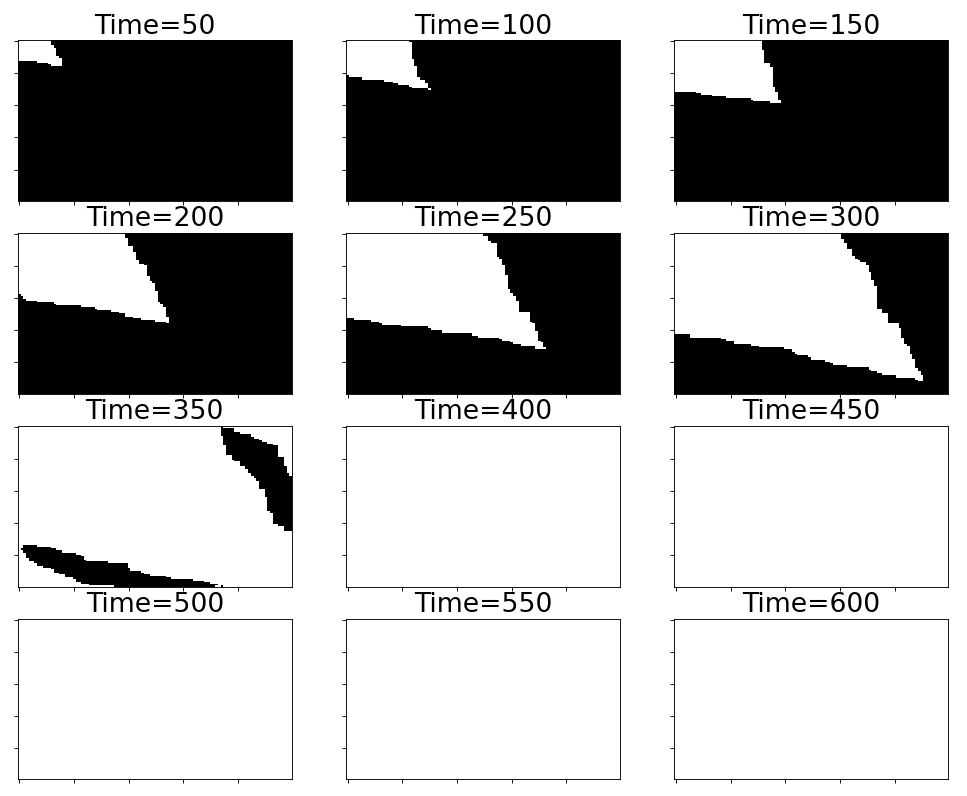}
    } \hspace{0.02\linewidth}
    \subfloat[{\small Simulation of policy $\tilde{\pi}^*$ for $\beta = 10$, $N = 100$, and $\kappa = 50000$.} \label{subfig:pi2_k50000}]{
        \includegraphics[width=0.3\linewidth]{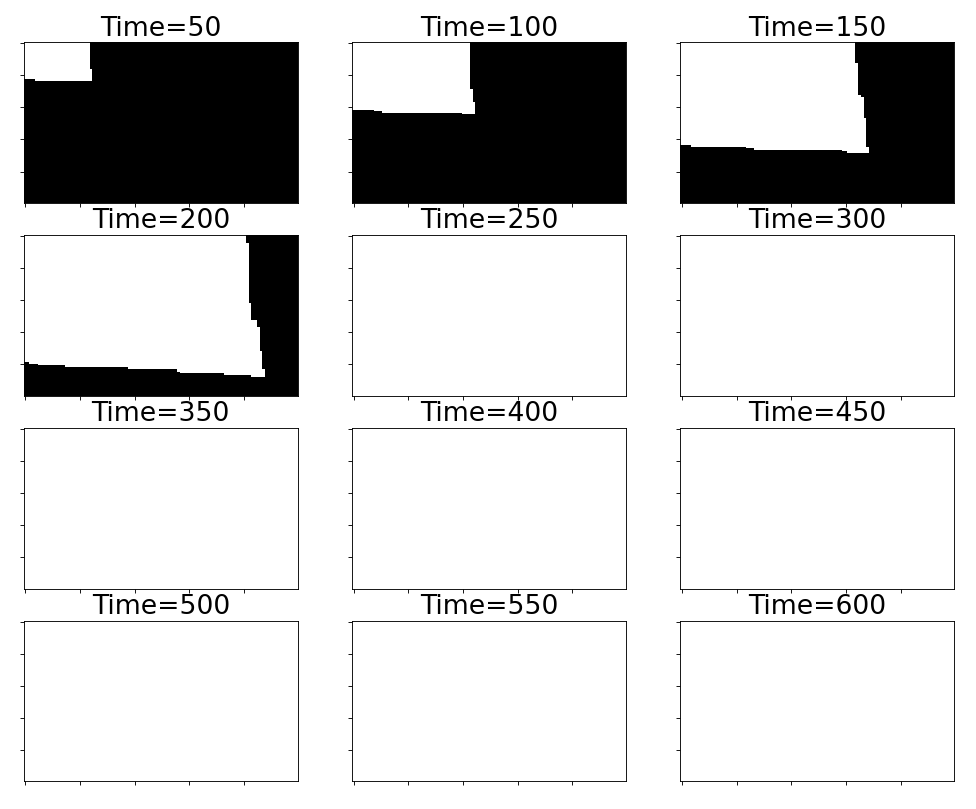}
    }

    \caption[Simulations of policy $\tilde{\pi}^*$ for various adjustment times]
    {\small Simulations of policy $\tilde{\pi}^*$ for various adjustment times: $\kappa = 5000$, $\kappa = 10000$, and $\kappa = 50000$, with $\beta = 10$ and $N = 100$.} 
    \label{policy_comparison_numerical_adjustment_times}
\end{figure}

Figure \ref{policy_comparison_numerical_temperatures} explores the behaviour of the policy $\tilde{\pi}^*$ for various temperatures. It is evident that the cluster of $+$-spins evolves more chaotically as the temperature increases.

\begin{figure}
    \centering
    \captionsetup[subfloat]{justification=centering, labelformat=empty, singlelinecheck=false, font=small}
    
    % Three images in a single row
    \subfloat[{\small Simulation of policy $\tilde{\pi}^*$ for $\beta = 2.5$, $N = 100$, and $\kappa = 5000$.} \label{subfig:piopt_k50000}]{
        \includegraphics[width=0.3\linewidth]{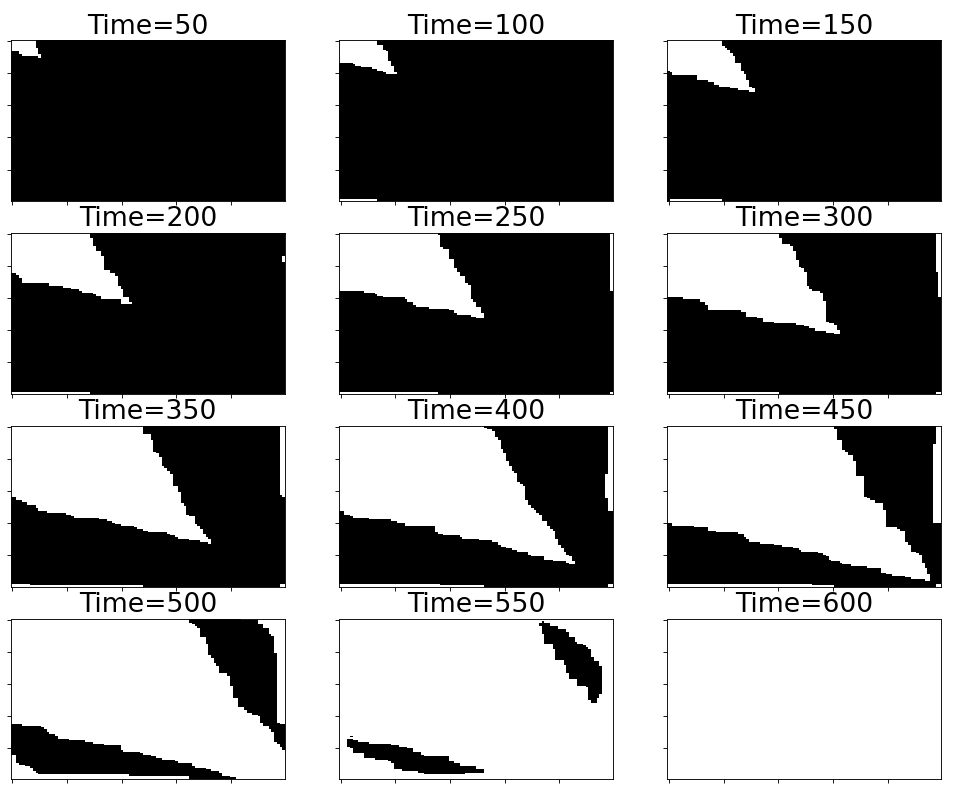}
    } \hspace{0.02\linewidth}
    \subfloat[{\small Simulation of policy $\tilde{\pi}^*$ for $\beta = 1.67$, $N = 100$, and $\kappa = 5000$.} \label{subfig:pi1_k50000}]{
        \includegraphics[width=0.3\linewidth]{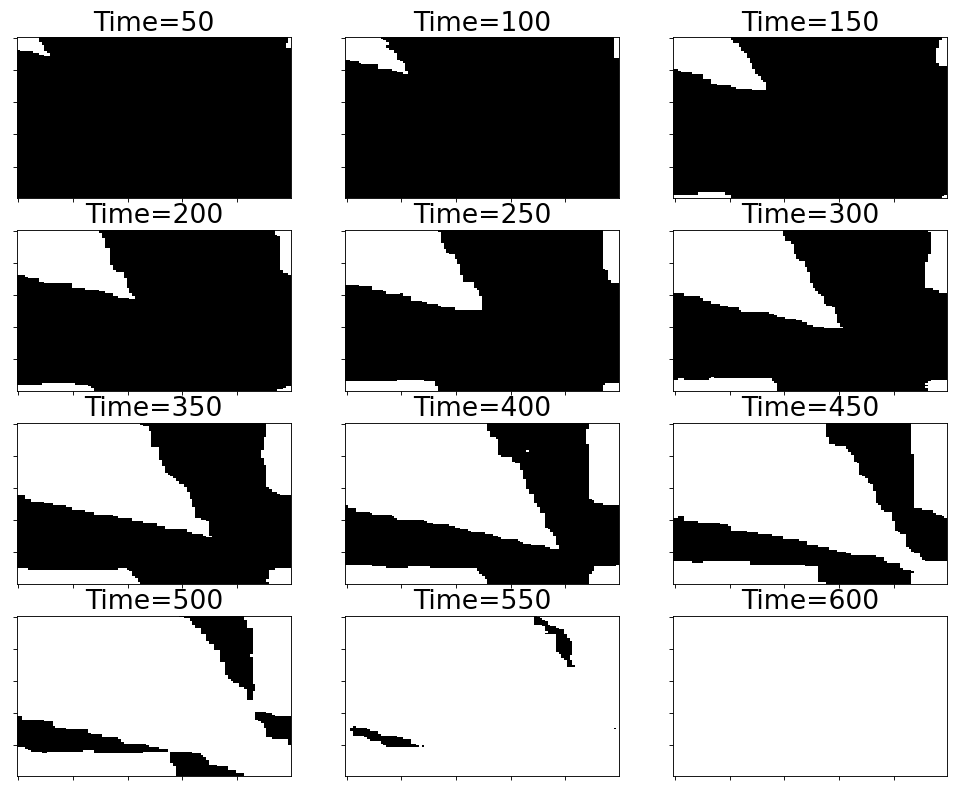}
    } \hspace{0.02\linewidth}
    \subfloat[{\small Simulation of policy $\tilde{\pi}^*$ for $\beta = 1.25$, $N = 100$, and $\kappa = 5000$.} \label{subfig:pi2_k50000}]{
        \includegraphics[width=0.3\linewidth]{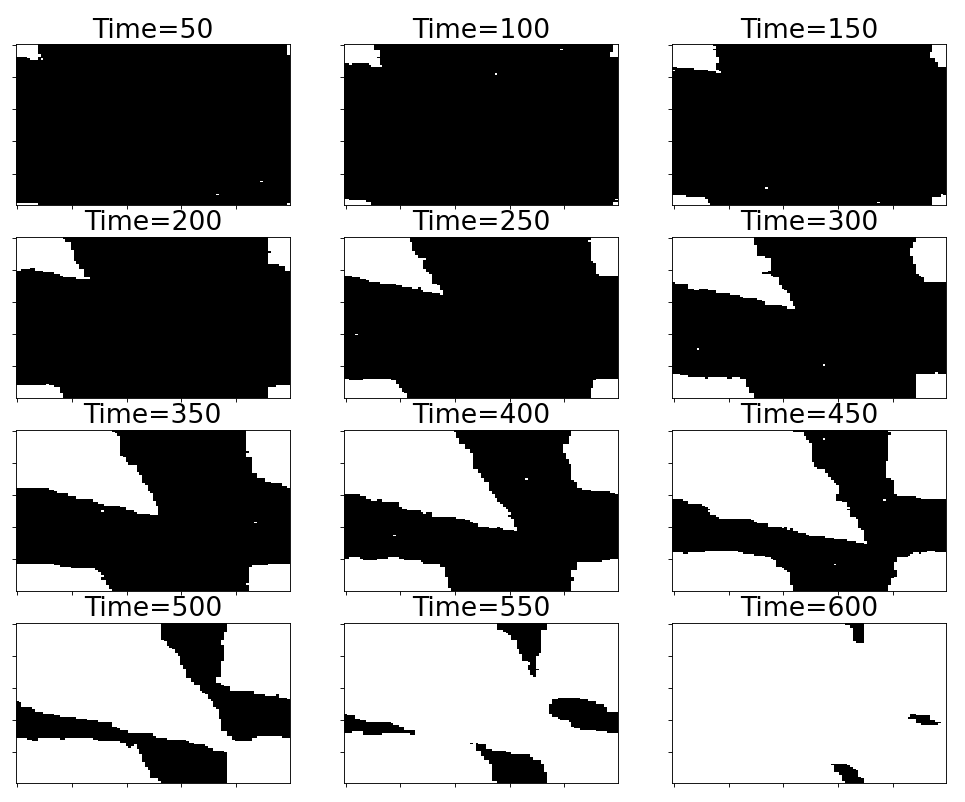}
    }

    \caption[Simulations of policies $\tilde{\pi}^*$ for various temperatures]
    {\small Simulations of policy $\tilde{\pi}^*$ for various temperatures: $\beta = 2.5$, $\beta = 1.67$, and $\beta = 1.25$, with $N = 100$ and $\kappa = 5000$.} 
    \label{policy_comparison_numerical_temperatures}
\end{figure}

To study the performance of the three policies more closely, we measure the first hitting times $\tau^{\sigma, \pi}_{\sigma^*}$ to the target configuration across $n$ experiments for different values of the adjustment time. Figure \ref{policy_comparison_adjustment_times} shows the resulting average hitting times for a starting configuration with a single 3x3 rectangle of $+$-spins for several values of the adjustment time and for $n = 20$, $\beta = 10$ and $N = 100$. Recall that the relation between these hitting times and the value function, given a discount factor $\lambda$, has been established in Section \ref{A reachability objective}.

\begin{figure}
\centering
        \includegraphics[width=0.5\linewidth]{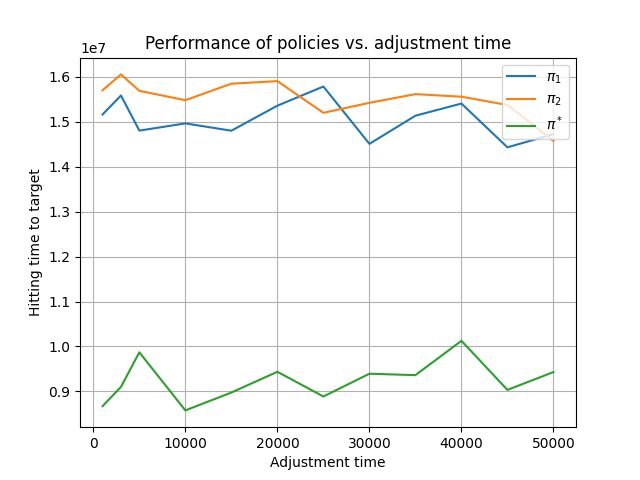}
        \caption{Hitting times to the target of policies $\pi_1$, $\pi_2$ and $\pi^*$ from state $(3, 3)$ for $n = 20$, $\beta = 10$ and $N = 100$.}
        \label{policy_comparison_adjustment_times}
    \end{figure} 

The results indicate that the policy $\tilde{\pi}^*$ derived from the optimal policy in the auxiliary MDP indeed achieves the goal faster than the policies $\pi_1$ and $\pi_2$ over a wide range of adjustment times.

\section{Analysis of the low-temperature Ising STMDP}
\label{Analysis of the low-temperature Ising STMDP sec}

In this section we provide a more formal analysis of the low-temperature Ising STMDP and prove our main result, i.e., Theorem \ref{Optimal_policy}. 
Section \ref{Additional notation and terminology} introduces some additional notation and terminology required for the complete analysis of the low-temperature Ising STMDP. In Section \ref{The auxiliary MDP}, we formalize the construction of the auxiliary MDP and show that it is indeed an accurate approximation of the original STMDP. Finally, Section \ref{Proof of main thm} reports an outline of the proof of Theorem \ref{Optimal_policy}. Some details of the proofs are deferred to the Supplementary Material.

\newpage
\subsection{Additional notation and terminology}
\label{Additional notation and terminology}

In order to formalize the analysis outlined in Section \ref{Definition of the Ising STMDP}, we introduce some additional notation and terminology. 

Given a spin $i \in V$, let $N_h(i) \subseteq V$ and $N_v(i) \subseteq V$ denote the sets of its horizontal neighbors and its vertical neighbors. We define the following distance measure $\delta: V \times V \rightarrow \mathbb{N}_0$ on the lattice:
    \begin{equation*}
        \delta((x,y), (x',y')) = \min\{|x'-x|, N - |x'-x|\} + \min\{|y'-y|, N-|y'-y|\}, 
    \end{equation*}
    where $(x,y), (x',y') \in V \times V$, taking into account a torus edge correction. 
    Furthermore, letting $P(V)$ denote the power set of $V$, we define a distance measure $\tilde{\delta}: V \times P(V) \rightarrow \mathbb{N}_0$ as 
    \begin{equation*}
        \tilde{\delta}((x,y), W) = \min_{(x',y') \in V}\delta((x,y), (x',y')), \quad (x,y) \in V,\quad W \subseteq V. 
    \end{equation*}
    Let $\mu((x,y), W)$ denote the set of spins in $W$ that are closest to $(x,y)$ with respect to this distance measure, i.e.,
    \begin{equation*}
        \mu((x,y), W) = \argmin_{(x',y') \in W}\delta((x,y), (x',y')), \quad (x,y) \in V,\quad W \subseteq V. 
    \end{equation*}
    In addition, we define a horizontal distance measure $\delta_h: V \times V \rightarrow \mathbb{N}^{\infty}_0$ as 
    \begin{equation*}
        \delta_h((x,y), (x',y')) = \begin{cases}
            \min\{|x'-x|, N-|x'-x|\}, &\text{if } y = y,' \\
            \infty, &\text{otherwise.}
        \end{cases}
    \end{equation*}
    Accordingly, let $\tilde{\delta}_h: V \times P(V) \rightarrow \mathbb{N}^{\infty}_0$ be defined as 
    \begin{equation*}
        \tilde{\delta}_h((x,y), W) = \min_{(x',y') \in V}\delta_h((x,y), (x',y')), \quad (x,y) \in V, \quad W \subseteq V,
    \end{equation*}
    and $\mu_h((x,y), W)$, $(x,y) \in V$, $W \subseteq V$ as
    \begin{equation*}
        \mu_h((x,y), W) = \begin{cases} 
        \emptyset, &\text{if } \delta_h((x,y), (x',y')) = \infty \\
        &\text{for all } (x',y') \in W, \\
        \argmin_{(x',y') \in W}\delta_h((x,y), (x',y')), &\text{otherwise}.
        \end{cases}
    \end{equation*}
    The vertical counterparts $\delta_v$, $\tilde{\delta}_v$ and $\mu_v$ are defined in an analogous way.

Given a configuration $\sigma \in S$, let $V_+(\sigma) \subseteq V$ denote the set of spins in state $+1$ in $\sigma$, i.e., $V_+(\sigma) = \{i \in V|\sigma(i) = +1\}$. Furthermore, let $R(\sigma)$ denote the spins in the smallest rectangle that circumscribes the set of $+$-spins.  We define the set of \textit{corner spins}, the sets of \textit{horizontal and vertical boundary spins} and the set of \textit{interior spins} of $\sigma$ as
\newpage
    \begin{align*}
        C(\sigma) &= \left\{i \in V| \sigma(i) = +1, \sum_{j \in N_h(i)} \sigma(j) = 0 \text{ and } \sum_{j \in N_v(i)}\sigma(j) = 0\right\},\\
        B_h(\sigma) &= \left\{i \in V| \sigma(i) = +1, \sum_{j \in N_h(i)} \sigma(j) = 2 \text{ and } \sum_{j \in N_v(i)}\sigma(j) = 0\right\}, \\
        B_v(\sigma) &= \left\{i \in V| \sigma(i) = +1, \sum_{j \in N_v(i)} \sigma(j) = 2 \text{ and } \sum_{j \in N_h(j)} \sigma(j) = 0\right\},\\
        I(\sigma) &= \left\{i \in V| \sigma(i) = \sigma(j) = +1 \text{ for all } j \in N(i)\right\}.
    \end{align*}
 Let $\hat{\sigma} \in S$ denote the configuration obtained from $\sigma$ by setting all spins in $R(\sigma)$ to $+1$, i.e.,
\begin{equation*}
    \hat{\sigma}(i) = \begin{cases}
        +1, &\text{if } i \in R(\sigma), \\
        \sigma(i), &\text{otherwise},
    \end{cases}
    \quad i \in V.
\end{equation*}
We proceed to introduce some notation and terminology specific to the low-temperature behaviour of the Ising STMDP. 
Let the low-temperature Metropolis dynamics be defined as
\begin{equation*}
    \tilde{p}(\sigma, \sigma') = \lim_{\beta \rightarrow \infty} p_{\beta}(\sigma, \sigma'), \quad \sigma, \sigma' \in S.
\end{equation*}
We call a configuration $\sigma' \in S$ a \textit{downhill configuration} of a configuration $\sigma \in S$ if there exists a sequence $\omega = (\sigma_0, \sigma_1, \ldots, \sigma_{\ell})$ for some $\ell \in \mathbb{N}$ that satisfies $\sigma_0 = \sigma$, $\sigma_{\ell} = \sigma'$, $H(\sigma_{t+1}) \leq H(\sigma_t)$ and $\sum\limits_{i \in V}|\sigma_{t+1}(i) - \sigma_t(i)| \leq 1$ for all $t = 1, \ldots, \ell -1$. Let the probability of following this sequence under the low-temperature Metropolis dynamics be denoted by $\tilde{p}(\omega)$, i.e.,
\begin{equation*}
    \tilde{p}(\omega) = \prod\limits_{t = 1}^{\ell-1} \tilde{p}(\sigma_t, \sigma_{t+1}).
\end{equation*}
Note that, using this formulation, a configuration $\sigma' \in S$ is a downhill configuration of $\sigma \in S$ if and only if there exists a sequence $\omega = (\sigma_0, \sigma_1, \ldots, \sigma_{\ell})$ for some $\ell \in \mathbb{N}$ such that $\sigma_0 = \sigma$, $\sigma_{\ell} = \sigma'$ and $\tilde{p}(\omega) > 0$. Similarly, we define $\sigma' \in S$ to be an \textit{$\ell$-step downhill configuration} of $\sigma \in S$ if there exists a sequence $\omega = (\sigma_0, \sigma_1, \ldots, \sigma_{\ell})$ of specified length $\ell \in \mathbb{N}$ that satisfies $\sigma_0 = \sigma$, $\sigma_{\ell} = \sigma'$ and $H(\sigma_{t+1}) \leq H(\sigma_t)$ and $\sum\limits_{i \in V}|\sigma_{t+1}(i) - \sigma_t(i)| \leq 1$ for all $t = 0, \ldots, \ell-1$. Let $\Gamma(\sigma) \subseteq S$ denote the set of all downhill configurations of $\sigma \in S$. Similarly, let $\Gamma_{\ell}(\sigma) \subseteq S$ denote the set of all $\ell$-step downhill configurations of $\sigma \in S$, where $\ell \in \mathbb{N}$.

We call a sequence of configurations $\omega = (\sigma_0, \sigma_1, \ldots, \sigma_{\ell})$, $\sigma_0, \ldots, \sigma_{\ell} \in S$, a \textit{downhill path} if $H(\sigma_{t+1}) \leq H(\sigma_t)$ and $\sum\limits_{i\in V}|\sigma_{t+1}(i) - \sigma_t(i)| = 1$ for all $t = 0, \ldots, \ell-1$.  For any $\sigma, \eta \in S$, let $\Omega(\sigma, \eta)$ denote the collection of all downhill paths leading from configuration $\sigma$ to $\eta$. Furthermore, for any $\sigma \in S$, let $\Omega(\sigma)$ denote the set of all downhill paths starting at configuration $\sigma$. Finally, let $\Omega$ denote the set of all downhill paths. Hence, $\Omega = \cup_{\sigma \in S} \Omega(\sigma)$ and $\Omega(\sigma) = \cup_{\eta \in S} \Omega(\sigma, \eta)$, $\sigma \in S$. For a set $W \subseteq V$, let $W^*$ denote the set of sequences of any length with elements in $W$. Given a downhill path $\omega = (\sigma_0, \sigma_1, \ldots, \sigma_{\ell})$, let $\zeta(\omega) = (x_1, x_2, \ldots, x_{\ell}) \in V^*$ denote the corresponding sequence of spins that are flipped along this path, i.e., $x_k$ is the spin that was flipped to obtain configuration $\sigma_{k}$ from configuration $\sigma_{k-1}$, $k = 1, \ldots, \ell$. We now define $\Theta(\sigma, \eta)$, $\Theta(\sigma)$, $\Theta \subseteq V^*$, $\sigma, \eta \in S$, as the sets of \textit{downhill sequences} corresponding to the paths in $\Omega(\sigma, \eta)$, $\Omega(\sigma)$ and $\Omega$. Note that $\zeta: \Omega \rightarrow \Theta$ is a one-to-one mapping from the set of downhill paths to the set of sequences.

We call a spin $i \in V$ \textit{susceptible} in a configuration $\sigma \in S$ if $H(\sigma^i) \leq H(\sigma)$. Note that a spin $i$ is susceptible in $\sigma$ if and only if $\tilde{p}(\sigma, \sigma^i) > 0$. Recall that in our setting, we imposed a small external magnetic field $h \in (0,1)$. This implies that a spin in state $+1$ is susceptible if and only if it has at least three neighboring spins in state -1. A spin in state -1, on the other hand, is susceptible if and only if it has at least two neighboring spins in state +1. Figure \ref{susceptible_spin} shows an example of a susceptible spin in state $+1$. The set of susceptible spins in a configuration $\sigma \in S$ is denoted by $\Delta(\sigma)$.

\begin{figure}
\centering
\includegraphics[width=0.2\linewidth]{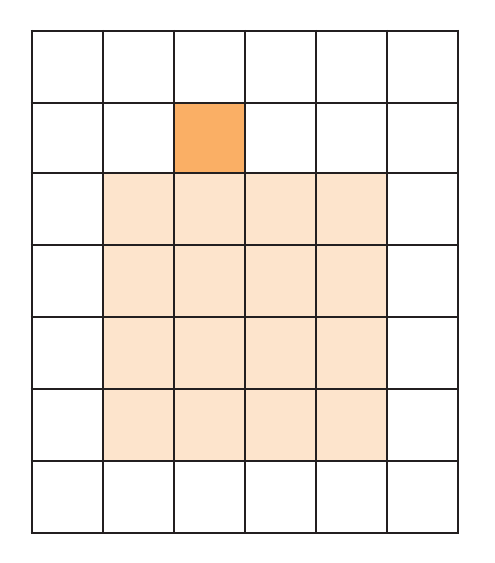}
\caption{Configuration drawn on the dual lattice. Each orange square corresponds to a spin in state $+1$ and each white square represents a spin in state $-1$. The dark orange spin, which occupies state $+1$, is susceptible. After all, it has three neighbors in state $-1$, so flipping this spin would lead to a lower energy configuration.}
\label{susceptible_spin}
\end{figure} 
 We call a configuration $\sigma \in S$ \textit{fragile} if at least one spin is susceptible in $\sigma$, that is, if $\Delta(\sigma) \neq \emptyset$. A configuration that is not fragile, we call \textit{robust}. Some examples of fragile and robust configurations are depicted in Figure \ref{fragile_vs_robust}. Let $U \subseteq S$ denote the set of robust configurations. We define $U^1$ to be the set of robust configurations that are not the all-minus configuration, denoted by $\sigma^-$, and in which the spins in state $+1$ form a single cluster. Theorem \ref{geom} characterizes these configurations geometrically as configurations in which the spins in state $+1$ form a rectangle. 

\begin{figure}
\centering
\includegraphics[width=0.7\linewidth]{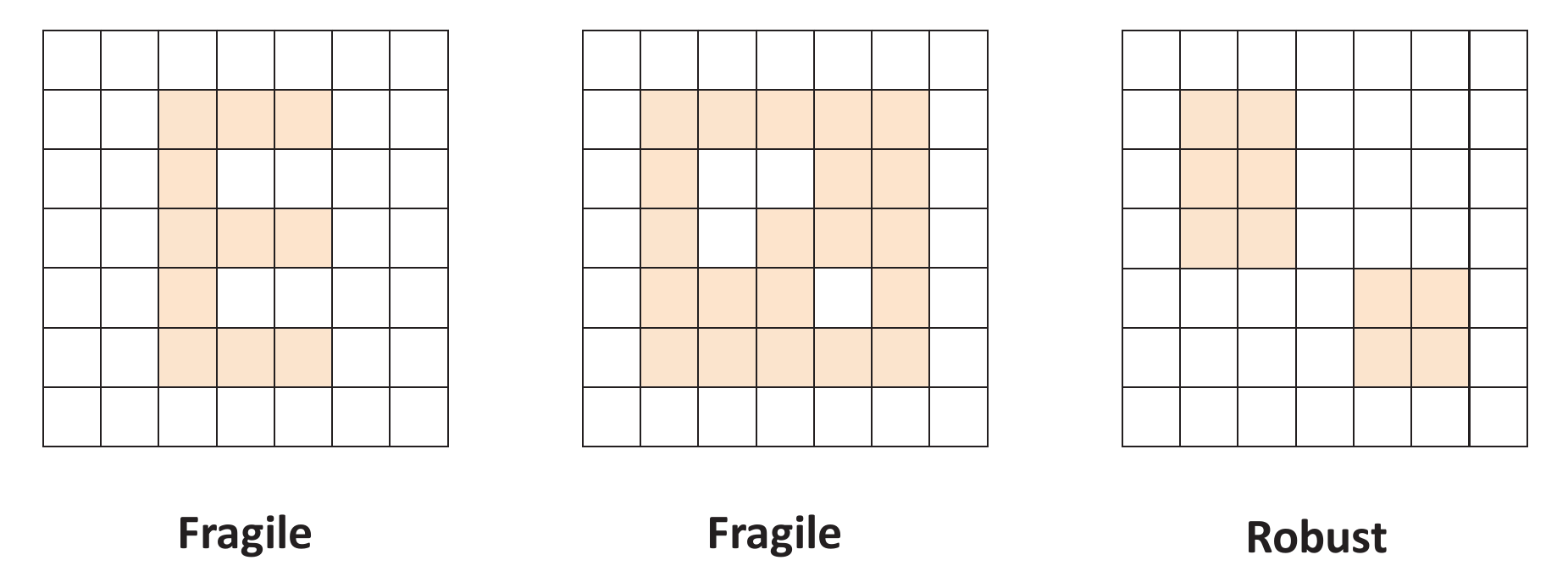}
\caption{Examples of fragile and robust configurations.}
\label{fragile_vs_robust}
\end{figure} 

Let $U(\sigma) \subseteq U$ denote the set of robust configurations that are downhill configurations of $\sigma \in S$, i.e., $U(\sigma) = U \cap \Gamma(\sigma)$. Accordingly, let $U_{\ell}(\sigma) = U \cap \Gamma_{\ell}(\sigma)$ denote the set of robust configurations that are $\ell$-step downhill configurations of $\sigma$ for some $\ell \in \mathbb{N}$. The sets $U^1(\sigma)$ and $U^1_{\ell}(\sigma)$ for $\sigma \in S$ and $\ell \in \mathbb{N}$ are defined in an analogous way, replacing $U$ with $U^1$. 

We call a downhill path $\omega = (\sigma_0, \sigma_1, \ldots, \sigma_{\ell})$, $\sigma_0, \ldots, \sigma_{\ell} \in S$, a \textit{closed downhill path} if it ends in a robust configuration, i.e., if $\sigma_{\ell} \in U$. We refer to the corresponding sequence of spins $\zeta(\omega)$ as a \textit{closed downhill sequence}. Let $\tilde{\Omega}(\sigma, \eta)$, $\tilde{\Omega}(\sigma)$ and $\tilde{\Omega}$ denote the sets of closed downhill paths, defined analogously to $\Omega(\sigma, \eta)$, $\Omega(\sigma)$ and $\Omega$, and let $\tilde{\Theta}(\sigma, \eta)$, $\tilde{\Theta}(\sigma)$ and $\tilde{\Theta}$ denote the corresponding sets of closed downhill sequences.   

Given a sequence $\mathbf{x} \in V^*$ and a configuration $\sigma \in S$, let $\sigma^{\mathbf{x}}$ denote the configuration obtained after consecutively flipping the spins in $\mathbf{x}$ from $\sigma$. For a set $W \subseteq V$ and a sequence $\mathbf{x} \in V^*$, let $\mathbf{x}^W$ denote the subsequence of $\mathbf{x}$ that consists only of those elements of $\mathbf{x}$ that are part of $W$. We refer to $\mathbf{x}^W$ as the \textit{restriction of} $\mathbf{x}$ to $W$. We now call a set $W \subseteq V$ \textit{essential} for a configuration $\sigma \in S$ if it satisfies the following two requirements:

\begin{itemize}
    \item All closed downhill sequences $\mathbf{x} \in \tilde{\Theta}(\sigma)$ that have the same restriction to $W$ lead to the same robust configuration, i.e., if $\mathbf{x}^W = \mathbf{y}^W$, for $\mathbf{x}, \mathbf{y} \in \tilde{\Theta}(\sigma)$, then $\sigma^{\mathbf{x}} = \sigma^{\mathbf{y}}$.
    \item The susceptibility of the spins in $W$ only depends on the evolution of the configuration on $W$, that is, for any downhill sequence $\mathbf{x} \in \Theta(\sigma)$, we have $\Delta(\sigma^{\mathbf{x}}) \cap W = \Delta(\sigma^{\mathbf{x}^W}) \cap W$. 
\end{itemize}
The above properties imply that the robust configuration at the end of a closed downhill path is fully determined by the evolution of the spins in the essential set. Furthermore, the dynamics governing the spins in the essential set, starting from the configuration $\sigma$, is also fully specified by the evolution of the spins in the set. As a consequence, when computing transition probabilities to the next robust configuration from a configuration $\sigma$, it suffices to focus on the essential set for $\sigma$, which greatly simplifies computations.

\subsection{Analysis of the low-temperature Ising STMDP}
\label{Analysis of the low-temperature Ising STMDP}
Using the concepts defined in Section \ref{Additional notation and terminology}, we now proceed to formalize the analysis of the low-temperature Ising STMDP. First, observe that for any $\ell \in \mathbb{N}$, we have $\tilde{P}_{\ell}(\sigma'|\sigma, d(\sigma)) = 0$ for all $\sigma' \notin \Gamma_{\ell}(\sigma^{d(\sigma)})$. Recalling in addition that the value function is finite, due to the fact that the reward function is bounded, the equations in expression (\ref{low_temp_value_function_equation}) reduce to 
\begin{equation}\label{low_temp_value_function_equation_simp}
    \tilde{v}^{\pi}_{\lambda, \kappa}(\sigma) = r(\sigma, d(\sigma)) + \lambda^{\kappa} \sum\limits_{\sigma' \in \Gamma_{\kappa}(\sigma^{d(\sigma)})}\tilde{P}_{\kappa}(\sigma'|\sigma, d(\sigma))\tilde{v}^{\pi}_{\lambda, \kappa}(\sigma'), \quad \sigma \in S.
\end{equation}
 We show that dropping the contributions of fragile configurations in the summation in expression (\ref{low_temp_value_function_equation_simp}) in fact leads to a valid approximation of the low-temperature value function. To this end, we define a transition probability kernel $Q_{\kappa}: S \times A \times S \rightarrow [0,1]$ by conditioning on the event that $\sigma_{\kappa} \in U$, given a starting configuration $\sigma \in S$ and action $a \in A$, i.e.,
\begin{equation}\label{Q}
    Q_{\kappa}(\sigma'|\sigma, a) := \begin{cases}
        \dfrac{\tilde{P}_{\kappa}(\sigma'|\sigma, a)}{\sum\limits_{\sigma'' \in U}\tilde{P}_{\kappa}(\sigma''|\sigma, a)}, &\text{if } \sigma' \in U,\\
        0, &\text{otherwise,}
    \end{cases} 
\end{equation}
for $\sigma, \sigma' \in S$, $a \in A$. 
In analogy to expression (\ref{operator_F}), we define the operator $\mathcal{G}_{\lambda, \kappa}^{\pi}: \mathbb{R}^{|S|} \rightarrow \mathbb{R}^{|S|}$ as
\begin{equation}\label{operator_G}
\mathcal{G}^{\pi}_{\lambda, \kappa}\mathbf{x} := \mathbf{r}^{\pi} + \lambda^{\kappa} \mathbf{Q}^{\pi}_{\kappa} \mathbf{x}, \quad \textbf{x} \in \mathbb{R}^{|S|}.
\end{equation}
Theorem \ref{Thm_P_to_Q} asserts that the fixed point of the operator $\mathcal{G}^{\pi}_{\lambda, \kappa}$ can approach the low-temperature value function $\tilde{v}^{\pi}_{\lambda, \kappa}$ arbitrarily closely and that the accuracy of this approximation can be controlled by selecting a sufficiently long adjustment time $\kappa$. 
\begin{theorem}\label{Thm_P_to_Q}
Consider the operators $\mathcal{F}^{\pi}_{\lambda, \kappa}$ and $\mathcal{G}^{\pi}_{\lambda, \kappa}$ defined in expressions (\ref{operator_F}) and (\ref{operator_G}) respectively. For all $\varepsilon > 0$, there exists a positive integer $K$ such that for all $\kappa \geq K$, we have
\begin{equation*}
    ||\mathcal{F}_{\lambda, \kappa}^{\pi}\mathbf{x} - \mathcal{G}_{\lambda, \kappa}^{\pi}\mathbf{x}||_{\infty} < \varepsilon ||\mathbf{x}||_{\infty} \quad \text{for all } \mathbf{x} \in \mathbb{R}^{|S|} \text{ and } \pi \in \Pi.
\end{equation*}
\end{theorem}
\begin{proof}
    Let $\pi = d^{\infty} \in \Pi$, $\mathbf{x} \in \mathbb{R}^{|S|}$ and $\varepsilon > 0$. We start by showing that there exists a positive integer $K$ such that for all $\kappa \geq K$, we have 
    \begin{equation}\label{ineq_P_E_kappa_sigma_a}
        \sum\limits_{\sigma' \in U}\tilde{P}_{\kappa}(\sigma'|\sigma, a) \geq 1 - \dfrac{\varepsilon}{2} \quad \text{for all } \sigma \in S, a \in A. 
    \end{equation}
    First, observe that $0<h<1$ implies that flipping a susceptible spin strictly decreases the energy of a configuration. Since the energy cannot drop below its global minimum $H_{\text{min}} = -N^2(1+h)$, each closed downhill path has finite length. Let $L$ denote the maximum length of such a path. Note that the probability of transitioning from a fragile configuration $\sigma \in S$ to a different configuration in the low-temperature limit is given by
    \begin{equation*}
        \sum\limits_{\substack{\sigma' \in S \\ \sigma' \neq \sigma}} \tilde{p}(\sigma, \sigma') = \dfrac{|\Delta(\sigma)|}{N^2}. 
    \end{equation*}
    Since each fragile configuration has at least 1 susceptible spin, the probability of taking a step along a downhill path is at least $1/N^2$. The minimum number of steps necessary to reach a robust configuration is at most $L$. Hence, we can lower bound the expression $\sum\limits_{\sigma' \in U}\tilde{P}_{\kappa}(\sigma'|\sigma, a)$ by the probability of obtaining at least $L$ successes in $\kappa$ Bernoulli trials with success probability $1/N^2$. That is,
    \begin{equation*}
        \sum\limits_{\sigma' \in U}\tilde{P}_{\kappa}(\sigma'|\sigma, a) \geq \mathbb{P}(X \geq L), \quad \text{for all } \sigma \in S, \quad a \in A,
    \end{equation*}
    where $X \sim \text{Bin}\left(\kappa, 1/N^2 \right)$ and $\kappa \geq L$. Hence,
    \begin{align*}
        \sum\limits_{\sigma' \in U}\tilde{P}_{\kappa}(\sigma'|\sigma, a) &\geq 1 - \mathbb{P}(X<L) = 1 - \sum\limits_{i=0}^{L-1} \binom{\kappa}{i}\left(\dfrac{1}{N^2}\right)^i\left(1-\dfrac{1}{N^2}\right)^{\kappa - i}.
    \end{align*}
    For a given $i = 0, \ldots, L-1$, we have
    \begin{equation*}
        0 \leq \lim_{\kappa \rightarrow \infty} \binom{\kappa}{i}\left(\dfrac{1}{N^2}\right)^i\left(1-\dfrac{1}{N^2}\right)^{\kappa - i} \leq \lim_{\kappa \rightarrow \infty} \dfrac{\kappa^i}{i!}\left(\dfrac{1}{N^2}\right)^ie^{(\kappa - i)\log{(1-1/N^2)}} = 0.
    \end{equation*}
    Hence,
    \begin{equation*}
        \lim_{\kappa \rightarrow \infty} \left[1 - \sum\limits_{i=0}^{L-1} \binom{\kappa}{i}\left(\dfrac{1}{N^2}\right)^i\left(1-\dfrac{1}{N^2}\right)^{\kappa - i}\right] = 1.
    \end{equation*}
    This implies that there exists a positive integer $K \geq L$ such that expression (\ref{ineq_P_E_kappa_sigma_a}) is satisfied for all $\kappa \geq K$.  
    
    We have
    \begin{equation*}
        ||\tilde{\mathbf{P}}^{\pi}_{\kappa} - \mathbf{Q}^{\pi}_{\kappa}||_{\infty} = \max_{\sigma \in S} \sum\limits_{\sigma' \in S} \left| \tilde{P}^{\pi}_{\kappa}(\sigma'|\sigma, d(\sigma)) - Q^{\pi}_{\kappa}(\sigma'|\sigma, d(\sigma)) \right|.
    \end{equation*}
    Since, $\tilde{P}^{\pi}_{\kappa}(\sigma'|\sigma, d(\sigma)) = Q^{\pi}_{\kappa}(\sigma'|\sigma, d(\sigma)) = 0$ for all $\sigma \in S$, $\sigma' \in S \setminus \Gamma_{\kappa}(\sigma^{d(\sigma)})$ and $Q^{\pi}_{\kappa}(\sigma'|\sigma, d(\sigma)) = 0$ for all $\sigma \in S$ and $\sigma' \in \Gamma_{\kappa}(\sigma^{d(\sigma)}) \setminus U_{\kappa}(\sigma^{d(\sigma)})$, we obtain
    \begin{align}\label{inf_norm}
         ||\tilde{\mathbf{P}}^{\pi}_{\kappa} - \mathbf{Q}^{\pi}_{\kappa}||_{\infty} &= \max_{\sigma \in S} \sum\limits_{\sigma' \in \Gamma_{\kappa}(\sigma^{d(\sigma)})} \left| \tilde{P}^{\pi}_{\kappa}(\sigma'|\sigma, d(\sigma)) - Q^{\pi}_{\kappa}(\sigma'|\sigma, d(\sigma)) \right| \\
        \nonumber&= \max_{\sigma \in S} \Big\{ \sum\limits_{\sigma' \in U_{\kappa}(\sigma^{d(\sigma)})} \left| \tilde{P}^{\pi}_{\kappa}(\sigma'|\sigma, d(\sigma)) - Q^{\pi}_{\kappa}(\sigma'|\sigma, d(\sigma)) \right| + \sum\limits_{\sigma' \in \Gamma_{\kappa}(\sigma^{d(\sigma)}) \setminus U_{\kappa}(\sigma^{d(\sigma)})} \tilde{P}^{\pi}_{\kappa}(\sigma'|\sigma, d(\sigma)) \Big\}.
    \end{align}
    For the first summation, using expression (\ref{Q}) yields
    \begin{align*}
        &\sum\limits_{\sigma' \in U_{\kappa}(\sigma^{d(\sigma)})} \left| \tilde{P}^{\pi}_{\kappa}(\sigma'|\sigma, d(\sigma)) - Q^{\pi}_{\kappa}(\sigma'|\sigma, d(\sigma)) \right| \\
        &=  \sum\limits_{\sigma' \in U_{\kappa}(\sigma^{d(\sigma)})} \left| \sum\limits_{\sigma'' \in U}\tilde{P}_{\kappa}(\sigma''|\sigma, d(\sigma))Q^{\pi}_{\kappa}(\sigma'|\sigma, d(\sigma)) - Q^{\pi}_{\kappa}(\sigma'|\sigma, d(\sigma))\right| \\
        \nonumber&= 1-\sum\limits_{\sigma' \in U}\tilde{P}_{\kappa}(\sigma'|\sigma, d(\sigma)).
    \end{align*}
    Choosing $\kappa \geq K$, we now obtain
    \begin{equation*}
         \sum\limits_{\sigma' \in U_{\kappa}(\sigma^{d(\sigma)})} \left| \tilde{P}^{\pi}_{\kappa}(\sigma'|\sigma, d(\sigma)) - Q^{\pi}_{\kappa}(\sigma'|\sigma, d(\sigma)) \right| \leq \dfrac{\varepsilon}{2}.
    \end{equation*}
    As for the second summation in expression (\ref{inf_norm}), we have, again for $\kappa \geq K$, 
    \begin{equation*}
        \sum\limits_{\sigma' \in \Gamma_{\kappa}(\sigma^{d(\sigma)}) \setminus U_{\kappa}(\sigma^{d(\sigma)})} \tilde{P}^{\pi}_{\kappa}(\sigma'|\sigma, d(\sigma)) = 1-\sum\limits_{\sigma' \in U}\tilde{P}_{\kappa}(\sigma'|\sigma, d(\sigma)) \leq \dfrac{\varepsilon}{2}. 
    \end{equation*}
    Hence, we obtain
    \begin{equation*}
        ||\tilde{\mathbf{P}}^{\pi}_{\kappa} - \mathbf{Q}^{\pi}_{\kappa}||_{\infty} \leq \varepsilon.
    \end{equation*}
    This now implies
    \begin{equation*}
        ||\mathcal{F}^{\pi}_{\lambda, \kappa}\mathbf{x} - \mathcal{G}^{\pi}_{\lambda, \kappa}\mathbf{x}||_{\infty} = ||\lambda^{\kappa}(\tilde{\mathbf{P}}^{\pi}_{\kappa} - \mathbf{Q}^{\pi}_{\kappa})\mathbf{x}||_{\infty} \leq \lambda^{\kappa}||\tilde{\mathbf{P}}^{\pi}_{\kappa} - \mathbf{Q}^{\pi}_{\kappa}||_{\infty} \cdot ||\mathbf{x}||_{\infty} < \varepsilon ||\mathbf{x}||_{\infty},
    \end{equation*}
    for all $\kappa \geq K$. 
\end{proof}

\subsubsection{The auxiliary MDP}
\label{The auxiliary MDP}
 Theorem \ref{Thm_P_to_Q} suggests that the fixed point of the operator $\mathcal{G}^{\pi}_{\lambda, \kappa}$ for some $\pi \in \Pi$ provides an accurate approximation of the low-temperature value function $\tilde{v}^{\pi}_{\lambda, \kappa}$ for a sufficiently long adjustment time~$\kappa$. Armed with this result, we construct an auxiliary MDP that serves as a caricature version of the original Ising STMDP and ranges only over the all-minus configuration, denoted by $\sigma^-$ together with the set $U^1$, i.e., the set of robust configurations in which the spins in state $+1$ form a single cluster. The following result provides a geometrical characterization of such robust configurations. 

\begin{theorem}\label{geom}
A configuration $\sigma \in S$ belongs to $U^1$ if and only if $\sigma \neq \sigma^-$ and the spins in state $+1$ form a rectangle of size $i \times j$, for some $i,j \in \{2, 3, \ldots, N-3, N-2, N\}$.     
\end{theorem}
\begin{proof}
    A configuration $\sigma \neq \sigma^-$ is robust if and only if each spin in state $-1$ has at most one neighbor in state $+1$ and each spin in state $+1$ has at least two neighbors in state $+1$. A configuration in which the spins in state $+1$ form one cluster satisfies these requirements if and only if these spins form a rectangle of size $i \times j$, where $i,j \in \{2, 3, \ldots, N-3, N-2, N\}$. 
\end{proof}
We now formalize the construction of the auxiliary MDP $(\hat{S}, \hat{A}, \hat{P}, \hat{r}$) described in Section~\ref{Definition of the Ising STMDP}. Recall that the state space  $\hat{S}$ consist of 2-dimensional vectors representing the size of a rectangle in the Ising STMDP, i.e.,
\begin{equation*}
    \hat{S} = \{(i,j)|i,j = 2, 3, \ldots, N-3, N-2, N\} \cup \{(0,0)\}.
\end{equation*}
Here, the vector $(0,0)$ corresponds to the all-minus configuration. To formally relate the original low-temperature Ising STMDP to the auxiliary MDP, we define a mapping $\mathcal{I}: U^1 \cup \{\sigma^-\} \rightarrow \hat{S}$ as $\mathcal{I}(\sigma) = (i,j)$ if $\sigma \in U^1$ and the spins in state $+1$ in $\sigma$ form a rectangle of size $i \times j$ and $\mathcal{I}(\sigma^-) = (0, 0)$. Given $(i,j) \in \hat{S}$, we now let $\sigma_{(i,j)} \in \mathcal{I}^{-1}((i,j))$, without loss of generality, denote the configuration in $\mathcal{I}^{-1}((i,j))$ in which the lower left corner of the rectangle is located at the origin of the lattice. 
The action space $\hat{A}_{(i,j)}$ corresponding to a state $(i,j) \in \hat{S}$ is defined as
\begin{equation*}
    \hat{A}(i,j) = \begin{cases}
        \{0\}, &\text{if } i = j = 0 \text{ or } i = j = N,\\
        \{a_{11}', a_{12}', a_{21}', a_{22}', a_0, \tilde{a}, 0\}, &\text{if } i = j = 2, \\
        \{a_{11}', a_{12}', a_{21}, a_{22}, a_{21}', a_{22}', a_0, \tilde{a}, 0\}, &\text{if } i = 2,\text{ } 3 \leq j \leq N-3, \\
        \{a_{11}', a_{21}, a_{22}, a_{21}', a_{22}', a_0, \tilde{a}, 0\}, &\text{if } i = 2,\text{ } j = N-2,\\
        \{a_{11}, a_{12}, a_{11}', a_{12}', a_{21}', a_{22}', a_0, \tilde{a}, 0\}, &\text{if } 3 \leq i \leq N-3, \text{ }j = 2, \\
        \{a_{11}, a_{12}, a_{11}', a_{12}', a_{21}', a_0, \tilde{a}, 0\}, &\text{if } i = N-2, \text{ }j = 2,\\
        \{a_{11}, a_{12}, a_{11}', a_{12}', a_{21}, a_{22}, a_{21}', a_{22}', a_0, \tilde{a}, 0\}, &\text{if } 3 \leq i, j \leq N-3, \\
        \{a_{11}, a_{12}, a_{11}', a_{12}', a_{21}, a_{21}', a_0, \tilde{a}, 0\},  &\text{if } i = N-2, \text{ }3 \leq j \leq N-3, \\
        \{a_{11}, a_{11}', a_{21}, a_{22}, a_{21}', a_{22}', a_0, \tilde{a}, 0\}, &\text{if }3 \leq i \leq N-3, \text{ }j = N-2, \\
        \{a_{11}, a_{11}', a_{21}, a_{21}', a_0, \tilde{a}, 0\},  &\text{if } i = j = N-2, \\
        \{a_{21}, 0\}, &\text{if } i = N-2, \text{ }j = N, \\
        \{a_{11}, a_{12}, 0\}, &\text{if } i = N,\text{ } 2 \leq j \leq N-3, \\
        \{a_{11}, 0\}, &\text{if } i = N, \text{ }j = N-2, \\
        \{a_{21}, a_{22}, 0\}, &\text{if } 2 \leq i \leq N-3,\text{ } j = N.\\
    \end{cases}
\end{equation*}
To clarify the connection to the original STMDP, we define a mapping $\mathcal{J}_{\sigma}: A \rightarrow \hat{A}_{\mathcal{I}(\sigma)}$ for each $\sigma \in U^1 \cup \{\sigma^-\}$ as 
\begin{equation*}
    \mathcal{J}_{\sigma}(a) = \begin{cases}
        a_{1\ell}, &\text{if } \tilde{\delta}_v(a, V_+(\sigma)) = \ell \text{ and } \mu_v(a, V_+(\sigma)) \notin C(\sigma),\\
        a_{1\ell}', &\text{if } \tilde{\delta}_v(a, V_+(\sigma)) = \ell \text{ and } \mu_v(a, V_+(\sigma)) \in C(\sigma),\\
        a_{2\ell}, &\text{if } \tilde{\delta}_h(a, V_+(\sigma)) = \ell \text{ and } \mu_h(a, V_+(\sigma)) \notin C(\sigma),\\
        a_{2\ell}', &\text{if } \tilde{\delta}_h(a, V_+(\sigma)) = \ell \text{ and } \mu_h(a, V_+) \in C(\sigma), \\
        a_0,  &\text{if } \tilde{\delta}(a, V_+(\sigma)) = 2 \text{ and } \tilde{\delta}_h(a, V_+(\sigma)) = \tilde{\delta}_v(a, V_+(\sigma)) = \infty,\\
        \tilde{a}, &\text{if } a \in C(\sigma), \\
        0, &\text{ otherwise,}
    \end{cases} 
\end{equation*}
for $\ell = 1, 2$, $a \in A$. Hence, $a_{1\ell}$ and $a_{2\ell}$, $\ell = 1, 2$, correspond to the actions of flipping a spin at distance $\ell$ from the horizontal and the vertical side of the rectangle, for which the closest spin belonging to the rectangle is not a corner spin of the rectangle. Actions $a_{1\ell}'$, $a_{2\ell}'$, $\ell = 1, 2$, correspond to the actions of flipping a spin at distance $\ell$ from the horizontal and the vertical side of the rectangle, where the closest spin belonging to the rectangle is a corner spin. The action $a_0$ represents the action of flipping a spin that is diagonally adjacent to the rectangle. The action $\tilde{a}$ corresponds to the action of flipping a corner spin of the rectangle. The action $0$ represents flipping any other spin or doing nothing. 
We refer to Figure \ref{auxiliary_action_space} for a visualization of the action space.

Recall that the reward function of the auxiliary process is given by 
\begin{equation*}
    \hat{r}(s, a) = \begin{cases}
        1, &\text{ if } s = (N, N), \\
        0, &\text{ otherwise},
    \end{cases}
\end{equation*}
for all $s \in \hat{S}$, $a \in \hat{A}$

Finally, we define the transition probability kernel $\hat{P}: \hat{S} \times \hat{A} \times \hat{S} \rightarrow [0,1]$, which is based on the the transition probability kernel $Q_{\kappa}: S \times A \times S \rightarrow [0,1]$ given by expression (\ref{Q}). To this end, we first provide some preliminary results. The following statement asserts that for any configuration $\sigma \in U^1$, action $a \in A$ and state $(i',j') \in \hat{S}$, there is at most one robust downhill configuration of $\sigma^a$ that corresponds to state $(i',j')$. 
\newpage
\begin{lemma}
    For any configuration $\sigma \in U^1$, action $a \in A$ and state $(i',j') \in \hat{S}$, the set $U^1(\sigma^a) \cap \mathcal{I}^{-1}((i',j'))$ is either empty or a singleton.
\end{lemma}

\begin{proof}
    Let $\sigma \in U^1$ and suppose that $\mathcal{I}(\sigma) = (i,j)$. Consider the case that $\mathcal{J}_{\sigma}(a) = 0$. This implies that $a \in \{0\} \cup [V_+(\sigma)\setminus C(\sigma)] \cup \{i \in V|\tilde{\delta}(a, V_+(\sigma)) > 2\}$. Suppose first that $a = 0$. In this case, we have $U^1(\sigma^a) \cap \mathcal{I}^{-1}((i',j')) = \{\sigma\}$ if $(i',j') = (i,j)$ and $U^1(\sigma^a) \cap \mathcal{I}^{-1}((i',j')) = \emptyset$ otherwise. Suppose now that either $a \in V_+(\sigma) \setminus C(\sigma)$ or $a$ satisfies $\tilde{\delta}(a, V_+(\sigma)) > 2$. In both cases, the only downhill path starting at configuration $\sigma^a$ is $(\sigma^a, \sigma)$, so we again obtain $U^1(\sigma^a) \cap \mathcal{I}^{-1}((i',j')) = \{\sigma\}$ if $(i',j') = (i,j)$ and $U^1(\sigma^a) \cap \mathcal{I}^{-1}((i',j')) = \emptyset$ otherwise. 

    We proceed to consider the case $\mathcal{J}_{\sigma}(a) = \tilde{a}$. If $i, j \geq 3$, then the only downhill path starting at configuration $\sigma^a$ is $(\sigma^a, \sigma)$, which again yields $U^1(\sigma^a) \cap \mathcal{I}^{-1}((i',j')) = \{\sigma\}$ if $(i',j') = (i,j)$ and $U^1(\sigma^a) \cap \mathcal{I}^{-1}((i',j')) = \emptyset$ otherwise. Suppose now that $a \in C(\sigma)$ and either $i = 2$, $j > 2$ or $i > 2$, $j = 2$. Without loss of generality, we consider the case $i = 2$ and $j > 2$. In this situation, $\sigma^a$ has two susceptible spins, namely $a$ itself and the unique horizontal neighbor $a'$ of $a$ that is in state $+1$. Flipping spin $a$ leads back to configuration $\sigma$, whereas flipping spin $a'$ leads to the robust configuration $\sigma^{\{a, a'\}}$, which satisfies $\mathcal{I}(\sigma^{\{a, a'\}}) = (i, j-1)$. Hence, we obtain 
        \begin{equation*}
            U^1(\sigma^a) \cap \mathcal{I}^{-1}((i',j')) = \begin{cases}
                \{\sigma\}, &\text{if } (i',j') = (i,j), \\
                \{\sigma^{\{a, a'\}}\}, &\text{if } (i',j') = (i, j-1), \\
                \emptyset, &\text{otherwise}.
            \end{cases}
        \end{equation*}
        Suppose that $i = j = 2$. In this case, $\sigma^a$ has three susceptible spins, namely $a$, the unique horizontal neighbor $a'$ of $a$ that is in state $+1$ and the unique vertical neighbor $a''$ of $a$ that is in state $+1$. If either spin $a'$ or spin $a''$ flips first from $\sigma^a$, any closed downhill path will end in the all-minus configuration $\sigma^-$, which corresponds to state $(0,0)$. On the other hand, flipping spin $a$ first leads back to configuration $\sigma$. Thus, we have 
        \begin{equation*}
            U^1(\sigma^a) \cap \mathcal{I}^{-1}((i',j')) = \begin{cases}
                \{\sigma\}, &\text{if } (i',j') = (i,j), \\
                \{\sigma^{-}\}, &\text{if } (i',j') = (0,0), \\
                \emptyset, &\text{otherwise}.
            \end{cases}
        \end{equation*} 
    Finally, consider the situation $\mathcal{J}_{\sigma}(a) \neq 0$. In this case, the only spins that can flip on a downhill path starting at configuration $\sigma^a$ are those in the set $R(\sigma^a) \setminus V_+(\sigma)$, i.e., the spins in the circumscribed rectangle of the set of $+$-spins in $\sigma^a$ which are not part of the rectangle in configuration $\sigma$. We consider $\mathcal{J}_{\sigma}(a) = a_{12}$. The remaining cases can be dealt with in a similar way. Let $b_{11}(\sigma)$ and $b_{12}(\sigma)$ denote the horizontal sets of spins lying at distance 1 and 2 from the rectangle in $\sigma$ respectively, as depicted in Figure \ref{Slices}. 
    
     \begin{figure}
     \centering
        \includegraphics[width=0.31\linewidth]{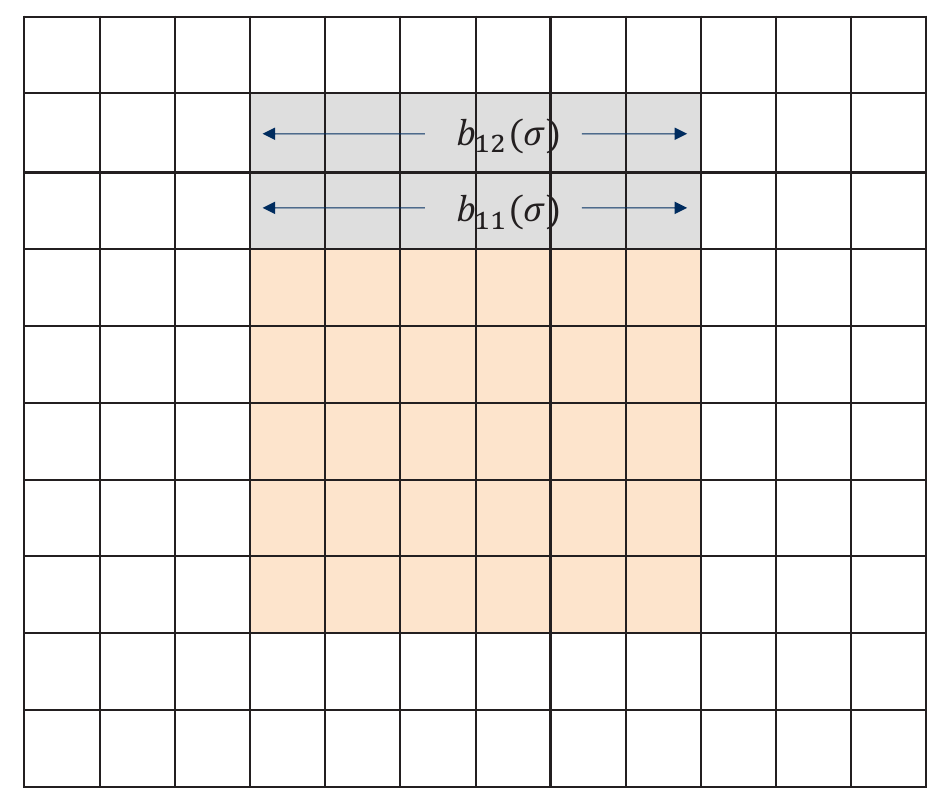}
        \caption{Sets $b_{11}(\sigma)$ and $b_{12}(\sigma)$.}
        \label{Slices}
    \end{figure}

    We have $R(\sigma^a) \setminus V_+(\sigma) = b_{11}(\sigma) \cup b_{12}(\sigma)$, so the only spins that can flip on a downhill path starting at configuration $\sigma^a$ are those in the set $b_{11}(\sigma) \cup b_{12}(\sigma)$. This implies that the only robust endpoints of downhill paths starting at $\sigma^a$ are the configurations $\sigma$ itself, $\sigma^{b_{11}(\sigma)}$ and $\sigma^{b_{11}(\sigma) \cup b_{12}(\sigma)}$. Thus, we obtain $U^1(\sigma^a) \cap \mathcal{I}^{-1}((i',j')) = \{\sigma\}$ if $(i',j') = (i,j)$, $U^1(\sigma^a) \cap \mathcal{I}^{-1}((i',j')) = \{\sigma^{b_{11}(\sigma)}\}$ if $(i',j') = (i, j+1)$, $U^1(\sigma^a) \cap \mathcal{I}^{-1}((i',j')) = \{\sigma^{b_{11}(\sigma) \cup b_{12}(\sigma)}\}$ if $(i',j') = (i, j+2)$ and $U^1(\sigma^a) \cap \mathcal{I}^{-1}((i',j')) = \emptyset$ otherwise. Analogous results for $\mathcal{J}_{\sigma}(a) \in \{a_{11}, a_{21}, a_{22}, a_{11}', a_{12}', a_{21}', a_{22}', a_0\}$ can be obtained in a similar way. It follows that the set $U^1(\sigma^a) \cap \mathcal{I}^{-1}((i',j'))$ is either empty or a singleton for any $\sigma \in U^1$ and $a \in A$.
\end{proof}

If the set $U^1(\sigma_{(i,j)}^a) \cap \mathcal{I}^{-1}((i',j'))$, for $(i,j), (i',j') \in \hat{S}$ and $a \in A$, is nonempty, let the unique configuration in this set be denoted by $\eta_{\{(i,j), a, (i',j')\}}$.

\newpage
\begin{lemma}\label{map_actions}
    Let $(i,j), (i',j') \in \hat{S}$ and $a, a' \in A$. If $\mathcal{J}_{\sigma_{(i,j)}}(a) = \mathcal{J}_{\sigma_{(i,j)}}(a')$, then 
    \begin{equation*}
        \lim_{\kappa \rightarrow \infty} Q_{\kappa}(\eta_{\{(i,j), a, (i',j')\}}|\sigma_{(i,j)}, a) = \lim_{\kappa \rightarrow \infty} Q_{\kappa}(\eta_{\{(i,j), a', (i',j')\}}|\sigma_{(i,j)}, a').
    \end{equation*}
\end{lemma}
\begin{proof}
    By symmetry, the statement is obvious for \[
\mathcal{J}_{\sigma_{(i,j)}}(a) = \mathcal{J}_{\sigma_{(i,j)}}(a') 
\in \{a_{11}', a_{12}', a_{21}', a_{22}', \tilde{a}, a_0\}.
\] We consider $\mathcal{J}_{\sigma_{(i,j)}}(a) = \mathcal{J}_{\sigma_{(i,j)}}(a') = a_{12}$. The result for the remaining cases can be obtained in a very similar way. Let the spin $a$ be labeled as 1, let its vertical neighbor that is adjacent to the rectangle be labeled as 2 and let the horizontal neighbors of spins 1 and 2 be labeled as 5, 6 and 3, 4 respectively, as illustrated in Figure \ref{fig_trans_probs_3}. 
    \begin{figure}
    \centering
        \includegraphics[width=0.308\linewidth]{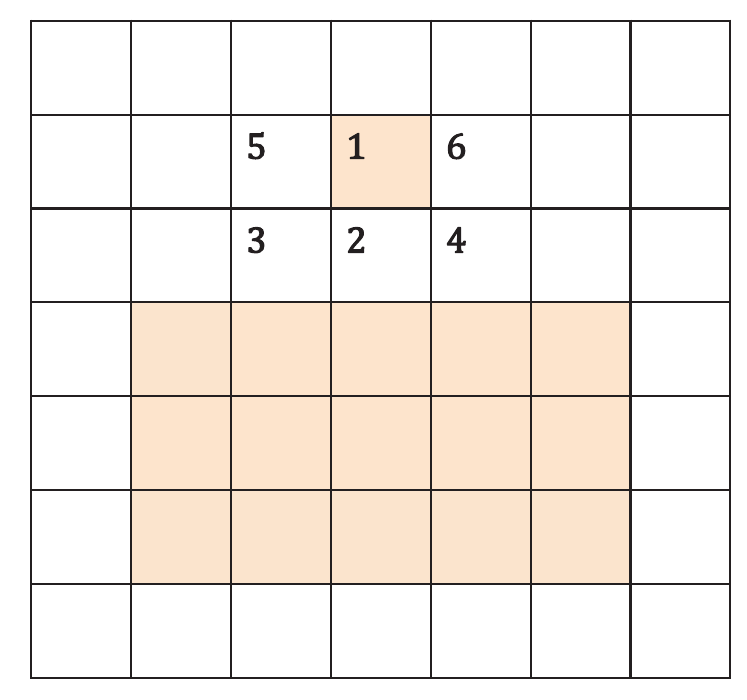}
        \caption{Post-decision configuration $\sigma_{(i,j)}^a$.}
        \label{fig_trans_probs_3}
    \end{figure} 
    Let $b_{11}(\sigma_{(i,j)})$ and $b_{12}(\sigma_{(i,j)})$ again denote the horizontal sets lying at distance 1 and 2 from the rectangle in $\sigma_{(i,j)}$ respectively, as depicted in Figure \ref{Slices}. Observe first of all that once two adjacent spins in one of these sets reach state $+1$, these spins can no longer flip to $-1$ on any downhill path. Furthermore, note that all remaining spins in this set will inevitably occupy state $+1$ in any robust configuration at the end of such a downhill path.   
    The only way to gain two $+$-spins in set $b_{11}(\sigma_{(i,j)})$ on a downhill path from configuration $\sigma_{(i,j)}^a$ is if spin 2 flips first, followed by either spin 3 or spin 4. The only way for set $b_{12}(\sigma_{(i,j)})$ to gain two spins on a downhill path from configuration $\sigma_{(i,j)}^a$ is through flipping either spin 5 or 6, which is only possible if spin 3 or 4 respectively has flipped to $+1$ and spin 1 has not flipped back in the mean time. It follows from these arguments that the robust configuration at the end of a closed downhill path is fully determined by the evolution of the spins in the set $W = \{1, 2, 3, 4, 5, 6\}$. That is, for any two closed downhill sequences $\mathbf{x}, \mathbf{y} \in \tilde{\Theta}(\sigma^a_{(i,j)})$ that satisfy $\mathbf{x}^W = \mathbf{y}^W$, we have $(\sigma^a_{(i,j)})^{\mathbf{x}} = (\sigma^a_{(i,j)})^{\mathbf{y}}$. In addition, the robust configuration we reach after taking action $a$ only depends on the rectangle through the spins that are adjacent to spins 2, 3 and 4. Since the configuration of the set of spins labeled 2, 3 and 4 is the same up to translation for each $a \in A$ that satisfies $\mathcal{J}_{\sigma_{(i,j)}}(a) = a_{12}$, we can conclude that
    \begin{equation*}
        \lim_{\kappa \rightarrow \infty} Q_{\kappa}(\eta_{\{(i,j), a, (i',j')\}}|\sigma_{(i,j)}, a) = \lim_{\kappa \rightarrow \infty} Q_{\kappa}(\eta_{\{(i,j), a', (i',j')\}}|\sigma_{(i,j)}, a')
    \end{equation*}
    if $\mathcal{J}_{\sigma_{(i,j)}}(a) = \mathcal{J}_{\sigma_{(i,j)}}(a') = a_{12}$.
\end{proof}
We now define the transition probability kernel $\hat{P}: \hat{S} \times \hat{A} \times \hat{S} \rightarrow [0,1]$ as
\begin{align}\label{hatP}
    \hat{P}((i',j')|&(i,j), \hat{a}) = \begin{cases}
        \lim\limits_{\kappa \rightarrow \infty} Q_{\kappa}(\eta_{\{(i,j), a, (i',j')\}}|\sigma_{(i,j)}, a), & \text{if } U^1(\sigma_{(i,j)}^a) \cap \mathcal{I}^{-1}((i',j')) \neq \emptyset, \\ 
        0, & \text{otherwise,}
    \end{cases}
\end{align}
for $(i,j), (i',j') \in \hat{S}$, $\hat{a} \in \hat{A}_{(i,j)}$, where $a \in \mathcal{J}^{-1}_{\sigma_{(i,j)}}(\hat{a})$. By Lemma \ref{map_actions}, each choice of $a \in \mathcal{J}^{-1}_{\sigma_{(i,j)}}(\hat{a})$ yields the same result and therefore expression (\ref{hatP}) is well-defined. 

\newpage
\noindent
The following result provides a way to compute the probabilities $\hat{P}((i',j')|(i,j), \hat{a})$ explicitly.

\begin{lemma}\label{Phat_formula}
    Consider $(i,j), (i',j') \in \hat{S}$ and $a \in A$ such that $U^1(\sigma^a_{(i,j)}) \cap \mathcal{I}^{-1}((i',j')) \neq~\emptyset$. Then, 
    \begin{align}\label{limQ}
        \lim_{\kappa \rightarrow \infty} Q_{\kappa}(&\eta_{\{(i,j), a, (i',j')\}}|\sigma_{(i,j)}, a) = \begin{cases}
            1, &  \text{if } \sigma_{(i,j)}^a = \eta_{\{(i,j), a, (i',j')\}}, \\
            \sum\limits_{\omega \in \Omega(\sigma_{(i,j)}^a, \eta_{\{(i,j), a, (i',j')\}})} \prod\limits_{\ell = 0}^{|\omega|-1} \dfrac{1}{|\Delta(\omega_{\ell})|}, & \text{otherwise}. 
        \end{cases}
    \end{align}
\end{lemma}

\begin{proof}
Note that a configuration $\sigma \in S$ is an absorbing state in the process that evolves according to this dynamics if and only if it is robust. Now, let $\tilde{p}^{\kappa}(\sigma, \sigma')$ denote the $\kappa$-step low-temperature Metropolis dynamics and define
\begin{equation*}
    \hat{p}(\sigma, \sigma') = \lim_{\kappa \rightarrow \infty} \tilde{p}^{\kappa}(\sigma, \sigma'), \quad \sigma, \sigma' \in S. 
\end{equation*}
It follows that $\hat{p}(\sigma, \sigma') > 0$ if and only if the configuration $\sigma'$ is robust. We first show that for any $\sigma, \sigma' \in S$ that satisfy $\sigma' \in U(\sigma)$ and $\sigma \neq \sigma'$, we have
\begin{equation}\label{hatp}
    \hat{p}(\sigma, \sigma') = \sum\limits_{\omega \in \Omega(\sigma, \sigma')} \prod\limits_{\ell = 0}^{|\omega|-1} \dfrac{1}{|\Delta(\omega_{\ell})|}.
\end{equation}
Recall that $\lim_{\beta \rightarrow \infty} p_{\beta}(\sigma, \sigma^i) = 0$ for all $i \notin \Delta(\sigma)$ and $\lim_{\beta \rightarrow \infty} p_{\beta}(\sigma, \sigma^i) = 1/N^2$ for all $i \in~\Delta(\sigma)$. Conditioning on the first time a susceptible spin, if any, is selected by the Metropolis dynamics during the adjustment period yields
\begin{equation*}
    \tilde{p}^{\kappa}(\sigma, \sigma') = \sum\limits_{t=1}^{\kappa} \left(1-\dfrac{|\Delta(\sigma)|}{N^2}\right)^{t-1} \dfrac{|\Delta(\sigma)|}{N^2} \sum\limits_{i \in \Delta(\sigma)} \dfrac{1}{|\Delta(\sigma)|} \tilde{p}^{\kappa-t}(\sigma^i, \sigma').
\end{equation*}
Taking the limit $\kappa \rightarrow \infty$, we now obtain
\begin{align*}
    \hat{p}(\sigma, \sigma') &= \sum\limits_{t=1}^{\infty} \left(1-\dfrac{|\Delta(\sigma)|}{N^2}\right)^{t-1} \dfrac{|\Delta(\sigma)|}{N^2} \sum\limits_{i \in \Delta(\sigma)} \dfrac{1}{|\Delta(\sigma)|} \hat{p}(\sigma^i, \sigma') = \sum\limits_{i \in \Delta(\sigma)}\dfrac{1}{|\Delta(\sigma)|)} \hat{p}(\sigma^i, \sigma').
\end{align*}
We show that this implies the validity of expression (\ref{hatp}) by induction over the length $L(\sigma, \sigma')$ of the longest downhill path that leads from $\sigma$ to $\sigma'$, i.e., $L(\sigma, \sigma') = \max_{\omega \in \Omega(\sigma, \sigma')}|\omega| -1$. Note that $\Omega(\sigma, \sigma') \neq \emptyset$ and $L(\sigma, \sigma') \geq 1$, since $\sigma' \in U(\sigma)$ and $\sigma \neq \sigma'$. For $L = 1$, we obtain
\begin{equation*}
    \hat{p}(\sigma, \sigma') = \sum\limits_{i \in \Delta(\sigma)}\dfrac{1}{|\Delta(\sigma)|}\hat{p}(\sigma^i, \sigma') = \sum\limits_{\substack{i \in \Delta(\sigma)\\ \sigma^i = \sigma'}} \dfrac{1}{|\Delta(\sigma)|} = \sum\limits_{\omega \in \Omega(\sigma, \sigma')} \prod\limits_{\ell = 0}^{|\omega|-1} \dfrac{1}{|\Delta(\omega_{\ell})|}.
\end{equation*}
Now, assume that expression (\ref{hatp}) holds if $L \leq n$. For $L = n+1$, we then obtain
\begin{align*}
    \hat{p}(\sigma, \sigma') &= \sum\limits_{i \in \Delta(\sigma)} \dfrac{1}{|\Delta (\sigma)|} \hat{p}(\sigma^i, \sigma').
\end{align*}
If $\sigma^i = \sigma'$, for $i \in \Delta(\sigma)$, then $\hat{p}(\sigma^i, \sigma') = 1$. On the other hand, if $\sigma^i \neq \sigma'$ for $i \in \Delta(\sigma)$, then either $\sigma' \notin U(\sigma^i)$ or $\sigma' \in U(\sigma^i)$ and $L(\sigma^i, \sigma') \leq n$. In the first case, we have $\hat{p}(\sigma^i, \sigma') = 0$. In the second case, the induction hypothesis implies that
\begin{equation*}
    \hat{p}(\sigma^i, \sigma') = \sum\limits_{\omega \in \Omega(\sigma^i, \sigma')} \prod\limits_{\ell = 0}^{|\omega|-1} \dfrac{1}{|\Delta(\omega_{\ell})|}.
\end{equation*}
Thus, we obtain
\begin{align*}
    \hat{p}(\sigma, \sigma') &= \sum\limits_{\substack{i \in \Delta(\sigma),\\ \sigma^i = \sigma'}} \dfrac{1}{|\Delta(\sigma)|} + \sum\limits_{\substack{i \in \Delta(\sigma)\\ \sigma^i \neq \sigma'\\ \sigma' \in U(\sigma^i)}} \dfrac{1}{|\Delta(\sigma)|} \sum\limits_{\omega \in \Omega(\sigma^i, \sigma')} \prod\limits_{\ell = 0}^{|\omega|-1} \dfrac{1}{|\Delta(\omega_{\ell})|} = \sum\limits_{\omega \in \Omega(\sigma, \sigma')}\prod\limits_{\ell = 0}^{|\omega|-1} \dfrac{1}{|\Delta(\omega_{\ell})|}.
\end{align*}
Now, consider $(i,j), (i',j') \in \hat{S}$ and $a \in A$ such that $U^1(\sigma^a_{(i,j)}) \cap \mathcal{I}^{-1}((i',j')) \neq \emptyset$. 
By definition of the transition probability kernel $Q_{\kappa}$, we have
\begin{equation}\label{Qk_phat}
    \lim_{\kappa \rightarrow \infty} Q_\kappa(\eta_{\{(i,j), a, (i',j')\}}|\sigma_{(i,j)}, a) = \hat{p}(\sigma_{(i,j)}^a, \eta_{\{(i,j), a, (i',j')\}}).
\end{equation}
First, suppose that $\eta_{\{(i,j), a, (i',j')\}} = \sigma_{(i,j)}^a$. Since $\eta_{\{(i,j), a, (i',j')\}}$ is a robust configuration, it immediately follows that
\begin{equation*}
    \lim_{\kappa \rightarrow \infty} Q_\kappa(\eta_{\{(i,j), a, (i',j')\}}|\sigma_{(i,j)}, a) = 1.
\end{equation*}
Secondly, suppose that $\eta_{\{(i,j), a, (i',j')\}} \neq \sigma_{(i,j)}^a$. It now follows from expression (\ref{hatp}) that
\begin{equation*}
    \lim_{\kappa \rightarrow \infty} Q_\kappa(\eta_{\{(i,j), a, (i',j')\}}|\sigma_{(i,j)}, a) = \sum\limits_{\omega \in \Omega(\sigma_{(i,j)}^a, \eta_{\{(i,j), a, (i',j')\}})} \prod\limits_{\ell = 0}^{|\omega|-1} \dfrac{1}{|\Delta(\omega_{\ell})|},
\end{equation*}
establishing expression (\ref{limQ}). 
\end{proof}

Lemma \ref{explicit_Phats} now provides the explicit expressions for the transition probabilities \\ $\hat{P}((i',j')|(i,j), \hat{a})$ for each $(i,j), (i',j') \in \hat{S}$, $\hat{a} \in \hat{A}_{(i, j)}$. 

\begin{lemma}\label{explicit_Phats}
The transition probability kernel $\hat{P}:\hat{S} \times \hat{A} \times \hat{S} \rightarrow [0,1]$ is given by

\begin{align}
\label{trans_probs_1}
    \hat{P}((i',j')|(i,j), a_{11}) &= 
    \begin{cases} 1/3, &\text{if } i'= i, j'= j,\\
    2/3, &\text{if } i'= i, j' = j+1, \\
    0, &\text{otherwise,}
    \end{cases} \quad \text{for } 
\begin{aligned}
    &i = 3, \ldots, N-2, N, \\
    &j = 2, \ldots, N-3,
\end{aligned}\\
\nonumber \\
\label{trans_probs_2}
    \hat{P}((i',j')|(i,j), a_{21}) &= \begin{cases} 1/3, &\text{if } i' = i, j'= j,\\
    2/3, &\text{if } i' = i+1, j'= j,\\
    0, &\text{ otherwise,}
    \end{cases} \quad \text{for } 
\begin{aligned}
    &i = 2, \ldots, N-3, \\
    &j = 3, \ldots, N-2, N, 
\end{aligned}\\
\nonumber \\
\label{trans_probs_3}
\hat{P}((i', j')|(i,j), a_{12}) &= 
\begin{cases} 5/9, &\text{if } i'=i, j'=j,\\
7/27, &\text{if } i'= i, j'=j+1,\\
5/27, &\text{if } i'= i, j'=j+2,\\
0, &\text{otherwise,}
\end{cases} \quad \text{for } 
\begin{aligned}
    &i = 3, \ldots, N-2, N, \\
    &j = 2, \ldots, N-4,
\end{aligned}\\
\nonumber \\
\label{trans_probs_4}
\hat{P}((i', j')|(i,j), a_{22}) &= 
\begin{cases} 5/9, &\text{if } i'=i, j'=j,\\
7/27, &\text{if } i'= i+1, j'=j,\\
5/27, &\text{if } i'= i+2, j'=j,\\
0, &\text{otherwise,}
\end{cases}\quad \text{for } 
\begin{aligned}
    &i = 2, \ldots, N-4, \\
    &j = 3, \ldots, N-2, N,
\end{aligned}\\
\nonumber \\
\label{trans_probs_5}
    \hat{P}((i', j')|(i,j), a_0) &= \begin{cases}
        4/9, &\text{if } i'=i, j'=j,\\
        1/9, &\text{if } i'=i+1, j'=j \\
        &\text{ or } i'=i, j'= j+1,\\
        1/3, &\text{if } i'=i+1, j'=j+1, \\
        0, &\text{otherwise,}
    \end{cases} \quad \text{for } 
\begin{aligned}
    &i = 2, \ldots, N-3, \\
    &j = 2, \ldots, N-3, 
\end{aligned}\\
\nonumber \\
\label{trans_probs_extra_1}
    \hat{P}((i', j')|(i,j), a_{11}') &= \begin{cases}
        1/2, &\text{if } i'=i, j'=j,\\
        1/2, &\text{if } i'=i, j'=j+1,\\
        0, &\text{otherwise,}
    \end{cases} \quad \text{for } 
\begin{aligned}
    &i = 2, \ldots, N-2, \\
    &j = 2, \ldots, N-3,
\end{aligned} \\
    \nonumber \\
\label{trans_probs_extra_2}
    \hat{P}((i', j')|(i,j), a_{21}') &= \begin{cases}
        1/2, &\text{if } i'=i, j'=j,\\
        1/2, &\text{if } i'=i+1, j'=j,\\
        0, &\text{otherwise,}
    \end{cases} \quad \text{for } 
\begin{aligned} 
    &i = 2, \ldots, N-3, \\
    &j = 2, \ldots, N-2, 
\end{aligned} \\
    \nonumber \\
\label{trans_probs_extra_3}
    \hat{P}((i',j')|(i,j), a_{12}') &= \begin{cases}
        5/8, &\text{if } i'=i, j'=j, \\
        1/4, &\text{if } i'=i, j'=j+1, \\
        1/8, &\text{if } i'=i, j'=j+2, \\
        0, &\text{otherwise,}
    \end{cases} \quad \text{for } 
\begin{aligned}
    &i = 2, \ldots, N-2, \\
    &j = 2, \ldots, N-4,
\end{aligned} \\
    \nonumber \\
\label{trans_probs_extra_4}
    \hat{P}((i',j')|(i,j), a_{22}') &= \begin{cases}
        5/8, &\text{if } i'=i, j'=j, \\
        1/4, &\text{if } i'=i+1, j'=j, \\
        1/8, &\text{if } i'=i+2, j'=j, \\
        0, &\text{otherwise,}
    \end{cases} \quad \text{for } 
\begin{aligned}
    &i = 2, \ldots, N-4, \\ 
    &j = 2, \ldots, N-2,
\end{aligned} \\
    \nonumber \\
\label{trans_probs_6}
    \hat{P}((i', j')|(i, N-2), a_{11}) &= 
    \begin{cases}
        1/4, &\text{if } i'=i, j'=N-2,\\
        3/4, &\text{if } i'=i, j'=N,\\
        0, &\text{otherwise,}
        \end{cases} \quad \text{for } i = 3, \ldots, N-2, N, \\ 
        \nonumber \\
\label{trans_probs_7}
    \hat{P}((i', j')|(N-2, j), a_{21}) &=
    \begin{cases}
        1/4, &\text{if } i'=N-2, j'=j,\\
        3/4, &\text{if } i'=N, j'=j,\\
        0, &\text{otherwise,}
    \end{cases} \quad \text{for } j = 3, \ldots, N-2, N, \\
    \nonumber \\
\label{trans_probs_8}
    \hat{P}((i',j')|(i, N-3), a_{12}) &= 
    \begin{cases}
        7/18, &\text{if } i'=i, j'=N-3,\\
        31/144, &\text{if } i'=i, j'=N-2\\
        19/48, &\text{if } i'=i, j'=N,\\
        0, &\text{otherwise,}
    \end{cases} \quad \text{for } i = 3, \ldots, N-2, N, \\
    \nonumber \\
\label{trans_probs_9}
    \hat{P}((i', j')|(N-3, j), a_{22}) &= 
    \begin{cases}
        7/18, &\text{if } i'=N-3, j'=j,\\
        31/144, &\text{if } i'=N-2, j'=j,\\
        19/48, &\text{if } i'= N, j'=j,\\
        0, &\text{otherwise,}
    \end{cases} \quad \text{for } j = 3, \ldots, N-2, N, \\
    \nonumber \\
\label{trans_probs_12}
    \hat{P}((i',j')|(i, N-2), a_0) &= 
    \begin{cases}
        5/12, &\text{if } i'=i, j'=j, \\
        1/8, &\text{if } i'=i, j'=N,\\
        1/9, &\text{if } i'=i+1, j'=j,\\
        25/72, &\text{if } i'=i+1, j'=N,\\
        0, &\text{otherwise,}
    \end{cases} \quad \text{for } i = 2, \ldots, N-3, \\
    \nonumber \\
\label{trans_probs_10}
    \hat{P}((i',j')|(N-2, j), a_0) &= 
    \begin{cases}
        5/12, &\text{if } i'=i, j'=j, \\
        1/8, &\text{if } i'=N, j'=j,\\
        1/9, &\text{if } i'=i, j'=j+1,\\
        25/72, &\text{if } i'=N, j'=j+1,\\
        0, &\text{otherwise,}
    \end{cases} \quad \text{for } j = 2, \ldots, N-3, \\
    \nonumber \\
\label{trans_probs_11}
    \hat{P}((i',j')|(N-2, N-2), a_0) &= 
    \begin{cases}
        7/18, &\text{if } i'=j'=N-2,\\
        1/8, &\text{if } i'=N, j'=N-2 \\
        &\text{or } i'=N-2, j'= N,\\
        13/36, &\text{if } i'=j'=N,\\
        0, &\text{otherwise,}
    \end{cases}\\
    \nonumber \\
\label{trans_probs_extra_5}
    \hat{P}((i',j')|(i, N-2), a_{11}') &= 
    \begin{cases}
        1/3, &\text{if } i'=i, j'=N-2,\\
        2/3, &\text{if } i'=i, j'=N,\\
        0, &\text{otherwise,}
    \end{cases} \quad \text{for } i = 2, \ldots, N-2, \\ 
    \nonumber \\
\label{trans_probs_extra_6}
    \hat{P}((i', j')|(N-2, j), a_{21}') &= 
    \begin{cases}
        1/3, &\text{if } i'=N-2, j'=j,\\
        2/3, &\text{if } i'=N, j'=j,\\
        0, &\text{otherwise,}
    \end{cases} \quad \text{for } j = 2, \ldots, N-2, \\
    \nonumber \\
\label{trans_probs_extra_7}
    \hat{P}((i',j')|(i, N-3), a_{12}') &= 
    \begin{cases}
        4/9, &\text{if } i'=i, j'=N-3,\\
        5/27, &\text{if } i'=i, j'=N-2,\\
        10/27, &\text{if } i'=i, j'=N,\\
        0, &\text{otherwise,}
    \end{cases} \quad \text{for } i = 2, \ldots, N-2, \\
    \nonumber \\
\label{trans_probs_extra_8}
    \hat{P}((i',j')|(N-3, j), a_{22}') &= 
    \begin{cases}
        4/9, &\text{if } i'=N-3, j'=j,\\
        5/27, &\text{if } i'=N-2, j'=j,\\
        10/27, &\text{if } i'=N, j'=j,\\
        0, &\text{otherwise,}
    \end{cases} \quad \text{for } j = 2, \ldots, N-2, \\
    \nonumber \\
\label{trans_probs_double_extra_1}
    \hat{P}((i',j')|(i,j), \tilde{a}) &= \begin{cases}
        1, &\text{if } (i',j') = (i,j), \\
        0, &\text{otherwise,}
    \end{cases} \quad \text{for } i, j = 3, \ldots, N-2, \\
    \nonumber \\
\label{trans_probs_double_extra_2}
    \hat{P}((i',j')|(2,j), \tilde{a}) &= 
    \begin{cases}
        1/2, &\text{if } (i',j') = (2,j), \\
        1/2, &\text{if } (i',j') = (2,j-1),\\
        0, &\text{otherwise,}
    \end{cases} \quad \text{for } j = 3, \ldots, N-2, \\
    \nonumber \\
\label{trans_probs_double_extra_3}
    \hat{P}((i',j')|(i,2), \tilde{a}) &= 
    \begin{cases}
        1/2, &\text{if } (i',j') = (i,2), \\
        1/2, &\text{if } (i',j') = (i-1,2),\\
        0, &\text{otherwise,}
    \end{cases} \quad \text{for } i = 3, \ldots, N-2, \\
    \nonumber \\
\label{trans_probs_double_extra_4}
    \hat{P}((i',j')|(2, 2), \tilde{a}) &= 
    \begin{cases}
        1/3, &\text{if } (i',j') = (2, 2), \\
        2/3, &\text{if } (i',j') = (0, 0), \\
        0, &\text{otherwise,}
    \end{cases} \quad \text{and } \\
    \hat{P}((i,j)|(i,j), 0) &= 1, \quad \text{for all } (i,j) \in \hat{S}.
\end{align}
 
\end{lemma}
\begin{proof}
    We give the proof of expression (\ref{trans_probs_3}). Let $(i,j) \in \hat{S}$, $i = 3, \ldots, N$, $j = 2, \ldots, N-4$ and let $a \in \mathcal{J}^{-1}_{\sigma_{(i,j)}}(a_{12})$. Consider again Figure \ref{fig_trans_probs_3}. First, the arguments laid out in the proof of Lemma \ref{map_actions} imply that $U^1(\sigma^a_{(i,j)}) \cap \mathcal{I}^{-1}((i',j')) \neq \emptyset$ if and only if $(i',j') = (i,j)$, $(i',j') = (i, j+1)$ or $(i',j') = (i, j+2)$. Hence, $\hat{P}((i',j')|(i,j), a_{12}) = 0$ if $(i',j') \notin \{(i, j), (i, j+1), (i, j+2)\}$. 
    
    Now, suppose that $(i',j') \in \{(i, j), (i, j+1), (i, j+2)\}$. 
    The proof of Lemma \ref{map_actions} implies that the transition probability $\hat{P}((i',j')|(i,j), a_{12})$ depends only on the evolution of the flipped spin, denoted by $a$, its vertical neighbor that is adjacent to the rectangle, the horizontal neighbors of the latter spin and the horizontal neighbors of $a$ itself, labeled 1, 2, 3, 4, 5 and 6 respectively, as illustrated in Figure \ref{fig_trans_probs_3}. Letting $W = \{1, 2, 3, 4, 5, 6\}$, we showed that for any two closed downhill sequences $\mathbf{x}, \mathbf{y} \in \tilde{\Theta}(\sigma^a_{(i,j)})$ that satisfy $\mathbf{x}^W = \mathbf{y}^W$, we have $(\sigma^a_{(i,j)})^{\mathbf{x}} = (\sigma^a_{(i,j)})^{\mathbf{y}}$. We now proceed to show that the susceptibility of the spins in the set $W = \{1, 2, 3, 4, 5, 6\}$ is not affected by the evolution of the configuration on the remainder of the lattice, i.e., that for any downhill sequence $\mathbf{x} \in \Theta(\sigma)$, we have $\Delta((\sigma^a_{(i,j)})^{\mathbf{x}}) \cap W = \Delta((\sigma^a_{(i,j)})^{\mathbf{x}^W}) \cap W$. Consider the sequence $\mathbf{y} = (2)$ and let $\mathbf{x} \in \Theta(\sigma)$ be any sequence that satisfies $\mathbf{x}^W = \mathbf{y}$. In configuration $\sigma^a_{(i,j)}$, the only susceptible spins are 1 and 2. This implies that the first element of $\mathbf{x}$ is spin 2. After flipping spin 2, the set of susceptible spins is $\{1, 3, 4\}$. Therefore, we have $\mathbf{x} = \mathbf{y} = (2)$ and thus $\Delta((\sigma^a_{(i,j)})^{\mathbf{x}}) \cap W = \Delta((\sigma^a_{(i,j)})^{\mathbf{x}^W}) \cap W$. Now consider the sequence $\mathbf{y} = (2, 4)$ and let $\mathbf{x} \in \Theta(\sigma^a_{(i,j)})$ again be any sequence that satisfies $\mathbf{x}^W = \mathbf{y}$. The same argument as before implies that the first two entries of $\mathbf{x}$ are the spins 2 and 4. At this point, the only spins outside of $W$ that can flip on any downhill path are the spins in the set $b_{11}(\sigma_{(i,j)})\setminus W$, where $b_{11}(\sigma_{(i,j)})$ is the horizontal set lying at distance 1 from the rectangle in $\sigma_{(i,j)}$, as depicted in Figure \ref{Slices}. Observe, however, that flipping any spins in this set leaves the set of susceptible spins that are part of $W$, namely $\{1, 6\}$, unaffected. We thus again obtain $\Delta((\sigma^a_{(i,j)})^{\mathbf{x}}) \cap W = \Delta((\sigma^a_{(i,j)})^{\mathbf{x}^W}) \cap W$. This argument can easily be extended to the sequence $\mathbf{y} = (2, 4, 1)$. Since after flipping the spins in this sequence, spins 2 and 4 can no longer flip back on any downhill path and the only spins that will become susceptible on any downhill path are those in the set $b_{11}(\sigma_{(i,j)})$, it follows that any closed downhill sequence $\mathbf{x} \in \tilde{\Theta}(\sigma^a_{(i,j)})$ that satisfies $\mathbf{x}^W = \mathbf{y}$ leads to the robust configuration $\sigma_{(i,j)}^{b_{11}(\sigma_{(i,j)})}$. 
    In Table \ref{trans_probs_3} we listed the restrictions to $W$ of all sequences in the set $\tilde{\Theta}(\sigma^a_{(i,j)})$. The same argument as above can be applied to any of the sequences in this table. We can thus conclude that for any downhill sequence $\mathbf{x} \in \Theta(\sigma)$, we have $\Delta((\sigma^a_{(i,j)})^{\mathbf{x}}) \cap W = \Delta((\sigma^a_{(i,j)})^{\mathbf{x}^W}) \cap W$. It follows that this set is essential for the configuration $\sigma_{(i,j)}^a$. Thus, to compute $\hat{P}((i',j')|(i,j), a_{12})$ it suffices to focus on the restrictions of the closed downhill sequences to $W$ that lead to the robust configuration $\eta_{\{(i,j), a, (i',j')\}}$. Let the set of these sequences be denoted by $\tilde{\Theta}^W(\sigma^a_{(i,j)}, \eta_{\{(i,j), a, (i',j')\}})$. To keep notation simple in the computation that follows, we abbreviate $\sigma^a_{(i,j)}$ and $\eta_{\{(i,j), a, (i',j')\}}$ as $\sigma^a$ and $\eta$. By the properties of the essential set, it follows that 
    \begin{align}\label{prune_paths}
        \sum\limits_{\omega \in \tilde{\Omega}(\sigma^a, \eta)} \prod\limits_{\ell = 0}^{|\omega|-1} \dfrac{1}{|\Delta(\omega_{\ell})|} &= \sum\limits_{\mathbf{x} \in \tilde{\Theta}(\sigma^a, \eta)} \dfrac{1}{|\Delta(\sigma^a)|} \prod\limits_{\ell = 1}^{|\mathbf{x}|-1} \dfrac{1}{|\Delta((\sigma^a)^{(x_1, \ldots, x_{\ell})})|} \\
        \nonumber &=\sum\limits_{\mathbf{y} \in \tilde{\Theta}^W(\sigma^a, \eta)} \dfrac{1}{|\Delta(\sigma^a) \cap W|} \prod\limits_{\ell = 1}^{|\mathbf{y}|-1} \dfrac{1}{|\Delta((\sigma^a)^{(y_1, \ldots, y_{\ell})}) \cap W|}.
    \end{align}
    For each sequence in $\mathbf{y} \in \tilde{\Theta}^W(\sigma^a, \eta)$, Table \ref{tab_trans_probs_3} gives the probability that this sequence is selected, i.e., the quantity $\dfrac{1}{|\Delta(\sigma^a \cap W)|} \prod\limits_{\ell = 1}^{|\mathbf{y}|-1} \dfrac{1}{|\Delta((\sigma^a)^{(y_1, \ldots, y_{\ell})}) \cap W|}$ and the state that corresponds to the robust configuration that the sequence leads to. Using Table \ref{tab_trans_probs_3} to evaluate expression (\ref{prune_paths}), we obtain expression (\ref{trans_probs_3}).
    
    The remaining transition probabilities can be computed in a similar way. Supporting figures depicting the essential sets and tables listing the relevant sequences of susceptible spins are provided in part A of the Supplementary Material.  \end{proof}
    
    \begin{table}
    \begin{center}
    \caption{Derivation of expression (\ref{trans_probs_3}), corresponding to Figure \ref{fig_trans_probs_3}. Similar derivation leads to expression (\ref{trans_probs_4}).}
    \centering
    \begin{tabular}{ llll } 
    \toprule
    Sequences of susceptible spins & Probability of selecting sequence & Next state & Transition probability \\
    \midrule
    $(1)$ & $1/2$ & \multirow{2}{4em}{$(i, j)$} & \multirow{2}{4em}{$5/9$} \\
    $(2, 1, 2)$ & $1/18$ & & \\
    \midrule
    $(2, 1, 3)$ & $1/18$ & \multirow{6}{4em}{$(i, j+1)$} & \multirow{6}{4em}{$7/27$} \\
    $(2, 1, 4)$ & $1/18$ & & \\
    $(2, 3, 1)$ & $1/18$ & & \\
    $(2, 3, 4, 1)$ & $1/54$ & & \\
    $(2, 4, 1)$ & $1/18$ & & \\
    $(2, 4, 3, 1)$ & $1/54$ & & \\
    \midrule 
    $(2, 3, 4, 5)$ & $1/54$ &\multirow{6}{4em}{$(i, j + 2)$} & \multirow{6}{4em}{$5/27$} \\
    $(2, 3, 4, 6)$ & $1/54$ & & \\
    $(2, 4, 3, 5)$ & $1/54$ & & \\
    $(2, 4, 3, 6)$ & $1/54$ & & \\
    $(2, 3, 5)$ & $1/18$ & & \\
    $(2, 4, 6)$ & $1/18$ & & \\
    \bottomrule

    \end{tabular}
    
    \label{tab_trans_probs_3}
    \end{center}
    \end{table}
     
\subsubsection{Proof of Theorem \ref{Optimal_policy}}
\label{Proof of main thm}

A convenient feature of the auxiliary MDP is the fact that it is not possible for a rectangle of size at least 3x3 to shrink, as by definition of the transition probability kernel, the probability that a state variable that is greater than 3 decreases in the auxiliary MDP is zero. This fact allows us to solve the Bellman equations recursively, starting from state $(N,N)$, which corresponds to the all-plus configuration. Before doing so, we first show the suboptimality of the actions 0 and $\tilde{a}$, such that we can leave these actions out of consideration when solving the Bellman equations.

\begin{lemma}\label{0_sub_opt}
Any stationary deterministic optimal policy $\pi^* = (d^*)^{\infty}$ satisfies $d^*(i,j) \neq 0$ for all $(i,j) \in \hat{S}$, $(i,j) \notin \{(0,0), (N, N)\}$. 
\end{lemma}
\begin{proof}
    Suppose that $\pi^* = (d^*)^{\infty}$ is an optimal policy, with $d^*(i, j) = 0$ for some $(i, j) \in \hat{S}$, $(i,j) \neq (N, N)$. By definition of the transition probabilities and the reward function, this implies $v^*(i, j) = 0$. Now, let $\pi = d^{\infty}$ denote another policy that satisfies $d(i,j) \neq 0$ for all $(i,j) \neq (0,0)$. Note that the definition of the transition probabilities imply that for any state $(i,j) \neq (0,0)$, there exists a sequence of states $(s_1, s_2, \ldots, s_k)$ for some $k \in \mathbb{N}$ that satisfies $s_1 = (i,j)$,  $s_k = (N, N)$ and $\hat{P}(s_{\ell + 1}|s_{\ell}, d(s_{\ell})) > 0)$ for all $\ell = 1, \ldots, k-1$. From the definition of the reward function, it now follows that $v^{\pi}(i,j) > 0$ and thus $v^{\pi}(i,j) > v^*(i,j)$, which contradicts the optimality of policy $\pi^*$. This completes the proof. 
\end{proof}

\begin{lemma}\label{atilde_sub_opt}
Any stationary deterministic optimal policy $\pi^* = (d^*)^{\infty}$ satisfies $d^*(i,j) \neq \tilde{a}$ for all $(i, j) \in \hat{S}$, $i, j \notin \{0, N\}$.
\end{lemma}
\begin{proof}
First consider $s = (i,j) \in \hat{S}$, where either $i,j = 3, \ldots, N-2$ or $(i,j) = (2, 2)$. Suppose that $\pi^* = (d^*)^{\infty}$ is an optimal policy that satisfies $d^*(s) = \tilde{a}$. By expressions (\ref{trans_probs_double_extra_1}), (\ref{trans_probs_double_extra_4}) and (\ref{reward}), this implies $v^*(s) = 0$. As shown in the proof of Lemma \ref{0_sub_opt}, there exists a policy $\pi = d^{\infty}$ that satisfies $v^{\pi}(i,j) > 0 = v^*(i,j)$, which contradicts the optimality of $\pi^*$. Hence $d^*(s) \neq \tilde{a}$.

We proceed to prove the statement for states $(i,j) \in \hat{S}$, where either $i = 2$ and $j = 3, \ldots, N-2$ or $i = 3, \ldots, N-2$ and $j = 2$. Suppose that $\pi^* = (d^*)^{\infty}$ is an optimal policy that satisfies $d^*(2,j) = \tilde{a}$ for some $j = 3, \ldots, N-2$. Let $j^*$ denote the largest such $j$ for which this holds. We assume that $3< j^* < N-2$. The argument can easily be extended to $j^* = 3$ and $j^* = N-2$. Note that rectangles of size $(i,j)$, where $i, j = 3, \ldots, N-2, N$ or $i = 2$, $j = j^* + 1, \ldots, N-2, N$, cannot shrink under policy $\pi^*$. In addition, by the definition of the reward function, the optimality of $\pi^*$ and the result of Lemma \ref{0_sub_opt}, we obtain $v^*(i,j) < v^*(i+1, j)$ and $v^*(i, j) < v^*(i, j+1)$ for $i, j = 3, \ldots, N-2, N$ or $i = 2$, $j = j^* + 1, \ldots, N-2, N$. We proceed to show that at least one of the following must hold:
\begin{enumerate}
    \item $d^*(2, j^*-1) \in \{a_{12}', a_{21}, a_{21}', a_{22}, a_{22}', a_0\}$,
    \item $d^*(2,j) \in \{a_{21}, a_{21}', a_{22}, a_{22}', a_0\}$, for some $j = 3, \ldots, j^*-2$,
    \item $d^*(2,2) = a_0$.
\end{enumerate}
If none of these statements holds, we have $d^*(2, j^*-1) \in \{a_{11}', \tilde{a}\}$, $d^*(2, j) \in \{a_{11}', a_{12}', \tilde{a}\}$ and $d^*(2, 2) \in \{a_{11}', a_{12}', a_{21}', a_{22}'\}$. It now follows from the definition of the transition probabilities and from symmetry that the only rectangles that can be reached from state $(2, j^*)$ are those in the set $\{(2, j), (j, 2)| j = 2, \ldots, j^*\}$. By the definition of the reward function, this implies $v^*(2, j^*) = 0$. Since there exists a policy $\pi = d^{\infty}$ that satisfies $v^{\pi}(s) > 0$ for all $s \in S\setminus\{(0,0\}$, this contradicts the optimality of $\pi^*$. Hence, one of the statements $1-3$ must hold. Let $j'$ denote the largest $j$ for which one of these statements is attained. We assume that $j' \in \{3, \dots, j^*-2\}$ and $d^*(2, j') = a_{21}$. All the remaining scenarios can be handled in a similar way. We construct a policy $\pi = d^{\infty}$, where $d: \hat{S} \rightarrow \hat{A}$ is given by
\begin{equation*}
    d(i,j) = \begin{cases}
        a_{21}, &\text{if } i = 2, \quad j = j^*, \\
        d^*(i,j), &\text{otherwise}.
    \end{cases}
\end{equation*}
That is, we replace the action in state $(2, j^*)$ by the action prescribed in state $(2, j')$. Now, let $\tau$ denote the first hitting time to the state $(2,j')$ in the process induced by policy $\pi^*$, started from state $(2, j^*)$. Note that $s_t \in \{(2, j)|j = j'+1, \ldots, j^*\}$ for all $t < \tau$ and that $\tau \geq j^*-j'$ by definition of $\pi^*$ and the transition probabilities. This implies that $v^*(2, j^*) \leq \lambda^{j^*-j'}v^*(2, j')$. Also, by the fact that $d(i,j) = d^*(i,j)$ for $i,j = 3, \ldots, N-2, N$, and for $i = 2$, $j = j^*+1, \ldots, N-2, N$ and the fact that these rectangles cannot shrink, we have $v^{\pi}(i,j) = v^*(i,j)$ for all $i,j = 3, \ldots, N-2, N$ and $i = 2$, $j = j^*+1, \ldots, N-2, N$. We now obtain, using expression (\ref{trans_probs_2}),
\begin{equation*}
    v^*(2, j^*) \leq \lambda^{j^* -j'}v^*(2, j') = \lambda^{j^*-j'}\dfrac{2}{3-\lambda}v^*(3, j'),
\end{equation*}
and
\begin{equation*}
    v^{\pi}(2, j^*) = \dfrac{2}{3-\lambda}v^{\pi}(3, j^*) = \dfrac{2}{3-\lambda}v^*(3, j^*).
\end{equation*}
Since $v^*(3, j') < v^*(3, j^*)$, this yields $v^*(2, j^*) < v^{\pi}(2, j^*)$, contradicting the optimality of $\pi^*$. Thus, $d^*(i,j) \neq \tilde{a}$ for all $(i,j) \in \hat{S}$, $i,j \notin \{0,N\}$. 
\end{proof}
It follows from Lemmas \ref{0_sub_opt} and \ref{atilde_sub_opt} that we can leave actions 0 and $\tilde{a}$ out of consideration. The following result implies the suboptimality of actions $a_{11}'$, $a_{12}'$, $a_{21}'$ and $a_{22}'$. For rectangles with side lengths greater than 2, it shows that these actions are always inferior to their counterparts $a_{11}$, $a_{12}$, $a_{21}$ and $a_{22}$ respectively. Hence, in finding the optimal policy, when considering a rectangle that has a side of length greater than 2, we leave the corresponding actions in the set $\{a_{11}'$, $a_{12}'$, $a_{21}'$, $a_{22}'\}$ out of consideration. 
\begin{lemma}\label{corner_vs_middle}
Let $v^*: \hat{S} \rightarrow \mathbb{R}$ denote the optimal value function of the auxiliary MDP. For any state $s = (x,y) \in \hat{S}$ with $3 \leq x \leq N-2$ and $y \leq N-2$, we have
\begin{equation}\label{a11_vs_a11'}
    \hat{r}(s, a_{11}') + \lambda \sum\limits_{s'\in \hat{S}}\hat{P}(s'|s, a_{11}')v^*(s') < \hat{r}(s, a_{11}) + \lambda \sum\limits_{s' \in \hat{S}}\hat{P}(s'|s, a_{11})v^*(s'),
\end{equation}
and for any state $s = (x,y) \in \hat{S}$ with $3 \leq x \leq N-2$ and $y \leq N-3$, we have
\begin{equation}\label{a12_vs_a12'}
    \hat{r}(s, a_{12}') + \lambda \sum\limits_{s'\in \hat{S}}\hat{P}(s'|s, a_{12}')v^*(s') < \hat{r}(s, a_{12}) + \lambda \sum\limits_{s' \in \hat{S}}\hat{P}(s'|s, a_{12})v^*(s').
\end{equation}
Similarly, for any state $s = (x,y) \in \hat{S}$ with $3 \leq y \leq N-2$ and $x \leq N-2$, we have
\begin{equation}\label{a21_vs_a21'}
    \hat{r}(s, a_{21}') + \lambda \sum\limits_{s'\in \hat{S}}\hat{P}(s'|s, a_{21}')v^*(s') < \hat{r}(s, a_{21}) + \lambda \sum\limits_{s' \in \hat{S}}\hat{P}(s'|s, a_{21})v^*(s'),
\end{equation}
and for any state $s = (x,y) \in \hat{S}$ with $3 \leq y \leq N-2$ and $x \leq N-3$, we have
\begin{equation}\label{a22_vs_a22'}
    \hat{r}(s, a_{22}') + \lambda \sum\limits_{s'\in \hat{S}}\hat{P}(s'|s, a_{22}')v^*(s') < \hat{r}(s, a_{22}) + \lambda \sum\limits_{s' \in \hat{S}}\hat{P}(s'|s, a_{22})v^*(s').
\end{equation}
\end{lemma}
\begin{proof}
    We prove the validity of expressions (\ref{a11_vs_a11'}) and (\ref{a12_vs_a12'}). Expressions (\ref{a21_vs_a21'}) and (\ref{a22_vs_a22'}) follow immediately by symmetry.
    We first consider expression (\ref{a11_vs_a11'}) for $y = N-2$. The definition of the reward function and expressions (\ref{trans_probs_6}) and (\ref{trans_probs_extra_5}) yield
    \begin{equation*}
        \hat{r}(s, a_{11}') + \lambda \sum\limits_{s'\in \hat{S}}\hat{P}(s'|s, a_{11}')v^*(s') = \dfrac{\lambda}{3}v^*(x,N-2) + \dfrac{2\lambda}{3}v^*(x, N)
    \end{equation*}
    and
    \begin{equation*}
        \hat{r}(s, a_{11}) + \lambda \sum\limits_{s'\in \hat{S}}\hat{P}(s'|s, a_{11})v^*(s') = \dfrac{\lambda}{4}v^*(x,N-2) + \dfrac{3\lambda}{4}v^*(x, N).
    \end{equation*}
    The facts that $\hat{r}(x, y) = 1$ if $(x, y) = (N, N)$ and $\hat{r}(x,y) = 0$ otherwise, together with the fact that for each $s \in \hat{S}$, $s \neq (N, N)$, there exists $s'\in \hat{S}$, $s' \neq s$ such that $\hat{P}(s'|s, d(s)) > 0$ if $d(s) \neq 0$ and the fact that rectangles cannot shrink, imply that $v^*(x, y) < v^*(x, y+1)$ for all $3 \leq y \leq N-1$. Therefore,
    \begin{align*}
        \dfrac{\lambda}{3}v^*(x,N-2) + \dfrac{2\lambda}{3}v^*(x, N) &= \left(\dfrac{\lambda}{3} - \dfrac{\lambda}{12}\right)v^*(x, N-2) + \dfrac{\lambda}{12}v^*(x, N-2) + \dfrac{2\lambda}{3}v^*(x, N) \\
        \nonumber&< \dfrac{\lambda}{4}v^*(x, N-2) + \dfrac{3\lambda}{4}v^*(x, N).
    \end{align*}
    Now suppose that $y < N-2$. Using the definition of the reward function and expressions~(\ref{trans_probs_1}) and (\ref{trans_probs_extra_1}), we obtain
    \begin{equation*}
        \hat{r}(s, a_{11}') + \lambda \sum\limits_{s'\in \hat{S}}\hat{P}(s'|s, a_{11}')v^*(s') = \dfrac{\lambda}{2}v^*(x,y) + \dfrac{\lambda}{2}v^*(x, y+1)
    \end{equation*}
    and
    \begin{equation*}
        \hat{r}(s, a_{11}) + \lambda \sum\limits_{s'\in \hat{S}}\hat{P}(s'|s, a_{11})v^*(s') = \dfrac{\lambda}{3}v^*(x,y) + \dfrac{2\lambda}{3}v^*(x, y+1).
    \end{equation*}
    Again invoking the fact that $v^*(x, y) < v^*(x, y+1)$ for all $3 \leq y \leq N-1$, we have
    \begin{align*}
        \dfrac{\lambda}{2}v^*(x,y) + \dfrac{\lambda}{2}v^*(x, y+1) &= \left(\dfrac{\lambda}{2} - \dfrac{\lambda}{6}\right)v^*(x,y) + \dfrac{\lambda}{6}v^*(x,y) + \dfrac{\lambda}{2}v^*(x,y+1) \\
        \nonumber&< \dfrac{\lambda}{3}v^*(x,y) + \dfrac{2\lambda}{3}v^*(x,y+1),
    \end{align*}
    which proves expression (\ref{a11_vs_a11'}).
    
    The validity of expression (\ref{a12_vs_a12'}) can be established in a similar way. We first consider the case $y = N-3$. From the definition of the reward function and expressions (\ref{trans_probs_8}) and (\ref{trans_probs_extra_7}), it follows that
    \begin{equation*}
        \hat{r}(s, a_{12}') + \lambda \sum\limits_{s'\in \hat{S}}\hat{P}(s'|s, a_{12}')v^*(s') = \dfrac{4\lambda}{9}v^*(x, N-3) + \dfrac{5\lambda}{27}v^*(x, N-2) + \dfrac{10\lambda}{27}v^*(x, N)
    \end{equation*}
    and
    \begin{equation*}
        \hat{r}(s, a_{12}) + \lambda \sum\limits_{s'\in \hat{S}}\hat{P}(s'|s, a_{12})v^*(s') = \dfrac{7\lambda}{18}v^*(x, N-3) + \dfrac{31\lambda}{144}v^*(x, N-2) + \dfrac{19\lambda}{48}v^*(x, N).
    \end{equation*}
    Again using the fact that $v^*(x, y) < v^*(x, y+1)$ for all $3 \leq y \leq N-1$, we obtain
    \begin{align*}
        &\dfrac{4\lambda}{9}v^*(x, N-3) + \dfrac{5\lambda}{27}v^*(x, N-2) + \dfrac{10\lambda}{27}v^*(x, N) \\
        \nonumber &= \dfrac{4\lambda}{9}v^*(x, N-3) + \left(\dfrac{5\lambda}{27} - \dfrac{11\lambda}{432}\right)v^*(x, N-2) + \dfrac{11\lambda}{432}v^*(x, N-2) + \dfrac{10\lambda}{27}v^*(x, N) \\
        \nonumber &< \dfrac{4\lambda}{9}v^*(x, N-3) + \dfrac{23\lambda}{144}v^*(x, N-2) + \dfrac{19\lambda}{48}v^*(x, N) \\
        \nonumber &= \left(\dfrac{4\lambda}{9} - \dfrac{\lambda}{18}\right)v^*(x, N-3) + \dfrac{\lambda}{18}v^*(x, N-3) + \dfrac{23\lambda}{144}v^*(x, N-2) + \dfrac{19\lambda}{48}v^*(x, N) \\
        \nonumber &< \dfrac{7\lambda}{18}v^*(x, N-3) + \dfrac{31\lambda}{144}v^*(x, N-2) + \dfrac{19\lambda}{48}v^*(x, N).
    \end{align*}
    Finally, for $y < N-3$, the definition of the reward function and expressions (\ref{trans_probs_3}) and (\ref{trans_probs_extra_3}) yield
    \begin{equation*}
        \hat{r}(s, a_{12}') + \lambda \sum\limits_{s'\in \hat{S}}\hat{P}(s'|s, a_{12}')v^*(s') = \dfrac{5\lambda}{8}v^*(x, y) + \dfrac{\lambda}{4}v^*(x, y+1) + \dfrac{\lambda}{8}v^*(x, y+2)
    \end{equation*}
    and
    \begin{equation*}
        \hat{r}(s, a_{12}) + \lambda \sum\limits_{s'\in \hat{S}}\hat{P}(s'|s, a_{12})v^*(s') = \dfrac{5\lambda}{9}v^*(x, y) + \dfrac{7\lambda}{27}v^*(x, y+1) + \dfrac{5\lambda}{27}v^*(x, y+2).
    \end{equation*}
    Invoking once more the fact that $v^*(x,y) < v^*(x, y+1)$ for all $3 \leq y \leq N-1$, we conclude
    \begin{align*}
        &\dfrac{5\lambda}{8}v^*(x, y) + \dfrac{\lambda}{4}v^*(x, y+1) + \dfrac{\lambda}{8}v^*(x, y+2)\\
        \nonumber &= \dfrac{5\lambda}{8}v^*(x,y) + \left(\dfrac{\lambda}{4} - \dfrac{13\lambda}{216}\right)v^*(x, y+1) + \dfrac{13\lambda}{216}v^*(x, y+1) + \dfrac{\lambda}{8}v^*(x, y+2) \\
        \nonumber &< \dfrac{5\lambda}{8}v^*(x,y) + \dfrac{41\lambda}{216}v^*(x, y+1) + \dfrac{5\lambda}{27}v^*(x, y+2) \\
        \nonumber &= \left(\dfrac{5\lambda}{8} - \dfrac{5\lambda}{72}\right)v^*(x,y) + \dfrac{5\lambda}{72}v^*(x,y) + \dfrac{41\lambda}{216}v^*(x, y+1) + \dfrac{5\lambda}{27}v^*(x, y+2) \\
        \nonumber &< \dfrac{5\lambda}{9}v^*(x, y) + \dfrac{7\lambda}{27}v^*(x, y+1) + \dfrac{5\lambda}{27}v^*(x, y+2).
    \end{align*}
\end{proof}
Lemma \ref{corner_vs_middle} implies that we can leave actions $a_{11}', a_{12}', a_{21}'$ and $a_{22}'$ out of consideration whenever the corresponding side of the rectangle has length greater than 2. In Lemma \ref{corner_vs_middle_small}, we provide a condition under which these actions are suboptimal when the side length equals 2 as well. First, we let $(\hat{S}, \hat{A}', \hat{P}', \hat{r}')$ denote an MDP that is identical to the auxiliary MDP, apart from the fact that the action spaces of rectangles that have a side of length 2 are extended with the corresponding actions in the set $\{a_{11}, a_{12}, a_{21}$, $a_{22}\}$. That is,
    \begin{equation*}
        \hat{A}'(i, j) = \begin{cases}
            \hat{A}(i,j), &\text{if } i, j \neq 2,\\
            \hat{A}(i,j) \cup \{a_{11}, a_{12}\}, &\text{if } i = 2, \quad j = 3, ..., N-3, \\
            \hat{A}(i,j) \cup \{a_{11}\}, &\text{if } i = 2, \quad j = N-2, \\
            \hat{A}(i,j) \cup \{a_{21}, a_{22}\}, &\text{if } i = 3, ..., N-3, \quad j = 2, \\
            \hat{A}(i,j) \cup \{a_{21}\}, &\text{if } i = N-2, \quad j = 2, \\
            \hat{A}(i,j) \cup \{a_{11}, a_{12}, a_{21}, a_{22}\}, &\text{if } i = j = 2. 
        \end{cases}
    \end{equation*}
    Let corresponding transition probabilities be defined analogously to expressions (\ref{trans_probs_1})--(\ref{trans_probs_4}), (\ref{trans_probs_6})--(\ref{trans_probs_9}), i.e., 
    \begin{equation*}
        \hat{P}'((i',j')|(2, j), a_{11}) = \begin{cases}
            1/3, &\text{if } i' = 2, j'= j, \\
            2/3, &\text{if } i'=i, j'=j+1, \\
            0, &\text{otherwise,}
        \end{cases}
    \end{equation*}
    for $j = 2, \ldots, N-3$ in accordance with expression (\ref{trans_probs_1}) and
    \begin{equation*}
        \hat{P}'((i',j')|(2, N-2), a_{11}) = \begin{cases}
            1/4, &\text{if } i'=2, j'= N-2, \\
            3/4, &\text{if } i'=2, j'=N, \\
            0, &\text{otherwise,}
        \end{cases}
    \end{equation*}
    in accordance with (\ref{trans_probs_6}). The remaining cases follow from symmetry. Let the reward function be defined analogously to expression (\ref{reward}). Lemma \ref{corner_vs_middle_small} shows that this artificially extended MDP can provide an easy way to discard actions $a_{11}', a_{12}', a_{21}'$ and $a_{22}'$ for rectangles with side length 2.   

\begin{lemma}\label{corner_vs_middle_small}
    Let $\pi^* = (d^*)^{\infty}$ denote an optimal policy in the MDP $(\hat{S}, \hat{A}, \hat{P}, \hat{r})$. Suppose that, for each optimal policy $\hat{\pi}^* = (\hat{d}^*)^{\infty}$ in the MDP $(\hat{S}, \hat{A}', \hat{P}', \hat{r}')$, all of the following conditions hold:
    \begin{enumerate}
        \item $\hat{d}^*(2,j) \notin \{a_{11}, a_{12}\}$, $j = 3, \ldots, N-3$, 
        \item $\hat{d}^*(2, N-2) \neq a_{11}$,
        \item $\hat{d}^*(i, 2) \notin \{a_{21}, a_{22}\}$, $i = 3, \ldots, N-3$,
        \item $\hat{d}^*(N-2, 2) \neq a_{21}$,
        \item $\hat{d}^*(2, 2) \notin \{a_{11}, a_{12}, a_{21}, a_{22}\}$.
    \end{enumerate}
    This implies
    \begin{enumerate}
        \item $d^*(2, j) \notin \{a_{11}', a_{12}'\}$, $j = 3, \ldots, N-3$,
        \item $d^*(2, N-2) \neq a_{11}'$, 
        \item $d^*(i,2) \notin \{a_{21}', a_{22}'\}$, $i = 3, \ldots, N-3$,
        \item $d^*(N-2, 2) \neq a_{21}'$,
        \item $d^*(2, 2) \notin \{a_{11}', a_{12}', a_{21}', a_{22}'\}$.
    \end{enumerate} 
\end{lemma}
\begin{proof}
    From the fact that none of the actions that were added to $\hat{A}$ to obtain the action space $\hat{A}'$ are optimal in the MDP $(\hat{S}, \hat{A}', \hat{P}', \hat{r})$, it follows that the optimal value function of the latter MDP is equivalent to that of the MDP $(\hat{S}, \hat{A}, \hat{P}, \hat{r})$. We show that this implies $d^*(s) \neq a_{11}'$ for $s = (2,j)$, $j = 3, \ldots, N-3$. The remaining cases follow from a similar argument. Since $a_{11}$ is not optimal in state $s$ in the MDP $(\hat{S}, \hat{A}', \hat{P}', \hat{r})$, there is an action $a^* \in \hat{A}'(s)\setminus \{a_{11}\}$ such that 
    \begin{equation}\label{a11_small_subopt}
        \hat{r}'(s, a_{11}) + \lambda \sum\limits_{s'\in \hat{S}}\hat{P}'(s'|s, a_{11})v^*(s') < \hat{r}'(s, a^*) + \lambda \sum\limits_{s' \in \hat{S}}\hat{P}'(s'|s, a^*)v^*(s'),
    \end{equation}
    where $v^*: S \rightarrow \mathbb{R}$ denotes the optimal value function of both the MDPs. The same arguments that led to Lemma \ref{corner_vs_middle} can be applied to show that 
    \begin{equation}\label{a11_vs_a11'_small}
        \hat{r}'(s, a_{11}') + \lambda \sum\limits_{s'\in \hat{S}}\hat{P}'(s'|s, a_{11}')v^*(s') < \hat{r}'(s, a_{11}) + \lambda \sum\limits_{s' \in \hat{S}}\hat{P}'(s'|s, a_{11})v^*(s'),
    \end{equation}
    It now follows from expressions (\ref{a11_small_subopt}) and (\ref{a11_vs_a11'_small}) that 
    \begin{equation*}
        \hat{r}'(s, a_{11}') + \lambda \sum\limits_{s'\in \hat{S}}\hat{P}'(s'|s, a_{11}')v^*(s') < \hat{r}'(s, a^*) + \lambda \sum\limits_{s' \in \hat{S}}\hat{P}'(s'|s, a^*)v^*(s').
    \end{equation*}
    Since this statement does not depend on the presence of action $a_{11}$ in the action space $\hat{A}'(s)$, we can conclude that for any stationary deterministic optimal policy $\pi^* = (d^*)^{\infty}$ in the MDP $(\hat{S}, \hat{A}, \hat{P}, \hat{r})$, we have $d^*(2,j) \neq a_{11}'$, $j = 3, \ldots, N-3$. 
\end{proof}

Now that actions $0$, $\tilde{a}$, $a_{11}'$, $a_{21}'$, $a_{12}'$ and $a_{22}'$ are dealt with, we unravel the structure of the optimal policy in the auxiliary MDP, which is specified in Theorem \ref{Optimal_policy}. Figure \ref{fig_optimal_policy} provides a visualization of this result.

\paragraph{Proof of Theorem \ref{Optimal_policy}}\quad \\
We show that the value function $v^{\pi^*}_{\lambda}: \hat{S} \rightarrow \mathbb{R}$ of a stationary, deterministic policy $\pi^* = (d^*)^{\infty}$ satisfies the Bellman optimality equations (\ref{Bellman_equations}) if and only if it has the specified form. The statement then follows immediately from Theorem \ref{Bellman_equations_thm}, since the state and action spaces are both finite. Note that it suffices to consider states of the form $(i,j)$, $i \geq j$. The analogous results for states of the form $(i,j)$, $i < j$ follow directly from symmetry. To obtain the desired result for states $(i, 2)$, $i = 2, \ldots, N-2$ (resp. $(2, j)$, $j = 2, \ldots, N-2$), we prove the statement for the MDP $(\hat{S}, \hat{A}', \hat{P}', \hat{r}')$. The result for the original auxiliary MDP $(\hat{S}, \hat{A}, \hat{P}, \hat{r})$ then follows immediately from Lemma \ref{corner_vs_middle_small}. 

Let $\Pi_1$ denote the space of stationary, deterministic policies $\pi_1 = d_1^{\infty}$ that satisfy $d_1(i,j) \in A^*_1(i,j)$ for all $(i,j) \in \hat{S}$. Similarly, let $\Pi_2$ denote the space of stationary, deterministic policies $\pi_2 = d_2^{\infty}$ that satisfy $d_2(i,j) \in A^*_2(i,j)$ for all $(i,j) \in \hat{S}$. We proceed to show that $v^{\pi^*}_{\lambda}: \hat{S} \rightarrow \mathbb{R}$ for $\lambda \in (\lambda_c, 1)$ satisfies the Bellman optimality equations if and only if $\pi^* \in \Pi_1$, that $v^{\pi^*}_{\lambda}: \hat{S} \rightarrow \mathbb{R}$ for $\lambda \in (0, \lambda_c)$ satisfies the Bellman optimality equations if and only if $\pi^* \in \Pi_2$ and that $v^{\pi^*}_{\lambda_c}: \hat{S} \rightarrow \mathbb{R}$ satisfies the Bellman optimality equations if and only if $\pi^* \in \Pi_1 \cup \Pi_2$. 

To this end, let $\pi_1 = d_1^{\infty}$ and $\pi_2 = d_2^{\infty}$ be two policies in $\Pi_1$ and $\Pi_2$ respectively. Note that $A^*_k(i,j)$, $k = 1, 2$ is a singleton for all $(i,j) \in \hat{S}$, $(i,j) \neq (N-3, N-3)$. Since actions $a_{12}$ and $a_{22}$ have the same effect in state $(N-3, N-3)$, we let $d_1(N-3, N-3) = d_2(N-3, N-3) = a_{22}$ without loss of generality. Using the transition probabilities (\ref{trans_probs_1}--\ref{trans_probs_extra_8}), we obtain the following recursive expressions for $v_{\lambda}^{\pi_1}$ and $v_{\lambda}^{\pi_2}$:
\begin{align}
    &\label{rec_1_main} v_{\lambda}^{\pi_k}(N,N) = \dfrac{1}{1-\lambda}, \quad k = 1, 2, \\
    &\label{rec_2}v_{\lambda}^{\pi_k}(N, N-2) = \dfrac{3\lambda}{4-\lambda}v_{\lambda}^{\pi_k}(N, N), \quad k = 1, 2,\\
    &\label{rec_3} v_{\lambda}^{\pi_k}(N, N-3) = \dfrac{31\lambda}{8(18-7\lambda)}v_{\lambda}^{\pi_k}(N, N-2) + \dfrac{57\lambda}{8(18-7\lambda)}v_{\lambda}^{\pi_k}(N, N), \quad k = 1, 2, \\
    &\label{rec_4}v_{\lambda}^{\pi_1}(N,j) = \dfrac{2\lambda}{3-\lambda}v_{\lambda}^{\pi_1}(N,j+1), \quad j = 2, \ldots, N-4, \\
    &\label{rec_5}v_{\lambda}^{\pi_2}(N,j) = \begin{cases}
        \dfrac{2\lambda}{3-\lambda}v_{\lambda}^{\pi_2}(N,j+1), &\text{if } j = N-4,\\
        \dfrac{7\lambda}{3(9-5\lambda)}v_{\lambda}^{\pi_2}(N, j+1) + \dfrac{5\lambda}{3(9-5\lambda)}v_{\lambda}^{\pi_2}(N, j+2),  &\text{if } j = 2, \ldots, N-5,
    \end{cases} \\
    &\label{rec_6} v_{\lambda}^{\pi_k}(N-2, N-2) = \dfrac{9\lambda}{2(18-7\lambda)}v_{\lambda}^{\pi_k}(N, N-2) + \dfrac{13\lambda}{2(18-7\lambda)}v_{\lambda}^{\pi_k}(N, N), \quad k = 1, 2, \\
    &\label{rec_7}v_{\lambda}^{\pi_k}(N-2, N-3) = \dfrac{31\lambda}{8(18-7\lambda)}v_{\lambda}^{\pi_k}(N-2, N-2) + \dfrac{57\lambda}{8(18-7\lambda)}v_{\lambda}^{\pi_k}(N-2, N), \quad k = 1, 2, \\
    &\label{rec_8} v_{\lambda}^{\pi_k}(N-2, j) = \dfrac{12}{12-5\lambda}\left(\dfrac{\lambda}{8}v_{\lambda}^{\pi_k}(N, j) + \dfrac{25\lambda}{72}v_{\lambda}^{\pi_k}(N, j+1) + \dfrac{\lambda}{9}v_{\lambda}^{\pi_k}(N-2, j+1)\right),\\
    \nonumber &\quad j = 2, \ldots, N-4,\quad  k = 1, 2, \\
    &\label{rec_9} v_{\lambda}^{\pi_k}(N-3, N-3) = \dfrac{31\lambda}{8(18-7\lambda)}v_{\lambda}^{\pi_k}(N-2, N-3) + \dfrac{57\lambda}{8(18-7\lambda)}v_{\lambda}^{\pi_k}(N, N-3), \quad k = 1, 2, \\
    &\label{rec_10_main} v_{\lambda}^{\pi_k}(i,j) = \dfrac{\lambda}{9-4\lambda}\Big(v_{\lambda}^{\pi_k}(i, j+1) + v_{\lambda}^{\pi_k}(i+1, j) + 3v_{\lambda}^{\pi_k}(i+1, j+1)\Big),\\
    \nonumber &\quad i = 2, \ldots, N-3, \quad j = 2, \ldots, N-4, \quad k = 1, 2.
\end{align}
Analogous expressions for states of the form $(i,j)$, $i<j$ follow directly from symmetry. Note that expressions (\ref{rec_1_main}--\ref{rec_10_main}) together with Corollary \ref{Value_hitting_time_indicator_reward} allow us to compute the quantity $\mathbb{E}[\lambda^{\tau^{(i,j), \pi_k}_{(N,N)}}]$ for $k = 1, 2$, $(i,j) \in \hat{S}$.
It suffices to show that for any state $s = (i,j) \in \hat{S}$, we have
\begin{equation}\label{Bellman_high}
    \hat{r}(s, a) + \lambda \sum\limits_{s' \in \hat{S}} \hat{P}(s'|s, a)v_{\lambda}^{\pi_1}(s') < \hat{r}(s, d_1(s)) +  \lambda \sum\limits_{s' \in \hat{S}} \hat{P}(s'|s, d_1(s))v_{\lambda}^{\pi_1}(s'),
\end{equation}
for all $a \in \hat{A}(i,j),   a \notin A^*_1(s),  \lambda > \lambda_c$,
\begin{equation}\label{Bellman_low}
    \hat{r}(s, a) + \lambda \sum\limits_{s' \in \hat{S}} \hat{P}(s'|s, a)v_{\lambda}^{\pi_2}(s') < \hat{r}(s, d_2(s)) +  \lambda \sum\limits_{s' \in \hat{S}} \hat{P}(s'|s, d_2(s))v_{\lambda}^{\pi_2}(s'),
\end{equation}
for all $a \in \hat{A}(i,j),   a \notin A^*_2(s),  \lambda < \lambda_c$ and
\begin{align}\label{Bellman_crit}
    v_{\lambda_c}^{\pi_1}(s) &= v_{\lambda_c}^{\pi_2}(s) := v_{\lambda_c}(s), \\
    \hat{r}(s, d_1(s)) + \lambda_c \sum\limits_{s' \in \hat{S}} \hat{P}(s'|s, d_1(s))v_{\lambda_c}(s') \nonumber &= \hat{r}(s, d_2(s)) +  \lambda \sum\limits_{s' \in \hat{S}} \hat{P}(s'|s, d_2(s))v_{\lambda_c}(s') \\
    \nonumber &> \hat{r}(s, a) + \lambda_c \sum\limits_{s' \in \hat{S}} \hat{P}(s'|s, a)v_{\lambda_c}(s'),
\end{align}
for all $a \in \hat{A}(i,j), a \notin A^*_1(s) \cup A^*_2(s)$. Note that the validity of these inequalities is trivial for states $(N, N)$ and $(N, N-2)$. \\
\\
Assuming that $v_{\lambda^c}^{\pi_1}(N,j) = v_{\lambda^c}^{\pi_2}(N,j) := v_{\lambda^c}(N,j)$, which will be proved at a later point, the inequalities for states $(N, j)$, $j = 2, \ldots, N-3$, reduce to
\begin{align*}
    &\lambda \sum\limits_{(i',j') \in \hat{S}} \hat{P}((i',j')|(N, N-3), a_{12})v_{\lambda}^{\pi_k}(i',j') > \lambda \sum\limits_{(i',j') \in \hat{S}} \hat{P}((i',j')|(N, N-3), a_{11})v_{\lambda}^{\pi_k}(i',j'), \\
    \nonumber &\quad \text{for }k = 1, \lambda \in [\lambda_c, 1) \text{ and } k =2, \lambda \in (0, \lambda_c],\\
    &\lambda \sum\limits_{(i',j') \in \hat{S}} \hat{P}((i',j')|(N, N-4), a_{11})v_{\lambda}^{\pi_k}(i',j') > \lambda \sum\limits_{(i',j') \in \hat{S}} \hat{P}((i',j')|(N, N-4), a_{12})v_{\lambda}^{\pi_k}(i',j'), \\
    \nonumber &\quad \text{for } k = 1, \lambda \in [\lambda_c, 1) \text{ and } k =2, \lambda \in (0, \lambda_c],\\
    &\lambda \sum\limits_{(i',j') \in \hat{S}}\hat{P}((i',j')|(N,j), a_{11})v_{\lambda}^{\pi_1}(i',j') > \lambda \sum\limits_{(i',j') \in \hat{S}}\hat{P}((i',j')|(N,j), a_{12})v_{\lambda}^{\pi_1}(i',j'),\\
    &\quad \text{for }\lambda \in (\lambda^c, 1), \quad j = 2, \ldots, N-5, \\
     &\lambda \sum\limits_{(i',j') \in \hat{S}}\hat{P}((i',j')|(i,j), a_{12})v_{\lambda}^{\pi_2}(i',j') > \lambda \sum\limits_{(i',j') \in \hat{S}}\hat{P}((i',j')|(i,j), a_{11})v_{\lambda}^{\pi_2}(i',j'), \\
     \nonumber &\quad \text{for } \lambda \in (0, \lambda^c) \text{ and } j = 2, \ldots, N-5, \text{ and }\\
     &\lambda_c \sum\limits_{(i',j') \in \hat{S}}\hat{P}((i',j')|(N,j), a_{11})v_{\lambda_c}(i',j') = \lambda_c \sum\limits_{(i',j') \in \hat{S}}\hat{P}((i',j')|(N,j), a_{12})v_{\lambda_c}(i',j'),\\ 
     &\quad \text{for } j = 2, \ldots, N-5.
\end{align*}
Inserting expressions (\ref{trans_probs_1}, \ref{trans_probs_3}, \ref{trans_probs_8}), we obtain the following set of inequalities:
\begin{align}
    &\label{cond_1}8v^{\pi_k}_{\lambda}(N, N-3) - 65v^{\pi_k}_{\lambda}(N, N-2) + 57v^{\pi_k}_{\lambda}(N, N)  > 0, \quad k = 1, \lambda \in [\lambda_c, 1) \text{ and } k =2, \lambda \in (0, \lambda_c],\\
     &\label{cond_2} -6v^{\pi_k}_{\lambda}(N, N-4) + 11v^{\pi_k}_{\lambda}(N, N-3) -5v^{\pi_k}_{\lambda}(N, N-2) > 0, \quad k = 1, \lambda \in [\lambda_c, 1) \text{ and } k =2, \lambda \in (0, \lambda_c],\\
    &\label{cond_3} -6v^{\pi_1}_{\lambda}(N, j) + 11v^{\pi_1}_{\lambda}(N, j+1) -5v^{\pi_1}_{\lambda}(N, j+2) > 0, \quad \lambda \in (\lambda_c, 1), \quad j = 2, \ldots, N-5, \\
     &\label{cond_4} 6v^{\pi_2}_{\lambda}(N, j) - 11v^{\pi_2}_{\lambda}(N, j+1) +5v^{\pi_2}_{\lambda}(N, j+2) > 0, \quad \lambda \in (0, \lambda_c), \quad j = 2, \ldots, N-5, \text{ and } \\
    &\label{cond_5}-6v_{\lambda_c}(N, j) + 11v_{\lambda_c}(N, j+1) -5v_{\lambda_c}(N, j+2) = 0, \quad j = 2, \ldots, N-5.
\end{align}
In a similar way, expressions (\ref{Bellman_high}--\ref{Bellman_crit}) for states $(N-2, j)$, $j = 2, \ldots, N-3$, using expressions (\ref{trans_probs_1}, \ref{trans_probs_3}, \ref{trans_probs_7}, \ref{trans_probs_8}, \ref{trans_probs_10}, \ref{trans_probs_11}), the fact that actions $a_{11}$ and $a_{21}$ have the same effect in state $(N-2, N-2)$ by symmetry and the fact that $v^{\pi_k}_{\lambda}(N-2, N) = v^{\pi_k}_{\lambda}(N, N-2)$ for $k = 1, 2$, $\lambda \in (0, 1)$, boil down to
\newpage
\begin{align}
    &\label{cond_6} 5v^{\pi_k}_{\lambda}(N-2, N-2) - 18v^{\pi_k}_{\lambda}(N, N-2) + 13v^{\pi_k}_{\lambda}(N, N) > 0, \\
    &\label{cond_7} 8v^{\pi_k}_{\lambda}(N-2, N-3) - 65v^{\pi_k}_{\lambda}(N-2, N-2)+57v^{\pi_k}_{\lambda}(N-2, N) > 0, \\
    &\label{cond_8} 20v^{\pi_k}_{\lambda}(N-2, N-3) + 31v^{\pi_k}_{\lambda}(N-2, N-2) + 57v^{\pi_k}_{\lambda}(N-2, N) - 108v^{\pi_k}_{\lambda}(N, N-3) >0, \\
    &\label{cond_9} -4v^{\pi_k}_{\lambda}(N-2, N-3) + 15v^{\pi_k}_{\lambda}(N-2, N-2) - 18v^{\pi_k}_{\lambda}(N, N-3) + 7v^{\pi_k}_{\lambda}(N, N-2) >0, \\
    &\label{cond_10} 6v^{\pi_k}_{\lambda}(N-2, j) - 40 v^{\pi_k}_{\lambda}(N-2, j+1) + 9v^{\pi_k}_{\lambda}(N, j) + 25v^{\pi_k}_{\lambda}(N, j+1) > 0, \quad j = 2, \ldots, N-4, \\
    &\label{cond_11}-30v^{\pi_k}_{\lambda}(N-2, j) -32v^{\pi_k}_{\lambda}(N-2, j+1) - 40v^{\pi_k}_{\lambda}(N-2, j+2) + 27v^{\pi_k}_{\lambda}(N, j)+ 75v^{\pi_k}_{\lambda}(N, j+1) > 0, \\
    \nonumber &\quad j = 2, \ldots, N-4, \\
    &\label{cond_12} 12v^{\pi_k}_{\lambda}(N-2, j) + 8v^{\pi_k}_{\lambda}(N-2, j+1) -45v^{\pi_k}_{\lambda}(N, j) + 25v^{\pi_k}_{\lambda}(N, j+1) > 0, \quad j = 2, \ldots, N-4,  
\end{align}
for $k = 1, \lambda \in [\lambda_c, 1) \text{ and } k =2, \lambda \in (0, \lambda_c]$.

For states $(N-3, j)$, $j = 2, \ldots, N-3$, by inserting expressions (\ref{trans_probs_1}, \ref{trans_probs_2}, \ref{trans_probs_3}, \ref{trans_probs_5}, \ref{trans_probs_9}) and using the fact that actions $a_{21}$ and $a_{11}$ have an equivalent effect in state $(N-3, N-3)$, inequalities (\ref{Bellman_high}--\ref{Bellman_crit}) reduce to
\begin{align}
    &\label{cond_13}8v^{\pi_k}_{\lambda}(N-3, N-3) - 65v^{\pi_k}_{\lambda}(N-2, N-3) + 57v^{\pi_k}_{\lambda}(N, N-3)  > 0, \\
    &\label{cond_14}-8v^{\pi_k}_{\lambda}(N-3, N-3) - v^{\pi_k}_{\lambda}(N-2, N-3) - 48v^{\pi_k}_{\lambda}(N-2, N-2) + 57v^{\pi_k}_{\lambda}(N, N-3)  > 0, \\
    &\label{cond_15} v^{\pi_k}_{\lambda}(N-3, j) - 5v^{\pi_k}_{\lambda}(N-3, j+1) + v^{\pi_k}_{\lambda}(N-2, j) + 3v^{\pi_k}_{\lambda}(N-2, j+1) > 0, \quad j = 2, \ldots, N-4, \\
    &\label{cond_16} -3v^{\pi_k}_{\lambda}(N-3, j) -4v^{\pi_k}_{\lambda}(N-3, j+1) - 5v^{\pi_k}_{\lambda}(N-3, j+2) + 3v^{\pi_k}_{\lambda}(N-2, j)\\
    \nonumber &+ 9v^{\pi_k}_{\lambda}(N-2, j+1) > 0, \quad j = 2, \ldots, N-4, \\
    &\label{cond_17} v^{\pi_k}_{\lambda}(N-3, j) + v^{\pi_k}_{\lambda}(N-3, j+1) - 5v^{\pi_k}_{\lambda}(N-2, j) + 3v^{\pi_k}_{\lambda}(N-2, j+1)  >0, \quad j = 2, \ldots, N-4, \\
    &\label{cond_18} 8v^{\pi_k}_{\lambda}(N-3, j) + 16v^{\pi_k}_{\lambda}(N-3, j+1) - 15v^{\pi_k}_{\lambda}(N-2, j) + 48v^{\pi_k}_{\lambda}(N-2, j+1) - 57v^{\pi_k}_{\lambda}(N, j)  > 0, \\
    \nonumber &\quad j = 2, \ldots, N-4,
\end{align}
for $k = 1, \lambda \in [\lambda_c, 1) \text{ and } k =2, \lambda \in (0, \lambda_c]$. \\
\\
Finally, inequalities (\ref{Bellman_high}--\ref{Bellman_crit}) for states $(i,j)$, $i, j = 2, \ldots, N-4$, using expressions (\ref{trans_probs_1}--\ref{trans_probs_5}), reduce to 
\begin{align}
    &\label{cond_19} v^{\pi_k}_{\lambda}(i,j) + v^{\pi_k}_{\lambda}(i+1, j) -5v^{\pi_k}_{\lambda}(i, j+1) + 3v^{\pi_k}_{\lambda}(i+1, j+1) > 0, \\
     &\label{cond_20} v^{\pi_k}_{\lambda}(i,j) - 5v^{\pi_k}_{\lambda}(i+1, j) + v^{\pi_k}_{\lambda}(i, j+1) + 3v^{\pi_k}_{\lambda}(i+1, j+1)> 0, \\
    &\label{cond_21} -3v^{\pi_k}_{\lambda}(i,j) + 3v^{\pi_k}_{\lambda}(i+1, j) -4v^{\pi_k}_{\lambda}(i, j+1) + 9v^{\pi_k}_{\lambda}(i+1, j+1) -5v^{\pi_k}_{\lambda}(i, j+2) > 0, \\
    &\label{cond_22} -3v^{\pi_k}_{\lambda}(i,j) - 4v^{\pi_k}_{\lambda}(i+1, j) + 3v^{\pi_k}_{\lambda}(i, j+1) + 9v^{\pi_k}_{\lambda}(i+1, j+1) -5v^{\pi_k}_{\lambda}(i+2, j) > 0,
\end{align}
for $k = 1, \lambda \in [\lambda_c, 1) \text{ and } k =2, \lambda \in (0, \lambda_c]$. \\
\\
Thus, in order to obtain the desired result, it suffices to prove that the value functions of policies $\pi_1$ and $\pi_2$ satisfy expressions (\ref{cond_1}--\ref{cond_22}), in addition to the equations $v^{\pi_1}_{\lambda^c}(N, j) = v^{\pi_2}_{\lambda^c}(N, j)$ for all $j = 2, \ldots,N-2, N$. To prove the set of (in)equalities, we make use of the recursive expressions (\ref{rec_1_main}--\ref{rec_10_main}) for $v^{\pi_1}_{\lambda}$ and $v^{\pi_2}_{\lambda}$. The proof crucially relies on the fact that rectangles of size at least $3 \times 3$ cannot shrink under the dynamics of the auxiliary MDP, which allows us to show the validity of the expressions by means of backward induction over the size of the rectangle. For states of the form $(N, j)$, $(N-2, j)$ and $(N-3, j)$, $j = 2, \ldots, N$, we use backward induction over the length of the vertical side of the rectangle. For states of the form $(i,j)$, $i, j = 2, \ldots, N-4$, i.e., expressions (\ref{cond_19}--\ref{cond_22}), we invoke a more involved induction argument, which can be outlined as follows:
\begin{itemize}
    \item \textit{Induction base:} First, we show that the expressions are satisfied for states of the form $(i, N-4)$ and $(N-4, j)$, $i, j = 2, \ldots, N-4$, using an embedded induction argument over the length of the shortest side of the rectangle. 
    \item \textit{Induction hypothesis}: We assume that the expressions are valid for states of the form $(i, n+1)$ and $(n+1, j)$, $i,j = 2, \ldots, n+1$ for some $n < N-4$. 
    \item \textit{Induction step:} First, we show that the induction hypothesis implies the correctness of the expressions for state $(n,n)$. Using this result, together with the induction hypothesis, we invoke another embedded induction argument over the length of the shortest side of the rectangle to show that the expressions hold for states of the form $(i,n)$ and $(n, j)$, $i,j = 2, \ldots, n-1$. 
\end{itemize}
The details of the induction proofs are provided in part B of the Supplementary Material. \qed

\section{Discussion}
\label{Discussion}

In this paper, we have introduced the spatiotemporal Markov decision process (STMDP), a new framework for sequential decision making problems that exhibit not only temporal, but also spatial interaction structures. Unlike the related FMDP and GMDP frameworks, the class of STMDPs includes processes that do not permit a factorisation of the state space. We have formulated and analyzed an STMDP inspired by the low-temperature Ising model on a finite, two-dimensional, square lattice, evolving according to the asynchronous Metropolis dynamics. Our analysis heavily relied on a reduction of the state space to local minima of the Hamiltonian, resulting in an auxiliary MDP. We argued that these local minima indeed make the largest contribution to the Bellman optimality equations if the adjustment time is sufficiently long. For the auxiliary MDP, we uncovered the structure of the exact optimal policy by solving the Bellman optimality equations in a recursive manner. Finally, we conducted numerical experiments on the performance of the analogue of this optimal policy in the original STMDP and compared it to the performance of alternative policies. The results of these experiments suggest that this policy obtained from the auxiliary MDP achieves the best performance over a range of different adjustment times. 

This work opens several interesting avenues for future research. First of all, a more rigorous analysis of the Ising STMDP for small values of the adjustment time could be performed, where the policy obtained from the auxiliary MDP may no longer be optimal. In addition, the Ising STMDP for higher temperatures would be a worthwile object of study, although much more difficult to handle. Furthermore, a potential direction would be to analyze the Ising STMDP for different starting and/or target configurations, for example consisting of multiple clusters of different shapes. Additionally, it would be of interest to study the Ising STMDP on the infinite lattice, in different dimensions or on different types of graphs. Finally, the idea of reducing the state space to local minima of the Hamiltonian may yield insights for a more general class of STMDPs, based on dynamics that are reversible with respect to a Gibbs measure.

\bibliographystyle{unsrtnat}
\bibliography{library_arxiv}  %%% Uncomment this line and comment out the ``thebibliography'' section below to use the external .bib file (using bibtex) .





% \end{thebibliography}

\newpage

\section*{Supplementary Material A: supporting tables and figures for computation of transition probabilities}

The derivations of expressions (\ref{trans_probs_1}, \ref{trans_probs_2}) and (\ref{trans_probs_5})--(\ref{trans_probs_11}) are based on similar arguments as in the proof of Lemma \ref{explicit_Phats} and are clarified in Figures \ref{fig_trans_probs_1}--\ref{fig_trans_probs_11} and the corresponding Tables \ref{tab_trans_probs_1}--\ref{tab_trans_probs_11}.

\begin{figure}[H]
\centering
\includegraphics[width=0.35\linewidth]{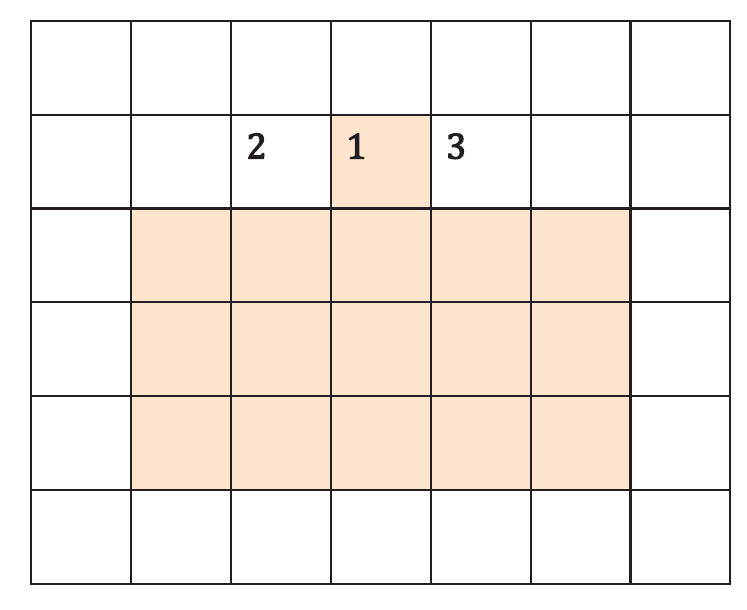}
\caption{Post-decision configuration after flipping a spin at distance 1 from the horizontal side of the rectangle.} 
\label{fig_trans_probs_1}
\end{figure}

\begin{table}[H]
\begin{center}
\caption{Derivation of expression (\ref{trans_probs_1}), corresponding to Figure \ref{fig_trans_probs_1}. Similar derivation leads to expression (\ref{trans_probs_2}).}
\begin{tabular}{ llll } 
\toprule
Sequences of susceptible spins & Probability of selecting sequence & Next state & Transition probability \\
\midrule
$(1)$ & $1/3$ & \multirow{1}{4em}{$(i,j)$}& \multirow{1}{4em}{$1/3$} \\
\midrule
$(2)$ & $1/3$ & \multirow{2}{4em}{$(i, j+1)$} & \multirow{2}{4em}{$2/3$} \\
$(3)$ & $1/3$ & & \\
\bottomrule
\end{tabular}
\label{tab_trans_probs_1}
\end{center}
\end{table}

\begin{figure}[H]
\centering
\includegraphics[width=0.35\linewidth]{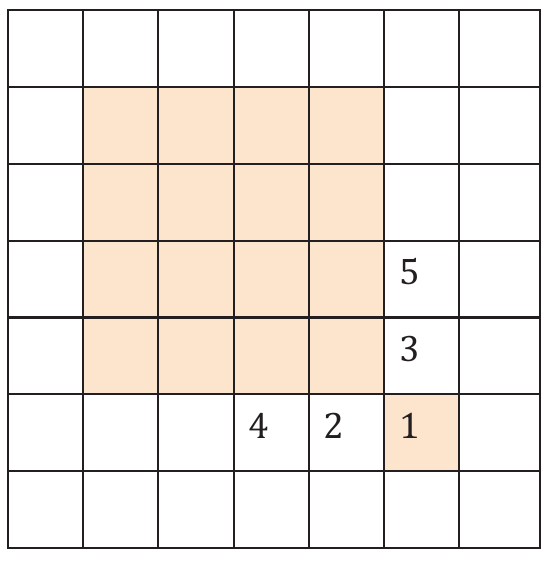}
\caption{Post-decision configuration corresponding to expression (\ref{trans_probs_5}).}
\label{fig_trans_probs_5}
\end{figure}

\begin{table}[H]
\begin{center}
\caption{Derivation of expression (\ref{trans_probs_5}), corresponding to Figure \ref{fig_trans_probs_5}.}
\begin{tabular}{ llll } 
\toprule
Sequences of susceptible spins & Probability of selecting sequence & Next state & Transition probability \\
\midrule
$(1)$ & $1/3$ & \multirow{3}{4em}{$(i,j)$}& \multirow{3}{4em}{$4/9$} \\
$(2, 1, 2)$ & $1/18$ & & \\
$(3, 1, 3)$ & $1/18$ & & \\
\midrule
$(2, 1, 4)$ & $1/18$ & \multirow{2}{4em}{$(i + 1, j)$} & \multirow{2}{4em}{$1/9$} \\
$(2, 4, 1)$ & $1/18$ & & \\
\midrule
$(3, 1, 5)$ & $1/18$ & \multirow{2}{4em}{$(i, j+1)$} & \multirow{2}{4em}{$1/9$} \\
$(3, 5, 1)$ & $1/18$ & & \\
\midrule
$(2, 3)$ & $1/9$ & \multirow{4}{6em}{$(i+1, j+1)$} & \multirow{4}{4em}{$1/3$} \\
$(3, 2)$ & $1/9$ & & \\
$(2, 4, 3)$ & $1/18$ & & \\
$(3, 5, 2)$ & $1/18$ & & \\
\bottomrule
\end{tabular}
\label{tab_trans_probs_5}
\end{center}
\end{table}

\begin{figure}[H]
\centering
\includegraphics[width=0.35\linewidth]{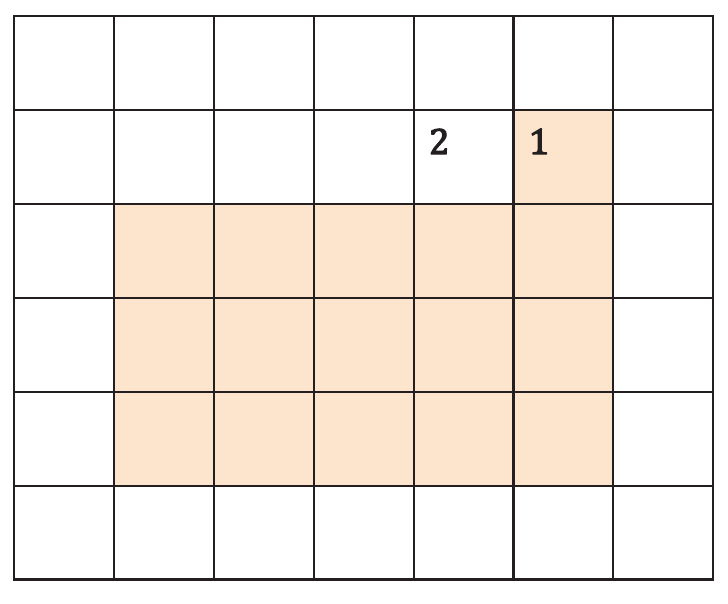}
\caption{Post-decision configuration corresponding to expression (\ref{trans_probs_extra_1}).}
\label{fig_trans_probs_extra_1}
\end{figure}

\begin{table}[H]
\begin{center}
\caption{Derivation of expression (\ref{trans_probs_extra_1}), corresponding to Figure \ref{fig_trans_probs_extra_1}. Similar derivation leads to expression (\ref{trans_probs_extra_2}).}
\begin{tabular}{ llll } 
\toprule
Sequences of susceptible spins & Probability of selecting sequence & Next state & Transition probability \\
\midrule
$(1)$ & $1/2$ & \multirow{1}{4em}{$(i,j)$}& \multirow{1}{4em}{$1/2$} \\
\midrule
$(2)$ & $1/2$ & \multirow{1}{4em}{$(i, j+1)$} & \multirow{1}{4em}{$1/2$} \\
\bottomrule

\end{tabular}
\label{tab_trans_probs_extra_1}
\end{center}
\end{table}

\begin{figure}[H]
\centering
\includegraphics[width=0.35\linewidth]{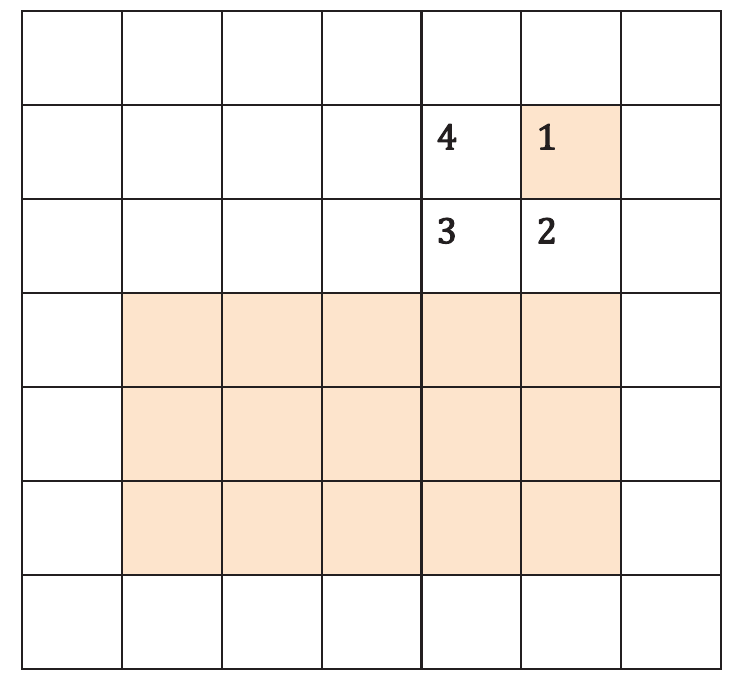}
\caption{Post-decision configuration corresponding to expression (\ref{trans_probs_extra_3}).}
\label{fig_trans_probs_extra_3}
\end{figure}

\begin{table}[H]
\begin{center}
\caption{Derivation of expression (\ref{trans_probs_extra_3}), corresponding to Figure \ref{fig_trans_probs_extra_3}. Similar derivation leads to expression (\ref{trans_probs_extra_4}).}
\begin{tabular}{ llll } 
\toprule
Sequences of susceptible spins & Probability of selecting sequence & Next state & Transition probability \\
\midrule
$(1)$ & $1/2$ & \multirow{2}{4em}{$(i,j)$}& \multirow{2}{4em}{$5/8$} \\
$(2, 1, 2)$ & $1/8$ & & \\
\midrule
$(2, 1, 3)$ & $1/8$ & \multirow{2}{4em}{$(i, j+1)$} & \multirow{2}{4em}{$1/4$} \\
$(2, 3, 1)$ & $1/8$ & & \\
\midrule
$(2, 3, 4)$ & $1/8$ & \multirow{1}{4em}{$(i, j+2)$} & \multirow{1}{4em}{$1/8$} \\
\bottomrule
\end{tabular}
\label{tab_trans_probs_extra_3}
\end{center}
\end{table}

\begin{figure}[H]
\centering
\includegraphics[width=0.35\linewidth]{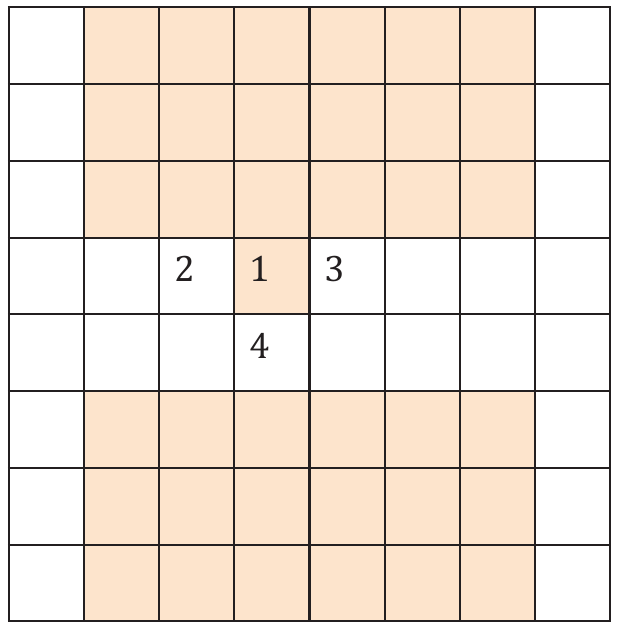}
\caption{Post-decision configuration corresponding to expression (\ref{trans_probs_6}).}
\label{fig_trans_probs_6}
\end{figure}

\begin{table}[H]
\begin{center}
\caption{Derivation of expression (\ref{trans_probs_6}), corresponding to Figure \ref{fig_trans_probs_6}. Similar derivation leads to expression (\ref{trans_probs_7}).}
\begin{tabular}{ llll } 
\toprule
Sequences of susceptible spins & Probability of selecting sequence & Next state & Transition probability \\
\midrule
$(1)$ & $1/4$ & \multirow{1}{6em}{$(i, N-2)$} & \multirow{1}{6em}{$1/4$} \\
\midrule
$(2)$ & $1/4$ & \multirow{3}{6em}{$(i, N)$} & \multirow{3}{6em}{$3/4$} \\
$(3)$ & $1/4$ & & \\
$(4)$ & $1/4$ & & \\
\bottomrule
\end{tabular}
\label{tab_trans_probs_6}
\end{center}
\end{table}

\begin{figure}[H]
\centering
\includegraphics[width=0.35\linewidth]{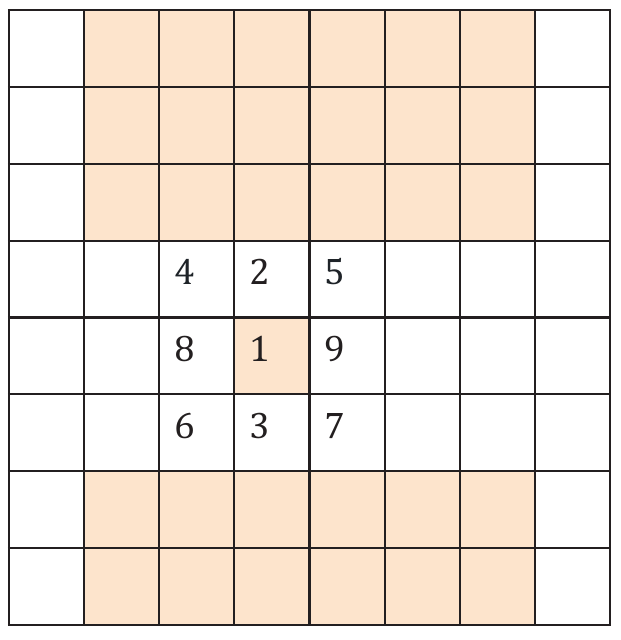}
\caption{Post-decision configuration corresponding to expression (\ref{trans_probs_8}).}
\label{fig_trans_probs_8}
\end{figure}

\begin{table}[H]
\begin{center}
\caption{Derivation of expression (\ref{trans_probs_8}), corresponding to Figure \ref{fig_trans_probs_8}. Similar derivation leads to expression (\ref{trans_probs_9}).}
\begin{tabular}{ llll } 
\toprule
Sequences of susceptible spins & Probability of selecting sequence & Next state & Transition probability \\
\midrule
$(1)$ & $1/3$ & \multirow{3}{6em}{$(i, N-3)$} & \multirow{3}{6em}{$7/18$} \\
$(2, 1, 2)$ & $1/36$ & & \\
$(3, 1, 3)$ & $1/36$ & & \\
\midrule
$(2, 1, 4)$ & $1/36$ & \multirow{12}{6em}{$(i, N-2)$} & \multirow{12}{6em}{$31/144$} \\
$(2, 1, 5)$ & $1/36$ & & \\
$(3, 1, 6)$ & $1/36$ & & \\
$(3, 1, 7)$ & $1/36$ & & \\
$(2, 4, 1)$ & $1/48$ & & \\
$(2, 5, 1)$ & $1/48$ & & \\
$(3, 6, 1)$ & $1/48$ & & \\
$(3, 7, 1)$ & $1/48$ & & \\
$(2, 4, 5, 1)$ & $1/192$ & & \\
$(2, 5, 4, 1)$ & $1/192$ & & \\
$(3, 6, 7, 1)$ & $1/192$ & & \\
$(3, 7, 6, 1)$ & $1/192$ & & \\
\midrule
$(2, 3)$ & $1/12$ & \multirow{22}{6em}{$(i, N)$} & \multirow{22}{6em}{$19/48$} \\
$(3, 2)$ & $1/12$ & & \\
$(2, 4, 3)$ & $1/48$ & & \\
$(2, 5, 3)$ & $1/48$ & & \\
$(3, 6, 2)$ & $1/48$ & & \\
$(3, 7, 2)$ & $1/48$ & & \\
$(2, 4, 5, 3)$ & $1/192$ & & \\
$(2, 5, 4, 3)$ & $1/192$ & & \\
$(3, 6, 7, 2)$ & $1/192$ & & \\
$(3, 7, 6, 2)$ & $1/192$ & & \\
$(2, 4, 5, 8)$ & $1/192$ & & \\
$(3, 6, 7, 8)$ & $1/192$ & & \\
$(3, 7, 6, 8)$ & $1/192$ & & \\
$(2, 5, 4, 8)$ & $1/192$ & & \\
$(2, 4, 5, 9)$ & $1/192$ & & \\
$(2, 5, 4, 9)$ & $1/192$ & & \\
$(3, 6, 7, 9)$ & $1/192$ & & \\
$(3, 7, 6, 9)$ & $1/192$ & & \\
$(2, 4, 8)$ & $1/48$ & & \\
$(2, 5, 9)$ & $1/48$ & & \\
$(3, 6, 8)$ & $1/48$ & & \\
$(3, 7, 9)$ & $1/48$ & & \\
\bottomrule
\end{tabular}
\label{tab_trans_probs_8}
\end{center}
\end{table}

\begin{figure}[H]
\centering
\includegraphics[width=0.35\linewidth]{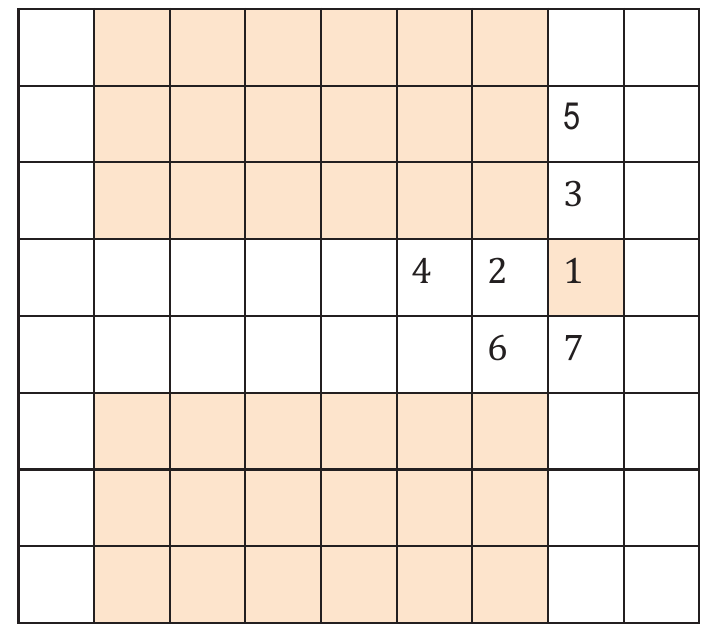}
\caption{Post-decision configuration corresponding to expression (\ref{trans_probs_12}).}
\label{fig_trans_probs_12}
\end{figure}

\begin{table}[H]
\begin{center}
\caption{Derivation of expression (\ref{trans_probs_12}), corresponding to Figure \ref{fig_trans_probs_12}. Similar derivation leads to expression (\ref{trans_probs_10}).} 
\begin{tabular}{ llll } 
\toprule
Sequences of susceptible spins & Probability of selecting sequence & Next state & Transition probability \\
\midrule
$(1)$ & $1/3$ & \multirow{3}{6em}{$(i, N-2)$} & \multirow{3}{6em}{$5/12$} \\
$(2, 1, 2)$ & $1/36$ & & \\
$(3, 1, 3)$ & $1/18$ & & \\
\midrule
$(2, 1, 4)$ & $1/36$ & \multirow{4}{6em}{$(i, N)$} & \multirow{4}{6em}{$1/8$} \\
$(2, 1, 6)$ & $1/36$ & & \\
$(2, 4, 1)$ & $1/24$ & & \\
$(2, 6, 1)$ & $1/36$ & & \\
\midrule
$(3, 1, 5)$ & $1/18$ & \multirow{2}{6em}{$(i+1, N-2)$} & \multirow{2}{6em}{$1/9$} \\
$(3, 5, 1)$ & $1/18$ & & \\
\midrule
$(2, 3)$ & $1/12$ & \multirow{6}{6em}{$(i+1, N)$} & \multirow{6}{6em}{$25/72$} \\
$(2, 4, 3)$ & $1/24$ & & \\
$(2, 6, 3)$ & $1/36$ & & \\
$(2, 6, 7)$ & $1/36$ & & \\
$(3, 2)$ & $1/9$ & & \\
$(3, 5, 2)$ & $1/18$ & & \\
\bottomrule
\end{tabular}
\label{tab_trans_probs_10}
\end{center}
\end{table}

\begin{figure}[H]
\centering
\includegraphics[width=0.35\linewidth]{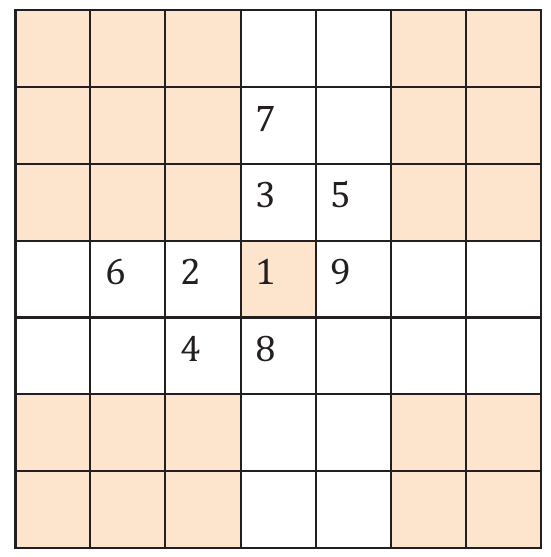}
\caption{Post-decision configuration corresponding to expression (\ref{trans_probs_11}).}
\label{fig_trans_probs_11}
\end{figure}

\begin{table}[H]
\begin{center}
\caption{Derivation of expression (\ref{trans_probs_11}), corresponding to Figure \ref{fig_trans_probs_11}. }
\begin{tabular}{ llll } 
\toprule
Sequences of susceptible spins & Probability of selecting sequence & Next state & Transition probability \\
\midrule
$(1)$ & $1/3$ & \multirow{3}{6em}{$(N-2, N-2)$} & \multirow{3}{6em}{$7/18$} \\
$(2, 1, 2)$ & $1/36$ & & \\
$(3, 1, 3)$ & $1/36$ & & \\
\midrule
$(2, 1, 4)$ & $1/36$ & \multirow{4}{6em}{$(N-2, N)$} & \multirow{4}{6em}{$1/8$ } \\
$(2, 1, 6)$ & $1/36$ & & \\
$(2, 4, 1)$ & $1/36$ & & \\
$(2, 6, 1)$ & $1/24$ & & \\
\midrule
$(3, 1, 5)$ & $1/36$ & \multirow{4}{6em}{$(N, N-2)$} & \multirow{4}{6em}{$1/8$} \\
$(3, 1, 7)$ & $1/36$ & & \\
$(3, 5, 1)$ & $1/36$ & & \\
$(3, 7, 1)$ & $1/24$ & & \\
\midrule
$(2, 3)$ & $1/12$ & \multirow{8}{6em}{$(N, N)$} & \multirow{8}{6em}{$13/36$} \\
$(3, 2)$ & $1/12$ & & \\
$(2, 4, 3)$ & $1/36$ & & \\
$(3, 5, 2)$ & $1/36$ & & \\
$(2, 4, 8)$ & $1/36$ & & \\
$(3, 5, 9)$ & $1/36$ & & \\
$(2, 6, 3)$ & $1/24$ & & \\
$(3, 7, 2)$ & $1/24$ & & \\
\bottomrule
\end{tabular}
\label{tab_trans_probs_11}
\end{center}
\end{table}

\begin{figure}[H]
\centering
\includegraphics[width=0.35\linewidth]{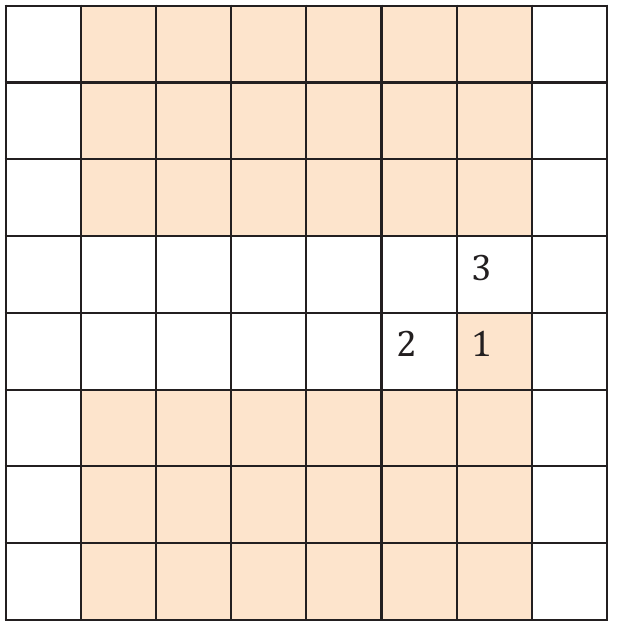}
\caption{Post-decision configuration corresponding to expression (\ref{trans_probs_extra_5}).}
\label{fig_trans_probs_extra_5}
\end{figure}

\begin{table}[H]
\begin{center}
\caption{Derivation of expression (\ref{trans_probs_extra_5}), corresponding to Figure \ref{fig_trans_probs_extra_5}. Similar derivation leads to expression (\ref{trans_probs_extra_6})}
\begin{tabular}{ llll } 
\toprule
Sequences of susceptible spins & Probability of selecting sequence & Next state & Transition probability \\
\midrule
$(1)$ & $1/3$ & \multirow{1}{6em}{$(i, N-2)$} & \multirow{1}{6em}{$1/3$} \\
\midrule
$(2)$ & $1/3$ & \multirow{2}{6em}{$(i, N)$} & \multirow{2}{6em}{$1/3$} \\
$(3)$ & $1/3$ & & \\
\bottomrule
\end{tabular}
\label{tab_trans_probs_extra_5}
\end{center}
\end{table}

\begin{figure}[H]
\centering
\includegraphics[width=0.35\linewidth]{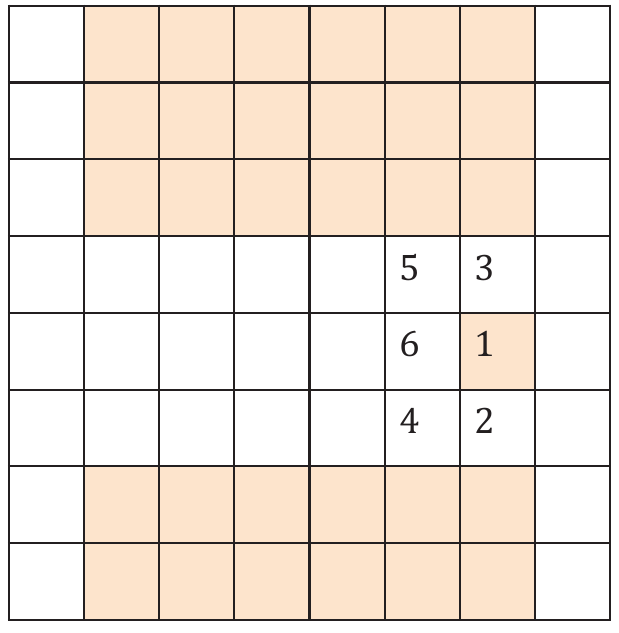}
\caption{Post-decision configuration corresponding to expression (\ref{trans_probs_extra_7}).}
\label{fig_trans_probs_extra_7}
\end{figure}

\begin{table}[H]
\begin{center}
\caption{Derivation of expression (\ref{trans_probs_extra_7}), corresponding to Figure \ref{fig_trans_probs_extra_7}. Similar derivation leads to expression (\ref{trans_probs_extra_8}).}
\begin{tabular}{ llll } 
\toprule
Sequences of susceptible spins & Probability of selecting sequence & Next state & Transition probability \\
\midrule
$(1)$ & $1/3$ & \multirow{3}{6em}{$(i, N-3)$} & \multirow{3}{6em}{$4/9$} \\
$(2, 1, 2)$ & $1/18$ & & \\
$(3, 1, 3)$ & $1/18$ & & \\
\midrule
$(2, 1, 4)$ & $1/18$ & \multirow{4}{6em}{$(i, N-2)$} &  \multirow{4}{6em}{$5/27$} \\
$(2, 4, 1)$ & $1/27$ & & \\
$(3, 1, 5)$ & $1/18$ & & \\
$(3, 5, 1)$ & $1/27$ & & \\
\midrule
$(2, 3)$ & $1/9$ & \multirow{6}{6em}{$(i, N)$} & \multirow{6}{6em}{$10/27$} \\
$(2, 4, 3)$ & $1/27$ & & \\
$(2, 4, 6)$ & $1/27$ & & \\
$(3, 2)$ & $1/9$ & & \\
$(3, 5, 2)$ & $1/27$ & & \\
$(3, 5, 6)$ & $1/27$ & & \\
\bottomrule
\end{tabular}
\label{tab_trans_probs_extra_7}
\end{center}
\end{table}

\begin{figure}[H]
\centering
\includegraphics[width=0.35\linewidth]{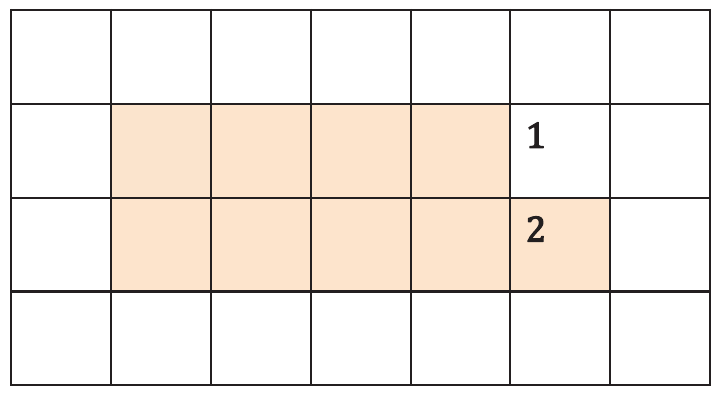}
\caption{Post-decision configuration corresponding to expression (\ref{trans_probs_double_extra_3}).}
\label{fig_trans_probs_double_extra_3}
\end{figure}

\begin{table}[H]
\begin{center}
\caption{Derivation of expression (\ref{trans_probs_double_extra_3}), corresponding to Figure \ref{fig_trans_probs_double_extra_3}. Similar derivation leads to expression (\ref{trans_probs_double_extra_2}).}
\begin{tabular}{ llll } 
\toprule
Sequences of susceptible spins & Probability of selecting sequence & Next state & Transition probability \\
\midrule
$(1)$ & $1/2$ & $(i,2)$ & $1/2$ \\
\midrule
$(2)$ & $1/2$ & $(i-1, 2)$ & $1/2$ \\
\bottomrule
\end{tabular}
\label{tab_trans_probs_double_extra_3}
\end{center}
\end{table}

\begin{figure}[H]
\centering
\includegraphics[width=0.3\linewidth]{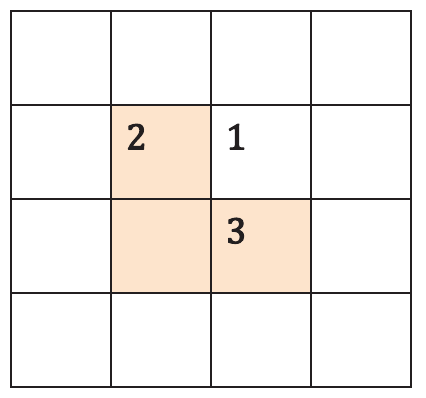}
\caption{Post-decision configuration corresponding to expression (\ref{trans_probs_double_extra_4}).}
\label{fig_trans_probs_double_extra_4}
\end{figure}

\begin{table}[H]
\begin{center}
\caption{Derivation of expression (\ref{trans_probs_double_extra_4}), corresponding to Figure \ref{fig_trans_probs_double_extra_4}.}
\begin{tabular}{ llll } 
\toprule
Sequences of susceptible spins & Probability of selecting sequence & Next state & Transition probability \\
\midrule
$(1)$ & $1/3$ & $(2,2)$ & $1/3$ \\
\midrule
$(2)$ & $1/3$ & \multirow{2}{2em}{$(0, 0)$} & \multirow{2}{2em}{$2/3$} \\
$(3)$ & $1/3$ & & \\
\bottomrule
\end{tabular}
\label{tab_trans_probs_double_extra_4}
\end{center}
\end{table}

\newpage

\section*{Supplementary Material B: Details of the proof of Theorem 3.1}

Let policies $\pi_1 = d_1^{\infty}$ and $\pi_2 = d_2^{\infty}$ denote two arbitrary policies in $\Pi_1$ and $\Pi_2$ respectively. Recall that the value functions $v^{\pi_1}_{\lambda}: \hat{S} \rightarrow \mathbb{R}$ and $v^{\pi_2}_{\lambda}: \hat{S} \rightarrow \mathbb{R}$ satisfy the following recursive expressions: 
\begin{align}
    &\label{rec_1} v_{\lambda}^{\pi_k}(N,N) = \dfrac{1}{1-\lambda}, \quad k = 1, 2, \\
    &\label{rec_2}v_{\lambda}^{\pi_k}(N, N-2) = \dfrac{3\lambda}{4-\lambda}v_{\lambda}^{\pi_k}(N, N), \quad k = 1, 2,\\
    &\label{rec_3} v_{\lambda}^{\pi_k}(N, N-3) = \dfrac{31\lambda}{8(18-7\lambda)}v_{\lambda}^{\pi_k}(N, N-2) + \dfrac{57\lambda}{8(18-7\lambda)}v_{\lambda}^{\pi_k}(N, N), \quad k = 1, 2, \\
    &\label{rec_4}v_{\lambda}^{\pi_1}(N,j) = \dfrac{2\lambda}{3-\lambda}v_{\lambda}^{\pi_1}(N,j+1), \quad j = 2, \ldots, N-4, \\
    &\label{rec_5}v_{\lambda}^{\pi_2}(N,j) = \begin{cases}
        \dfrac{2\lambda}{3-\lambda}v_{\lambda}^{\pi_2}(N,j+1), &\text{if } j = N-4,\\
        \\
        \dfrac{7\lambda}{3(9-5\lambda)}v_{\lambda}^{\pi_2}(N, j+1) + \dfrac{5\lambda}{3(9-5\lambda)}v_{\lambda}^{\pi_2}(N, j+2), &\text{if } j = 2, \ldots, N-5,
    \end{cases} \\
    &\label{rec_6} v_{\lambda}^{\pi_k}(N-2, N-2) = \dfrac{9\lambda}{2(18-7\lambda)}v_{\lambda}^{\pi_k}(N, N-2) + \dfrac{13\lambda}{2(18-7\lambda)}v_{\lambda}^{\pi_k}(N, N), \quad k = 1, 2, \\
    &\label{rec_7}v_{\lambda}^{\pi_k}(N-2, N-3) = \dfrac{31\lambda}{8(18-7\lambda)}v_{\lambda}^{\pi_k}(N-2, N-2) + \dfrac{57\lambda}{8(18-7\lambda)}v_{\lambda}^{\pi_k}(N-2, N), \quad k = 1, 2, \\
    &\label{rec_8} v_{\lambda}^{\pi_k}(N-2, j) = \dfrac{12}{12-5\lambda}\Big(\dfrac{\lambda}{8}v_{\lambda}^{\pi_k}(N, j) + \dfrac{25\lambda}{72}v_{\lambda}^{\pi_k}(N, j+1) + \dfrac{\lambda}{9}v_{\lambda}^{\pi_k}(N-2, j+1)\Big),\\
    \nonumber &\quad j = 2, \ldots, N-4,\quad  k = 1, 2, \\
    &\label{rec_9} v_{\lambda}^{\pi_k}(N-3, N-3) = \dfrac{31\lambda}{8(18-7\lambda)}v_{\lambda}^{\pi_k}(N-2, N-3) + \dfrac{57\lambda}{8(18-7\lambda)}v_{\lambda}^{\pi_k}(N, N-3), \quad k = 1, 2, \\
    &\label{rec_10} v_{\lambda}^{\pi_k}(i,j) = \dfrac{\lambda}{9-4\lambda}\Big(v_{\lambda}^{\pi_k}(i, j+1) + v_{\lambda}^{\pi_k}(i+1, j) + 3v_{\lambda}^{\pi_k}(i+1, j+1)\Big), \\
    \nonumber &\quad i = 2, \ldots, N-3, \quad j = 2, \ldots, N-4, \quad k = 1, 2.
\end{align}

In Table \ref{explicit_v_stars}, explicit expressions for these value functions are recorded for several states, which will be required at a later point in the proof. Using expressions (\ref{rec_1}-\ref{rec_10}), we show that $v^{\pi_1}_{\lambda}$ and $v^{\pi_2}_{\lambda}$ satisfy $v^{\pi_1}_{\lambda^c}(N,j) = v^{\pi_2}_{\lambda^c}(N, j) := v_{\lambda^c}(N,j)$ for all $j = 2, \ldots, N$ and expressions (\ref{cond_1}-\ref{cond_5}, \ref{cond_6}-\ref{cond_12}, \ref{cond_13}-\ref{cond_18}, \ref{cond_19}-\ref{cond_22}), which are recollected below:  
\begin{align}
    \label{1}& 8v^{\pi_k}_{\lambda}(N, N-3) - 65v^{\pi_k}_{\lambda}(N, N-2) + 57v^{\pi_k}_{\lambda}(N, N) > 0, \quad \text{for } k = 1, \lambda \in [\lambda_c, 1) \text{ and } k =2, \lambda \in (0, \lambda_c], \\
    \label{2}& -6v^{\pi_k}_{\lambda}(N, N-4) + 11v^{\pi_k}_{\lambda}(N, N-3) -5v^{\pi_k}_{\lambda}(N, N-2) > 0, \quad \text{for } k = 1, \lambda \in [\lambda_c, 1) \\
    \nonumber & \quad \text{and } k =2, \lambda \in (0, \lambda_c], \\
    \label{3}& -6v^{\pi_1}_{\lambda}(N, j) + 11v^{\pi_1}_{\lambda}(N, j+1) -5v^{\pi_1}_{\lambda}(N, j+2) > 0, \quad \lambda \in (\lambda^c, 1), \quad j = 2, \ldots, N-5, \\
    \label{4}& 6v^{\pi_2}_{\lambda}(N, j) - 11v^{\pi_2}_{\lambda}(N, j+1) +5v^{\pi_2}_{\lambda}(N, j+2) > 0, \quad \lambda \in (0, \lambda^c), \quad j = 2, \ldots, N-5,  \\
    \label{5}& -6v_{\lambda_c}(N, j) + 11v_{\lambda^c}(N, j+1) -5v_{\lambda_c}(N, j+2) = 0, \quad j = 2, \ldots, N-5, \\
    \label{6}& 5v^{\pi_k}_{\lambda}(N-2, N-2) - 18v^{\pi_k}_{\lambda}(N, N-2) + 13v^{\pi_k}_{\lambda}(N, N) > 0, \quad\text{for }k = 1, \lambda \in [\lambda_c, 1)\\
    \nonumber &\quad \text{and } k =2, \lambda \in (0, \lambda_c],\\
    \label{7}&  8v^{\pi_k}_{\lambda}(N-2, N-3) - 65v^{\pi_k}_{\lambda}(N-2, N-2)+57v^{\pi_k}_{\lambda}(N-2, N) > 0, \quad \text{for }k = 1, \lambda \in [\lambda_c, 1) \\
    \nonumber &\quad \text{and } k =2, \lambda \in (0, \lambda_c], \\
    \label{8}& 20v^{\pi_k}_{\lambda}(N-2, N-3) + 31v^{\pi_k}_{\lambda}(N-2, N-2) + 57v^{\pi_k}_{\lambda}(N-2, N) - 108v^{\pi_k}_{\lambda}(N, N-3) >0, \\
    \nonumber &\quad\text{for }k = 1, \lambda \in [\lambda_c, 1) \text{ and } k =2, \lambda \in (0, \lambda_c], \\
    \label{9}&-4v^{\pi_k}_{\lambda}(N-2, N-3) + 15v^{\pi_k}_{\lambda}(N-2, N-2) - 18v^{\pi_k}_{\lambda}(N, N-3) + 7v^{\pi_k}_{\lambda}(N, N-2) >0, \\
    \nonumber &\quad \text{for }k = 1, \lambda \in [\lambda_c, 1) \text{ and } k =2, \lambda \in (0, \lambda_c], \\
    \label{10}& 6v^{\pi_k}_{\lambda}(N-2, j) - 40 v^{\pi_k}_{\lambda}(N-2, j+1) + 9v^{\pi_k}_{\lambda}(N, j) + 25v^{\pi_k}_{\lambda}(N, j+1) > 0, \quad j = 2, \ldots, N-4,\\
    &\nonumber \quad \text{for } k = 1, \lambda \in [\lambda_c, 1) \text{ and } k =2, \lambda \in (0, \lambda_c], \\
    \label{11}&  -30v^{\pi_k}_{\lambda}(N-2, j) -32v^{\pi_k}_{\lambda}(N-2, j+1) - 40v^{\pi_k}_{\lambda}(N-2, j+2) + 27v^{\pi_k}_{\lambda}(N, j) + 75v^{\pi_k}_{\lambda}(N, j+1) > 0, \\
    \nonumber &\quad j = 2, \ldots, N-4, \text{ for } k = 1, \lambda \in [\lambda_c, 1) \quad \text{and } k =2, \lambda \in (0, \lambda_c], \\
    \label{12}& 12v^{\pi_k}_{\lambda}(N-2, j) + 8v^{\pi_k}_{\lambda}(N-2, j+1) -45v^{\pi_k}_{\lambda}(N, j) + 25v^{\pi_k}_{\lambda}(N, j+1) > 0, \quad j = 2, \ldots, N-4, \\
    \nonumber &\quad \text{for } k = 1, \lambda \in [\lambda_c, 1) \text{ and } k =2, \lambda \in (0, \lambda_c]\\
    \label{13}& 8v^{\pi_k}_{\lambda}(N-3, N-3) - 65v^{\pi_k}_{\lambda}(N-2, N-3) + 57v^{\pi_k}_{\lambda}(N, N-3) > 0, \quad \text{for }k = 1, \lambda \in [\lambda_c, 1)\\
    \nonumber &\quad \text{and } k =2, \lambda \in (0, \lambda_c],  \\
    \label{14}& -8v^{\pi_k}_{\lambda}(N-3, N-3) - v^{\pi_k}_{\lambda}(N-2, N-3) - 48v^{\pi_k}_{\lambda}(N-2, N-2) + 57v^{\pi_k}_{\lambda}(N, N-3) > 0, \\
    \nonumber &\quad \text{for }k = 1, \lambda \in [\lambda_c, 1) \text{ and } k =2, \lambda \in (0, \lambda_c], \\
    \label{15} & v^{\pi_k}_{\lambda}(N-3, j) - 5v^{\pi_k}_{\lambda}(N-3, j+1) + v^{\pi_k}_{\lambda}(N-2, j) + 3v^{\pi_k}_{\lambda}(N-2, j+1) > 0,  j = 2, \ldots, N-4, \\
    \nonumber &\quad \text{for } k = 1, \lambda \in [\lambda_c, 1) \text{ and } k =2, \lambda \in (0, \lambda_c],  \\
    \label{16} & -3v^{\pi_k}_{\lambda}(N-3, j) -4v^{\pi_k}_{\lambda}(N-3, j+1) - 5v^{\pi_k}_{\lambda}(N-3, j+2) + 3v^{\pi_k}_{\lambda}(N-2, j) \\
    \nonumber &\quad + 9v^{\pi_k}_{\lambda}(N-2, j+1) >0, \quad j = 2, \ldots, N-4 \\
    \nonumber &\quad \text{for } k = 1, \lambda \in [\lambda_c, 1) \quad \text{and } k =2, \lambda \in (0, \lambda_c],  \\
    \label{17}& v^{\pi_k}_{\lambda}(N-3, j) + v^{\pi_k}_{\lambda}(N-3, j+1) - 5v^{\pi_k}_{\lambda}(N-2, j) + 3v^{\pi_k}_{\lambda}(N-2, j+1) >0, \\
    \nonumber &\quad j = 2, \ldots, N-4, \text{ for } k = 1, \lambda \in [\lambda_c, 1) \text{ and } k =2, \lambda \in (0, \lambda_c] \\
    \label{18}& 8v^{\pi_k}_{\lambda}(N-3, j) + 16v^{\pi_k}_{\lambda}(N-3, j+1) - 15v^{\pi_k}_{\lambda}(N-2, j) + 48v^{\pi_k}_{\lambda}(N-2, j+1) - 57v^{\pi_k}_{\lambda}(N, j) > 0, \\
    \nonumber &\quad j = 2, \ldots, N-4, \text{ for } k = 1, \lambda \in [\lambda_c, 1) \quad \text{and } k =2, \lambda \in (0, \lambda_c], \\
    \label{19}& v^{\pi_k}_{\lambda}(i,j) + v^{\pi_k}_{\lambda}(i+1, j) -5v^{\pi_k}_{\lambda}(i, j+1) + 3v^{\pi_k}_{\lambda}(i+1, j+1) > 0, \quad \lambda \in (0,1), \quad  i,j = 2, \ldots, N-4, \\
    \nonumber &\quad \text{for } k = 1, \lambda \in [\lambda_c, 1) \text{ and } k =2, \lambda \in (0, \lambda_c], \\
    \label{20}& v^{\pi_k}_{\lambda}(i,j) - 5v^{\pi_k}_{\lambda}(i+1, j) + v^{\pi_k}_{\lambda}(i, j+1) + 3v^{\pi_k}_{\lambda}(i+1, j+1) > 0, \quad \lambda \in (0,1), \quad i,j = 2, \ldots, N-4, \\
    \nonumber &\quad \text{for } k = 1, \lambda \in [\lambda_c, 1) \text{ and } k =2, \lambda \in (0, \lambda_c],\\
    \label{21}& -3v^{\pi_k}_{\lambda}(i,j) + 3v^{\pi_k}_{\lambda}(i+1, j) -4v^{\pi_k}_{\lambda}(i, j+1) + 9v^{\pi_k}_{\lambda}(i+1, j+1) -5v^{\pi_k}_{\lambda}(i, j+2) > 0, \\
    \nonumber &\quad i,j = 2, \ldots, N-4 \text{ for } k = 1, \lambda \in [\lambda_c, 1) \quad \text{and } k =2, \lambda \in (0, \lambda_c] \\
    \label{22}&-3v^{\pi_k}_{\lambda}(i,j) - 4v^{\pi_k}_{\lambda}(i+1, j) + 3v^{\pi_k}_{\lambda}(i, j+1) + 9v^{\pi_k}_{\lambda}(i+1, j+1) -5v^{\pi_k}_{\lambda}(i+2, j) > 0, \\
    \nonumber &\quad i,j = 2, \ldots, N-4, \text{ for } k = 1, \lambda \in [\lambda_c, 1) \quad \text{and } k =2, \lambda \in (0, \lambda_c].
\end{align}
We start by showing that $v^{\pi_1}_{\lambda_c}(N,j) = v^{\pi_2}_{\lambda_c}(N,j) := v_{\lambda_c}(N,j)$, for all $j = 2, \ldots, N-2$. This statement is trivial for $j = N-4, N-3, N-2, N$, by the fact that $d^1(N, j) = d^2(N, j)$ for these values of $j$ and the fact that rectangles cannot shrink. Let $v_{\lambda_c}(N,j) := v^{\pi_1}_{\lambda_c}(N,j) = v^{\pi_2}_{\lambda_c}(N, j)$ for $j = N-4, N-3, N-2, N$. Using expression (\ref{rec_4}) and recalling that $\lambda_c = 15/17$, we compute $v_{\lambda_c}(N, N-4) = 19550/3551$ and $v_{\lambda_c}(N N-3) = 23460/3551$. We now use a backward induction argument to show that $v^{\pi_1}_{\lambda_c}(N,j) = v^{\pi_2}_{\lambda_c}(N,j)$ holds for $j = 2, \ldots, N-5$. First of all, we verify its validity for $j = N-5$. By expression (\ref{rec_4}), we have
\begin{equation*}
    v^{\pi_1}_{\lambda_c}(N, N-5) = \dfrac{2\lambda_c}{3-\lambda_c}v^{\pi_1}_{\lambda_c}(N, N-4) = \dfrac{48875}{10653}.
\end{equation*}
By expression (\ref{rec_5}), we obtain
\begin{equation*}
    v^{\pi_2}_{\lambda_c}(N, N-5) = \dfrac{7\lambda_c}{3(9-5\lambda_c)}v^{\pi_2}_{\lambda_c}(N, N-4) + \dfrac{5\lambda_c}{3(9-5\lambda_c)}v^{\pi_2}_{\lambda_c}(N, N-3) = \dfrac{48875}{10653}.
\end{equation*}
Thus, we have
\begin{equation*}
    v^{\pi_1}_{\lambda^c}(N, N-5) = v^{\pi_2}_{\lambda^c}(N, N-5) = \dfrac{48875}{10653} := v_{\lambda^c}(N, N-5). 
\end{equation*}
Similarly, we find
\begin{equation*}
    v^{\pi_1}_{\lambda_c}(N, N-6) = v^{\pi_2}_{\lambda^c}(N, N-6) = \dfrac{244375}{63918} := v_{\lambda_c}(N, N-6).
\end{equation*}
Now, suppose that $v^{\pi_1}_{\lambda_c}(N, j) = v^{\pi_2}_{\lambda_c}(N, j) := v_{\lambda_c}(N,j)$ for all $j \geq n+1$, $n = 2, \ldots, N-7$. We proceed to show that this implies $v^{\pi_1}_{\lambda_c}(N, n) = v^{\pi_2}_{\lambda_c}(N, n) := v_{\lambda_c}(N,n)$.
Using expression (\ref{rec_4}) and the induction hypothesis, we obtain
\begin{align*}
    v^{\pi_1}_{\lambda_c}(N, n) &= \dfrac{2\lambda_c}{3-\lambda_c}v_{\lambda_c}(N, n+1) \\
                                &= \dfrac{7\lambda_c}{3(9-5\lambda_c)}v_{\lambda_c}(N, n+1) + \left(\dfrac{2\lambda_c}{3-\lambda_c} - \dfrac{7\lambda_c}{3(9-5\lambda_c)}\right)\dfrac{2\lambda_c}{3-\lambda_c}v_{\lambda_c}(N, n+2)\\
                                &= \dfrac{7\lambda_c}{3(9-5\lambda^c)}v_{\lambda_c}(N, n+1) + \dfrac{2\lambda_c^2(33-23\lambda_c)}{3(3-\lambda_c)^2(9-5\lambda_c)}v_{\lambda_c}(N, n+2) \\
                                &= \dfrac{7\lambda_c}{3(9-5\lambda_c)}v_{\lambda_c}(N, n+1) + \dfrac{5\lambda_c}{3(9-5\lambda_c)}v_{\lambda_c}(N, n+2),
\end{align*}
where we used the fact that
\begin{equation*}
   \dfrac{2\lambda_c^2(33-23\lambda_c)}{3(3-\lambda_c)^2(9-5\lambda_c)} = \dfrac{25}{78} = \dfrac{5\lambda_c}{3(9-5\lambda_c)}. 
\end{equation*} 
Thus, by expression (\ref{rec_5}) and the induction hypothesis, we have $v^{\pi_1}_{\lambda^c}(N, n) = v^{\pi_2}_{\lambda^c}(N, n) := v_{\lambda_c}(N, n)$. It follows that $v^{\pi_1}_{\lambda^c}(N, j) = v^{\pi_2}_{\lambda^c}(N, j) := v_{\lambda_c}(N, j)$ for all $j = 2, \ldots, N-2$. Note that this immediately implies the validity of equation (\ref{5}) above.\\
\\
We now turn our attention to the remaining inequalities listed above. Inequalities (\ref{1}), (\ref{2}), (\ref{6}), (\ref{7}), (\ref{8}), (\ref{9}), (\ref{13}) and (\ref{14}) can be verified easily using the explicit expressions for the value functions collected in Table \ref{explicit_v_stars}. We now proceed to prove the remaining inequalities by means of induction over the size of the rectangle.

\paragraph*{Inequality (\ref{3})}
First, we verify the validity of inequality (\ref{3}) for $j = N-5$ using the explicit expressions for $v^{\pi_1}_{\lambda}(N, N-5)$, $v^{\pi_1}_{\lambda}(N, N-4)$ and $v^{\pi}_{\lambda}(N, N-3)$, $\lambda \in (\lambda_c, 1)$ recorded in Table \ref{explicit_v_stars}. Now, assume that it holds for $j = n+1$ for some $n = 2, \ldots, N-6$. Using expression (\ref{rec_4}), we obtain for $j = n$:
\begin{align*}
    &-6v^{\pi_1}_{\lambda}(N, n) + 11v^{\pi_1}_{\lambda}(N, n+1) -5v^{\pi_1}_{\lambda}(N, n+2)\\
    &= \dfrac{2\lambda}{3-\lambda}\Big(-6v^{\pi_1}_{\lambda}(N, n+1) + 11v^{\pi_1}_{\lambda}(N, n+2) -5v^{\pi_1}_{\lambda}(N, n+3)\Big).
\end{align*}
From the induction hypothesis, it now follows that
\begin{equation*}
    -6v^{\pi_1}_{\lambda}(N, n) + 11v^{\pi_1}_{\lambda}(N, n+1) -5v^{\pi_1}_{\lambda}(N, n+2) > 0.
\end{equation*}
Thus, inequality (\ref{3}) holds for all $j = 2, \ldots, N-5$. 

\paragraph*{Inequality (\ref{4})}
We first verify the correctness of inequality (\ref{4}) for $j = N-5$ and $j = N-6$ using the explicit expressions collected in Table \ref{explicit_v_stars}. Suppose now that it is valid for all $j \in \{n+1, \ldots, N-5\}$ for some $n = 2, \ldots, N-7$. Using expression (\ref{rec_5}), we obtain
\begin{align*}
    & 6v^{\pi_2}_{\lambda}(N, n) - 11v^{\pi_2}_{\lambda}(N, n+1) +5v^{\pi_2}_{\lambda}(N, n+2) \\
    \nonumber &= 6\left(\dfrac{7\lambda}{3(9-5\lambda)}v^{\pi_2}_{\lambda}(N, n+1) + \dfrac{5\lambda}{3(9-5\lambda)}v^{\pi_2}_{\lambda}(N, n+2)\right) \\
    \nonumber &- 11\left(\dfrac{7\lambda}{3(9-5\lambda)}v^{\pi_2}_{\lambda}(N, n+2) + \dfrac{5\lambda}{3(9-5\lambda)}v^{\pi_2}_{\lambda}(N, n+3)\right) \\
    \nonumber &+5\left(\dfrac{7\lambda}{3(9-5\lambda)}v^{\pi_2}_{\lambda}(N, n+3) + \dfrac{5\lambda}{3(9-5\lambda)}v^{\pi_2}_{\lambda}(N, n+4)\right) \\
    \nonumber &= \dfrac{7\lambda}{3(9-5\lambda)}\Big(6v^{\pi_2}_{\lambda}(N, n+1) - 11v^{\pi_2}_{\lambda}(N, n+2) +5v^{\pi_2}_{\lambda}(N, n+3)\Big) \\
    \nonumber &+ \dfrac{5\lambda}{3(9-5\lambda)}\Big(6v^{\pi_2}_{\lambda}(N, n+2) - 11v^{\pi_2}_{\lambda}(N, n+3) +5v^{\pi_2}_{\lambda}(N, n+4)\Big).
\end{align*}
Invoking the induction hypothesis now yields
\begin{equation*}
    6v^{\pi_2}_{\lambda}(N, n) - 11v^{\pi_2}_{\lambda}(N, n+1) +5v^{\pi_2}_{\lambda}(N, n+2) > 0.
\end{equation*}
It follows that inequality (\ref{4}) holds for all $j = 2, \ldots, N-5$.

\paragraph*{Inequality (\ref{10})}
We start by verifying the validity of inequality (\ref{10}) for $j = N-4$ using the explicit expressions for the value functions in Table \ref{explicit_v_stars}. Then, we assume that it holds for $j = n+1$ for some $n = 2, \ldots, N-5$ and show that its validity carries over to $j = n$. Expression (\ref{rec_8}) implies
\begin{align*}
    & 6v^{\pi_k}_{\lambda}(N-2, n) - 40 v^{\pi_k}_{\lambda}(N-2, n+1) + 9v^{\pi_k}_{\lambda}(N, n) + 25v^{\pi_k}_{\lambda}(N, n+1) \\
    \nonumber &= \dfrac{72}{12-5\lambda}\left(\dfrac{\lambda}{8}v^{\pi_k}_{\lambda}(N, n) + \dfrac{25\lambda}{72}v^{\pi_k}_{\lambda}(N, n+1) + \dfrac{\lambda}{9}v^{\pi_k}_{\lambda}(N-2, n+1)\right) \\
    \nonumber &\quad - \dfrac{480}{12-5\lambda}\left(\dfrac{\lambda}{8}v^{\pi_k}_{\lambda}(N, n+1) + \dfrac{25\lambda}{72}v^{\pi_k}_{\lambda}(N, n+2) + \dfrac{\lambda}{9}v^{\pi_k}_{\lambda}(N-2, n+2)\right) +9v^{\pi_k}_{\lambda}(N, n) + 25v^{\pi_k}_{\lambda}(N, n+1) \\
    \nonumber &= \dfrac{4\lambda}{3(12-5\lambda)}\Big(6v^{\pi_k}_{\lambda}(N-2, n+1) - 40 v^{\pi_k}_{\lambda}(N-2, n+2) + 9v^{\pi_k}_{\lambda}(N, n+1)+ 25v^{\pi_k}_{\lambda}(N, n+2)\Big) \\
    \nonumber &\quad+ \dfrac{36(3-\lambda)}{12-5\lambda}v^{\pi_k}_{\lambda}(N, n)  + \dfrac{4(75-43\lambda)}{12-5\lambda}v^{\pi_k}_{\lambda}(N, n+1) - \dfrac{200\lambda}{12-5\lambda}v^{\pi_k}_{\lambda}(N, n+2),
\end{align*}
for $k = 1, 2$, $\lambda \in (0,1)$. 
From the induction hypothesis, it follows that it suffices to show that
\begin{equation}\label{exp10_cond1}
    \dfrac{36(3-\lambda)}{12-5\lambda}v^{\pi_k}_{\lambda}(N, \ell)  + \dfrac{4(75-43\lambda)}{12-5\lambda}v^{\pi_k}_{\lambda}(N, \ell+1) - \dfrac{200\lambda}{12-5\lambda}v^{\pi_k}_{\lambda}(N, \ell+2) \geq 0,
\end{equation}
for all $\ell = 2, \ldots, N-5$. We prove this statement by means of another induction argument, distinguishing between $k = 1$ and $k = 2$. For $k = 1$, we first verify the validity of expression (\ref{exp10_cond1}) for $\ell = N-5$, using the explicit expressions for $v^{\pi_1}_{\lambda}(N, N-5)$, $v^{\pi_1}_{\lambda}(N, N-4)$ and $v^{\pi_1}_{\lambda}(N, N-3)$ recorded in Table \ref{explicit_v_stars}. Now, assume that it holds for $\ell = \tilde{\ell}+1$ for some $\tilde{\ell} = 2, \ldots, N-6$. By expression (\ref{rec_4}) and the induction hypothesis, we obtain
\begin{align*}
    &\dfrac{36(3-\lambda)}{12-5\lambda}v^{\pi_1}_{\lambda}(N, \tilde{\ell})  + \dfrac{4(75-43\lambda)}{12-5\lambda}v^{\pi_1}_{\lambda}(N, \tilde{\ell}+1) - \dfrac{200\lambda}{12-5\lambda}v^{\pi_1}_{\lambda}(N, \tilde{\ell}+2) \\
    \nonumber &= \dfrac{2\lambda}{3-\lambda}\Big(\dfrac{36(3-\lambda)}{12-5\lambda}v^{\pi_1}_{\lambda}(N, \tilde{\ell}+1)  + \dfrac{4(75-43\lambda)}{12-5\lambda}v^{\pi_1}_{\lambda}(N, \tilde{\ell}+2) - \dfrac{200\lambda}{12-5\lambda}v^{\pi_1}_{\lambda}(N, \tilde{\ell}+3)\Big) > 0.
\end{align*}
This establishes the validity of expression (\ref{exp10_cond1}) for all $\ell = 2, \ldots, N-5$ and $k = 1$. \\
We now turn our attention to $k = 2$. In this case, we verify the validity of expression (\ref{exp10_cond1}) for $\ell = N-5$ and $\ell = N-6$, using Table \ref{explicit_v_stars}, and make the slightly stronger assumption that it holds for all $\ell \geq \tilde{\ell}+1$ for some $\tilde{\ell} = 23 \ldots, N-7$. By expression (\ref{rec_5}), we obtain 
\begin{align*}
    &\dfrac{36(3-\lambda)}{12-5\lambda}v^{\pi_2}_{\lambda}(N, \tilde{\ell})  + \dfrac{4(75-43\lambda)}{12-5\lambda}v^{\pi_2}_{\lambda}(N, \tilde{\ell}+1) - \dfrac{200\lambda}{12-5\lambda}v^{\pi_2}_{\lambda}(N, \tilde{\ell}+2) \\
    \nonumber &= \dfrac{36(3-\lambda)}{12-5\lambda}\left(\dfrac{7\lambda}{3(9-5\lambda)}v^{\pi_2}_{\lambda}(N, \tilde{\ell}+1) + \dfrac{5\lambda}{3(9-5\lambda)}v^{\pi_2}_{\lambda}(N, \tilde{\ell}+2)\right) \\
    \nonumber &+ \dfrac{4(75-43\lambda)}{12-5\lambda}\left(\dfrac{7\lambda}{3(9-5\lambda)}v^{\pi_2}_{\lambda}(N, \tilde{\ell}+2) + \dfrac{5\lambda}{3(9-5\lambda)}v^{\pi_2}_{\lambda}(N, \tilde{\ell}+3)\right) \\
    \nonumber &- \dfrac{200\lambda}{12-5\lambda}\left(\dfrac{7\lambda}{3(9-5\lambda)}v^{\pi_2}_{\lambda}(N, \tilde{\ell}+3) + \dfrac{5\lambda}{3(9-5\lambda)}v^{\pi_2}_{\lambda}(N, \tilde{\ell}+4)\right) \\
    \nonumber &= \dfrac{7\lambda}{3(9-5\lambda)}\Big(\dfrac{36(3-\lambda)}{12-5\lambda}v^{\pi_2}_{\lambda}(N, \tilde{\ell}+1)  + \dfrac{4(75-43\lambda)}{12-5\lambda}v^{\pi_2}_{\lambda}(N, \tilde{\ell}+2) - \dfrac{200\lambda}{12-5\lambda}v^{\pi_2}_{\lambda}(N, \tilde{\ell}+3)\Big) \\
    \nonumber &+ \dfrac{5\lambda}{3(9-5\lambda)}\Big(\dfrac{36(3-\lambda)}{12-5\lambda}v^{\pi_2}_{\lambda}(N, \tilde{\ell}+2) + \dfrac{4(75-43\lambda)}{12-5\lambda}v^{\pi_2}_{\lambda}(N, \tilde{\ell}+3) - \dfrac{200\lambda}{12-5\lambda}v^{\pi_2}_{\lambda}(N, \tilde{\ell}+4)\Big).
\end{align*}
The induction hypothesis now implies
\begin{equation*}
    \dfrac{36(3-\lambda)}{12-5\lambda}v^{\pi_2}_{\lambda}(N, \tilde{\ell})  + \dfrac{4(75-43\lambda)}{12-5\lambda}v^{\pi_2}_{\lambda}(N, \tilde{\ell}+1) - \dfrac{200\lambda}{12-5\lambda}v^{\pi_2}_{\lambda}(N, \tilde{\ell}+2) \geq 0,
\end{equation*}
which proves the correctness of inequality (\ref{exp10_cond1}) for all $\ell = 2, \ldots, N-5$, $k = 2$ and $\lambda \in (0, \lambda^c]$. It now follows that
\begin{equation*}
    6v^{\pi_k}_{\lambda}(N-2, n) - 40 v^{\pi_k}_{\lambda}(N-2, n+1) + 9v^{\pi_k}_{\lambda}(N, n) + 25v^{\pi_k}_{\lambda}(N, n+1) > 0,
\end{equation*}
for $k = 1$, $\lambda \in [\lambda_c, 1)$ and $k = 2$, $\lambda \in (0, \lambda_c]$. Thus, inequality (\ref{10}) is valid for all $j = 2, \ldots, N-4$, for $k =1$, $\lambda \in [\lambda_c, 1)$ and $k =2$, $\lambda \in (0, \lambda_c]$. 

\paragraph*{Inequality (\ref{11})}
First, we verify the correctness of inequality (\ref{11}) for $j = N-4$ and $j = N-5$, for $k = 1$, $\lambda \in [\lambda_c, 1)$ and $k = 2$, $\lambda \in (0, \lambda_c]$, using Table \ref{explicit_v_stars}. Now, assume that it holds for all $j = n+1, \ldots N-4$, for some $n = 2, \ldots, N-6$, for $k = 1$, $\lambda \in [\lambda_c, 1)$ and $k = 2$, $\lambda \in (0, \lambda_c]$. By expression (\ref{rec_8}) we obtain
\begin{align*}
    &-30v^{\pi_k}_{\lambda}(N-2, n) -32v^{\pi_k}_{\lambda}(N-2, n+1) - 40v^{\pi_k}_{\lambda}(N-2, n+2) + 27v^{\pi_k}_{\lambda}(N, n) + 75v^{\pi_k}_{\lambda}(N, n+1) \\
    \nonumber &= -\dfrac{360}{12-5\lambda}\left(\dfrac{\lambda}{8}v^{\pi_k}_{\lambda}(N, n) + \dfrac{25\lambda}{72}v^{\pi_k}_{\lambda}(N, n+1) + \dfrac{\lambda}{9}v^{\pi_k}_{\lambda}(N-2, n+1)\right) \\
    \nonumber & - \dfrac{384}{12-5\lambda}\left(\dfrac{\lambda}{8}v^{\pi_k}_{\lambda}(N, n+1) + \dfrac{25\lambda}{72}v^{\pi_k}_{\lambda}(N, n+2) + \dfrac{\lambda}{9}v^{\pi_k}_{\lambda}(N-2, n+2)\right) \\
    \nonumber &- \dfrac{480}{12-5\lambda}\left(\dfrac{\lambda}{8}v^{\pi_k}_{\lambda}(N, n+2) + \dfrac{25\lambda}{72}v^{\pi_k}_{\lambda}(N, n+3) + \dfrac{\lambda}{9}v^{\pi_k}_{\lambda}(N-2, n+3)\right) + 27v^{\pi_k}_{\lambda}(N, n) + 75v^{\pi_k}_{\lambda}(N, n+1) \\
    \nonumber &= \dfrac{4\lambda}{3(12-5\lambda)}\Big(-30v^{\pi_k}_{\lambda}(N-2, n+1) -32v^{\pi_k}_{\lambda}(N-2, n+2) - 40v^{\pi_k}_{\lambda}(N-2, n+3) + 27v^{\pi_k}_{\lambda}(N, n+1) \\
    \nonumber &+ 75v^{\pi_k}_{\lambda}(N, n+2)\Big) + \dfrac{36(9-5\lambda)}{12-5\lambda}v^{\pi_k}_{\lambda}(N, n) + \dfrac{900-584\lambda}{12-5\lambda}v^{\pi_k}_{\lambda}(N, n+1) - \dfrac{880\lambda}{3(12-5\lambda)}v^{\pi_k}_{\lambda}(N, n+2)\\
    \nonumber &- \dfrac{500\lambda}{3(12-5\lambda)}v^{\pi_k}_{\lambda}(N, n+3),
\end{align*}
for $k = 1$, $\lambda \in [\lambda_c, 1)$ and $k = 2$, $\lambda \in (0, \lambda_c]$. From the induction hypothesis, it follows that it suffices to show that
\begin{align}\label{exp11_cond1}
    &\dfrac{36(9-5\lambda)}{12-5\lambda}v^{\pi_k}_{\lambda}(N, \ell) + \dfrac{900-584\lambda}{12-5\lambda}v^{\pi_k}_{\lambda}(N, \ell+1) - \dfrac{880\lambda}{3(12-5\lambda)}v^{\pi_k}_{\lambda}(N, \ell+2) - \dfrac{500\lambda}{3(12-5\lambda)}v^{\pi_k}_{\lambda}(N, \ell+3) \geq 0,
\end{align}
for all $\ell = 2, \ldots, N-6$, for $k = 1$, $\lambda \in [\lambda_c, 1)$ and $k = 2$, $\lambda \in (0, \lambda_c]$. We first consider the case $k = 1$, $\lambda \in [\lambda_c, 1)$. After verifying the correctness of expression (\ref{exp11_cond1}) for $\ell = N-6$, using Table \ref{explicit_v_stars}, we assume that it holds for $\ell = \tilde{\ell} + 1$, for some $\tilde{\ell} = 2, \ldots, N-7$. Expression (\ref{rec_4}) and this induction hypothesis imply
\begin{align*}
    &\dfrac{36(9-5\lambda)}{12-5\lambda}v^{\pi_1}_{\lambda}(N, \tilde{\ell}) + \dfrac{900-584\lambda}{12-5\lambda}v^{\pi_1}_{\lambda}(N, \tilde{\ell}+1) - \dfrac{880\lambda}{3(12-5\lambda)}v^{\pi_1}_{\lambda}(N, \tilde{\ell}+2) - \dfrac{500\lambda}{3(12-5\lambda)}v^{\pi_1}_{\lambda}(N, \tilde{\ell}+3) \\
    \nonumber &= \dfrac{2\lambda}{3-\lambda}\Bigg(\dfrac{36(9-5\lambda)}{12-5\lambda}v^{\pi_1}_{\lambda}(N, \tilde{\ell}+1) + \dfrac{900-584\lambda}{12-5\lambda}v^{\pi_1}_{\lambda}(N, \tilde{\ell}+2) - \dfrac{880\lambda}{3(12-5\lambda)}v^{\pi_1}_{\lambda}(N, \tilde{\ell}+3) \\
    \nonumber &- \dfrac{500\lambda}{3(12-5\lambda)}v^{\pi_1}_{\lambda}(N, \tilde{\ell}+4)\Bigg) \geq 0.
\end{align*}
Hence, expression (\ref{exp11_cond1}) holds for all $\ell = 2, \ldots, N-6$, for $k = 1$, $\lambda \in [\lambda_c, 1)$. 

We proceed to consider the case $k = 2$, $\lambda \in (0, \lambda_c]$. First of all, we again use Table \ref{explicit_v_stars} to verify the validity of expression (\ref{exp11_cond1}) for $\ell = N-6$ and $\ell = N-7$. Then, we make the slightly stronger assumption that it is satisfied for all $\ell = \tilde{\ell}+1, \ldots N-6$, for some $\tilde{\ell} = 2, \ldots, N-8$. Using expression (\ref{rec_5}) we obtain
\begin{align*}
    &\dfrac{36(9-5\lambda)}{12-5\lambda}v^{\pi_2}_{\lambda}(N, \tilde{\ell}) + \dfrac{900-584\lambda}{12-5\lambda}v^{\pi_2}_{\lambda}(N, \tilde{\ell}+1) - \dfrac{880\lambda}{3(12-5\lambda)}v^{\pi_2}_{\lambda}(N, \tilde{\ell}+2)- \dfrac{500\lambda}{3(12-5\lambda)}v^{\pi_2}_{\lambda}(N, \tilde{\ell}+3) \\
    \nonumber &= \dfrac{36(9-5\lambda)}{12-5\lambda} \left(\dfrac{7\lambda}{3(9-5\lambda)}v^{\pi_2}_{\lambda}(N, \tilde{\ell}+1) + \dfrac{5\lambda}{3(9-5\lambda)}v^{\pi_2}_{\lambda}(N, \tilde{\ell}+2)\right) + \dfrac{900-584\lambda}{12-5\lambda}\Bigg(\dfrac{7\lambda}{3(9-5\lambda)}v^{\pi_2}_{\lambda}(N, \tilde{\ell}+2) \\
    \nonumber &+ \dfrac{5\lambda}{3(9-5\lambda)}v^{\pi_2}_{\lambda}(N, \tilde{\ell}+3)\Bigg) - \dfrac{880\lambda}{3(12-5\lambda)}\left(\dfrac{7\lambda}{3(9-5\lambda)}v^{\pi_2}_{\lambda}(N, \tilde{\ell}+3) + \dfrac{5\lambda}{3(9-5\lambda)}v^{\pi_2}_{\lambda}(N, \tilde{\ell}+4)\right) \\
    \nonumber &- \dfrac{500\lambda}{3(12-5\lambda)}\left(\dfrac{7\lambda}{3(9-5\lambda)}v^{\pi_2}_{\lambda}(N, \tilde{\ell}+4) + \dfrac{5\lambda}{3(9-5\lambda)}v^{\pi_2}_{\lambda}(N, \tilde{\ell}+5)\right) \\ 
    \nonumber &= \dfrac{7\lambda}{3(9-5\lambda)}\Bigg(\dfrac{36(9-5\lambda)}{12-5\lambda}v^{\pi_2}_{\lambda}(N, \tilde{\ell}+1) + \dfrac{900-584\lambda}{12-5\lambda}v^{\pi_2}_{\lambda}(N, \tilde{\ell}+2) - \dfrac{880\lambda}{3(12-5\lambda)}v^{\pi_2}_{\lambda}(N, \tilde{\ell}+3)\\
    \nonumber &- \dfrac{500\lambda}{3(12-5\lambda)}v^{\pi_2}_{\lambda}(N, \tilde{\ell}+4)\Bigg) + \dfrac{5\lambda}{3(9-5\lambda)}\Bigg(\dfrac{36(9-5\lambda)}{12-5\lambda}v^{\pi_2}_{\lambda}(N, \tilde{\ell}+2) + \dfrac{900-584\lambda}{12-5\lambda}v^{\pi_2}_{\lambda}(N, \tilde{\ell}+3) \\
    \nonumber &- \dfrac{880\lambda}{3(12-5\lambda)}v^{\pi_2}_{\lambda}(N, \tilde{\ell}+4)  - \dfrac{500\lambda}{3(12-5\lambda)}v^{\pi_2}_{\lambda}(N, \tilde{\ell}+5)\Bigg).
\end{align*}
The induction hypothesis now implies that
\begin{align*}
    &\dfrac{36(9-5\lambda)}{12-5\lambda}v^{\pi_2}_{\lambda}(N, \tilde{\ell}) + \dfrac{900-584\lambda}{12-5\lambda}v^{\pi_2}_{\lambda}(N, \tilde{\ell}+1) - \dfrac{880\lambda}{3(12-5\lambda)}v^{\pi_2}_{\lambda}(N, \tilde{\ell}+2) - \dfrac{500\lambda}{3(12-5\lambda)}v^{\pi_2}_{\lambda}(N, \tilde{\ell}+3) \geq 0.
\end{align*}
Thus, expression (\ref{exp11_cond1}) holds for all $\ell = 2, \ldots, N-6$ for $k =2$, $\lambda \in (0, \lambda_c]$ as well. From this, we can conclude that 
\begin{align*}
    &-30v^{\pi_k}_{\lambda}(N-2, n) -32v^{\pi_k}_{\lambda}(N-2, n+1) - 40v^{\pi_k}_{\lambda}(N-2, n+2) + 27v^{\pi_k}_{\lambda}(N, n) + 75v^{\pi_k}_{\lambda}(N, n+1) > 0,
\end{align*}
for $k = 1$, $\lambda \in [\lambda_c, 1)$ and $k =2$, $\lambda \in (0, \lambda_c]$. Thus, we have established that inequality (\ref{11}) holds for all $j = 2, \ldots, N-4$ for $k = 1$, $\lambda \in [\lambda_c, 1)$ and $k =2$, $\lambda \in (0, \lambda_c]$.

\paragraph*{Inequality (\ref{12})}
We again start by verifying the validity of inequality (\ref{12}) for $j = N-4$, for $k = 1$, $\lambda \in [\lambda_c, 1)$ and $k = 2$, $\lambda \in (0, \lambda_c]$, using Table \ref{explicit_v_stars}. Assume that it is satisfied for $j = n+1$, for some $n = 2, \ldots, N-5$, for $k =1$, $\lambda \in [\lambda_c, 1)$ and $k =2$, $\lambda \in (0, \lambda_c]$. Expression (\ref{rec_8}) yields
\begin{align*}
    &12v^{\pi_k}_{\lambda}(N-2, n) + 8v^{\pi_k}_{\lambda}(N-2, n+1) -45v^{\pi_k}_{\lambda}(N, n) + 25v^{\pi_k}_{\lambda}(N, n+1) \\
    \nonumber &= \dfrac{144}{12-5\lambda}\left(\dfrac{\lambda}{8}v^{\pi_k}_{\lambda}(N, n) + \dfrac{25\lambda}{72}v^{\pi_k}_{\lambda}(N, n+1) + \dfrac{\lambda}{9}v^{\pi_k}_{\lambda}(N-2, n+1)\right) \\
    \nonumber &+ \dfrac{96}{12-5\lambda} \left(\dfrac{\lambda}{8}v^{\pi_k}_{\lambda}(N, n+1) + \dfrac{25\lambda}{72}v^{\pi_k}_{\lambda}(N, n+2) + \dfrac{\lambda}{9}v^{\pi_k}_{\lambda}(N-2, n+2)\right) -45v^{\pi_k}_{\lambda}(N, n) + 25v^{\pi_k}_{\lambda}(N, n+1) \\
    \nonumber &= \dfrac{4\lambda}{3(12-5\lambda)}\Big(12v^{\pi_k}_{\lambda}(N-2, n+1) + 8v^{\pi_k}_{\lambda}(N-2, n+2) -45v^{\pi_k}_{\lambda}(N, n+1) + 25v^{\pi_k}_{\lambda}(N, n+2)\Big) \\
    \nonumber &- \dfrac{27(20-9\lambda)}{12-5\lambda}v^{\pi_k}_{\lambda}(N, n) + \dfrac{3(100-\lambda)}{12-5\lambda}v^{\pi_k}_{\lambda}(N, n+1),
\end{align*}
for $k = 1, 2$, $\lambda \in (0,1)$. The induction hypothesis implies that a sufficient condition for inequality (\ref{12}) to hold for $j = n$ is given by
\begin{equation}\label{exp12_cond1}
    - \dfrac{27(20-9\lambda)}{12-5\lambda}v^{\pi_k}_{\lambda}(N, \ell) + \dfrac{3(100-\lambda)}{12-5\lambda}v^{\pi_k}_{\lambda}(N, \ell+1) \geq 0,
\end{equation}
for all $\ell = 2, \ldots, N-5$, for $k = 1$, $\lambda \in [\lambda_c, 1)$ and $k =2$, $\lambda \in (0, \lambda^c]$. First, we consider the case $k = 1$, $\lambda \in [\lambda_c, 1)$. After verifying the validity of expression (\ref{exp12_cond1}) for $\ell = N-5$, using Table \ref{explicit_v_stars}, we assume that it holds for $\ell = \tilde{\ell} +1$ for some $\tilde{\ell} = 2, \ldots, N-6$. It now follows from expression (\ref{rec_4}) and this induction hypothesis that
\begin{align*}
    &- \dfrac{27(20-9\lambda)}{12-5\lambda}v^{\pi_1}_{\lambda}(N, \tilde{\ell}) + \dfrac{3(100-\lambda)}{12-5\lambda}v^{\pi_1}_{\lambda}(N, \tilde{\ell}+1) \\
    \nonumber &= \dfrac{2\lambda}{3-\lambda}\left( - \dfrac{27(20-9\lambda)}{12-5\lambda}v^{\pi_1}_{\lambda}(N, \tilde{\ell}+1) + \dfrac{3(100-\lambda)}{12-5\lambda}v^{\pi_1}_{\lambda}(N, \tilde{\ell}+2)\right) \geq 0,
\end{align*}
for $\lambda \in [\lambda^c, 1)$. This implies that expression (\ref{exp12_cond1}) is satisfied for all $\ell = 2, \ldots, N-5$, for $k = 1$, $\lambda \in [\lambda^c, 1)$. \\
We proceed to consider the case $k =2$, $\lambda \in (0, \lambda^c]$. Again, we use Table \ref{explicit_v_stars} to verify the correctness of expression (\ref{exp12_cond1}) for $\ell = N-5$ and $\ell = N-6$. Assume that it is satisfied for $\ell = \tilde{\ell}+1, \ldots, N-5$, where $\tilde{\ell} = 2, \ldots, N-7$. Using expression (\ref{rec_5}), we obtain
\begin{align*}
     &- \dfrac{27(20-9\lambda)}{12-5\lambda}v^{\pi_2}_{\lambda}(N, \tilde{\ell}) + \dfrac{3(100-\lambda)}{12-5\lambda}v^{\pi_2}_{\lambda}(N, \tilde{\ell}+1) \\
     \nonumber &= - \dfrac{27(20-9\lambda)}{12-5\lambda}\left(\dfrac{7\lambda}{3(9-5\lambda)}v^{\pi_2}_{\lambda}(N, \tilde{\ell}+1) + \dfrac{5\lambda}{3(9-5\lambda)}v^{\pi_2}_{\lambda}(N, \tilde{\ell}+2)\right) \\
     \nonumber &+ \dfrac{3(100-\lambda)}{12-5\lambda}\left(\dfrac{7\lambda}{3(9-5\lambda)}v^{\pi_2}_{\lambda}(N, \tilde{\ell}+2) + \dfrac{5\lambda}{3(9-5\lambda)}v^{\pi_2}_{\lambda}(N, \tilde{\ell}+3)\right)\\
     \nonumber &= \dfrac{7\lambda}{3(9-5\lambda)}\left(- \dfrac{27(20-9\lambda)}{12-5\lambda}v^{\pi_2}_{\lambda}(N, \tilde{\ell}+1) + \dfrac{3(100-\lambda)}{12-5\lambda}v^{\pi_2}_{\lambda}(N, \tilde{\ell}+2)\right) \\
     \nonumber &+ \dfrac{5\lambda}{3(9-5\lambda)}\left(- \dfrac{27(20-9\lambda)}{12-5\lambda}v^{\pi_2}_{\lambda}(N, \tilde{\ell}+2) + \dfrac{3(100-\lambda)}{12-5\lambda}v^{\pi_2}_{\lambda}(N, \tilde{\ell}+3)\right).
\end{align*}
By the induction hypothesis, it now follows that
\begin{equation*}
    - \dfrac{27(20-9\lambda)}{12-5\lambda}v^{\pi_2}_{\lambda}(N, \tilde{\ell}) + \dfrac{300-3\lambda}{12-5\lambda}v^{\pi_2}_{\lambda}(N, \tilde{\ell}+1) \geq 0,
\end{equation*}
for all $\lambda \in (0, \lambda_c]$. This implies that expression (\ref{exp12_cond1}) holds for all $\ell = 2, \ldots, N-5$, for $k =2$, $\lambda \in (0, \lambda_c]$ as well. Thus, we can conclude that
\begin{equation*}
    12v^{\pi_k}_{\lambda}(N-2, n) + 8v^{\pi_k}_{\lambda}(N-2, n+1) -45v^{\pi_k}_{\lambda}(N, n) + 25v^{\pi_k}_{\lambda}(N, n+1) > 0,
\end{equation*}
for $k = 1$, $\lambda \in [\lambda_c, 1)$ and $k =2$, $\lambda \in (0, \lambda_c]$. This establishes the validity of inequality (\ref{12}) for all $j = 2, \ldots, N-4$, for $k = 1$, $\lambda \in [\lambda_c, 1)$ and $k =2$, $\lambda \in (0, \lambda_c]$.

\paragraph{Inequality (\ref{15})}
First of all, we verify the correctness of inequality (\ref{15}) for $j = N-4$, for $k = 1$, $\lambda \in [\lambda_c, 1)$ and $k =2$, $\lambda \in (0, \lambda_c]$ using Table \ref{explicit_v_stars}. We proceed to assume that it holds for $j = n+1$ for some $n = 2, \ldots, N-5$, for $k = 1$, $\lambda \in [\lambda_c, 1)$ and $k =2$, $\lambda \in (0, \lambda_c]$. Using expression (\ref{rec_10}), we obtain
\begin{align*}
    &v^{\pi_k}_{\lambda}(N-3, n) - 5v^{\pi_k}_{\lambda}(N-3, n+1) + v^{\pi_k}_{\lambda}(N-2, n) + 3v^{\pi_k}_{\lambda}(N-2, n+1) \\
    \nonumber &= \dfrac{\lambda}{9-4\lambda}\Big(v^{\pi_k}_{\lambda}(N-2, n) + v^{\pi_k}_{\lambda}(N-3, n+1) + 3v^{\pi_k}_{\lambda}(N-2, n+1)\Big) \\
    \nonumber &- \dfrac{5\lambda}{9-4\lambda}\Big(v^{\pi_k}_{\lambda}(N-2, n+1) + v^{\pi_k}_{\lambda}(N-3, n+2) + 3v^{\pi_k}_{\lambda}(N-2, n+2)\Big) + v^{\pi_k}_{\lambda}(N-2, n) + 3v^{\pi_k}_{\lambda}(N-2, n+1) \\
    \nonumber &= \dfrac{\lambda}{9-4\lambda}\Big(v^{\pi_k}_{\lambda}(N-3, n+1) - 5v^{\pi_k}_{\lambda}(N-3, n+2) + v^{\pi_k}_{\lambda}(N-2, n+1) + 3v^{\pi_k}_{\lambda}(N-2, n+2)\Big)\\
    \nonumber &+ \dfrac{3(3-\lambda)}{9-4\lambda}v^{\pi_k}_{\lambda}(N-2, n) + \dfrac{3(9-5\lambda)}{9-4\lambda}v^{\pi_k}_{\lambda}(N-2, n+1) - \dfrac{18\lambda}{9-4\lambda}v^{\pi_k}_{\lambda}(N-2, n+2), 
\end{align*}
for $k = 1, 2$, $\lambda \in (0,1)$. 
From the induction hypothesis, it follows that it suffices to show that
\begin{equation}\label{exp15_cond1}
    \dfrac{3(3-\lambda)}{9-4\lambda}v^{\pi_k}_{\lambda}(N-2, \ell) + \dfrac{3(9-5\lambda)}{9-4\lambda}v^{\pi_k}_{\lambda}(N-2, \ell+1) - \dfrac{18\lambda}{9-4\lambda}v^{\pi_k}_{\lambda}(N-2, \ell+2) \geq 0,
\end{equation}
holds for all $\ell = 2, \ldots, N-5$, for $k = 1$, $\lambda \in [\lambda_c, 1)$ and for $k = 2$, $\lambda \in (0, \lambda_c]$. We prove this statement by means of another induction argument. First of all, we verify its validity for $\ell = N-5$ using Table \ref{explicit_v_stars}. Now suppose that it is true for $\ell = \tilde{\ell} + 1$, for some $\tilde{\ell} = 2, \ldots, N-6$, for $k = 1$, $\lambda \in [\lambda_c, 1)$ and for $k = 2$, $\lambda \in (0, \lambda_c]$. Using expression (\ref{rec_8}), we obtain 
\begin{align*}
    &\dfrac{3(3-\lambda)}{9-4\lambda}v^{\pi_k}_{\lambda}(N-2, \tilde{\ell}) + \dfrac{3(9-5\lambda)}{9-4\lambda}v^{\pi_k}_{\lambda}(N-2, \tilde{\ell}+1) - \dfrac{18\lambda}{9-4\lambda}v^{\pi_k}_{\lambda}(N-2, \tilde{\ell}+2) \\
    \nonumber &=  \dfrac{3(3-\lambda)}{9-4\lambda} \dfrac{12}{12-5\lambda}\left(\dfrac{\lambda}{8}v^{\pi_k}_{\lambda}(N, \tilde{\ell}) + \dfrac{25\lambda}{72}v^{\pi_k}_{\lambda}(N, \tilde{\ell}+1) + \dfrac{\lambda}{9}v^{\pi_k}_{\lambda}(N-2, \tilde{\ell}+1)\right) \\
    \nonumber &+ \dfrac{3(9-5\lambda)}{9-4\lambda}\dfrac{12}{12-5\lambda}\left(\dfrac{\lambda}{8}v^{\pi_k}_{\lambda}(N, \tilde{\ell}+1) + \dfrac{25\lambda}{72}v^{\pi_k}_{\lambda}(N, \tilde{\ell}+2) + \dfrac{\lambda}{9}v^{\pi_k}_{\lambda}(N-2, \tilde{\ell}+2)\right) \\
    \nonumber &-\dfrac{18\lambda}{9-4\lambda}\dfrac{12}{12-5\lambda}\left(\dfrac{\lambda}{8}v^{\pi_k}_{\lambda}(N, \tilde{\ell}+2) + \dfrac{25\lambda}{72}v^{\pi_k}_{\lambda}(N, \tilde{\ell}+3) + \dfrac{\lambda}{9}v^{\pi_k}_{\lambda}(N-2, \tilde{\ell}+3)\right) \\
    \nonumber &= \dfrac{4\lambda}{3(12-5\lambda)}\Big(\dfrac{3(3-\lambda)}{9-4\lambda}v^{\pi_k}_{\lambda}(N-2, \tilde{\ell}+1) + \dfrac{3(9-5\lambda)}{9-4\lambda}v^{\pi_k}_{\lambda}(N-2, \tilde{\ell}+2) - \dfrac{18\lambda}{9-4\lambda}v^{\pi_k}_{\lambda}(N-2, \tilde{\ell}+3)\Big) \\
    \nonumber &+ \dfrac{9\lambda(3-\lambda)}{2(12-5\lambda)(9-4\lambda)}v^{\pi_k}_{\lambda}(N, \tilde{\ell}) + \dfrac{\lambda(78-35\lambda)}{(9-4\lambda)(12-5\lambda)}v^{\pi_k}_{\lambda}(N, \tilde{\ell}+1) + \dfrac{\lambda(225-179\lambda)}{2(9-4\lambda)(12-5\lambda)}v^{\pi_k}_{\lambda}(N, \tilde{\ell}+2) \\
    \nonumber &- \dfrac{75\lambda^2}{(12-5\lambda)(9-4\lambda)}v^{\pi_k}_{\lambda}(N, \tilde{\ell}+3),
\end{align*}
for $k = 1, 2$, $\lambda \in (0,1)$. By the induction hypothesis, it follows that it suffices to show that
\begin{align}\label{exp15_cond2}
     &\dfrac{9\lambda(3-\lambda)}{2(12-5\lambda)(9-4\lambda)}v^{\pi_k}_{\lambda}(N, m) + \dfrac{\lambda(78-35\lambda)}{(9-4\lambda)(12-5\lambda)}v^{\pi_k}_{\lambda}(N, m+1) + \dfrac{\lambda(225-179\lambda)}{2(9-4\lambda)(12-5\lambda)}v^{\pi_k}_{\lambda}(N, m+2) \\
     \nonumber &- \dfrac{75\lambda^2}{(12-5\lambda)(9-4\lambda)}v^{\pi_k}_{\lambda}(N, m+3) \geq 0,
\end{align}
for all $m = 2, \ldots, N-6$, for $k = 1$, $\lambda \in [\lambda_c, 1)$ and $k =2$, $\lambda \in (0, \lambda_c]$. We prove this statement by yet another induction argument. Consider the case $k = 1$, $\lambda \in [\lambda_c, 1)$. Table \ref{explicit_v_stars} yields that expression (\ref{exp15_cond2}) holds with equality for $m = N-6$. Now, suppose that it holds with equality for $m = \tilde{m}+1$, for some $\tilde{m} = 2, \ldots, N-7$. By expression (\ref{rec_4}) and this induction hypothesis, we obtain
\begin{align*}
    &\dfrac{9\lambda(3-\lambda)}{2(12-5\lambda)(9-4\lambda)}v^{\pi_1}_{\lambda}(N, \tilde{m}) + \dfrac{\lambda(78-35\lambda)}{(9-4\lambda)(12-5\lambda)}v^{\pi_1}_{\lambda}(N, \tilde{m}+1) + \dfrac{\lambda(225-179\lambda)}{2(12-5\lambda)(9-4\lambda)}v^{\pi_1}_{\lambda}(N, \tilde{m}+2)\\
    &- \dfrac{75\lambda^2}{(12-5\lambda)(9-4\lambda)}v^{\pi_1}_{\lambda}(N, \tilde{m}+3) \\
    \nonumber &= \dfrac{2\lambda}{3-\lambda}\Bigg[\dfrac{9\lambda(3-\lambda)}{2(12-5\lambda)(9-4\lambda)}v^{\pi_1}_{\lambda}(N, \tilde{m}+1) + \dfrac{\lambda(78-35\lambda)}{(9-4\lambda)(12-5\lambda)}v^{\pi_1}_{\lambda}(N, \tilde{m}+2) \\
    \nonumber &+  \dfrac{\lambda(225-179\lambda)}{2(12-5\lambda)(9-4\lambda)}v^{\pi_1}_{\lambda}(N, \tilde{m}+3) - \dfrac{75\lambda^2}{(12-5\lambda)(9-4\lambda)}v^{\pi_1}_{\lambda}(N, \tilde{m}+4)\Bigg] = 0,
\end{align*}
for all $\lambda \in [\lambda_c, 1)$. Hence, expression (\ref{exp15_cond2}) is satisfied with equality for all $m = 2, \ldots, N-6$, for $k = 1$, $\lambda \in [\lambda_c, 1)$. \\
Now, consider the case $k = 2$, $\lambda \in (0, \lambda_c]$. Again using Table \ref{explicit_v_stars}, we verify the correctness of expression (\ref{exp15_cond2}) for $m = N-6$ and $m = N-7$. Assume that it holds for all $m \geq \tilde{m}+1$, for some $\tilde{m} = 2, \ldots, N-8$. Expression (\ref{rec_5}) yields
\begin{align*}
    &\dfrac{9\lambda(3-\lambda)}{2(12-5\lambda)(9-4\lambda)}v^{\pi_2}_{\lambda}(N, \tilde{m}) + \dfrac{\lambda(78-35\lambda)}{(9-4\lambda)(12-5\lambda)}v^{\pi_2}_{\lambda}(N, \tilde{m}+1) + \dfrac{\lambda(225-179\lambda)}{2(12-5\lambda)(9-4\lambda)}v^{\pi_2}_{\lambda}(N, \tilde{m}+2)\\
    &- \dfrac{75\lambda^2}{(12-5\lambda)(9-4\lambda)}v^{\pi_2}_{\lambda}(N, \tilde{m}+3) \\
    \nonumber &= \dfrac{9\lambda(3-\lambda)}{2(12-5\lambda)(9-4\lambda)}\left(\dfrac{7\lambda}{3(9-5\lambda)}v^{\pi_2}_{\lambda}(N, \tilde{m}+1) + \dfrac{5\lambda}{3(9-5\lambda)}v^{\pi_2}_{\lambda}(N, \tilde{m}+2)\right) \\
    \nonumber &+ \dfrac{\lambda(78-35\lambda)}{(9-4\lambda)(12-5\lambda)}\left(\dfrac{7\lambda}{3(9-5\lambda)}v^{\pi_2}_{\lambda}(N, \tilde{m}+2) + \dfrac{5\lambda}{3(9-5\lambda)}v^{\pi_2}_{\lambda}(N, \tilde{m}+3)\right) \\
    \nonumber &+\dfrac{\lambda(225-179\lambda)}{2(12-5\lambda)(9-4\lambda)}\left(\dfrac{7\lambda}{3(9-5\lambda)}v^{\pi_2}_{\lambda}(N, \tilde{m}+3) + \dfrac{5\lambda}{3(9-5\lambda)}v^{\pi_2}_{\lambda}(N, \tilde{m}+4)\right) \\
    \nonumber &- \dfrac{75\lambda^2}{(12-5\lambda)(9-4\lambda)}\left(\dfrac{7\lambda}{3(9-5\lambda)}v^{\pi_2}_{\lambda}(N, \tilde{m}+4) + \dfrac{5\lambda}{3(9-5\lambda)}v^{\pi_2}_{\lambda}(N, \tilde{m}+5)\right) \\
    \nonumber &= \dfrac{7\lambda}{3(9-5\lambda)}\Bigg[\dfrac{9\lambda(3-\lambda)}{2(12-5\lambda)(9-4\lambda)}v^{\pi_2}_{\lambda}(N, \tilde{m}+1) + \dfrac{\lambda(78-35\lambda)}{(9-4\lambda)(12-5\lambda)}v^{\pi_2}_{\lambda}(N, \tilde{m}+2) \\
    &+ \dfrac{\lambda(225-179\lambda)}{2(12-5\lambda)(9-4\lambda)}v^{\pi_2}_{\lambda}(N, \tilde{m}+3) - \dfrac{75\lambda^2}{(12-5\lambda)(9-4\lambda)}v^{\pi_2}_{\lambda}(N, \tilde{m}+4)\Bigg] \\
    \nonumber &+\dfrac{5\lambda}{3(9-5\lambda)}\Bigg[\dfrac{9\lambda(3-\lambda)}{2(12-5\lambda)(9-4\lambda)}v^{\pi_2}_{\lambda}(N, \tilde{m}+2) + \dfrac{\lambda(78-35\lambda)}{(9-4\lambda)(12-5\lambda)}v^{\pi_2}_{\lambda}(N, \tilde{m}+3) \\
    \nonumber &+ \dfrac{\lambda(225-179\lambda)}{2(12-5\lambda)(9-4\lambda)}v^{\pi_2}_{\lambda}(N, \tilde{m}+4) - \dfrac{75\lambda^2}{(12-5\lambda)(9-4\lambda)}v^{\pi_2}_{\lambda}(N, \tilde{m}+5)\Bigg].
\end{align*}
By the induction hypothesis, we now obtain
\begin{align*}
    &\dfrac{9\lambda(3-\lambda)}{2(12-5\lambda)(9-4\lambda)}v^{\pi_2}_{\lambda}(N, \tilde{m}) + \dfrac{\lambda(78-35\lambda)}{(9-4\lambda)(12-5\lambda)}v^{\pi_2}_{\lambda}(N, \tilde{m}+1)+ \dfrac{\lambda(225-179\lambda)}{2(12-5\lambda)(9-4\lambda)}v^{\pi_2}_{\lambda}(N, \tilde{m}+2)\\
    &- \dfrac{75\lambda^2}{(12-5\lambda)(9-4\lambda)}v^{\pi_2}_{\lambda}(N, \tilde{m}+3) \geq 0,
\end{align*}
for all $\lambda \in (0, \lambda_c]$. This implies that expression (\ref{exp15_cond2}) holds for all $m = 2, \ldots, N-6$ for $k =2$, $\lambda \in (0, \lambda_c]$ as well. It follows that expression (\ref{exp15_cond1}) is valid for all $\ell = 2, \ldots, N-5$, for $k =1$, $\lambda \in [\lambda_c, 1)$, $k =2$, $\lambda \in (0, \lambda_c]$. This in turn implies the validity of inequality (\ref{15}) for all $j = 2, \ldots, N-4$, for $k =1$, $\lambda \in [\lambda_c, 1)$, $k =2$, $\lambda \in (0, \lambda_c]$.

\paragraph*{Inequality (\ref{16})} We again start by verifying inequality (\ref{16}) for $j = N-4$ and $j = N-5$, for $k =1$, $\lambda \in [\lambda_c, 1)$ and $k =2$, $\lambda \in (0, \lambda_c]$, using Table \ref{explicit_v_stars}. Suppose that it is satisfied for $j = n+1$ for some $n = 2, \ldots, N-6$, for $k =1$, $\lambda \in [\lambda_c, 1)$ and $k =2$, $\lambda \in (0, \lambda_c]$. Expressions (\ref{rec_8}) and (\ref{rec_10}) yield
\begin{align*}
    &-3v^{\pi_k}_{\lambda}(N-3, n) -4v^{\pi_k}_{\lambda}(N-3, n+1) - 5v^{\pi_k}_{\lambda}(N-3, n+2) + 3v^{\pi_k}_{\lambda}(N-2, n) + 9v^{\pi_k}_{\lambda}(N-2, n+1)\\
    \nonumber &= -\dfrac{3\lambda}{9-4\lambda}\Big(v^{\pi_k}_{\lambda}(N-2, n) + v^{\pi_k}_{\lambda}(N-3, n+1) + 3v^{\pi_k}_{\lambda}(N-2, n+1)\Big) \\
    \nonumber &- \dfrac{4\lambda}{9-4\lambda}\Big(v^{\pi_k}_{\lambda}(N-2, n+1) + v^{\pi_k}_{\lambda}(N-3, n+2) + 3v^{\pi_k}_{\lambda}(N-2, n+2)\Big) \\
    \nonumber &- \dfrac{5\lambda}{9-4\lambda}\Big(v^{\pi_k}_{\lambda}(N-2, n+2) + v^{\pi_k}_{\lambda}(N-3, n+3) + 3v^{\pi_k}_{\lambda}(N-2, n+3)\Big) \\ 
    \nonumber &+ \dfrac{36}{12-5\lambda}\left(\dfrac{\lambda}{8}v^{\pi_k}_{\lambda}(N, n) + \dfrac{25\lambda}{72}v^{\pi_k}_{\lambda}(N, n+1) + \dfrac{\lambda}{9}v^{\pi_k}_{\lambda}(N-2, n+1)\right) \\
    \nonumber &+ \dfrac{108}{12-5\lambda}\left(\dfrac{\lambda}{8}v^{\pi_k}_{\lambda}(N, n+1) + \dfrac{25\lambda}{72}v^{\pi_k}_{\lambda}(N, n+2) + \dfrac{\lambda}{9}v^{\pi_k}_{\lambda}(N-2, n+2)\right) \\
    \nonumber &= \dfrac{\lambda}{9-4\lambda}\Big(-3v^{\pi_k}_{\lambda}(N-3, n+1) -4v^{\pi_k}_{\lambda}(N-3, n+2) - 5v^{\pi_k}_{\lambda}(N-3, n+3) + 3v^{\pi_k}_{\lambda}(N-2, n+1) \\
    \nonumber &+ 9v^{\pi_k}_{\lambda}(N-2, n+2)\Big) - \dfrac{3\lambda}{9-4\lambda}v^{\pi_k}_{\lambda}(N-2, n) - \left(\dfrac{16\lambda}{9-4\lambda} - \dfrac{4\lambda}{12-5\lambda}\right)v^{\pi_k}_{\lambda}(N-2, n+1) \\
    \nonumber &- \left(\dfrac{26\lambda}{9-4\lambda} - \dfrac{12\lambda}{12-5\lambda}\right)v^{\pi_k}(N-2, n+2) - \dfrac{15\lambda}{9-4\lambda}v^{\pi_k}_{\lambda}(N-2, n+3) + \dfrac{9\lambda}{2(12-5\lambda)}v^{\pi_k}_{\lambda}(N, n) \\
    \nonumber &+ \dfrac{26\lambda}{12-5\lambda}v^{\pi_k}_{\lambda}(N, n+1) + \dfrac{75\lambda}{2(12-5\lambda)}v^{\pi_k}_{\lambda}(N, n+2),  
\end{align*}
for $k = 1$, $\lambda \in [\lambda_c, 1)$ and $k = 2$, $\lambda \in (0, \lambda_c]$. The induction hypothesis now implies that it suffices to show that
\begin{align}\label{exp16_cond1}
    &- \dfrac{3\lambda}{9-4\lambda}v^{\pi_k}_{\lambda}(N-2, \ell) - \left(\dfrac{16\lambda}{9-4\lambda} - \dfrac{4\lambda}{12-5\lambda}\right)v^{\pi_k}_{\lambda}(N-2, \ell+1) \\
    \nonumber &- \left(\dfrac{26\lambda}{9-4\lambda} - \dfrac{12\lambda}{12-5\lambda}\right)v^{\pi_k}(N-2, \ell+2) - \dfrac{15\lambda}{9-4\lambda}v^{\pi_k}_{\lambda}(N-2, \ell+3) + \dfrac{9\lambda}{2(12-5\lambda)}v^{\pi_k}_{\lambda}(N, \ell) \\
    \nonumber &+ \dfrac{26\lambda}{12-5\lambda}v^{\pi_k}_{\lambda}(N, \ell+1) + \dfrac{75\lambda}{2(12-5\lambda)}v^{\pi_k}_{\lambda}(N, \ell+2) \geq 0,
\end{align}
for all $\ell = 2, \ldots, N-6$, for $k =1$, $\lambda \in [\lambda_c, 1)$ and $k =2$, $\lambda \in (0, \lambda_c]$. We prove this statement by means of an embedded induction argument. First of all, we use Table \ref{explicit_v_stars} to verify its validity for $\ell = N-6$. Now, we assume that it holds for $\ell = \tilde{\ell}+1$, for some $\tilde{\ell} = 2, \ldots, N-7$, for $k =1$, $\lambda \in [\lambda_c, 1)$ and $k =2$, $\lambda \in (0, \lambda^c]$. Using expression (\ref{rec_8}), we obtain
\begin{align*}
    &- \dfrac{3\lambda}{9-4\lambda}v^{\pi_k}_{\lambda}(N-2, \tilde{\ell}) - \left(\dfrac{16\lambda}{9-4\lambda} - \dfrac{4\lambda}{12-5\lambda}\right)v^{\pi_k}_{\lambda}(N-2, \tilde{\ell}+1) - \left(\dfrac{26\lambda}{9-4\lambda} - \dfrac{12\lambda}{12-5\lambda}\right)v^{\pi_k}(N-2, \tilde{\ell}+2) \\
    &- \dfrac{15\lambda}{9-4\lambda}v^{\pi_k}_{\lambda}(N-2, \tilde{\ell}+3) + \dfrac{9\lambda}{2(12-5\lambda)}v^{\pi_k}_{\lambda}(N, \tilde{\ell}) + \dfrac{26\lambda}{12-5\lambda}v^{\pi_k}_{\lambda}(N, \tilde{\ell}+1) + \dfrac{75\lambda}{2(12-5\lambda)}v^{\pi_k}_{\lambda}(N, \tilde{\ell}+2) \\
    \nonumber &= - \dfrac{3\lambda}{9-4\lambda} \dfrac{12}{12-5\lambda}\left(\dfrac{\lambda}{8}v^{\pi_k}_{\lambda}(N, \tilde{\ell}) + \dfrac{25\lambda}{72}v^{\pi_k}_{\lambda}(N, \tilde{\ell}+1) + \dfrac{\lambda}{9}v^{\pi_k}_{\lambda}(N-2, \tilde{\ell}+1)\right) \\
    \nonumber &- \left(\dfrac{16\lambda}{9-4\lambda} - \dfrac{4\lambda}{12-5\lambda}\right) \dfrac{12}{12-5\lambda}\Bigg(\dfrac{\lambda}{8}v^{\pi_k}_{\lambda}(N, \tilde{\ell}+1) + \dfrac{25\lambda}{72}v^{\pi_k}_{\lambda}(N, \tilde{\ell}+2) + \dfrac{\lambda}{9}v^{\pi_k}_{\lambda}(N-2, \tilde{\ell}+2)\Bigg) \\
    &- \left(\dfrac{26\lambda}{9-4\lambda} - \dfrac{12\lambda}{12-5\lambda}\right)\dfrac{12}{12-5\lambda}\Bigg(\dfrac{\lambda}{8}v^{\pi_k}_{\lambda}(N, \tilde{\ell}+2) + \dfrac{25\lambda}{72}v^{\pi_k}_{\lambda}(N, \tilde{\ell}+3) + \dfrac{\lambda}{9}v^{\pi_k}_{\lambda}(N-2, \tilde{\ell}+3)\Bigg) \\
    &-\dfrac{15\lambda}{9-4\lambda}\dfrac{12}{12-5\lambda}\Bigg(\dfrac{\lambda}{8}v^{\pi_k}_{\lambda}(N, \tilde{\ell}+3) + \dfrac{25\lambda}{72}v^{\pi_k}_{\lambda}(N, \tilde{\ell}+4) + \dfrac{\lambda}{9}v^{\pi_k}_{\lambda}(N-2, \tilde{\ell}+4)\Bigg) + \dfrac{9\lambda}{2(12-5\lambda)}v^{\pi_k}_{\lambda}(N, \tilde{\ell}) \\
    &+ \dfrac{26\lambda}{12-5\lambda}v^{\pi_k}_{\lambda}(N, \tilde{\ell}+1) + \dfrac{75\lambda}{2(12-5\lambda)}v^{\pi_k}_{\lambda}(N, \tilde{\ell}+2) \\
    \nonumber &= \dfrac{4\lambda}{3(12-5\lambda)}\Bigg[- \dfrac{3\lambda}{9-4\lambda}v^{\pi_k}_{\lambda}(N-2, \tilde{\ell}+1) - \left(\dfrac{16\lambda}{9-4\lambda} - \dfrac{4\lambda}{12-5\lambda}\right)v^{\pi_k}_{\lambda}(N-2, \tilde{\ell}+2)\\
    \nonumber &- \left(\dfrac{26\lambda}{9-4\lambda} - \dfrac{12\lambda}{12-5\lambda}\right)v^{\pi_k}(N-2, \tilde{\ell}+3) - \dfrac{15\lambda}{9-4\lambda}v^{\pi_k}_{\lambda}(N-2, \tilde{\ell}+4) + \dfrac{9\lambda}{2(12-5\lambda)}v^{\pi_k}_{\lambda}(N, \tilde{\ell}+1) \\
    &+ \dfrac{26\lambda}{12-5\lambda}v^{\pi_k}_{\lambda}(N, \tilde{\ell}+2) + \dfrac{75\lambda}{2(12-5\lambda)}v^{\pi_k}_{\lambda}(N, \tilde{\ell}+3)\Bigg] + \dfrac{9\lambda(9-5\lambda)}{2(9-4\lambda)(12-5\lambda)}v^{\pi_k}_{\lambda}(N, \tilde{\ell})\\
    &+ \dfrac{\lambda(468-281\lambda)}{2(9-4\lambda)(12-5\lambda)}v^{\pi_k}_{\lambda}(N, \tilde{\ell}+1) + \dfrac{\lambda(2025-1534\lambda)}{6(9-4\lambda)(12-5\lambda)}v^{\pi_k}_{\lambda}(N, \tilde{\ell}+2) - \dfrac{785\lambda^2}{6(9-4\lambda)(12-5\lambda)}v^{\pi_k}_{\lambda}(N, \tilde{\ell}+3) \\
    \nonumber &- \dfrac{125\lambda^2}{2(9-4\lambda)(12-5\lambda)}v^{\pi_k}_{\lambda}(N, \tilde{\ell}+4),
\end{align*}
for $k = 1$, $\lambda \in [\lambda_c, 1)$ and $k = 2$, $\lambda \in (0, \lambda_c]$. The induction hypothesis implies that in order to prove expression (\ref{exp16_cond1}) for $\ell = \tilde{\ell}$, it suffices to show that
\begin{align}\label{exp16_cond2}
    &\dfrac{9\lambda(9-5\lambda)}{2(9-4\lambda)(12-5\lambda)}v^{\pi_k}_{\lambda}(N, m) + \dfrac{\lambda(468-281\lambda)}{2(9-4\lambda)(12-5\lambda)}v^{\pi_k}_{\lambda}(N, m+1) \\
    \nonumber &+ \dfrac{\lambda(2025-1534\lambda)}{6(9-4\lambda)(12-5\lambda)}v^{\pi_k}_{\lambda}(N, m+2) - \dfrac{785\lambda^2}{6(9-4\lambda)(12-5\lambda))}v^{\pi_k}_{\lambda}(N, m+3) \\
    \nonumber &- \dfrac{125\lambda^2}{2(9-4\lambda)(12-5\lambda)}v^{\pi_k}_{\lambda}(N, m+4) \geq 0,
\end{align}
for all $m = 2, \ldots, N-7$, for $k =1$, $\lambda \in [\lambda_c, 1)$ and $k =2$, $\lambda \in (0, \lambda_c]$. We prove this statement by means of another induction argument. First, we consider the case $k =1$, $\lambda \in [\lambda_c, 1)$. After verifying expression (\ref{exp16_cond2}) for $m = N-7$, using Table \ref{explicit_v_stars}, we assume that it holds for $m = \tilde{m}+1$, for some $\tilde{m} = 2, \ldots, N-8$. This induction hypothesis, together with expression (\ref{rec_4}) now implies 
\begin{align*}
   &\dfrac{9\lambda(9-5\lambda)}{2(9-4\lambda)(12-5\lambda)}v^{\pi_1}_{\lambda}(N, \tilde{m}) + \dfrac{\lambda(468-281\lambda)}{2(9-4\lambda)(12-5\lambda)}v^{\pi_1}_{\lambda}(N, \tilde{m}+1) + \dfrac{\lambda(2025-1534\lambda)}{6(9-4\lambda)(12-5\lambda)}v^{\pi_1}_{\lambda}(N, \tilde{m}+2) \\
   &- \dfrac{785\lambda^2}{6(9-4\lambda)(12-5\lambda)}v^{\pi_1}_{\lambda}(N, \tilde{m}+3) - \dfrac{125\lambda^2}{2(9-4\lambda)(12-5\lambda)}v^{\pi_1}_{\lambda}(N, \tilde{m}+4)\\
   &= \dfrac{2\lambda}{3-\lambda}\Bigg[\dfrac{9\lambda(9-5\lambda)}{2(9-4\lambda)(12-5\lambda)}v^{\pi_1}_{\lambda}(N, \tilde{m}+1) + \dfrac{\lambda(468-281\lambda)}{2(9-4\lambda)(12-5\lambda)}v^{\pi_1}_{\lambda}(N, \tilde{m}+2) \\
   \nonumber &+ \dfrac{\lambda(2025-1534\lambda)}{6(9-4\lambda)(12-5\lambda)}v^{\pi_1}_{\lambda}(N, \tilde{m}+3) - \dfrac{785\lambda^2}{6(9-4\lambda)(12-5\lambda)}v^{\pi_1}_{\lambda}(N, \tilde{m}+4) \\
   \nonumber &- \dfrac{125\lambda^2}{2(9-4\lambda)(12-5\lambda)}v^{\pi_1}_{\lambda}(N, \tilde{m}+5)\Bigg] \geq 0.
\end{align*}
Thus, expression (\ref{exp16_cond2}) holds for all $m = 2, \ldots, N-7$, for $k =1$, $\lambda \in [\lambda_c, 1)$. We now direct our attention to the case $k = 2$, $\lambda \in (0, \lambda^c]$. From Table \ref{explicit_v_stars}, it follows that expression (\ref{exp16_cond2}) holds with equality for $m = N-7$ and $m = N-8$. We proceed to assume that it holds with equality for all $m = \tilde{m}+1, \ldots, N-7$, for some $\tilde{m} = 2, \ldots, N-9$. By expression (\ref{rec_5}), we have
\begin{align*}
   &\dfrac{9\lambda(9-5\lambda)}{2(9-4\lambda)(12-5\lambda)}v^{\pi_2}_{\lambda}(N, \tilde{m}) + \dfrac{\lambda(468-281\lambda)}{2(9-4\lambda)(12-5\lambda)}v^{\pi_2}_{\lambda}(N, \tilde{m}+1) + \dfrac{\lambda(2025-1534\lambda)}{6(9-4\lambda)(12-5\lambda)}v^{\pi_2}_{\lambda}(N, \tilde{m}+2) \\
   &- \dfrac{785\lambda^2}{6(9-4\lambda)(12-5\lambda)}v^{\pi_2}_{\lambda}(N, \tilde{m}+3) - \dfrac{125\lambda^2}{2(9-4\lambda)(12-5\lambda)}v^{\pi_2}_{\lambda}(N, \tilde{m}+4) \\
    \nonumber &= \dfrac{9\lambda(9-5\lambda)}{2(9-4\lambda)(12-5\lambda)}\left[\dfrac{7\lambda}{3(9-5\lambda)}v^{\pi_2}_{\lambda}(N, \tilde{m}+1) + \dfrac{5\lambda}{3(9-5\lambda)}v^{\pi_2}_{\lambda}(N, \tilde{m}+2)\right] \\
    \nonumber &+ \dfrac{\lambda(468-281\lambda)}{2(9-4\lambda)(12-5\lambda)}\left[\dfrac{7\lambda}{3(9-5\lambda)}v^{\pi_2}_{\lambda}(N, \tilde{m}+2) + \dfrac{5\lambda}{3(9-5\lambda)}v^{\pi_2}_{\lambda}(N, \tilde{m}+3)\right] \\
    \nonumber &+ \dfrac{\lambda(2025-1534\lambda)}{6(9-4\lambda)(12-5\lambda)}\left[\dfrac{7\lambda}{3(9-5\lambda)}v^{\pi_2}_{\lambda}(N, \tilde{m}+3) + \dfrac{5\lambda}{3(9-5\lambda)}v^{\pi_2}_{\lambda}(N, \tilde{m}+4)\right] \\
    \nonumber &-\dfrac{785\lambda^2}{6(9-4\lambda)(12-5\lambda)} \left[\dfrac{7\lambda}{3(9-5\lambda)}v^{\pi_2}_{\lambda}(N, \tilde{m}+4) + \dfrac{5\lambda}{3(9-5\lambda)}v^{\pi_2}_{\lambda}(N, \tilde{m}+5)\right] \\ 
    \nonumber &- \dfrac{125\lambda^2}{2(9-4\lambda)(12-5\lambda)}\left[\dfrac{7\lambda}{3(9-5\lambda)}v^{\pi_2}_{\lambda}(N, \tilde{m}+5) + \dfrac{5\lambda}{3(9-5\lambda)}v^{\pi_2}_{\lambda}(N, \tilde{m}+6)\right] \\
    \nonumber &= \dfrac{7\lambda}{3(9-5\lambda)}\Bigg[\dfrac{9\lambda(9-5\lambda)}{2(9-4\lambda)(12-5\lambda)}v^{\pi_k}_{\lambda}(N, \tilde{m}+1) + \dfrac{\lambda(468-281\lambda)}{2(9-4\lambda)(12-5\lambda)}v^{\pi_k}_{\lambda}(N, \tilde{m}+2) \\
    \nonumber &+ \dfrac{\lambda(2025-1534\lambda)}{6(9-4\lambda)(12-5\lambda)}v^{\pi_k}_{\lambda}(N, \tilde{m}+3) - \dfrac{785\lambda^2}{6(9-4\lambda)(12-5\lambda)}v^{\pi_k}_{\lambda}(N, \tilde{m}+4) \\
    \nonumber &- \dfrac{125\lambda^2}{2(12-5\lambda)(9-4\lambda)}v^{\pi_k}_{\lambda}(N, \tilde{m}+5)\Bigg] + \dfrac{5\lambda}{3(9-5\lambda)}\Bigg[\dfrac{9\lambda(9-5\lambda)}{2(9-4\lambda)(12-5\lambda)}v^{\pi_k}_{\lambda}(N, \tilde{m}+2) \\
    \nonumber &+ \dfrac{\lambda(468-281\lambda)}{2(9-4\lambda)(12-5\lambda)}v^{\pi_k}_{\lambda}(N, \tilde{m}+3) + \dfrac{\lambda(2025-1534\lambda)}{6(9-4\lambda)(12-5\lambda)}v^{\pi_k}_{\lambda}(N, \tilde{m}+4) \\
    \nonumber &- \dfrac{785\lambda^2}{6(9-4\lambda)(12-5\lambda)}v^{\pi_k}_{\lambda}(N, \tilde{m}+5) - \dfrac{125\lambda^2}{2(12-5\lambda)(9-4\lambda)}v^{\pi_k}_{\lambda}(N, \tilde{m}+6)\Bigg].
\end{align*}
From the induction hypothesis, it now follows that
\begin{align*}
    &\dfrac{9\lambda(9-5\lambda)}{2(9-4\lambda)(12-5\lambda)}v^{\pi_2}_{\lambda}(N, \tilde{m}) + \dfrac{\lambda(468-281\lambda)}{2(9-4\lambda)(12-5\lambda)}v^{\pi_2}_{\lambda}(N, \tilde{m}+1) + \dfrac{\lambda(2025-1534\lambda)}{6(9-4\lambda)(12-5\lambda)}v^{\pi_2}_{\lambda}(N, \tilde{m}+2) \\
    &- \dfrac{785\lambda^2}{6(9-4\lambda)(12-5\lambda)}v^{\pi_2}_{\lambda}(N, \tilde{m}+3) - \dfrac{125\lambda^2}{2(9-4\lambda)(12-5\lambda)}v^{\pi_2}_{\lambda}(N, \tilde{m}+4) = 0,
\end{align*}
for all $\lambda \in (0, \lambda_c]$. Hence, expression (\ref{exp16_cond2}) is satisfied for all $m = 2, \ldots, N-7$, for $k =2$, $\lambda \in (0, \lambda_c]$ as well. It follows that expression (\ref{exp16_cond1}) holds for all $\ell = 2, \ldots, N-6$, for $k=1$, $\lambda \in [\lambda_c, 1)$ and $k=2$, $\lambda \in (0, \lambda_c]$. This in turn implies the validity of expression (\ref{16}) for all $j = 2, \ldots, N-4$, for $k=1$, $\lambda \in [\lambda_c, 1)$ and $k=2$, $\lambda \in (0, \lambda_c]$.

\paragraph*{Inequality (\ref{17})}
First of all, we use Table \ref{explicit_v_stars} to verify inequality (\ref{17}) for $j = N-4$, for $k =1$, $\lambda \in [\lambda_c, 1)$ and $k =2$, $\lambda \in (0, \lambda_c]$. We now assume that it holds for $j = n+1$, for some $n = 2, \ldots, N-5$, for $k =1$, $\lambda \in [\lambda_c, 1)$ and $k =2$, $\lambda \in (0, \lambda_c]$. By expressions (\ref{rec_8}) and (\ref{rec_10}), we obtain
\begin{align*}
    &v^{\pi_k}_{\lambda}(N-3, n) + v^{\pi_k}_{\lambda}(N-3, n+1) - 5v^{\pi_k}_{\lambda}(N-2, n) + 3v^{\pi_k}_{\lambda}(N-2, n+1) \\
    \nonumber &= \dfrac{\lambda}{9-4\lambda}\Big(v^{\pi_k}_{\lambda}(N-2, n) + v^{\pi_k}_{\lambda}(N-3, n+1) + 3v^{\pi_k}_{\lambda}(N-2, n+1)\Big) \\
    \nonumber &+ \dfrac{\lambda}{9-4\lambda}\Big(v^{\pi_k}_{\lambda}(N-2, n+1) + v^{\pi_k}_{\lambda}(N-3, n+2) + 3v^{\pi_k}_{\lambda}(N-2, n+2)\Big) \\
    \nonumber &- \dfrac{60}{12-5\lambda}\left[\dfrac{\lambda}{8}v^{\pi_k}_{\lambda}(N, n) + \dfrac{25\lambda}{72}v^{\pi_k}_{\lambda}(N, n+1) + \dfrac{\lambda}{9}v^{\pi_k}_{\lambda}(N-2, n+1)\right] \\
    \nonumber &+ \dfrac{36}{12-5\lambda}\left[\dfrac{\lambda}{8}v^{\pi_k}_{\lambda}(N, n+1) + \dfrac{25\lambda}{72}v^{\pi_k}_{\lambda}(N, n+2) + \dfrac{\lambda}{9}v^{\pi_k}_{\lambda}(N-2, n+2)\right] \\
    \nonumber &= \dfrac{\lambda}{9-4\lambda}\Big(v^{\pi_k}_{\lambda}(N-3, n+1) + v^{\pi_k}_{\lambda}(N-3, n+2) - 5v^{\pi_k}_{\lambda}(N-2, n+1) + 3v^{\pi_k}_{\lambda}(N-2, n+2)\Big) \\
    \nonumber &+ \dfrac{\lambda}{9-4\lambda}v^{\pi_k}_{\lambda}(N-2, n) - \left(\dfrac{20\lambda}{3(12-5\lambda)} - \dfrac{9\lambda}{9-4\lambda}\right)v^{\pi_k}_{\lambda}(N-2, n+1) + \dfrac{4\lambda}{12-5\lambda}v^{\pi_k}_{\lambda}(N-2, n+2) \\
    &- \dfrac{15\lambda}{2(12-5\lambda)}v^{\pi_k}_{\lambda}(N, n) - \dfrac{49\lambda}{3(12-5\lambda)}v^{\pi_k}_{\lambda}(N, n+1) + \dfrac{25\lambda}{2(12-5\lambda)}v^{\pi_k}_{\lambda}(N, n+2),
\end{align*}
for $k = 1$, $\lambda \in [\lambda_c, 1)$ and $k = 2$, $\lambda \in (0, \lambda_c]$. By the induction hypothesis, it suffices to show that 
\begin{align}\label{exp17_cond1}
    &\dfrac{\lambda}{9-4\lambda}v^{\pi_k}_{\lambda}(N-2, \ell) - \left(\dfrac{20\lambda}{3(12-5\lambda)} - \dfrac{9\lambda}{9-4\lambda}\right)v^{\pi_k}_{\lambda}(N-2, \ell+1) + \dfrac{4\lambda}{12-5\lambda}v^{\pi_k}_{\lambda}(N-2, \ell+2) \\
    \nonumber &- \dfrac{15\lambda}{2(12-5\lambda)}v^{\pi_k}_{\lambda}(N, \ell) - \dfrac{49\lambda}{3(12-5\lambda)}v^{\pi_k}_{\lambda}(N, \ell+1) + \dfrac{25\lambda}{2(12-5\lambda)}v^{\pi_k}_{\lambda}(N, \ell+2) \geq 0,
\end{align}
for all $\ell = 2, \ldots, N-5$, for $k =1$, $\lambda \in [\lambda_c, 1)$ and $k =2$, $\lambda \in (0, \lambda_c]$. We prove this statement by another induction argument. First, we verify its validity for $\ell = N-5$ using Table \ref{explicit_v_stars}. Now, suppose that it holds for $\ell = \tilde{\ell} +1$, for some $\tilde{\ell} = 2, \ldots, N-6$, for $k =1$, $\lambda \in [\lambda_c, 1)$ and $k =2$, $\lambda \in (0, \lambda_c]$. Using expression (\ref{rec_8}), we obtain
\begin{align*}
    &\dfrac{\lambda}{9-4\lambda}v^{\pi_k}_{\lambda}(N-2, \tilde{\ell}) - \left(\dfrac{20\lambda}{3(12-5\lambda)} - \dfrac{9\lambda}{9-4\lambda}\right)v^{\pi_k}_{\lambda}(N-2, \tilde{\ell}+1) + \dfrac{4\lambda}{12-5\lambda}v^{\pi_k}_{\lambda}(N-2, \tilde{\ell}+2)\\
    &- \dfrac{15\lambda}{2(12-5\lambda)}v^{\pi_k}_{\lambda}(N, \tilde{\ell}) - \dfrac{49\lambda}{3(12-5\lambda)}v^{\pi_k}_{\lambda}(N, \tilde{\ell}+1) + \dfrac{25\lambda}{2(12-5\lambda)}v^{\pi_k}_{\lambda}(N, \tilde{\ell}+2) \\
    \nonumber &= \dfrac{\lambda}{9-4\lambda}\dfrac{12}{12-5\lambda}\Bigg[\dfrac{\lambda}{8}v^{\pi_k}_{\lambda}(N, \tilde{\ell}) + \dfrac{25\lambda}{72}v^{\pi_k}_{\lambda}(N, \tilde{\ell}+1) +\dfrac{\lambda}{9}v^{\pi_k}_{\lambda}(N-2, \tilde{\ell}+1)\Bigg] \\
    &- \left(\dfrac{20\lambda}{3(12-5\lambda)} - \dfrac{9\lambda}{9-4\lambda}\right)\dfrac{12}{12-5\lambda}\Bigg[\dfrac{\lambda}{8}v^{\pi_k}_{\lambda}(N, \tilde{\ell}+1) + \dfrac{25\lambda}{72}v^{\pi_k}_{\lambda}(N, \tilde{\ell}+2) + \dfrac{\lambda}{9}v^{\pi_k}_{\lambda}(N-2, \tilde{\ell}+2)\Bigg] \\
    &+\dfrac{4\lambda}{12-5\lambda}\dfrac{12}{12-5\lambda}\Bigg[\dfrac{\lambda}{8}v^{\pi_k}_{\lambda}(N, \tilde{\ell}+2) + \dfrac{25\lambda}{72}v^{\pi_k}_{\lambda}(N, \tilde{\ell}+3) + \dfrac{\lambda}{9}v^{\pi_k}_{\lambda}(N-2, \tilde{\ell}+3)\Bigg] - \dfrac{15\lambda}{2(12-5\lambda)}v^{\pi_k}_{\lambda}(N, \tilde{\ell}) \\
    &- \dfrac{49\lambda}{3(12-5\lambda)}v^{\pi_k}_{\lambda}(N, \tilde{\ell}+1) + \dfrac{25\lambda}{2(12-5\lambda)}v^{\pi_k}_{\lambda}(N, \tilde{\ell}+2) \\
    \nonumber &= \dfrac{4\lambda}{3(12-5\lambda)}\Bigg[\dfrac{\lambda}{9-4\lambda}v^{\pi_k}_{\lambda}(N-2, \tilde{\ell}+1) - \left(\dfrac{20\lambda}{3(12-5\lambda)} - \dfrac{9\lambda}{9-4\lambda}\right)v^{\pi_k}_{\lambda}(N-2, \tilde{\ell}+2) \\
    \nonumber &+ \dfrac{4\lambda}{12-5\lambda}v^{\pi_k}_{\lambda}(N-2, \tilde{\ell}+3) - \dfrac{15\lambda}{2(12-5\lambda)}v^{\pi_k}_{\lambda}(N, \tilde{\ell}+1) - \dfrac{49\lambda}{3(12-5\lambda)}v^{\pi_k}_{\lambda}(N, \tilde{\ell}+2) \\
    &+ \dfrac{25\lambda}{2(12-5\lambda)}v^{\pi_k}_{\lambda}(N, \tilde{\ell}+3)\Bigg] -\dfrac{9\lambda(15-7\lambda)}{2(9-4\lambda)(12-5\lambda)}v^{\pi_k}_{\lambda}(N, \tilde{\ell}) - \dfrac{\lambda(147-83\lambda)}{(9-4\lambda)(12-5\lambda)}v^{\pi_k}_{\lambda}(N, \tilde{\ell}+1) \\
    &+ \dfrac{25\lambda(9-\lambda)}{2(9-4\lambda)(12-5\lambda)}v^{\pi_k}_{\lambda}(N, \tilde{\ell}+2),
\end{align*}
for $k = 1, 2$ and $\lambda \in (0,1)$. From the induction hypothesis, it follows that a sufficient condition for expression (\ref{exp17_cond1}) to hold for $\ell = \tilde{\ell}$ is
\begin{align}\label{exp17_cond2}
&-\dfrac{9\lambda(15-7\lambda)}{2(9-4\lambda)(12-5\lambda)}v^{\pi_k}_{\lambda}(N, m) - \dfrac{\lambda(147-83\lambda)}{(9-4\lambda)(12-5\lambda)}v^{\pi_k}_{\lambda}(N, m+1) + \dfrac{25\lambda(9-\lambda)}{2(9-4\lambda)(12-5\lambda)}v^{\pi_k}_{\lambda}(N, m+2) \geq 0,
\end{align}
for all $m = 2, \ldots, N-6$, for $k =1$, $\lambda \in [\lambda_c, 1)$ and $k =2$, $\lambda \in (0, \lambda_c]$. We prove this statement by another embedded induction argument. Consider the case $k =1$, $\lambda \in [\lambda_c, 1)$. We start by verifying the validity of expression (\ref{exp17_cond2}) for $m = N-6$ using Table \ref{explicit_v_stars}. Now, assume that it holds for $m = \tilde{m}+1$, for some $\tilde{m} = 2, \ldots, N-7$. Expression (\ref{rec_4}) and this induction hypothesis now yield
\begin{align*}
    & -\dfrac{9\lambda(15-7\lambda)}{2(9-4\lambda)(12-5\lambda)}v^{\pi_1}_{\lambda}(N, \tilde{m}) - \dfrac{\lambda(147-83\lambda)}{(9-4\lambda)(12-5\lambda)}v^{\pi_1}_{\lambda}(N, \tilde{m}+1) + \dfrac{25\lambda(9-\lambda)}{2(9-4\lambda)(12-5\lambda)}v^{\pi_1}_{\lambda}(N, \tilde{m}+2) \\
    \nonumber &= \dfrac{2\lambda}{3-\lambda}\Bigg[-\dfrac{9\lambda(15-7\lambda)}{2(9-4\lambda)(12-5\lambda)}v^{\pi_1}_{\lambda}(N, \tilde{m}+1) - \dfrac{\lambda(147-83\lambda)}{(9-4\lambda)(12-5\lambda)}v^{\pi_1}_{\lambda}(N, \tilde{m}+2) \\
    \nonumber &+ \dfrac{25\lambda(9-\lambda)}{2(9-4\lambda)(12-5\lambda)}v^{\pi_1}_{\lambda}(N, \tilde{m}+3)\Bigg] \geq 0,
\end{align*}
for all $\lambda \in [\lambda_c, 1)$. Thus, expression (\ref{exp17_cond2}) is satisfied for all $m = 2, \ldots, N-6$, for $k = 1$, $\lambda \in [\lambda_c, 1)$. \\
We proceed to consider the case $k =2$, $\lambda \in (0, \lambda _c]$. Again using Table \ref{explicit_v_stars}, we verify the correctness of expression (\ref{exp17_cond2}) for $m = N-6$ and $m = N-7$. Suppose now that it is satisfied for all $m \geq \tilde{m}+1$, for some $\tilde{m} = 2, \ldots, N-8$. Using expression (\ref{rec_5}), we obtain
\begin{align*}
    & -\dfrac{9\lambda(15-7\lambda)}{2(9-4\lambda)(12-5\lambda)}v^{\pi_2}_{\lambda}(N, \tilde{m}) - \dfrac{\lambda(147-83\lambda)}{(9-4\lambda)(12-5\lambda)}v^{\pi_2}_{\lambda}(N, \tilde{m}+1) + \dfrac{25\lambda(9-\lambda)}{2(9-4\lambda)(12-5\lambda)}v^{\pi_2}_{\lambda}(N, \tilde{m}+2) \\
    \nonumber &=  -\dfrac{9\lambda(15-7\lambda)}{2(9-4\lambda)(12-5\lambda)}\left[\dfrac{7\lambda}{3(9-5\lambda)}v^{\pi_2}_{\lambda}(N, \tilde{m}+1) + \dfrac{5\lambda}{3(9-5\lambda)}v^{\pi_2}_{\lambda}(N, \tilde{m}+2)\right] \\
    \nonumber &- \dfrac{\lambda(147-83\lambda)}{(9-4\lambda)(12-5\lambda)}\left[\dfrac{7\lambda}{3(9-5\lambda)}v^{\pi_2}_{\lambda}(N, \tilde{m}+2) + \dfrac{5\lambda}{3(9-5\lambda)}v^{\pi_2}_{\lambda}(N, \tilde{m}+3)\right] \\
    \nonumber &+ \dfrac{25\lambda(9-\lambda)}{2(9-4\lambda)(12-5\lambda)}\left[\dfrac{7\lambda}{3(9-5\lambda)}v^{\pi_2}_{\lambda}(N, \tilde{m}+3) + \dfrac{5\lambda}{3(9-5\lambda)}v^{\pi_2}_{\lambda}(N, \tilde{m}+4)\right] \\
    \nonumber &= \dfrac{7\lambda}{3(9-5\lambda)}\Bigg[-\dfrac{9\lambda(15-7\lambda)}{2(9-4\lambda)(12-5\lambda)}v^{\pi_2}_{\lambda}(N, \tilde{m}+1) - \dfrac{\lambda(147-83\lambda)}{(9-4\lambda)(12-5\lambda)}v^{\pi_2}_{\lambda}(N, \tilde{m}+2)\\
    \nonumber &+ \dfrac{25\lambda(9-\lambda)}{2(9-4\lambda)(12-5\lambda)}v^{\pi_2}_{\lambda}(N, \tilde{m}+3)\Bigg] +\dfrac{5\lambda}{3(9-5\lambda)}\Bigg[-\dfrac{9\lambda(15-7\lambda)}{2(9-4\lambda)(12-5\lambda)}v^{\pi_2}_{\lambda}(N, \tilde{m}+2) \\
    \nonumber &- \dfrac{\lambda(147-83\lambda)}{(9-4\lambda)(12-5\lambda)}v^{\pi_2}_{\lambda}(N, \tilde{m}+3) + \dfrac{25\lambda(9-\lambda)}{2(9-4\lambda)(12-5\lambda)}v^{\pi_2}_{\lambda}(N, \tilde{m}+4)\Bigg]. 
\end{align*}
It now follows from the induction hypothesis that
\begin{align*}
    & -\dfrac{9\lambda(15-7\lambda)}{2(9-4\lambda)(12-5\lambda)}v^{\pi_2}_{\lambda}(N, \tilde{m}) - \dfrac{\lambda(147-83\lambda)}{(9-4\lambda)(12-5\lambda)}v^{\pi_2}_{\lambda}(N, \tilde{m}+1) + \dfrac{25\lambda(9-\lambda)}{2(9-4\lambda)(12-5\lambda)}v^{\pi_2}_{\lambda}(N, \tilde{m}+2) \geq 0.
\end{align*}
Hence, expression (\ref{exp17_cond2}) holds for all $m = 2, \ldots, N-6$, for $k = 2$, $\lambda \in (0, \lambda_c]$ as well. This implies that expression (\ref{exp17_cond1}) is valid for all $\ell = 2, \ldots, N-5$ for $k = 1$, $\lambda \in [\lambda_c, 1)$ and for $k = 2$, $\lambda \in (0, \lambda_c]$. This in turn establishes the correctness of expression (\ref{17}) for all $j = 2, \ldots, N-4$, for $k = 1$, $\lambda \in [\lambda_c, 1)$ and for $k = 2$, $\lambda \in (0, \lambda_c]$.

\paragraph*{Inequality (\ref{18})}
Again, we start by verifying inequality (\ref{18}) for $j = N-4$ using Table \ref{explicit_v_stars}, for $k = 1$, $\lambda \in [\lambda_c, 1)$ and $k = 2$, $\lambda \in (0, \lambda_c]$. Now, suppose that it is satisfied for $j = n+1$, for some $n = 2, \ldots, N-5$, for $k = 1$, $\lambda \in [\lambda_c, 1)$ and $k = 2$, $\lambda \in (0, \lambda_c]$. Using expressions (\ref{rec_8}) and (\ref{rec_10}), we obtain
\begin{align*}
    &8v^{\pi_k}_{\lambda}(N-3, n) + 16v^{\pi_k}_{\lambda}(N-3, n+1) - 15v^{\pi_k}_{\lambda}(N-2, n) + 48v^{\pi_k}_{\lambda}(N-2, n+1) - 57v^{\pi_k}_{\lambda}(N, n) \\
    \nonumber & =\dfrac{8\lambda}{9-4\lambda}\Big(v^{\pi_k}_{\lambda}(N-2, n) + v^{\pi_k}_{\lambda}(N-3, n+1) + 3v^{\pi_k}_{\lambda}(N-2, n+1)\Big) \\
    \nonumber &+ \dfrac{16\lambda}{9-4\lambda}\Big(v^{\pi_k}_{\lambda}(N-2, n+1) + v^{\pi_k}_{\lambda}(N-3, n+2) + 3v^{\pi_k}_{\lambda}(N-2, n+2)\Big) \\
    \nonumber &- \dfrac{180}{12-5\lambda}\left(\dfrac{\lambda}{8}v^{\pi_k}_{\lambda}(N, n) + \dfrac{25\lambda}{72}v^{\pi_k}_{\lambda}(N, n+1) + \dfrac{\lambda}{9}v^{\pi_k}_{\lambda}(N-2, n+1)\right) \\
    \nonumber &+ \dfrac{576}{12-5\lambda}\left(\dfrac{\lambda}{8}v^{\pi_k}_{\lambda}(N, n+1) + \dfrac{25\lambda}{72}v^{\pi_k}_{\lambda}(N, n+2) + \dfrac{\lambda}{9}v^{\pi_k}_{\lambda}(N-2, n+2)\right) - 57v^{\pi_k}_{\lambda}(N, n) \\
    \nonumber &= \dfrac{\lambda}{9-4\lambda}\Big(8v^{\pi_k}_{\lambda}(N-3, n+1) + 16v^{\pi_k}_{\lambda}(N-3, n+2) - 15v^{\pi_k}_{\lambda}(N-2, n+1) + 48v^{\pi_k}_{\lambda}(N-2, n+2) \\
    &- 57v^{\pi_k}_{\lambda}(N, n+1)\Big) + \dfrac{8\lambda}{9-4\lambda}v^{\pi_k}_{\lambda}(N-2, n) + \left[\dfrac{55\lambda}{9-4\lambda} - \dfrac{20\lambda}{12-5\lambda}\right]v^{\pi_k}_{\lambda}(N-2, n+1) \\
    \nonumber &+ \dfrac{64\lambda}{12-5\lambda}v^{\pi_k}_{\lambda}(N-2, n+2) -\dfrac{1368-525\lambda}{2(12-5\lambda)}v^{\pi_k}_{\lambda}(N, n) + \left[\dfrac{19\lambda}{2(12-5\lambda)}+ \dfrac{57\lambda}{9-4\lambda}\right]v^{\pi_k}_{\lambda}(N, n+1) \\
    \nonumber &+ \dfrac{200}{12-5\lambda}v^{\pi_k}_{\lambda}(N, n+2),
\end{align*}
for $k = 1$, $\lambda \in [\lambda_c, 1)$ and $k = 2$, $\lambda \in (0, \lambda_c]$. By the induction hypothesis, it follows that it suffices to show that
\begin{align}\label{exp18_cond1}
    &\dfrac{8\lambda}{9-4\lambda}v^{\pi_k}_{\lambda}(N-2, \ell) + \left[\dfrac{55\lambda}{9-4\lambda} - \dfrac{20\lambda}{12-5\lambda}\right]v^{\pi_k}_{\lambda}(N-2, \ell+1) + \dfrac{64\lambda}{12-5\lambda}v^{\pi_k}_{\lambda}(N-2, \ell+2)\\
    \nonumber &- \dfrac{1368-525\lambda}{2(12-5\lambda)}v^{\pi_k}_{\lambda}(N, \ell) + \left[\dfrac{19\lambda}{2(12-5\lambda)}+ \dfrac{57\lambda}{9-4\lambda}\right]v^{\pi_k}_{\lambda}(N, \ell+1) + \dfrac{200}{12-5\lambda}v^{\pi_k}_{\lambda}(N, \ell+2) \geq 0,
\end{align}
for all $\ell = 2, \ldots, N-5$, for $k = 1$, $\lambda \in [\lambda_c, 1)$ and $k = 2$, $\lambda \in (0, \lambda_c]$. We prove this statement by an embedded induction argument. First of all, we use Table \ref{explicit_v_stars} to verify its validity for $\ell = N-5$. Now, we assume that it holds for $\ell = \tilde{\ell} + 1$ for some $\tilde{\ell} = 2, \ldots, N-6$, for $k = 1$, $\lambda \in [\lambda_c, 1)$ and $k = 2$, $\lambda \in (0, \lambda_c]$. Expression (\ref{rec_8}) yields 
\begin{align*}
    &\dfrac{8\lambda}{9-4\lambda}v^{\pi_k}_{\lambda}(N-2, \tilde{\ell}) + \left[\dfrac{55\lambda}{9-4\lambda} - \dfrac{20\lambda}{12-5\lambda}\right]v^{\pi_k}_{\lambda}(N-2, \tilde{\ell}+1) + \dfrac{64\lambda}{12-5\lambda}v^{\pi_k}_{\lambda}(N-2, \tilde{\ell}+2) \\
    \nonumber &- \dfrac{1368-525\lambda}{2(12-5\lambda)}v^{\pi_k}_{\lambda}(N, \tilde{\ell}) + \left[\dfrac{19\lambda}{2(12-5\lambda)}+ \dfrac{57\lambda}{9-4\lambda}\right]v^{\pi_k}_{\lambda}(N, \tilde{\ell}+1) + \dfrac{200}{12-5\lambda}v^{\pi_k}_{\lambda}(N, \tilde{\ell}+2) \\
    \nonumber &= \dfrac{8\lambda}{9-4\lambda}\dfrac{12}{12-5\lambda}\left[\dfrac{\lambda}{8}v^{\pi_k}_{\lambda}(N, \tilde{\ell}) + \dfrac{25\lambda}{72}v^{\pi_k}_{\lambda}(N, \tilde{\ell}+1) + \dfrac{\lambda}{9}v^{\pi_k}_{\lambda}(N-2, \tilde{\ell}+1)\right] \\
    \nonumber &+ \left[\dfrac{55\lambda}{9-4\lambda} - \dfrac{20\lambda}{12-5\lambda}\right]\dfrac{12}{12-5\lambda}\left[\dfrac{\lambda}{8}v^{\pi_k}_{\lambda}(N, \tilde{\ell}+1) + \dfrac{25\lambda}{72}v^{\pi_k}_{\lambda}(N, \tilde{\ell}+2) + \dfrac{\lambda}{9}v^{\pi_k}_{\lambda}(N-2, \tilde{\ell}+2)\right] \\
    \nonumber &+\dfrac{64\lambda}{12-5\lambda} \dfrac{12}{12-5\lambda}\left[\dfrac{\lambda}{8}v^{\pi_k}_{\lambda}(N, \tilde{\ell}+2) + \dfrac{25\lambda}{72}v^{\pi_k}_{\lambda}(N, \tilde{\ell}+3) + \dfrac{\lambda}{9}v^{\pi_k}_{\lambda}(N-2, \tilde{\ell}+3)\right] \\
    \nonumber &- \dfrac{1368-525\lambda}{2(12-5\lambda)}v^{\pi_k}_{\lambda}(N, \tilde{\ell}) + \left[\dfrac{19\lambda}{2(12-5\lambda)}+ \dfrac{57\lambda}{9-4\lambda}\right]v^{\pi_k}_{\lambda}(N, \tilde{\ell}+1) + \dfrac{200}{12-5\lambda}v^{\pi_k}_{\lambda}(N, \tilde{\ell}+2) \\
    \nonumber &= \dfrac{4\lambda}{3(12-5\lambda)}\Bigg[\dfrac{8\lambda}{9-4\lambda}v^{\pi_k}_{\lambda}(N-2, \tilde{\ell}+1) + \left[\dfrac{55\lambda}{9-4\lambda} - \dfrac{20\lambda}{12-5\lambda}\right]v^{\pi_k}_{\lambda}(N-2, \tilde{\ell}+2) \\
    \nonumber &+ \dfrac{64\lambda}{12-5\lambda}v^{\pi_k}_{\lambda}(N-2, \tilde{\ell}+3) - \dfrac{1368-525\lambda}{2(12-5\lambda)}v^{\pi_k}_{\lambda}(N, \tilde{\ell}+1) \\
    &+ \left[\dfrac{19\lambda}{2(12-5\lambda)}+ \dfrac{57\lambda}{9-4\lambda}\right]v^{\pi_k}_{\lambda}(N, \tilde{\ell}+2) + \dfrac{200}{12-5\lambda}v^{\pi_k}_{\lambda}(N, \tilde{\ell}+3)\Bigg] \\
    \nonumber &- \dfrac{3(4104 - 3399 \lambda + 692\lambda^2)}{2(9-4\lambda)(12-5\lambda)}v^{\pi_k}_{\lambda}(N, \tilde{\ell}) + \dfrac{\lambda(8721 - 3067\lambda)}{6(9-4\lambda)(12-5\lambda)}v^{\pi_k}_{\lambda}(N, \tilde{\ell}+1) \\
    \nonumber & + \dfrac{10800 - 4800\lambda +919\lambda^2}{6(9-4\lambda)(12-5\lambda)}v^{\pi_k}_{\lambda}(N, \tilde{\ell}+2) - \dfrac{800\lambda(1-\lambda)}{3(12-5\lambda)^2}v^{\pi_k}_{\lambda}(N, \tilde{\ell}+3),
\end{align*}
for $k = 1$, $\lambda \in [\lambda_c, 1)$ and $k = 2$, $\lambda \in (0, \lambda_c]$. The induction hypothesis now implies that it suffices to show that
\begin{align}\label{exp18_cond2}
    &- \dfrac{3(4104 - 3399 \lambda + 692\lambda^2)}{2(9-4\lambda)(12-5\lambda)}v^{\pi_k}_{\lambda}(N, m) + \dfrac{\lambda(8721 - 3067\lambda)}{6(9-4\lambda)(12-5\lambda)}v^{\pi_k}_{\lambda}(N, m+1) \\
    \nonumber &+ \dfrac{10800 - 4800\lambda +919\lambda^2}{648 - 558\lambda + 120\lambda^2}v^{\pi_k}_{\lambda}(N, m+2) - \dfrac{800\lambda(1-\lambda)}{3(12-5\lambda)^2}v^{\pi_k}_{\lambda}(N, m+3) \geq 0,
\end{align}
for all $m = 2, \ldots, N-6$, for $k = 1$, $\lambda \in [\lambda_c, 1)$ and for $k = 2$, $\lambda \in (0, \lambda_c]$. We prove this statement by yet another induction argument. First of all, consider the case $k =1$, $\lambda \in [\lambda_c, 1)$. We start by verifying the correctness of expression (\ref{exp18_cond2}) for $m = N-6$. Now, assume that it is satisfied for $m = \tilde{m}+1$ for some $\tilde{m} = 2, \ldots, N-7$. Using expression (\ref{rec_4}) and this induction hypothesis, we now obtain
\begin{align*}
    &- \dfrac{3(4104 - 3399 \lambda + 692\lambda^2)}{2(9-4\lambda)(12-5\lambda)}v^{\pi_1}_{\lambda}(N, \tilde{m}) + \dfrac{\lambda(8721 - 3067\lambda)}{6(9-4\lambda)(12-5\lambda)}v^{\pi_1}_{\lambda}(N, \tilde{m}+1) \\
    \nonumber &+ \dfrac{10800 - 4800\lambda +919\lambda^2}{648 - 558\lambda + 120\lambda^2}v^{\pi_1}_{\lambda}(N, \tilde{m}+2) - \dfrac{800\lambda(1-\lambda)}{3(12-5\lambda)^2}v^{\pi_1}_{\lambda}(N, \tilde{m}+3) \\
    \nonumber &=\dfrac{2\lambda}{3-\lambda}\Bigg[- \dfrac{3(4104 - 3399 \lambda + 692\lambda^2)}{2(9-4\lambda)(12-5\lambda)}v^{\pi_1}_{\lambda}(N, \tilde{m}+1) + \dfrac{\lambda(8721 - 3067\lambda)}{6(9-4\lambda)(12-5\lambda)}v^{\pi_1}_{\lambda}(N, \tilde{m}+2) \\
    \nonumber &+ \dfrac{10800 - 4800\lambda +919\lambda^2}{648 - 558\lambda + 120\lambda^2}v^{\pi_1}_{\lambda}(N, \tilde{m}+3) - \dfrac{800\lambda(1-\lambda)}{3(12-5\lambda)^2}v^{\pi_1}_{\lambda}(N, \tilde{m}+4)\Bigg] \geq 0.
\end{align*}
Hence, expression (\ref{exp18_cond2}) holds for all $m = 2, \ldots, N-6$, for $k = 1$, $\lambda \in [\lambda_c, 1)$. \\
We proceed to consider the case $k = 2$, $\lambda \in (0, \lambda_c]$. Again using Table \ref{explicit_v_stars}, we verify the validity of expression (\ref{exp18_cond2}) for $m = N-6$ and $m = N-7$. Now, we suppose that it is true for all $m = \tilde{m}+1, \ldots, N-6$ for some $\tilde{m} = 2, \ldots, N-8$. By expression (\ref{rec_5}), we have
\begin{align*}
    &- \dfrac{3(4104 - 3399 \lambda + 692\lambda^2)}{2(9-4\lambda)(12-5\lambda)}v^{\pi_2}_{\lambda}(N, \tilde{m}) + \dfrac{\lambda(8721 - 3067\lambda)}{6(9-4\lambda)(12-5\lambda)}v^{\pi_2}_{\lambda}(N, \tilde{m}+1) \\
    \nonumber &+ \dfrac{10800 - 4800\lambda +919\lambda^2}{648 - 558\lambda + 120\lambda^2}v^{\pi_2}_{\lambda}(N, \tilde{m}+2) - \dfrac{800\lambda(1-\lambda)}{3(12-5\lambda)^2}v^{\pi_2}_{\lambda}(N, \tilde{m}+3) \\
    \nonumber &= - \dfrac{3(4104 - 3399 \lambda + 692\lambda^2)}{2(9-4\lambda)(12-5\lambda)}\left[\dfrac{7\lambda}{3(9-5\lambda)}v^{\pi_2}_{\lambda}(N, \tilde{m}+1) + \dfrac{5\lambda}{3(9-5\lambda)}v^{\pi_2}_{\lambda}(N, \tilde{m}+2)\right] \\
    \nonumber &+ \dfrac{\lambda(8721 - 3067\lambda)}{6(9-4\lambda)(12-5\lambda)}\left[\dfrac{7\lambda}{3(9-5\lambda)}v^{\pi_2}_{\lambda}(N, \tilde{m}+2) + \dfrac{5\lambda}{3(9-5\lambda)}v^{\pi_2}_{\lambda}(N, \tilde{m}+3)\right] \\
    \nonumber &+ \dfrac{10800 - 4800\lambda +919\lambda^2}{648 - 558\lambda + 120\lambda^2}\left[\dfrac{7\lambda}{3(9-5\lambda)}v^{\pi_2}_{\lambda}(N, \tilde{m}+3) + \dfrac{5\lambda}{3(9-5\lambda)}v^{\pi_2}_{\lambda}(N, \tilde{m}+4)\right] \\
    \nonumber &- \dfrac{800\lambda(1-\lambda)}{3(12-5\lambda)^2}\left[\dfrac{7\lambda}{3(9-5\lambda)}v^{\pi_2}_{\lambda}(N, \tilde{m}+4) + \dfrac{5\lambda}{3(9-5\lambda)}v^{\pi_2}_{\lambda}(N, \tilde{m}+5)\right] \\
    \nonumber &= \dfrac{7\lambda}{3(9-5\lambda)}\Bigg[- \dfrac{3(4104 - 3399 \lambda + 692\lambda^2)}{2(9-4\lambda)(12-5\lambda)}v^{\pi_2}_{\lambda}(N, \tilde{m}+1) + \dfrac{\lambda(8721 - 3067\lambda)}{6(9-4\lambda)(12-5\lambda)}v^{\pi_2}_{\lambda}(N, \tilde{m}+2) \\
    &+ \dfrac{10800 - 4800\lambda +919\lambda^2}{648 - 558\lambda + 120\lambda^2}v^{\pi_2}_{\lambda}(N, \tilde{m}+3) - \dfrac{800\lambda(1-\lambda)}{3(12-5\lambda)^2}v^{\pi_2}_{\lambda}(N, \tilde{m}+4)\Bigg] \\
    &+ \dfrac{5\lambda}{3(9-5\lambda)}\Bigg[- \dfrac{3(4104 - 3399 \lambda + 692\lambda^2)}{2(9-4\lambda)(12-5\lambda)}v^{\pi_2}_{\lambda}(N, \tilde{m}+2) + \dfrac{\lambda(8721 - 3067\lambda)}{6(9-4\lambda)(12-5\lambda)}v^{\pi_2}_{\lambda}(N, \tilde{m}+3) \\
    &+ \dfrac{10800 - 4800\lambda +919\lambda^2}{648 - 558\lambda + 120\lambda^2}v^{\pi_2}_{\lambda}(N, \tilde{m}+4) - \dfrac{800\lambda(1-\lambda)}{3(12-5\lambda)^2}v^{\pi_2}_{\lambda}(N, \tilde{m}+5)\Bigg],
\end{align*}
for all $\lambda \in (0, \lambda_c]$. Invoking the induction hypothesis, we can now conclude that 
\begin{align*}
    &- \dfrac{3(4104 - 3399 \lambda + 692\lambda^2)}{2(9-4\lambda)(12-5\lambda)}v^{\pi_2}_{\lambda}(N, \tilde{m}) + \dfrac{\lambda(8721 - 3067\lambda)}{6(9-4\lambda)(12-5\lambda)}v^{\pi_2}_{\lambda}(N, \tilde{m}+1) \\
    \nonumber &+ \dfrac{10800 - 4800\lambda +919\lambda^2}{648 - 558\lambda + 120\lambda^2}v^{\pi_2}_{\lambda}(N, \tilde{m}+2) - \dfrac{800\lambda(1-\lambda)}{3(12-5\lambda)^2}v^{\pi_2}_{\lambda}(N, \tilde{m}+3) \geq 0,
\end{align*}
for all $\lambda \in (0, \lambda_c]$. Thus, expression (\ref{exp18_cond2}) holds for all $m = 2, \ldots, N-6$ for $k = 2$, $\lambda \in [\lambda_c, 1)$ as well. Hence, expression (\ref{exp18_cond1}) is satisfied for all $\ell = 2, \ldots, N-5$, for $k =1$, $\lambda \in [\lambda_c, 1)$ and for $k = 2$, $\lambda \in (0, \lambda_c]$, which in turn implies the validity of inequality (\ref{18}) for all $j = 2, \ldots, N-4$, for $k =1$, $\lambda \in [\lambda_c, 1)$ and for $k = 2$, $\lambda \in (0, \lambda_c]$.

We finalize the proof by showing the validity of inequalities (\ref{19}) and (\ref{21}). The correctness of inequalities (\ref{20}) and (\ref{22}) then follow immediately from symmetry. In the proofs of inequalities (\ref{19}) and (\ref{21}), we make use of the more involved induction argument outlined in the main part of the proof of Theorem \ref{Optimal_policy}.

\paragraph*{Inequality (\ref{19})}

\begin{itemize}
    \item \textit{Induction base:} We show that inequality (\ref{19}) holds for states $(i, N-4)$, $(N-4, j)$, $i, j = 2, \ldots, N-4$, for $k =1$, $\lambda \in [\lambda_c, 1)$ and $k=2$, $\lambda \in (0, \lambda_c]$. To this end, we use an embedded induction argument over the length of the shortest side of the rectangle. By symmetry, it suffices to handle the case $(N-4, j)$, $j = 2, \ldots, N-4$. First of all, we verify the validity of the inequality for state $(N-4, N-4)$ using Table \ref{explicit_v_stars}. Now, suppose that it is true for state $(N-4, n+1)$ for some $n = 2, \ldots, N-5$, for $k =1$, $\lambda \in [\lambda_c, 1)$ and $k=2$, $\lambda \in (0, \lambda_c]$. By expression (\ref{rec_10}), we have
    \begin{align*}
        &v^{\pi_k}_{\lambda}(N-4, n) + v^{\pi_k}_{\lambda}(N-3, n) - 5v^{\pi_k}_{\lambda}(N-4, n+1) + 3v^{\pi_k}_{\lambda}(N-3, n+1) \\
        \nonumber &= \dfrac{\lambda}{9-4\lambda}\Big(v^{\pi_k}_{\lambda}(N-3, n) + v^{\pi_k}_{\lambda}(N-4, n+1) +3v^{\pi_k}_{\lambda}(N-3, n+1)\Big) \\
        \nonumber &+ \dfrac{\lambda}{9-4\lambda}\Big(v^{\pi_k}_{\lambda}(N-2, n) + v^{\pi_k}_{\lambda}(N-3, n+1) +3v^{\pi_k}_{\lambda}(N-2, n+1)\Big) \\
        \nonumber &- \dfrac{5\lambda}{9-4\lambda}\Big(v^{\pi_k}_{\lambda}(N-3, n+1) + v^{\pi_k}_{\lambda}(N-4, n+2) +3v^{\pi_k}_{\lambda}(N-3, n+2)\Big) \\
        \nonumber &+ \dfrac{3\lambda}{9-4\lambda}\Big(v^{\pi_k}_{\lambda}(N-2, n+1) + v^{\pi_k}_{\lambda}(N-3, n+2) +3v^{\pi_k}_{\lambda}(N-2, n+2)\Big),
    \end{align*}
    for $k =1$, $\lambda \in [\lambda_c, 1)$ and $k=2$, $\lambda \in (0, \lambda_c]$. Through rearranging, we obtain
    \begin{align*}
        &v^{\pi_k}_{\lambda}(N-4, n) + v^{\pi_k}_{\lambda}(N-3, n) - 5v^{\pi_k}_{\lambda}(N-4, n+1) + 3v^{\pi_k}_{\lambda}(N-3, n+1) \\
        \nonumber &=\dfrac{\lambda}{9-4\lambda}\Big(v^{\pi_k}_{\lambda}(N-3, n) + v^{\pi_k}_{\lambda}(N-2, n) - 5v^{\pi_k}_{\lambda}(N-3, n+1) + 3v^{\pi_k}_{\lambda}(N-2, n+1)\Big) \\
        \nonumber &+ \dfrac{\lambda}{9-4\lambda}\Big(v^{\pi_k}_{\lambda}(N-4, n+1) + v^{\pi_k}_{\lambda}(N-3, n+1) - 5v^{\pi_k}_{\lambda}(N-4, n+2) + 3v^{\pi_k}_{\lambda}(N-3, n+2)\Big) \\
        &+ \dfrac{3\lambda}{9-4\lambda}\Big(v^{\pi_k}_{\lambda}(N-3, n+1) + v^{\pi_k}_{\lambda}(N-2, n+1) - 5v^{\pi_k}_{\lambda}(N-3, n+2) + 3v^{\pi_k}_{\lambda}(N-2, n+2)\Big),
    \end{align*}
    for $k =1$, $\lambda \in [\lambda_c, 1)$ and $k=2$, $\lambda \in (0, \lambda_c]$. The induction hypothesis and inequality (\ref{15}) now yield,
    \begin{equation*}
        v^{\pi_k}_{\lambda}(N-4, n) + v^{\pi_k}_{\lambda}(N-3, n) - 5v^{\pi_k}_{\lambda}(N-4, n+1) + 3v^{\pi_k}_{\lambda}(N-3, n+1) >0,
    \end{equation*}
    for $k = 1$, $\lambda \in [\lambda_c, 1)$ and $k =2$, $\lambda \in (0, \lambda_c]$. Hence, inequality (\ref{19}) holds for all states $(N-4, j)$, $j = 2, \ldots, N-4$, for $k = 1$, $\lambda \in [\lambda_c, 1)$ and $k =2$, $\lambda \in (0, \lambda_c]$.

    \item \textit{Induction hypothesis}: Suppose that inequality (\ref{19}) holds for states $(i, n+1)$, $(n+1, i)$ for all $i = 2, \ldots, n+1$ for some $n = 2, \ldots, N-5$. We refer to this induction hypothesis as the \textit{global induction hypothesis}.

    \item \textit{Induction step}: First of all, we show that the induction hypothesis implies the validity of inequality (\ref{19}) for state $(n, n)$. Using expression (\ref{rec_10}), we obtain
    \begin{align*}
        &v^{\pi_k}_{\lambda}(n, n) + v^{\pi_k}_{\lambda}(n+1, n) - 5v^{\pi_k}_{\lambda}(n, n+1) + 3v^{\pi_k}_{\lambda}(n+1, n+1) \\
        \nonumber &= \dfrac{\lambda}{9-4\lambda}\Big(v^{\pi_k}_{\lambda}(n+1, n) + v^{\pi_k}_{\lambda}(n, n+1) + 3v^{\pi_k}_{\lambda}(n+1, n+1)\Big) \\
        \nonumber &+ \dfrac{\lambda}{9-4\lambda}\Big(v^{\pi_k}_{\lambda}(n+2, n) + v^{\pi_k}_{\lambda}(n+1, n+1) + 3v^{\pi_k}_{\lambda}(n+2, n+1)\Big)\\
        \nonumber &- \dfrac{5\lambda}{9-4\lambda}\Big(v^{\pi_k}_{\lambda}(n+1, n+1) + v^{\pi_k}_{\lambda}(n, n+2) + 3v^{\pi_k}_{\lambda}(n+1, n+2)\Big)\\
        \nonumber &+\dfrac{3\lambda}{9-4\lambda}\Big(v^{\pi_k}_{\lambda}(n+2, n+1) + v^{\pi_k}_{\lambda}(n+1, n+2) + 3v^{\pi_k}_{\lambda}(n+2, n+2)\Big)\\
        \nonumber &= \dfrac{\lambda}{9-4\lambda}\Big(v^{\pi_k}_{\lambda}(n+1, n) + v^{\pi_k}_{\lambda}(n+2, n) - 5v^{\pi_k}_{\lambda}(n+1, n+1) + 3v^{\pi_k}_{\lambda}(n+2, n+1)\Big) \\
        \nonumber &+ \dfrac{\lambda}{9-4\lambda}\Big(v^{\pi_k}_{\lambda}(n, n+1) + v^{\pi_k}_{\lambda}(n+1, n+1) - 5v^{\pi_k}_{\lambda}(n, n+2) + 3v^{\pi_k}_{\lambda}(n+1, n+2)\Big) \\
        \nonumber &+ \dfrac{3\lambda}{9-4\lambda}\Big(v^{\pi_k}_{\lambda}(n+1, n+1) + v^{\pi_k}_{\lambda}(n+2, n+1) - 5v^{\pi_k}_{\lambda}(n+1, n+2) + 3v^{\pi_k}_{\lambda}(n+2, n+2)\Big) > 0,
    \end{align*}
    for $k = 1$, $\lambda \in [\lambda_c, 1)$ and $\lambda \in (0, \lambda_c]$ by the induction hypothesis, establishing the correctness of inequality (\ref{19}) for state $(n, n)$.
    
    Now, in addition to the global induction hypothesis, we make the assumption that inequality (\ref{19}) holds for some state $(n, \ell + 1)$, where $\ell = 2, \ldots, n-1$. We show that, under the global and additional induction hypotheses, the validity of inequality (\ref{19}) for state $(n, \ell +1)$ carries over to state $(n, \ell)$. This implies the correctness of inequality (\ref{19}) for all states $(i, n)$, $(n, i)$, $i = 2, \ldots, n$ under the global induction hypothesis. Using expression (\ref{rec_10}), we obtain
    \begin{align*}
        &v^{\pi_k}_{\lambda}(n, \ell) + v^{\pi_k}_{\lambda}(n+1, \ell) -5v^{\pi_k}_{\lambda}(n, \ell + 1) + 3v^{\pi_k}_{\lambda}(n+1, \ell + 1) \\
        \nonumber &= \dfrac{\lambda}{9-4\lambda}\Big(v^{\pi_k}_{\lambda}(n+1, \ell) + v^{\pi_k}_{\lambda}(n, \ell+1) + 3v^{\pi_k}_{\lambda}(n+1, \ell + 1)\Big) \\
        \nonumber &+ \dfrac{\lambda}{9-4\lambda}\Big(v^{\pi_k}_{\lambda}(n+2, \ell) + v^{\pi_k}_{\lambda}(n+1, \ell +1) +3v^{\pi_k}_{\lambda}(n+2, \ell + 1)\Big) \\
        \nonumber &-\dfrac{5\lambda}{3(9-4\lambda)}\Big(v^{\pi_k}_{\lambda}(n+1, \ell+1) + v^{\pi_k}_{\lambda}(n, \ell+2) + 3v^{\pi_k}_{\lambda}(n+1, \ell + 2)\Big) \\
        \nonumber &+ \dfrac{3\lambda}{9-4\lambda}\Big(v^{\pi_k}_{\lambda}(n+2, \ell +1) + v^{\pi_k}_{\lambda}(n+1, \ell +2) +3v^{\pi_k}_{\lambda}(n+2, \ell + 2)\Big) \\
        \nonumber &= \dfrac{\lambda}{9-4\lambda}\Big(v^{\pi_k}_{\lambda}(n+1, \ell) + v^{\pi_k}_{\lambda}(n+2, \ell) -5v^{\pi_k}_{\lambda}(n+1, \ell + 1) + 3v^{\pi_k}_{\lambda}(n+2, \ell + 1)\Big) \\
        \nonumber &+ \dfrac{\lambda}{9-4\lambda}\Big(v^{\pi_k}_{\lambda}(n, \ell +1) + v^{\pi_k}_{\lambda}(n+1, \ell+1) -5v^{\pi_k}_{\lambda}(n, \ell + 2) + 3v^{\pi_k}_{\lambda}(n+1, \ell + 2)\Big) \\
        \nonumber &+ \dfrac{3\lambda}{9-4\lambda}\Big(v^{\pi_k}_{\lambda}(n+1, \ell +1) + v^{\pi_k}_{\lambda}(n+2, \ell+1) -5v^{\pi_k}_{\lambda}(n+1, \ell + 2) + 3v^{\pi_k}_{\lambda}(n+2, \ell + 2)\Big) > 0,
    \end{align*}
    by the global and additional induction hypotheses. It follows that inequality (\ref{19}) is satisfied for states $(n, i)$, $(i, n)$ for all $i = 2, \ldots, n$ under the global induction hypothesis. From the full induction argument, we can now conclude that inequality (\ref{19}) holds for all states $(i,j)$, $i, j = 2, \ldots, N-4$, for $k = 1$, $\lambda \in [\lambda_c, 1)$ and $k = 2$, $\lambda \in (0, \lambda_c]$.  
    \end{itemize}
    
    \paragraph*{Inequality (\ref{21})}
    \begin{itemize}
        \item \textit{Induction base:} We prove that inequality (\ref{21}) is valid for states $(i, N-4)$, $(N-4, j)$, $i, j = 2, \ldots, N-4$, for $k = 1$, $\lambda \in [\lambda_c, 1)$ and $k = 2$, $\lambda \in (0, \lambda_c]$, again by means of an embedded induction argument over the length of the shortest side of the rectangle. By symmetry, it is sufficient to prove the statement for states $(N-4, j)$, $j = 2, \ldots, N-4$. First, we use Table \ref{explicit_v_stars} to verify the correctness of inequality (\ref{21}) for state $(N-4, N-4)$. Now, assume that it holds for state $(N-4, n+1)$ for some $n = 2, \ldots, N-5$, for $k = 1$, $\lambda \in [\lambda_c, 1)$ and $k = 2$, $\lambda \in (0, \lambda_c]$. Expression (\ref{rec_10}) yields
        \begin{align*}
            & -3v^{\pi_k}_{\lambda}(N-4,n) + 3v^{\pi_k}_{\lambda}(N-3, n) -4v^{\pi_k}_{\lambda}(N-4, n+1) + 9v^{\pi_k}_{\lambda}(N-3, n+1) -5v^{\pi_k}_{\lambda}(N-4, n+2) \\
            \nonumber &=-\dfrac{3\lambda}{9-4\lambda}\Big(v^{\pi_k}_{\lambda}(N-3, n) + v^{\pi_k}_{\lambda}(N-4, n+1) + 3v^{\pi_k}_{\lambda}(N-3, n+1)\Big) \\
            \nonumber &+\dfrac{3\lambda}{9-4\lambda}\Big(v^{\pi_k}_{\lambda}(N-2, n) + v^{\pi_k}_{\lambda}(N-3, n+1) + 3v^{\pi_k}_{\lambda}(N-2, n+1)\Big) \\
            \nonumber &-\dfrac{4\lambda}{9-4\lambda}\Big(v^{\pi_k}_{\lambda}(N-3, n+1) + v^{\pi_k}_{\lambda}(N-4, n+2) + 3v^{\pi_k}_{\lambda}(N-3, n+2)\Big) \\
            \nonumber &+\dfrac{9\lambda}{9-4\lambda}\Big(v^{\pi_k}_{\lambda}(N-2, n+1) + v^{\pi_k}_{\lambda}(N-3, n+2) + 3v^{\pi_k}_{\lambda}(N-2, n+2)\Big) \\
            \nonumber &- \dfrac{5\lambda}{9-4\lambda}\Big(v^{\pi_k}_{\lambda}(N-3, n+2) + v^{\pi_k}_{\lambda}(N-4, n+3) + 3v^{\pi_k}_{\lambda}(N-3, n+3)\Big).           
        \end{align*}
        Through rearranging, we obtain
        \begin{align*}
            & -3v^{\pi_k}_{\lambda}(N-4,n) + 3v^{\pi_k}_{\lambda}(N-3, n) -4v^{\pi_k}_{\lambda}(N-4, n+1) + 9v^{\pi_k}_{\lambda}(N-3, n+1)  -5v^{\pi_k}_{\lambda}(N-4, n+2) \\
            \nonumber &= \dfrac{\lambda}{9-4\lambda}\Big(-3v^{\pi_k}_{\lambda}(N-3,n) + 3v^{\pi_k}_{\lambda}(N-2, n) -4v^{\pi_k}_{\lambda}(N-3, n+1) + 9v^{\pi_k}_{\lambda}(N-2, n+1) \\
            &-5v^{\pi_k}_{\lambda}(N-3, n+2)\Big) +\dfrac{\lambda}{9-4\lambda}\Big(-3v^{\pi_k}_{\lambda}(N-4,n+1) + 3v^{\pi_k}_{\lambda}(N-3, n+1) -4v^{\pi_k}_{\lambda}(N-4, n+2) \\
            &+ 9v^{\pi_k}_{\lambda}(N-3, n+2) -5v^{\pi_k}_{\lambda}(N-4, n+3)\Big) + \dfrac{3\lambda}{9-4\lambda}\Big(-3v^{\pi_k}_{\lambda}(N-3,n+1) + 3v^{\pi_k}_{\lambda}(N-2, n+1) \\
            &-4v^{\pi_k}_{\lambda}(N-3, n+2) + 9v^{\pi_k}_{\lambda}(N-2, n+2) -5v^{\pi_k}_{\lambda}(N-3, n+3)\Big).
        \end{align*}
        From the induction hypothesis and inequality (\ref{16}), it now follows that
        \begin{align*}
            & -3v^{\pi_k}_{\lambda}(N-4,n) + 3v^{\pi_k}_{\lambda}(N-3, n) -4v^{\pi_k}_{\lambda}(N-4, n+1) + 9v^{\pi_k}_{\lambda}(N-3, n+1) \\
            &-5v^{\pi_k}_{\lambda}(N-4, n+2) > 0.
        \end{align*}
        Thus, inequality (\ref{21}) holds for all states $(N-4, j)$, $j = 2, \ldots, N-4$ for $k = 1$, $\lambda \in [\lambda_c, 1)$ and $k = 2$, $\lambda \in (0, \lambda_c]$. 

        \item \textit{Induction hypothesis:} Suppose that inequality (\ref{21}) is valid for states $(i, n+1)$, $(n+1, j)$ for all $i, j = 2, \ldots, n+1$ for some $n = 2, \ldots, N-5.$ We refer to this induction hypothesis as the \textit{global induction hypothesis.}

        \item \textit{Induction step:} First of all, we show that the induction hypothesis implies the correctness of inequality (\ref{21}) for state $(n, n)$. By expression (\ref{rec_10}), we have
        \begin{align*}
            & -3v^{\pi_k}_{\lambda}(n,n) + 3v^{\pi_k}_{\lambda}(n+1, n) -4v^{\pi_k}_{\lambda}(n, n+1) + 9v^{\pi_k}_{\lambda}(n+1, n+1) -5v^{\pi_k}_{\lambda}(n, n+2) \\
            \nonumber &= -\dfrac{3\lambda}{9-4\lambda}\Big(v^{\pi_k}_{\lambda}(n+1, n) + v^{\pi_k}_{\lambda}(n, n+1) + 3v^{\pi_k}_{\lambda}(n+1, n+1)\Big) \\
            \nonumber &+ \dfrac{3\lambda}{9-4\lambda}\Big(v^{\pi_k}_{\lambda}(n+2, n) + v^{\pi_k}_{\lambda}(n+1, n+1) + 3v^{\pi_k}_{\lambda}(n+2, n+1)\Big) \\
            \nonumber &- \dfrac{4\lambda}{9-4\lambda}\Big(v^{\pi_k}_{\lambda}(n+1, n+1) + v^{\pi_k}_{\lambda}(n, n+2) + 3v^{\pi_k}_{\lambda}(n+1, n+2)\Big) \\
            \nonumber &+ \dfrac{9\lambda}{9-4\lambda}\Big(v^{\pi_k}_{\lambda}(n+2, n+1) + v^{\pi_k}_{\lambda}(n+1, n+2) + 3v^{\pi_k}_{\lambda}(n+2, n+2)\Big) \\
            \nonumber &-\dfrac{5\lambda}{9-4\lambda}\Big(v^{\pi_k}_{\lambda}(n+1, n+2) + v^{\pi_k}_{\lambda}(n, n+3) + 3v^{\pi_k}_{\lambda}(n+1, n+3)\Big).
        \end{align*}
        Rearranging yields
        \begin{align*}
            & -3v^{\pi_k}_{\lambda}(n,n) + 3v^{\pi_k}_{\lambda}(n+1, n) -4v^{\pi_k}_{\lambda}(n, n+1) + 9v^{\pi_k}_{\lambda}(n+1, n+1) -5v^{\pi_k}_{\lambda}(n, n+2) \\
            \nonumber &= \dfrac{\lambda}{9-4\lambda}\Big(-3v^{\pi_k}_{\lambda}(n+1,n) + 3v^{\pi_k}_{\lambda}(n+2, n) -4v^{\pi_k}_{\lambda}(n+1, n+1) + 9v^{\pi_k}_{\lambda}(n+2, n+1)\\
            &-5v^{\pi_k}_{\lambda}(n+1, n+2)\Big) + \dfrac{\lambda}{9-4\lambda}\Big(-3v^{\pi_k}_{\lambda}(n,n+1) + 3v^{\pi_k}_{\lambda}(n+1, n+1) \\
            &-4v^{\pi_k}_{\lambda}(n, n+2) + 9v^{\pi_k}_{\lambda}(n+1, n+2) -5v^{\pi_k}_{\lambda}(n, n+3)\Big) + \dfrac{3\lambda}{9-4\lambda}\Big(-3v^{\pi_k}_{\lambda}(n+1,n+1)\\
            &+ 3v^{\pi_k}_{\lambda}(n+2, n+1)-4v^{\pi_k}_{\lambda}(n+1, n+2) + 9v^{\pi_k}_{\lambda}(n+2, n+2) -5v^{\pi_k}_{\lambda}(n+1, n+3)\Big).
        \end{align*}
    Invoking the induction hypothesis three times, we now obtain
    \begin{equation*}
        -3v^{\pi_k}_{\lambda}(n,n) + 3v^{\pi_k}_{\lambda}(n+1, n) -4v^{\pi_k}_{\lambda}(n, n+1) + 9v^{\pi_k}_{\lambda}(n+1, n+1) -5v^{\pi_k}_{\lambda}(n, n+2) > 0,
    \end{equation*}
    for $k =1$, $\lambda \in [\lambda_c, 1)$ and $\lambda \in (0, \lambda_c]$. Hence, inequality (\ref{21}) is satisfied for state $(n, n)$. 
    
    We proceed to make the assumption that inequality (\ref{21}) holds for some state $(n, \ell +1)$, where $\ell = 2, \ldots, n-1$, in addition to the global induction hypothesis. We show that the global and additional induction hypotheses imply the validity of inequality (\ref{21}) for state $(n, \ell)$. Invoking expression (\ref{rec_10}), we obtain
    \begin{align*}
        & -3v^{\pi_k}_{\lambda}(n,\ell) + 3v^{\pi_k}_{\lambda}(n+1, \ell) -4v^{\pi_k}_{\lambda}(n, \ell+1) + 9v^{\pi_k}_{\lambda}(n+1, \ell+1) -5v^{\pi_k}_{\lambda}(n, \ell+2) \\
        \nonumber &= -\dfrac{3\lambda}{9-4\lambda}\Big(v^{\pi_k}_{\lambda}(n+1, \ell) + v^{\pi_k}_{\lambda}(n, \ell+1) + 3v^{\pi_k}_{\lambda}(n+1, \ell+1)\Big) \\
        \nonumber &+ \dfrac{3\lambda}{9-4\lambda}\Big(v^{\pi_k}_{\lambda}(n+2, \ell) + v^{\pi_k}_{\lambda}(n+1, \ell+1) + 3v^{\pi_k}_{\lambda}(n+2, \ell+1)\Big) \\
        \nonumber &- \dfrac{4\lambda}{9-4\lambda}\Big(v^{\pi_k}_{\lambda}(n+1, \ell+1) + v^{\pi_k}_{\lambda}(n, \ell+2) + 3v^{\pi_k}_{\lambda}(n+1, \ell+2)\Big) \\
        \nonumber &+ \dfrac{9\lambda}{9-4\lambda}\Big(v^{\pi_k}_{\lambda}(n+2, \ell+1) + v^{\pi_k}_{\lambda}(n+1, \ell+2) + 3v^{\pi_k}_{\lambda}(n+2, \ell+2)\Big) \\
        \nonumber &-\dfrac{5\lambda}{9-4\lambda}\Big(v^{\pi_k}_{\lambda}(n+1, \ell+2) + v^{\pi_k}_{\lambda}(n, \ell+3) + 3v^{\pi_k}_{\lambda}(n+1, \ell+3)\Big) \\
        \nonumber &= \dfrac{\lambda}{9-4\lambda}\Big(-3v^{\pi_k}_{\lambda}(n+1,\ell) + 3v^{\pi_k}_{\lambda}(n+2, \ell) -4v^{\pi_k}_{\lambda}(n+1, \ell+1) + 9v^{\pi_k}_{\lambda}(n+2, \ell+1) \\
        \nonumber &-5v^{\pi_k}_{\lambda}(n+1, \ell+2)\Big) + \dfrac{\lambda}{9-4\lambda}\Big(-3v^{\pi_k}_{\lambda}(n,\ell+1) + 3v^{\pi_k}_{\lambda}(n+1, \ell+1) -4v^{\pi_k}_{\lambda}(n, \ell+2)\\
        &+ 9v^{\pi_k}_{\lambda}(n+1, \ell+2) -5v^{\pi_k}_{\lambda}(n, \ell+3)\Big) + \dfrac{3\lambda}{9-4\lambda}\Big(-3v^{\pi_k}_{\lambda}(n+1,\ell+1)\\
        &+ 3v^{\pi_k}_{\lambda}(n+2, \ell+1) -4v^{\pi_k}_{\lambda}(n+1, \ell+2) + 9v^{\pi_k}_{\lambda}(n+2, \ell+2) -5v^{\pi_k}_{\lambda}(n+1, \ell+3)\Big). 
    \end{align*}
    It now follows from the global and additional induction hypotheses that
    \begin{equation*}
        -3v^{\pi_k}_{\lambda}(n,\ell) + 3v^{\pi_k}_{\lambda}(n+1, \ell) -4v^{\pi_k}_{\lambda}(n, \ell+1) + 9v^{\pi_k}_{\lambda}(n+1, \ell+1) -5v^{\pi_k}_{\lambda}(n, \ell+2) > 0.
    \end{equation*}
    Thus, inequality (\ref{21}) holds for all states $(n, i)$, $(i, n)$, $i = 2, \ldots, n$ under the global induction hypothesis. The full induction argument now implies that inequality (\ref{21}) is valid for all states $(i,j)$, $i,j = 2, \ldots, N-4$ for $k=1$, $\lambda \in [\lambda_c, 1)$ and $k = 2$, $\lambda \in (0, \lambda_c]$. 
    \end{itemize}

%\begin{longtable}{| p{.20\textwidth} | p{.20\textwidth} | p{.60\textwidth}}
\begin{longtable}{lll}
%\begin{table}[H]
%\begin{center}
%\begin{tabular}{ |c|l|l| } 
\toprule
$(i,j)$ & $k, \lambda$ & $v^{\pi_k}_{\lambda}(i,j)$\\
\midrule
\midrule
& & \\
$(N, N)$ & $k = 1, 2$, $\lambda \in (0,1)$ & $\dfrac{1}{1-\lambda}$ \\ 
& & \\\midrule
& & \\
$(N, N-2)$ & $k = 1, 2$, $\lambda \in (0,1)$ & $\dfrac{3\lambda}{(1-\lambda)(4-\lambda)}$ \\
& & \\ \midrule
& & \\
$(N, N-3)$ & $k = 1, 2$, $\lambda \in (0,1)$ & $\dfrac{3\lambda(19+3\lambda)}{2(4-\lambda)(1-\lambda)(18-7\lambda)}$ \\
& & \\ \midrule
& & \\
$(N, N-4)$ & $k = 1, 2$, $\lambda \in (0,1)$ & $\dfrac{3\lambda^2(19+3\lambda)}{(18-7\lambda)(1-\lambda)(3-\lambda)(4-\lambda)}$ \\ 
& & \\ \midrule
& & \\
$(N, N-5)$ & $k =1$, $[\lambda_c, 1)$  & $\dfrac{6\lambda^3(19+3\lambda)}{(18-7\lambda)(1-\lambda)(3-\lambda)^2(4-\lambda)}$  \\
& & \\
& $k = 2$, $(0, \lambda_c]$ & $\dfrac{3\lambda^2(5+3\lambda)(19+3\lambda)}{2(4-\lambda)(3-\lambda)(1-\lambda)(9-5\lambda)(18-7\lambda)}$ \\
& & \\ \midrule
& & \\
$(N, N-6)$ & $k = 1$, $[\lambda_c, 1)$ & $\dfrac{12\lambda^4(19+3\lambda)}{(18-7\lambda)(1-\lambda)(3-\lambda)^3(4-\lambda)}$  \\
& & \\
 & $k = 2$, $(0, \lambda_c]$ & $\dfrac{\lambda^3(19+3\lambda)(125-29\lambda)}{2(9-5\lambda)^2(4-\lambda)(3-\lambda)(1-\lambda)(18-7\lambda)}$ \\
 & & \\ \midrule
 & & \\
$(N, N-7)$ & $k = 1$, $[\lambda_c, 1)$ & $\dfrac{24\lambda^5(19+3\lambda)}{(18-7\lambda)(1-\lambda)(3-\lambda)^4(4-\lambda)}$ \\
& & \\
 & $k = 2$, $(0, \lambda_c]$ & $\dfrac{\lambda^3(19+3\lambda)(675+905\lambda-428\lambda^2)}{6(4-\lambda)(3-\lambda)(1-\lambda)(9-5\lambda)^3(18-7\lambda)}$  \\
 & & \\ \midrule
 & & \\
 $(N, N-8)$ & $k = 1$, $[\lambda_c, 1)$ &  $\dfrac{48\lambda^6(19+3\lambda)}{(18-7\lambda)(1-\lambda)(3-\lambda)^5(4-\lambda)}$ \\
 & & \\
 & $k=2$, $(0, \lambda_c]$ & $\dfrac{\lambda^4(19+3\lambda)(21600-6955\lambda - 821\lambda^2)}{18(9-5\lambda)^4(4-\lambda)(3-\lambda)(1-\lambda)(18-7\lambda)}$ \\
 & & \\ \midrule
& & \\
$(N-2, N-2)$ & $k = 1, 2$, $\lambda \in (0,1)$ & $\dfrac{\lambda(7\lambda + 26)}{(4-\lambda)(1-\lambda)(18-7\lambda)}$ \\
& & \\ \midrule
& & \\
$(N-2, N-3)$ & $k = 1, 2$, $\lambda \in (0,1)$ & $\dfrac{\lambda^2(971-245\lambda)}{2(18-7\lambda)^2(1-\lambda)(4-\lambda)}$ \\ 
& & \\ \midrule
& & \\
$(N-2, N-4)$ & $k = 1, 2$, $\lambda \in (0,1)$ & $\dfrac{\lambda^2(2401\lambda^3 - 16714\lambda^2 - 1653\lambda + 76950)}{12(18-7\lambda)^2(4-\lambda)(3-\lambda)(1-\lambda)(12-5\lambda)}$  \\
& & \\ \midrule
& & \\
$(N-2, N-5)$ & $k = 1$, $[\lambda_c, 1)$ & $\dfrac{\lambda^3(3231900 - 2216520 \lambda + 167001 \lambda^2 + 109700 \lambda^3 - 11417 \lambda^4)}{18(18-7\lambda)^2(12-5\lambda)^2(4-\lambda)(3-\lambda)^2(1-\lambda)}$  \\
& & \\
 & $k = 2$, $(0, \lambda_c]$ & $\dfrac{\lambda^3(21053520 -21676626\lambda + 5403027 \lambda^2 + 601490 \lambda^3 - 258755 \lambda^4)}{36(18-7\lambda)^2(12-5\lambda)^2(4-\lambda)(3-\lambda)(1-\lambda)(9-5\lambda)}$ \\
 & & \\ \midrule
 & & \\
$(N-2, N-6)$ & $k = 1$, $[\lambda_c, 1)$ & $\dfrac{\lambda^4}{27(18-7\lambda)^2(4-\lambda)(3-\lambda)^3(1-\lambda)(12-5\lambda)^3}\Big(119118600 $ \\
& &$- 135213192\lambda + 45728820\lambda^2 - 1408311\lambda^3 - 1454032\lambda^4$\\
& & $+ 122059\lambda^5\Big)$ \\
& & \\ 
& $k = 2$, $\lambda \in (0, \lambda_c]$ &  $\dfrac{\lambda^3}{108(18-7\lambda)^2(9-5\lambda)^2(4-\lambda)(3-\lambda)(1-\lambda)(12-5\lambda)^3}$\\
& & $\cdot \Big(149508000 - 269088480 \lambda -1961414838 \lambda^2 + 1470626901 \lambda^3$\\
& & $- 283284744\lambda^4$\\
& & $- 34167685\lambda^5 + 11723950 \lambda^6 \Big) $ \\
& & \\ \midrule 
& & \\
$(N-3, N-3)$ & $k = 1, 2$, $\lambda \in (0,1)$ & $\dfrac{\lambda^2(29241+8296\lambda - 5593\lambda^2)}{8(4-\lambda)(1-\lambda)(18-7\lambda)^3}$  \\ 
& & \\\midrule
& & \\
$(N-3, N-4)$ & $k = 1, 2$, $\lambda \in (0,1)$ & $\dfrac{\lambda^3}{24(4-\lambda)(3-\lambda)(1-\lambda)(9-4\lambda)(12-5\lambda)(18-7\lambda)^3}$\\
& & $\cdot \Big(28579716 - 34122249\lambda + 14845917\lambda^2 - 2786335\lambda^3$\\
& & $+ 191191\lambda^4\Big)$\\
& & \\ \midrule
& & \\
$(N-3, N-5)$ & $k = 1$, $[\lambda_c, 1)$ & $\dfrac{\lambda^3}{72(12-5\lambda)^2(9-4\lambda)^2(4-\lambda)(3-\lambda)^2(1-\lambda)(18-7\lambda)^3}$\\
& & $\cdot \Big(8077903200 - 7782658200 \lambda - 3094709976 \lambda^2 + 7348557735 \lambda^3$\\
& & $- 4088759886\lambda^4 + 1083552000\lambda^5 - 142794154\lambda^6 + 7639681\lambda^7\Big)$ \\
& & \\
& $k = 2$, $(0, \lambda_c]$ & $\dfrac{\lambda^3}{72(12-5\lambda)^2(9-4\lambda)^2(4-\lambda)(3-\lambda)(1-\lambda)(9-5\lambda)(18-7\lambda)^3} $ \\
& & $\cdot \Big(24233709600 - 28194716520\lambda - 7657684272\lambda^2 + 27365238693\lambda^3$\\
& & $- 17216824914\lambda^4 + 4906807638\lambda^5 - 649171070\lambda^6 + 30101645\lambda^7\Big)$ \\
& & \\ \midrule
& & \\
$(N-4, N-4)$ & $k = 1, 2$, $\lambda \in (0,1)$ & $\dfrac{\lambda^3}{24(9-4\lambda)^2(4-\lambda)(3-\lambda)(1-\lambda)(12-5\lambda)(18-7\lambda)^3}\Big(85266756$ \\
& & $- 20495835 \lambda - 73183797 \lambda^2 + 55332465 \lambda^3 - 14767511 \lambda^4$\\
& & $+ 1389122 \lambda^5\Big)$ \\
& & \\
\bottomrule
%\end{tabular}
\caption{Some explicit expressions for the optimal value function.}
\label{explicit_v_stars}
%\end{center}
%\end{table}
\end{longtable}

\end{document}